\documentclass[journal]{IEEEtran}
\usepackage{easyReview}
\usepackage{amssymb}
\usepackage{amsmath}
\usepackage{amsfonts}
\usepackage{graphicx}
\usepackage{color}
\usepackage{caption}
\usepackage{multirow}
\usepackage{diagbox}
\usepackage{lineno}
\usepackage{subcaption}
\usepackage{mathrsfs}
\usepackage{amscd}
\usepackage{url}
\usepackage{enumitem}
\usepackage{booktabs}
\usepackage{bm}
\newtheorem{theorem}{Theorem}[section]

\newtheorem{definition}{Definition}[section]

\newtheorem{lemma}{Lemma}[section]

\newtheorem{remark}{Remark}[section]

\newenvironment{proof}[1][Proof]{\noindent\textbf{#1.} }{\ \hfill\rule{0.3em}{0.5em}}

\newcommand*{\R}{\mathbb{R}}

\def\beq#1{\begin{equation}\label{#1}}
	\def\eeq{\end{equation}}
\def\bep{\begin{proof}}
	
	\def\ep{\end{proof}}
\def\reff#1{(\ref{#1})}

\def\M{\mathcal M}

\def\A{{\mathcal A}}
\def\B{\mathcal B}

\def\ignore#1{}
\def\X{\mathcal X}
\def\Y{\mathcal Y}
\def\G{\mathcal G}

\def\F{\mathcal F}

\def\U{\mathcal U}
\def\V{\mathcal V}
\def\P{\mathcal P}
\def\Q{\mathcal Q}
\def\rank{\operatorname {rank}}
\def\bR{\mathbb R}

\newcommand{\re}{\mathbb{R}}

\input epsf

\ifCLASSINFOpdf
\else
\fi
\begin{document}

\title{Tensor factorization based method for low rank matrix completion and its application on tensor completion}
%
%
%

\author{\IEEEauthorblockN{Quan~Yu,
		Xinzhen~Zhang$ ^{*} $\thanks{* Corresponding author.}}
\thanks{Quan Yu and Xinzhen Zhang are with School of Mathematics, Tianjin University, Tianjin 300354, P.R. China. (e-mail:quanyu@tju.edu.cn; xzzhang@tju.edu.cn). This work was supported by NSFC under Grant 11871369 and the Tianjin Research Innovation Project for Postgraduate Students under Grant 2020YJSS140. }
}
%
%

\markboth{}
{Shell \MakeLowercase{\textit{et al.}}: Bare Demo of IEEEtran.cls for IEEE Journals}
%



\maketitle

\begin{abstract}Low rank matrix and tensor completion problems are to recover the
	incomplete two and higher order data by using their
	low rank structures. The  essential problem in the matrix and tensor completion problems  is how to improve the efficiency.
To this end, we first establish the relationship between matrix rank and  tensor tubal rank, and then reformulate matrix completion problem as a tensor completion problem. For the reformulated tensor completion problem, we adopt a two-stage strategy based on tensor factorization algorithm. In this way,   a matrix completion problem of big size can be  solved via  some matrix computations of smaller sizes.
	For a third order tensor completion problem, to fully exploit the low rank structures, we introduce the double tubal rank which combines the tubal rank and the rank of the mode-3 unfolding matrix. For the mode-3 unfolding matrix rank,  we follow the idea of matrix completion.  Based on this, we establish a novel model and modify the tensor factorization based algorithm for third order tensor completion. Extensive numerical experiments demonstrate that the	proposed methods outperform state-of-the-art methods in terms of both accuracy and  running time.
\end{abstract}


%
\IEEEpeerreviewmaketitle

\section{Introduction}
\IEEEPARstart{M}{atrix} and tensor completion have received much attention in recent years, which have many applications, such as in hyperspectral data recovery \cite{GRY11}, image/video inpainting \cite{LMWY13,SZ17,YZ21,YZH20,YZM20,ZHJ19}, image classification \cite{CdlTCB15,LLTX15} and high dynamic range (HDR) imaging \cite{LLM14,OLTK15,TT12}. In general, such matrix and tensor data have low rank structures. Hence the problems are modeled as the rank minimization problems. Unfortunately, the rank minimization problem is NP-hard in general due to the combinational nature of the function $ \rank(\cdot) $ even for matrix rank.

Nuclear norm is known to be the tightest convex relaxation of matrix rank function \cite{RFP10}. Hence the matrix completion problem is relaxed as a nuclear norm minimization with efficient numerical methods \cite{CCS10,LST12,MGC11,MHT10,TY10,MZXB21}. But these nuclear norm minimization methods require computing matrix singular value decomposition (SVD), which become increasingly expensive with the increasing sizes of the underlying matrices. To cut down the computational cost, low rank matrix factorization methods have been proposed in \cite{HWL20,KMO10,RR13,ZS18,ZLS12}. However,  matrix decomposition methods also need expensive computation for large scale matrix data.

As a higher order generalization of matrix completion,  tensor completion has attracted much more attention recently \cite{BZNC16,BPTD17,JNZH20,Zha19,HZW21,JHZ17}. Compared to matrix rank, there are various definitions for tensor rank, including CANDECOMP/PARAFAC (CP)
rank \cite{Hit27}, Tucker rank \cite{Tuc66}, TT rank \cite{Ose11}, triple rank \cite{QCBZ21} and tubal rank \cite{KBHH13}.  Since it is generally NP-hard to compute the CP rank \cite{HL13}, it is hard to apply CP rank to the tensor completion problem. Although the TT rank can be computed by TT singular value decomposition, it always has a fixed pattern, which might not be the optimum for specific data tensor \cite{ZZX16}. The Tucker rank is defined on the rank of unfolding matrices, which are of big sizes. On the other hand, unfolding a tensor as a matrix would destroy the original multi-way structure of the data, leading to vital information loss and degrading performance efficiency \cite{MSL13,MHWG14}. Recently, tubal rank
becomes more and more popular since the low tubal rank tensor completion can be solved via updating matrices of smaller sizes at each iteration \cite{ZLLZ18}.
However, tubal rank is defined on the third mode, which ignores the low rank  structures on the other two modes \cite{YZH20}. 
To exploit the low rank structures, 
 \cite{ZHZ20} and \cite{ZHZ20a} proposed 3-tubal rank and tensor fibered rank, respectively,  which considered the three modes at the same time. Though this type of rank reveals more low rank structures of the tensor, the low rank structures they considered overlapped (see Lemma \ref{lem:wstnn}), so redundant running time is generated.

Based on these analyses, in this paper, we first propose a novel model for low rank matrix completion  problem. 
For a large scale matrix, we reshape it  as a third order tensor. Then we establish the relationship between matrix rank and tubal rank of the reshaped tensor. Based on this relationship, we reformulate the matrix completion problem as a third order tensor completion problem. Then we  propose a two-stage  tensor factorization based algorithm
 to the reformulated tensor completion problem. 
By this way,  a matrix completion problem of big size can be dealt with by computing matrix factorization of smaller sizes, which drastically reduces the consumed time. 

For the tensor completion problem, we consider the tubal rank  and the mode-3 unfolding matrix rank together for fully exploiting the low rank structures of the tensor. 
For the mode-3 unfolding matrix rank, we adopt the strategy of matrix completion problems. Thus, we introduce a new tensor rank, named double tubal rank. See the definition of tensor double tubal rank in \eqref{ttl}. 
Based on these, we modify the proposed tensor factorization based algorithm  for the tensor completion based on double tubal rank. 

In summary, our main contributions include:
\begin{itemize}
	\item[(1)] We reformulate the matrix completion problem as a third order tensor completion problem. Then we propose a  tensor factorization based algorithm. 
In this way, a big matrix completion problem can be solved by computing some smaller matrices, which greatly improves the efficiency of matrix completion problems.

	\item[(2)]  For  a third order tensor, we introduce the tensor double tubal rank. Compared with tubal rank, 3-tubal rank \cite{ZHZ20} and tensor fibered rank \cite{ZHZ20a}, double tubal rank can fully exploit the  low rank structures  without redundancy. 
	Based on the introduced double tubal rank, we  modify the proposed   tensor factorization based algorithm. 
	
\item[(3)] In the proposed algorithms, we adopt  the two-stage  strategy, in which a good initial point is generated in the first stage and the  convergence is accelerated in the second stage.
	
	\item[(4)] The proposed algorithms converge to  KKT points.   Extensive numerical experiments demonstrate  the outperforms of our proposed algorithms over the other compared algorithms.
\end{itemize}
The outline of this paper is given as follows. We recall the basic notations on tensor  in Section 2. 
In Section 3, we establish the relation between matrix rank and tubal rank of the reshaped tensor, and then reformulate the matrix completion problem as a tensor completion problem. For the reformulated tensor completion problem, a two-stage tensor factorization based algorithm is proposed. Section 4 introduces double tubal rank and then presents a new model for low rank tensor completion. For the presented model, we modify the  two-stage  
tensor factorization based algorithm.  Extensive simulation results are reported to demonstrate the validity of our proposed algorithms in Section 5.

\section{NOTATIONS AND PRELIMINARIES}
This section recalls some basic knowledge on tensors. We first give the basic notations and then present the tubal rank, 3-tubal rank (tensor fibered rank), and Tucker rank. We state them here in detail for the readers' convenience.

\subsection{Notations}
For a positive integer $n$,
$[n]:=\{1,2,\ldots, n\}$. Scalars, vectors and matrices are denoted as lowercase letters ($a,b,c,\ldots$), boldface lowercase letters ($\bm{a} ,\bm{b},\bm{c},\ldots$) and uppercase letters ($A,B,C,\ldots $), respectively.
Third order tensors are denoted as ($\mathcal{A},\mathcal{B},\mathcal{C},\ldots$). For a third order tensor $\mathcal{A} \in \R^{n_1\times n_2\times n_3}$, we use the Matlab notations $  \mathcal{A}(:, :, k)  $ to denote its $ k $-th frontal slice, denoted by $ A^{(k)} $ for all $k\in { [n_3]}$. The inner product of two tensors  $ \mathcal{A},\,\mathcal{B} \in {\R^{{n_1} \times {n_2} \times {n_3}}}$ is the sum of products of their entries, i.e.
$$\left\langle {\mathcal{A},\mathcal{B}} \right\rangle  = \sum\limits_{i = 1}^{{n_1}} {\sum\limits_{j = 1}^{{n_2}} {\sum\limits_{k = 1}^{{n_3}} {{\mathcal{A}_{ijk}}{\mathcal{B}_{ijk}}} } }. $$
The Frobenius norm is ${\left\| \mathcal{A} \right\|_F} = \sqrt {\left\langle {\mathcal{A},\mathcal{A}} \right\rangle } $.
For a matrix $A$, $ A^* $ and $ A^\dagger $ represent the conjugate transpose and the pseudo-inverse of $ A $, respectively.

\subsection{$T$-product, tubal rank and 3-tubal rank (tensor fibered rank)}
Discrete Fourier Transformation
(DFT) plays a key role in tensor-tensor product (t-product). For $\A\in \mathbb{R}^{n_1 \times n_2 \times n_3}$, let ${{\bar \A}} \in {{\mathbb C}^{{n_1} \times {n_2} \times {n_3}}}$ be the result of
Discrete Fourier transformation (DFT) of ${{ \A}} \in {{\mathbb R}^{{n_1} \times {n_2} \times {n_3}}}$
along the 3rd dimension. Specifically, let  $F_{n_3}=[f_1,\dots, f_{n_3}]\in \mathbb C^{n_3\times n_3}$, where
$$f_i=\left[ \omega^{0\times (i-1)}; \omega^{1\times (i-1)};\dots; \omega^{(n_3-1)\times (i-1)}\right] \in \mathbb C^{n_3}$$
with $\omega=e^{-\frac{2\pi \mathfrak{b}}{n_3}}$ and $\mathfrak{b}=\sqrt{-1}$. Then $ \bar \A(i,j,:)=F_{n_3}\A(i,j,:) $,
which can be computed by Matlab command ``$\bar \A=fft(\A,[\; ],3)$''. Furthermore, $\A$ can be computed by $\bar \A$ with the inverse DFT $ \A=ifft({\bar \A},[\; ],3) $.

\begin{lemma}\label{lem:v}\cite{RR04}
	Given any real vector $\bm{v} \in \mathbb{R}^{n_3}$, the associated $\bar{\bm{v}}=F_{n_3} \bm{v} \in \mathbb{C}^{n_3}$ satisfies
	$$
	\bar{v}_{1} \in \mathbb{R} \text { and } \operatorname{conj} \left(\bar{v}_{i}\right)=\bar{v}_{n_3-i+2},\; i=2, \ldots,\left\lfloor\frac{n_3+1}{2}\right\rfloor.
	$$
\end{lemma}
By using Lemma \ref{lem:v}, the frontal slices of $ \bar\A $ have the following properties:
\begin{equation}\label{conj}
	\left\{\begin{array}{l}
		\bar{{A}}^{(1)} \in \mathbb{R}^{n_{1} \times n_{2}}, \\
		\operatorname{conj} \left(\bar{{A}}^{(i)}\right)=\bar{{A}}^{\left(n_{3}-i+2\right)},\; i=2, \ldots,\left\lfloor\frac{n_{3}+1}{2}\right\rfloor.
	\end{array}\right.	
\end{equation}
For $\A\in \mathbb{R}^{n_1\times n_2\times n_3}$, we define matrix ${{\bar A}} \in {{\mathbb C}^{{n_1}{n_3} \times {n_2}{n_3}}}$ as
\begin{equation}\label{bdiag}
	{{\bar A}} = bdia{g}(\bar {{\mathcal{A}}} ) \hfill \\
	= \left[ {\begin{array}{*{20}{c}}
			{\bar A^{(1)}}&{}&{}&{} \\
			{}&{\bar A^{(2)}}&{}&{} \\
			{}&{}& \ddots &{} \\
			{}&{}&{}&{\bar A^{({n_3})}}
	\end{array}} \right]. \end{equation}
Here, $ bdiag(\cdot) $ is an operator which maps the tensor $ {{\bar {\mathcal A}}} $ to the block diagonal matrix $ \bar A$. The block circulant matrix $bcirc({\A}) \in {{\mathbb R}^{{n_1}{n_3} \times {n_2}{n_3}}}$ of $\A$ is defined  as
$$bcirc({\A}) = \left[ {\begin{array}{*{20}{c}}
		{A^{(1)}}&{A^{(n_3)}}& \cdots &{A^{(2)}}\\
		{A^{(2)}}&{A^{(1)}}& \cdots &{A^{(3)}}\\
		\vdots & \vdots & \ddots & \vdots \\
		{A^{(n_3)}}&{A^{({n_3} - 1)}}& \cdots &{A^{(1)}}
\end{array}} \right].$$

Based on these notations, the $T$-product is presented as follows.
\begin{definition}\label{def:T-pro}\textbf{($T$-product)} \cite{KM11}
	For $\A\in \mathbb{R}^{n_1\times r\times n_3}$ and $\B\in \mathbb R^{r\times n_2\times n_3}$, define
	$$\A\ast\B:=fold\left(bcirc(\A)\ \cdot unfold(\B)\right) \in \mathbb{R}^{n_1\times n_2\times n_3}.$$
	Here
	$$ unfold(\B) = \left[B^{(1)};B^{(2)}; \ldots ;B^{(n_3)}\right],$$
	and its inverse operator ``fold" is defined by $$fold(unfold(\B)) = \B.$$
\end{definition}
Tensor multi-rank and tubal rank are now introduced.
\begin{definition}\label{definition2.6}{\bfseries (Tensor multi-rank and tubal rank)} \cite{KBHH13}
	For tensor $\A \in {{\mathbb R}^{{n_1} \times {n_2} \times {n_3}}}$, let $r_k=\rank\left(\bar A^{(k)}\right)$ for all $k\in {[n_3]}$.
	Then multi-rank of $\A$ is defined as $\rank_{m}(\A)=(r_1,\ldots,r_{n_3})$. The tensor tubal
	rank is defined as $ \rank_t(\A)=\max\left\lbrace r_k|k\in[n_3]\right\rbrace  $.
\end{definition}

Then, we introduce 3-tubal rank (tensor fibered rank).
\begin{definition}{\bfseries (3-tubal rank/tensor fibered rank)} \cite{ZHZ20,ZHZ20a}
	For tensor $\A \in {{\mathbb R}^{{n_1} \times {n_2} \times {n_3}}}$, its 3-tubal rank (tensor fibered rank) as follows:
	$$
	3\text{-}\rank_t\left(\A\right)=\left(\rank_t\left(\A\right), \rank_t\left(\A_{(13)}\right), \rank_t\left(\A_{(23)}\right) \right) ,
	$$
	where $\mathcal{A}\left(i,j,k\right)=\mathcal{A}_{(13)}\left(i, k, j\right)=\mathcal{A}_{(23)}\left(j,k,i\right)$.
\end{definition}

Finally, we offer a lemma that will be utilized to simplify models and do theoretical analysis.
\begin{lemma}\label{lem:equ} \cite{KM11}
	Suppose that $\A,\, \B$ are tensors such that $\F:=\A\ast \B$ is well defined as in Definition \ref{def:T-pro}. Let $ \bar A,\bar B, \bar F $ be the block diagonal matrices defined as in \eqref{bdiag}. Then
	\begin{enumerate}
		\item [(1).]$\left\| \A \right\|_F^2 = \frac{1}{n_3}\left\| {{{\bar A}}} \right\|_F^2$;
		\item [(2).]${\mathcal F} = \A{ \ast }\B$ and $\bar F= {\bar A}{\bar B}$ are equivalent.
	\end{enumerate}		
\end{lemma}

\begin{lemma}\label{lem:f}\cite{ZLLZ18}
	Suppose that $\A\in\R^{n_1\times r\times n_3}$ and $\B\in\R^{r\times n_2\times n_3}$. Then	$ \rank_t\left(\A\ast\B\right)\le\min\left\lbrace\rank_t\left(\A\right),\rank_t\left(\B\right)\right\rbrace$.
\end{lemma}

\subsection{Tucker rank}
In this subsection, we are ready to present some notations on Tucker rank decomposition. More details can be found in Kolda and Bader's review on tensor decompositions \cite{KB09}.

The mode-$s$ unfolding $A_{(s)}$ of tensor $\mathcal{A}\in\R^{n_1\times n_2\times n_3}$ is a matrix in $\mathbb R^{n_s\times N_s}$ with its $(i, j)$-th element being $\A_{i_1\dots i_{s-1}ii_{s+1}\dots i_3}$, where
$j=1+\sum_{k\neq s} (i_k-1)\bar n_k$, $\bar n_k=\prod_{l<s} n_l
$ and $N_s=\prod_{k\neq s} n_k$. The unfolding matrix can be obtained by ``tens2mat($ \A,s $)'' in Matlab. The opposite operation ``$ fol{d_s} $" is defined as $ fol{d_s}({A_{(s)}}): = \A $.

Based on the definition of mode-$s$ unfolding matrix, the Tucker rank of tensor is defined as follows.
\begin{definition}\label{TCrankdef}
	For a tensor $\mathcal{A}\in\mathbb R^{n_1\times n_2\times n_3}$, let $A_{(i)}\in \mathbb R^{n_i\times N_i}$ be the mode-i unfolding matrix. The Tucker rank of $\A$ is
	\[\rank_{tc}(\mathcal{A})=\left(\rank (A_{(1)}),\rank(A_{(2)}),\rank(A_{(3)})\right).\]
\end{definition}
Next, we recall the definition of $k$-mode product.
\begin{definition}\label{modprod}
	For a tensor $\mathcal{A}\in\mathbb R^{n_1\times n_2\times n_3}$ and a matrix
	$B\in\mathbb R^{J_k\times n_k}$, the mode-$ k $ product of $\mathcal{A}$ with $B$ is a tensor of $n_1\times\ldots\times n_{k-1}\times J_k\times n_{k+1}\times\ldots \times n_3$ with its entries
	$$(\mathcal{A}\times_k B)_{i_1i_2i_3}=\sum\limits_{j_k=1}^{n_i} \A_{i_1i_2\ldots i_{k-1}j_ki_{k+1}\ldots i_3} B_{i_kj_k}.$$
\end{definition}
Easy to find that, for suitable matrices $B^1$ and $B^2$, it holds for  $$\mathcal{T}\times_i B^1\times_i  B^2=\mathcal{T}\times_i\left(B^2B^1\right).$$
Based on these notations, we are ready to present an equivalent definition of Tucker decomposition of tensor as follows.
\begin{definition}\label{TCdecom}
	Suppose that
	\begin{equation}\label{Tuckde}
		\mathcal{A}=\mathcal{G}\times_1 U^1\times_2U^2\times_3 U^3,	
	\end{equation}
	where $\mathcal{G}\in\mathbb R^{r_1\times r_2\times r_3}$, orthogonal matrix $U^i\in \mathbb R^{n_i\times r_i}$ and $r_i=\rank \left( A_{(i)}\right) $ for all $i\in[3]$.  Such $\mathcal{G}$ is called the core tensor and (\ref{Tuckde}) is called a Tucker rank decomposition of $\mathcal{A}$.
\end{definition}

\section{Matrix completion}\label{smc}
Given a partially observed matrix $ M\in \bR^{n_1\times h} $,  low rank matrix completion problem can be formulated as a constrained rank minimization problem, that is,
\begin{equation}\label{mc}
	\min_{X\in \bR^{n_1\times h}}~\operatorname{rank}(X), \quad \mbox{\rm s.t.} \quad  P_{\tilde\Omega}(X-M)=0,	
\end{equation}
where $ \tilde\Omega $ is the index subset of observed entries of matrix, $  P_{\tilde\Omega}(\cdot) $ is a projection operator
that keeps the entries of matrix in $ \tilde\Omega $ and makes other entries zero.
When $ n_1 $ and $ h $ are very large, the  required cost to recover matrix $ X $ will be very expensive.  To lower the cost, we reshape the matrix as a third order tensor as follows.
 For a given integer $n_2$, we add a zero matrix $0\in \bR^{n_1\times l}$ in $X$ with the smallest $l$ such that $X:=[X,0]\in \bR^{n_1\times (h+l)} $  and $n_3:=(h+l)/n_2$ is an integer.
Therefore, we reshape the matrix $ X\in \bR^{n_1\times h} $ as a tensor $ \X\in\bR^{n_1\times n_2\times n_3} $ such that
\begin{equation}\label{m-t}
	\begin{aligned}
		X^{(k)}=X\left(:,(k-1)n_2+1:kn_2\right),\,k\in [n_3].			
	\end{aligned}
\end{equation}
See Figure \ref{fig:mm} for clearness.
\begin{figure}
	\centering
	\includegraphics[width=0.8\linewidth]{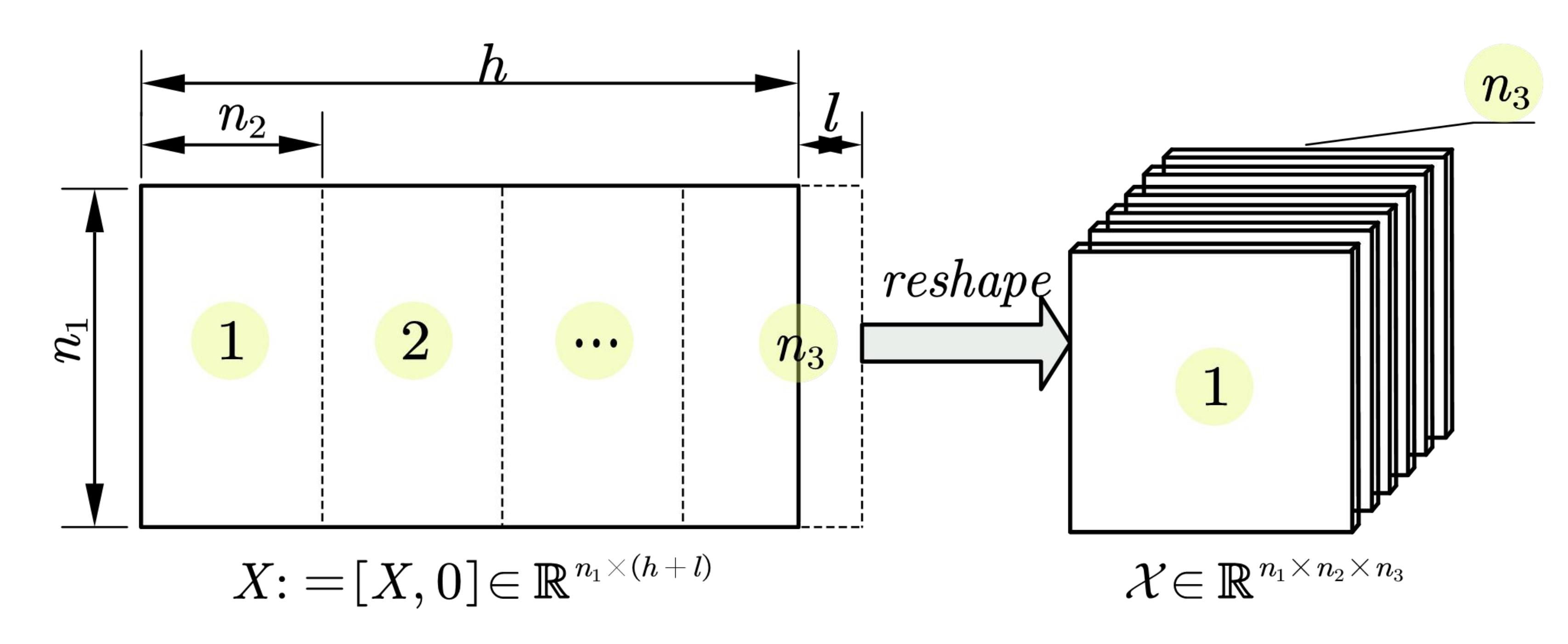}
	\caption{Reshaping the matrix $ X $ into the tensor $ \X $. }
	\label{fig:mm}
\end{figure}

Now we are ready to  establish the relationship between $ \rank\left(X\right)  $ and $ \rank_t\left(\X\right)$. For this aim, we need the following results.
\begin{lemma}\label{rank-equ}
	Suppose that $ \A \in {{\mathbb R}^{{n_1} \times {n_2} \times {n_3}}} $ and $\bar \A=fft(\A,[\; ],3)$, then $ \rank\left({{\bar A}_{(1)}}\right) = \rank\left({A_{(1)}}\right) $.
\end{lemma}
\begin{proof}
	By $\bar \A=fft(\A,[\; ],3)$, we have $ \bar \A = \mathcal{A}{ \times _3}{F_{n_3}}  $. Let $ \A=\G\times_{1} U^1 \times_{2} U^2 \times_{3} U^3 $ be a Tucker rank decomposition. Then
	\[ \bar \A = \mathcal{A}{ \times _3}{F_{n_3}}=\G\times_{1} U^1 \times_{2} U^2 \times_{3}\left(F_{n_3}U^3\right), \]
	which leads to $ \rank\left({{\bar A}_{(1)}}\right) \le\rank\left(U^1\right) = \rank\left({A_{(1)}}\right) $. Similarly, with $ \A = \bar \A{ \times _3}{F_{n_3}^{-1}}  $, there holds \[ \rank\left({A_{(1)}}\right) \le \rank\left({{\bar A}_{(1)}}\right).\] In conclusion, the lemma is established now.
\end{proof}

\begin{lemma}\label{rank-inequ}
	Suppose that matrix $ X\in \bR^{n_1\times h} $ and tensor $ \X\in\R^{n_1\times n_2\times n_3}$ obtained by reshaping  matrix $X$ with
	(\ref{m-t}).
	Then
	\begin{equation}
		\begin{aligned}
			&\rank_t(\X)\leq\rank(X)\le n_3 \rank_t(\X), \\
			&\rank\left( X \right)\leq \left\|\rank_m(\X)\right\|_1 \leq  n_3\rank\left( X \right).
		\end{aligned}
	\end{equation}
\end{lemma}
\begin{proof}
	Let $\bar \X=fft(\X,[\; ],3)$, then
	\begin{equation}
		\begin{aligned}\label{rank-inequ-1}
			\rank\left( X \right)&=\rank\left( {{X_{(1)}}} \right)=\rank\left( {{{\bar X}_{(1)}}} \right)\\&=\rank\left(\left[{{{\bar X}^{(1)}},{{\bar X}^{(2)}}, \ldots ,{{\bar X}^{(n_3)}}}\right]\right),		
		\end{aligned}
	\end{equation}
	where the first equality follows from the way of  the reshaped tensor $ \X $, the second equality is due to Lemma \ref{rank-equ} and the third equality comes from $ {{{\bar X}_{(1)}}}=\left[{{{\bar X}^{(1)}},{{\bar X}^{(2)}}, \ldots ,{{\bar X}^{(n_3)}}}\right] $.
	
	Observe that
	\begin{equation}\label{rank-inequ-2}
		\begin{aligned}
			&\rank\left(\left[{{{\bar X}^{(1)}},{{\bar X}^{(2)}}, \ldots ,{{\bar X}^{(n_3)}}}\right]\right)\\\le& \sum_{k=1}^{n_3} \rank\left({{\bar X}^{(k)}}\right)\le n_3 \rank_t(\X)			
		\end{aligned}
	\end{equation}
	and
	\begin{equation}\label{rank-inequ-3}
		\begin{aligned}
			&\rank\left(\left[{{{\bar X}^{(1)}},{{\bar X}^{(2)}}, \ldots ,{{\bar X}^{(n_3)}}}\right] \right)\\\ge&\max\left\lbrace\rank\left({{\bar X}^{(k)}}\right)|k\in[n_3]\right\rbrace=\rank_t(\X).				
		\end{aligned}
	\end{equation}
	By \eqref{rank-inequ-1}, \eqref{rank-inequ-2} and \eqref{rank-inequ-3}, it follows \[ \rank_t(\X)\le\rank(X)\le n_3\rank_t(\X).\]
	
	On the other hand,  \eqref{rank-inequ-1} and \eqref{rank-inequ-3} mean that
	\begin{equation}\label{rank-inequ-4}
		n_3\rank\left( X \right)\ge n_3\rank_t(\X)\ge\sum_{k=1}^{n_3} \rank\left({{\bar X}^{(k)}}\right).	
	\end{equation}
	Together with \eqref{rank-inequ-1} and \eqref{rank-inequ-2}, it holds
	\begin{equation*}
		\begin{aligned}
			&n_3\rank\left( X \right)\ge\sum_{k=1}^{n_3} \rank\left({{\bar X}^{(k)}}\right)\\=&\left\|\rank_m(\X)\right\|_1\ge\rank\left( X \right).			
		\end{aligned}
	\end{equation*}
\end{proof}

Based on these analyses, we consider the following tensor completion problem for solving the matrix completion problem
\eqref{mc}:
\begin{equation}\label{tc}
	\min_{\X\in\bR^{n_1\times n_2 \times n_3}}~\operatorname{rank}_t(\X), \quad \mbox{\rm s.t.} \quad  P_{\Omega}(\X-\M)=0,		
\end{equation}
where $ \M\in\R^{n_1\times n_2\times n_3 } $ is a tensor by reshaping  matrix $ M$  in the same way of  reshaped tensor $\X$.

According to Lemma \ref{lem:f}, we consider the following tensor factorization model to solve \eqref{tc}
\begin{equation}\label{Q:m}
	\min_{\X,\P,\Q}~\frac{1}{2}\left\| \P \ast\Q - \X
	\right\|_F^2, \quad \mbox{\rm s.t.} \quad  P_{\Omega}(\X-\M)=0.		
\end{equation}
We use the alternating minimization algorithm to optimize \eqref{Q:m}. Update $ \X$, for fixed tensors $\P $ and $ \Q $ by
\begin{equation}\label{M-X}
	\begin{aligned}
		\mathcal{X} &=\underset{P_{\Omega}(\X-\M)=0}{\operatorname{argmin}} \frac{1}{2}\|\mathcal{P} * \mathcal{Q}-\mathcal{X}\|_{F}^{2} =P_{\Omega^c}(\mathcal{P} * \mathcal{Q})+P_{\Omega}(\mathcal{M}).
	\end{aligned}
\end{equation}

Now we present how to update $\P$ and $\Q$, which is similar to Algorithm TCTF proposed in Section 3 of \cite{ZLLZ18}. For the ease of the reader, we present the details here. We rewrite (\ref{Q:m}) as a corresponding matrix version. Assume that $ \operatorname{rank}_{m}(\mathcal{X})=\bm{r} $ and $ \operatorname{rank}_{t}(\mathcal{X})=\hat{r} $, where $ \bm{r}_{k}=\operatorname{rank}\left(\bar{X}^{(k)}\right),\, k\in[n_3]$ and $\hat{r}=\max\left\lbrace \bm{r}_{1}, \ldots, \bm{r}_{n_{3}}\right\rbrace$. For each $k$,  $\bar X^{(k)}$ can be factorized as a product of two matrices $\hat P^{(k)}$ and $\hat Q^{(k)}$ of smaller sizes, where $\hat P^{(k)}\in {\mathbb{C}}^{n_1\times \bm r_k}$ and $\hat Q^{(k)}\in {\mathbb{C}}^{\bm r_k\times n_{2}}$ are the $k$-th block diagonal matrices of $\hat P\in {\mathbb{C}}^{n_{1}n_3\times \left(\sum_{k=1}^{n_3}\bm r_k\right)}$ and $\hat Q\in {\mathbb{C}}^{\left(\sum_{k=1}^{n_3}\bm r_k\right)\times n_2n_3}$.
Let $\bar P^{(k)}=[\hat P^{(k)}, 0]\in \mathbb{C}^{n_1\times \hat r}$, $\bar Q^{(k)}=[\hat Q^{(k)};0]\in \mathbb{C}^{\hat r\times n_2}$ and $\bar P,\,\bar Q$ be the block diagonal matrices with the $k$-th block diagonal matrices $\bar P^{(k)}$ and $\bar Q^{(k)}$, respectively. Then $ \hat P\hat Q=\bar P\bar Q$. Together with Lemma \ref{lem:equ}, it follows
\begin{equation*}
	\begin{aligned}
		\left\|\P*\Q-\X\right\|_F^2&=\frac{1}{n_3}\left\|\bar P\bar Q-\bar X\right\|_F^2= \frac{1}{n_3}\left\|\hat P\hat Q-\bar X\right\|_F^2 \\&=\frac{1}{n_3}\sum\limits_{k=1}^{n_3}\left\|\hat P^{(k)}\hat Q^{(k)}-\bar X^{(k)}\right\|_F^2.		
	\end{aligned}
\end{equation*}
Therefore, (\ref{Q:m}) can be rewritten as
\begin{equation}
	\min\limits_{\hat\P,\hat\Q}\; \frac{1}{2n_3}\sum\limits_{k = 1}^{n_3}\left\|\hat P^{(k)}\hat Q^{(k)} - \bar X^{(k)}\right\|_F^2, \quad \mbox{\rm s.t.} \quad  P_{\Omega}(\X-\M)=0.
\end{equation}
Combining  with \eqref{conj}, we can update $ \hat P $ and $ \hat Q $ as follows:
{\small
\begin{equation}\label{P}
	\begin{aligned}
		\hat P^{(k)} = \left\{ \begin{gathered}
			{{\bar X}^{(k)}}{\left( {{{\hat Q}^{(k)}}} \right)^*}{\left( {{{\hat Q}^{(k)}}{{\left( {{{\hat Q}^{(k)}}} \right)}^*}} \right)^\dag },\,k = 1, \ldots, \left\lceil \frac{n_3+1}{2} \right\rceil, \hfill \\
			conj\left( {{{\hat P}^{({n_3} - k + 2)}}} \right),\,k = \left\lceil \frac{n_3+1}{2} \right\rceil+1, \ldots ,{n_3}, \hfill \\
		\end{gathered}  \right.		
	\end{aligned}
\end{equation}
\begin{equation}\label{Q}
	\begin{aligned}
		\hat Q^{(k)} = \left\{ \begin{gathered}
			{\left( {{\left( {\hat P^{(k)}} \right)}^*}\hat P^{(k)} \right)^\dagger}{\left( {\hat P^{(k)}} \right)^*}\bar X^{(k)},\,k = 1, \ldots, \left\lceil \frac{n_3+1}{2} \right\rceil, \hfill \\
			conj\left( {{{\hat Q}^{({n_3} - k + 2)}}} \right),\,k = \left\lceil \frac{n_3+1}{2} \right\rceil+1, \ldots ,{n_3}. \hfill \\
		\end{gathered}  \right.		
	\end{aligned}
\end{equation}}

One can perform \eqref{P}, \eqref{Q} and \eqref{M-X} to update $ \P $, $ \Q $ and $ \X $ in different manners. Directly applying the APG method proposed in \cite{XYWZ12} leads to the order of $ \P $, $ \Q $, $ \X $. However, since $ \X $ interacts with $ \P $ and $ \Q $, updating it more frequently is expected to speed up the convergence of the algorithm. Hence, a more efficient way would be to update the variables in the order of $ \P $, $ \X $, $ \Q $, $ \X $. The convergence behavior with two different updating orders on a synthetic tensor and the USC-SIPI image database\footnote{http://sipi.usc.edu/database/.} was shown in Figure \ref{fig:three}. From the figure, we see that the updating order $ \P $, $ \Q $, $ \X $ final effect comparably well as that with the order $ \P $, $ \X $, $ \Q $, $ \X $. However, the former convergence speeds are much worse than the latter. We further notice that although the update sequence $ \P $, $ \X $, $ \Q $, $ \X $ converges faster, it takes more iteration time for each step, and the reason for the faster convergence is due to the fact that the first few steps can produce a good value. For this reason, we adopt the two-stage strategy:
updating order $ \P $, $ \X $, $ \Q $, $ \X $ in the first few steps, and $ \P $, $ \Q $, $ \X $ in the subsequent steps. We denote this algorithm  by TCTF-M. Similarly, we can see  the convergence behavior of TCTF-M with the best performance.
\begin{figure*}[htbp]
	\centering
	\begin{subfigure}[b]{1\linewidth}
		\begin{subfigure}[b]{0.5\linewidth}
			\centering
			\includegraphics[width=\linewidth]{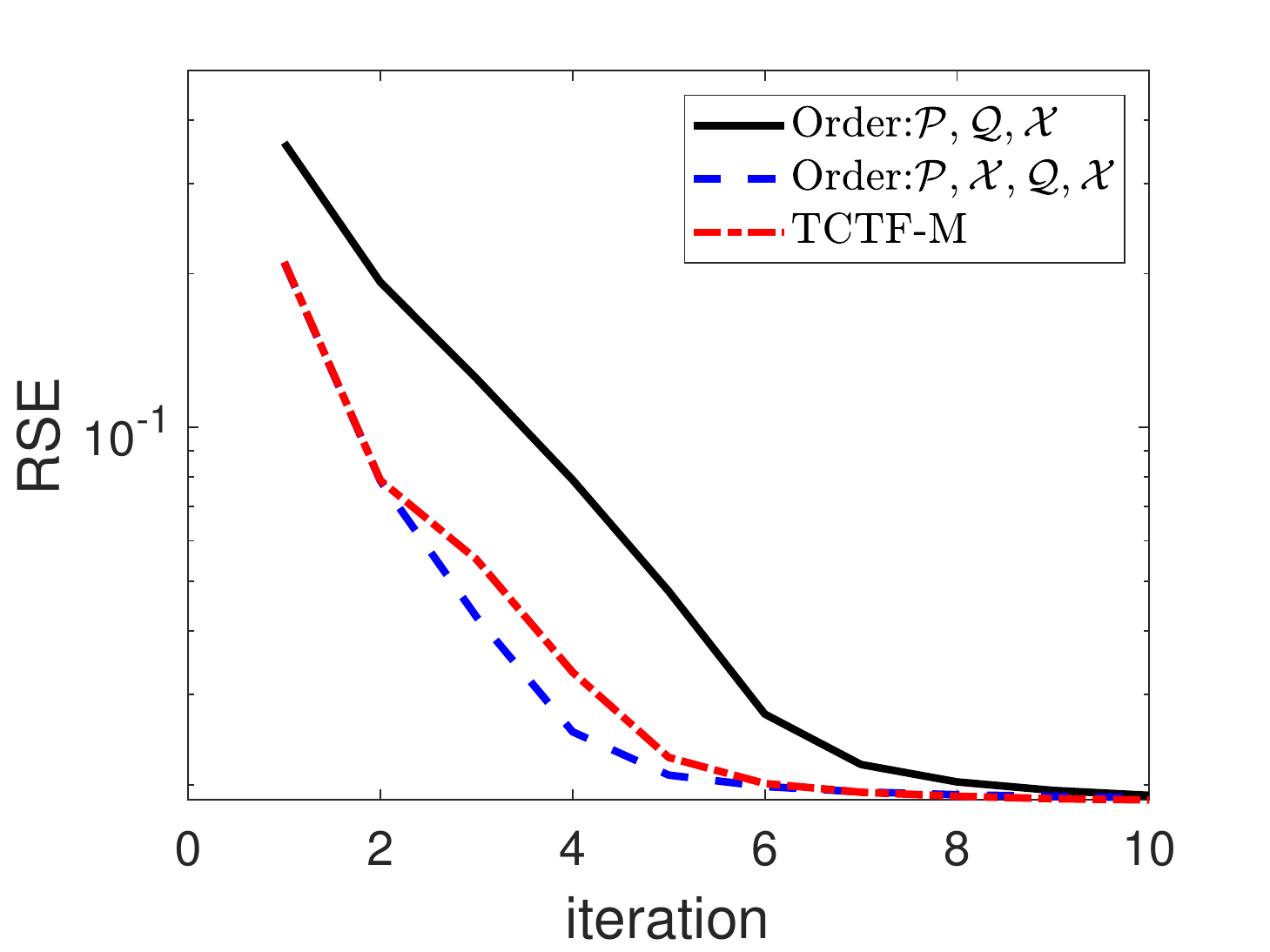}
			\includegraphics[width=\linewidth]{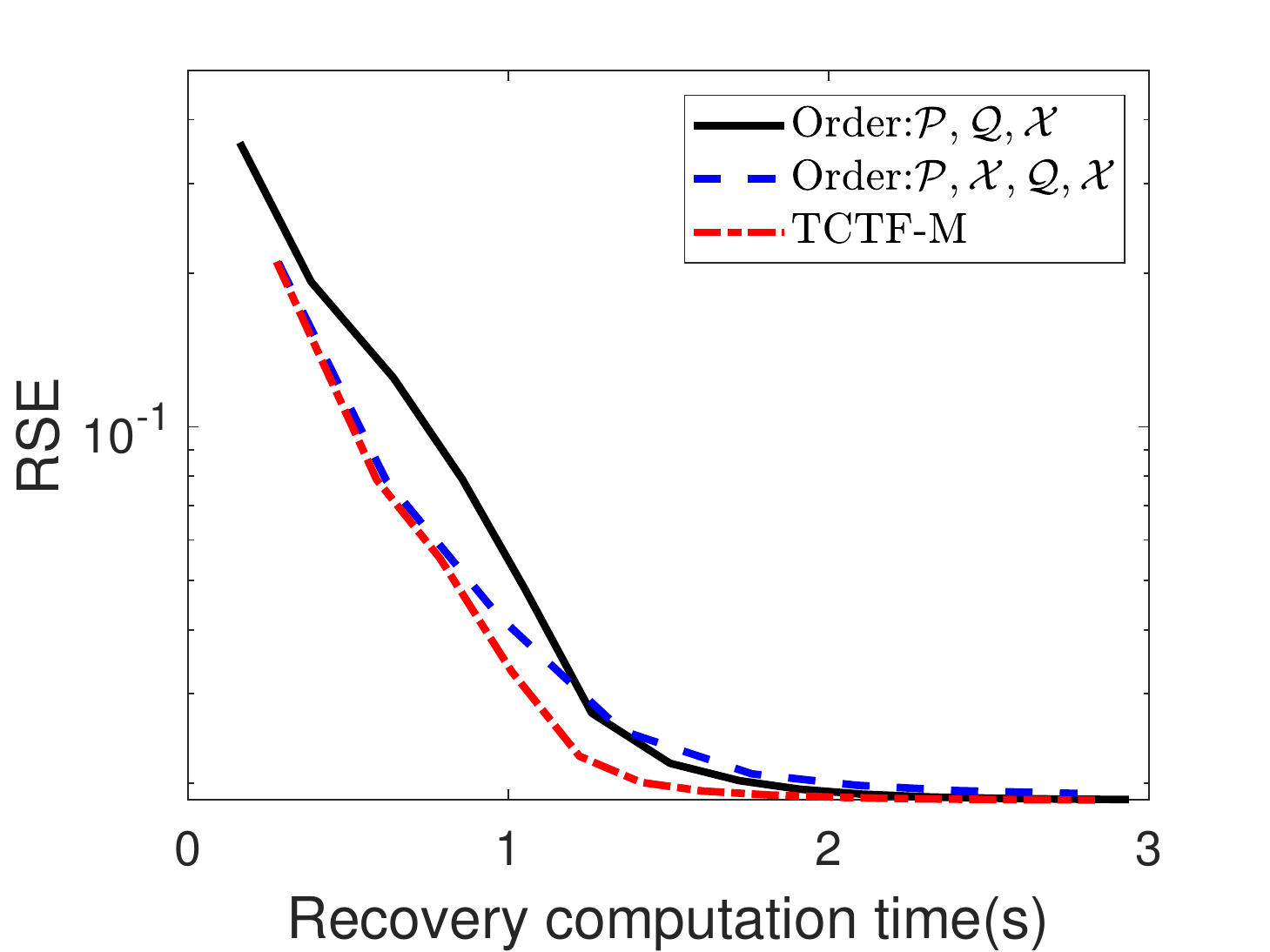}
			\caption{A Gaussian random tensor.} 
		\end{subfigure}   	
		\begin{subfigure}[b]{0.5\linewidth}
			\centering
			\includegraphics[width=\linewidth]{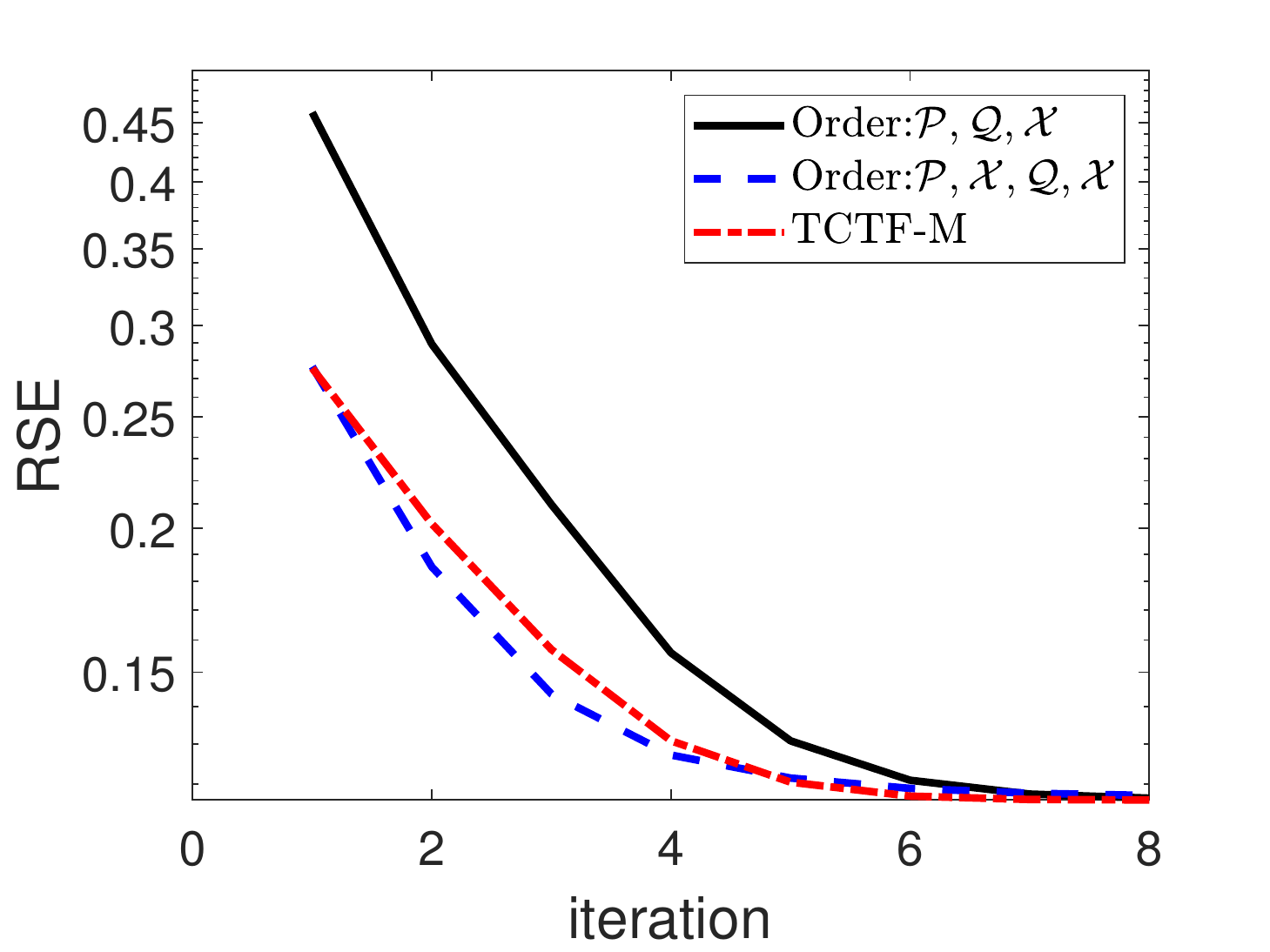}	
			\includegraphics[width=\linewidth]{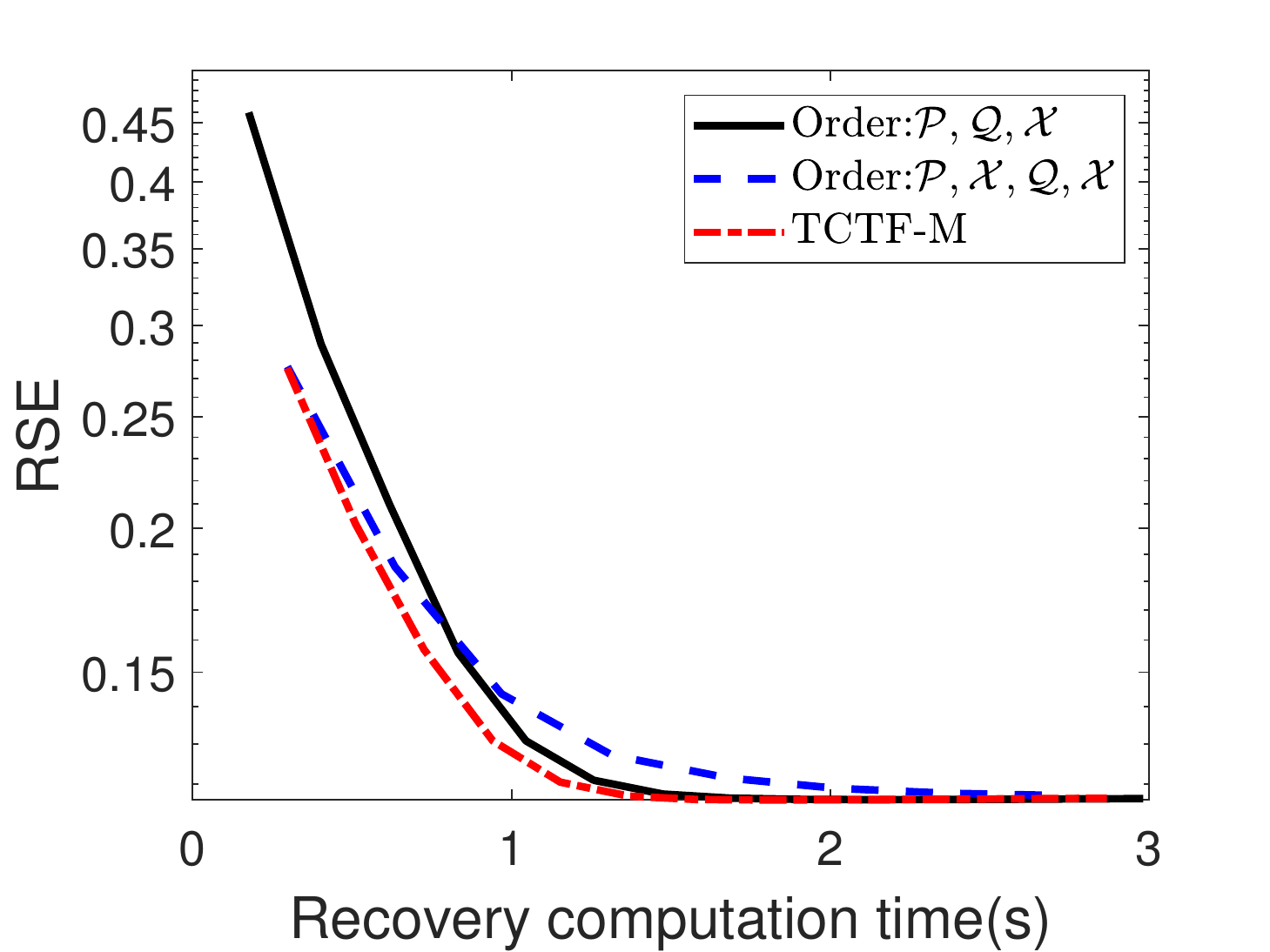}	
			\caption{The USC-SIPI image database.} 
		\end{subfigure}
	\end{subfigure}
	\vfill
	\caption{Results with three different orders.}
	\label{fig:three}
\end{figure*}
For convenience of notation, we outline the pseudocode of TCTF-M as follows.
\begin{table}[htbp]
	\centering
	\begin{tabular}{l}
		\toprule
		\toprule
		{\bfseries Algorithm 3.1} Matrix Completion Algorithm (TCTF-M)       \\
		\midrule
		{\bfseries Input:} The matrix (tensor) data $M \in {{\mathbb R}^{{n_1} \times h}}\,({\mathcal M} \in {{\mathbb R}^{{n_1} \times {n_2} \times {n_3}}})$, \\\qquad\quad\:\! the observed set
		$\tilde\Omega\, (\Omega) $ and $ t_0 $.                             \\
		{\bfseries Initialization:} $\X^0,\,\hat P^0,\,\hat Q^0$ and the multi-rank  $ \bm r_\X^0\in\R^{n_3}$.                                       \\
		{\bfseries While not converge do}                \\		
		\qquad  $ \bm{1.} $ Fix $\hat Q^t$ and $\mathcal X^t $ to update $\hat P^{t+1}$ by \eqref{P}.     \\
		\qquad  $ \bm{2.} $ If $ t\le t_0 $ then \\
		\qquad \quad\qquad Fix $\hat P^{t+1}$ and $ \hat Q^{t} $ to compute $\mathcal X^{t} $ by \eqref{M-X}. \\
		\qquad  $ \bm{3.} $ Fix $\hat P^{t+1}$ and $\X^{t} $ to update $\hat Q^{t+1}$ by \eqref{Q}.     \\
		\qquad  $ \bm{4.} $ Adopt the rank decreasing scheme to \\\qquad\quad\:\! adjust $ \bm r^t_\X $,
		adjust the sizes of $\hat P^{t+1},\,\hat Q^{t+1}$. \\
		\qquad  $\bm {5.}$ Fix $\hat P^{t+1}$ and $ \hat Q^{t+1} $ to compute $\mathcal X^{t+1} $ by \eqref{M-X}.\\
		\qquad  $ \bm{6.} $ Check the stop criterion: ${\left\| {\X^{t + 1} - {\X^t}} \right\|_F}/{\left\| {{\X^t}} \right\|_F} < \varepsilon $.  \\
		\qquad  $ \bm{7.} $ $t \leftarrow t + 1$.            \\
		{\bfseries end while}                             \\
		{\bfseries Output:}  $\mathcal X^{t + 1}  $.      \\		
		\bottomrule
		\bottomrule
	\end{tabular}
\end{table}
\begin{remark}
	In general, we do not know the true multi-tubal rank of optimal tensor $\X$ in advance. Thus, it is necessary to  estimate the multi-rank of tensor $\X$. In this paper, we adopt the same rank  estimation and rank decreasing strategy proposed in \cite{WYZ12,XHYS15,ZLLZ18}.
	
	Compared to TCTF, only half of matrices $\hat P^{(k)}$ and $\hat Q^{(k)}$ are calculated in \eqref{P} and \eqref{Q}. The reduction decreases the computational cost of $\hat P^{t+1}$ and $\hat Q^{t+1}$  when $ n_3 $ is large.
	When $ t\ge t_0 $, in each iteration, the complexity of TCTF-M is $\mathcal{O}\left(r\left(n_1+n_2\right) n_3 \log n_3+r n_1n_2\left\lceil\frac{n_3+1}{2} \right\rceil \right)$, where $ r=\operatorname{rank}_t(\X) $.
\end{remark}

Finally, we present the convergence results of Algorithm 3.1, whose proof is from \cite{ZLLZ18}.

\begin{theorem}
	Assume that $ g\left(\hat P,\hat Q,\X\right) =\frac{1}{2n_3}\left\|\hat P\hat Q-\bar X\right\|_F^2 =\frac{1}{2n_3}\sum\limits_{k=1}^{n_3}\left\|\hat P^{(k)}\hat Q^{(k)}-\bar X^{(k)}\right\|_F^2 $ is the objective function and the sequence $\left\{\P^t,\Q^t,\X^t\right\}$ generated by Algorithm 3.1 is bounded, Then
	it satisfies the following properties:
	\begin{itemize}
		\item[(1)] $g^t:=g\left(\hat P^t,\hat Q^t,\X^t\right)$ is monotonically decreasing. Actually, it
		satisfies the following inequality:	
		$$ g^t-g^{t+1} \ge \frac{1}{2n_3}\left\|\hat{P}^{t+1} \hat{Q}^{t+1}-\hat{P}^{t} \hat{Q}^{t}\right\|_F^2\ge 0. $$
		\item[(2)] Any accumulation point $ \left(\P_\star,\Q_\star,\X_\star\right) $ of the sequence $\left\{\P^t,\Q^t,\X^t\right\}$ is a KKT point of problem \eqref{Q:m}.
	\end{itemize}	
\end{theorem}


\section{Tensor completion}

In this section, we first establish the relationship between tubal rank and Tucker rank of the third order tensor. According to such relationship, we improve the tubal rank to double tubal rank and then establish the low rank tensor completion problem with the introduced double tubal rank.

\subsection{Tensor completion model based on double tubal rank}
From Lemma \ref{rank-inequ}, the following results is direct.
\begin{lemma}\label{lem:tl-ineq}
	For a tensor $ \X \in {{\mathbb R}^{{n_1} \times {n_2} \times {n_3}}} $, it holds
	\begin{equation}\label{eq:tl-ineq}
		\rank_t(\X)\le\rank\left(X_{(i)}\right)\le n_3\rank_t(\X),\quad i\in[2].	
	\end{equation}
\end{lemma}

Compared to Tucker rank, tubal rank does not involve the low rank structure
information of the mode-3 unfolding matrix from Lemma
\ref{lem:tl-ineq}. Hence, we define an improved tensor  rank as follows:
\begin{equation}\label{ttr}
	\rank_{ttr}\left(\X\right)=\left(\rank_t(\X),\rank(X_{(3)})\right).	
\end{equation}
Based on the Lemma \ref{rank-inequ}, we change \eqref{ttr} into double tubal rank:
\begin{equation}\label{ttl}
	\rank_{dt}\left(\X\right)=\left(\rank_t(\X),\rank_t(\tilde\X)\right),	
\end{equation}
where $ \tilde\X\in\R^{n_3\times p\times q }\,(pq=n_1n_2) $ is a tensor by reshaping the unfolding matrix $ X_{(3)} $  satisfying \eqref{m-t} and hence $ \tilde X_{(1)}=X_{(3)} $.

Next, we discuss the relationship between Tucker rank and double tubal rank.
\begin{lemma}\label{lem:tc-dt}
	Suppose that $ \X \in {{\mathbb R}^{{n_1} \times {n_2} \times {n_3}}} $ and $ \rank_{dt}\left(\X\right) $ is defined as in \eqref{ttl}. Then
	\begin{equation*}
		\begin{aligned}
			&\rank_t(\X)\le\rank\left(X_{(i)}\right)\le n_3\rank_t(\X),\quad i\in[2],\\
			&\rank_t(\tilde\X)\le\rank\left(X_{(3)}\right)\le n_3\rank_t(\tilde\X).
		\end{aligned}
	\end{equation*}
\end{lemma}
\begin{proof}
	The result is immediate from Lemma \ref{rank-inequ} and Lemma \ref{lem:tl-ineq}.
\end{proof}

According to this lemma, the proposed double tubal rank can learn the global correlations within multi-dimensional data as well as the Tucker rank. In the next lemma, we prove a connection between double tubal rank and 3-tubal rank (tensor fibered rank).

\begin{lemma}\label{lem:wstnn}
	For a tensor $ \X \in {{\mathbb R}^{{n_1} \times {n_2} \times {n_3}}} $, we have
		\begin{equation*}
		\begin{aligned}
			&\rank_t(\tilde\X)/n_2\le\rank_t(\X_{(13)})\le q\rank_t(\tilde\X),\\
			&\rank_t(\tilde\X)/n_1\le\rank_t(\X_{(23)})\le q\rank_t(\tilde\X).
		\end{aligned}
	\end{equation*} 
In particular, when $ \tilde\X \in {{\mathbb R}^{{n_3} \times {n_1} \times {n_2}}} $, $ \rank_t(\tilde\X)=\rank_t(\X_{(13)}) $.
\end{lemma}
\begin{proof}
By the definition of $ \X_{(13)} $ and Lemma \ref{lem:tl-ineq}, we have
\begin{equation*}
	\begin{aligned}
	&\rank_t(\X_{(13)})\le \rank(X_{(3)})\le n_2\rank_t(\X_{(13)}), \\
	&\rank_t(\tilde\X)\le \rank(\tilde X_{(1)})\le q\rank_t(\tilde \X).
\end{aligned}
\end{equation*}	
Combining the above inequality and $ \tilde X_{(1)}=X_{(3)} $, one has
\begin{equation*}
	\rank_t(\tilde\X)/n_2\le\rank_t(\X_{(13)})\le q\rank_t(\tilde\X).
\end{equation*}
Similar to the analysis above, we obtain
\begin{equation*}
	\rank_t(\tilde\X)/n_1\le\rank_t(\X_{(23)})\le q\rank_t(\tilde\X).
\end{equation*}

\end{proof}

Double tubal rank is a vector and its corresponding low rank tensor completion model is a vector optimization problem. To keep things simple, we adopt the weighted rank $ \rank_t(\X)+\gamma\rank_t(\tilde\X)$ with a positive parameter $\gamma$  as a measure of tensor rank, and the low rank tensor completion problem can be modeled as 
\begin{equation}\label{mt-vector-com}
	\begin{aligned}
		\min \limits_{{\mathcal X}}&\; \rank_t(\X)+\gamma\rank_t(\tilde\X)\\
		\mbox{\rm s.t.}& \;{P_\Omega }\left( {{\mathcal X} - {\mathcal M}} \right) = 0.		
	\end{aligned}
\end{equation}
Clearly, \eqref{mt-vector-com} reduces to the classical low tubal rank tensor completion model when $\gamma=0$.

According to Lemma \ref{lem:f}, we consider the following tensor factorization model
\begin{equation}\label{lm}
	\begin{aligned}
		\min&\;\frac{1}{2}\left\| \P \ast\Q - \X
		\right\|_F^2+\frac{\gamma}{2}\left\| \U \ast\V - \tilde\X
		\right\|_F^2 \\
		\mbox{\rm s.t.}& \; {P_\Omega }(\mathcal{X} - \mathcal{M}) = 0.		
	\end{aligned}
\end{equation}

Now, we are ready to update $\X,\,\P,\,\Q,\,\U,\,\V$. First of all, we update $\X$ by
{\small
\begin{equation}\label{X}
	\begin{aligned}
		\X=&\mathop {\operatorname{argmin}} \limits_{{P_\Omega }(\X - \mathcal M) = 0}\frac{1}{2}\left\|\P \ast\Q -\X\right\|_F^2+\frac{\gamma}{2}\left\| \U \ast\V - \tilde\X
		\right\|_F^2	\\
		=&\mathop {\operatorname{argmin}} \limits_{{P_\Omega }(\X - \mathcal M) = 0}\frac{1}{2}\left\|\P \ast\Q -\X\right\|_F^2+\frac{\gamma}{2}\left\| fold_3\left[\left(\U \ast\V\right)_{(1)}\right] - \X
		\right\|_F^2	\\
		=&\frac{1}{1+\gamma}P_{\Omega^c}\left(\P \ast\Q+\gamma fold_3\left[\left(\U \ast\V\right)_{(1)}\right]\right)
		+{P_\Omega }(\mathcal M).
	\end{aligned}
\end{equation}}

Furthermore, $\P$ and $\Q$ can be updated by solving the following problem
\begin{equation}\label{PQ}
	\mathop {\operatorname{argmin}}\limits_{\P,\,\Q}\;\frac{1}{2}\left\| \P \ast\Q - \X
	\right\|_F^2.
\end{equation}
Clearly, $\P$ and $\Q$ can be updated by \eqref{P} and \eqref{Q} respectively.

Similarly, we can update $ \hat U $ and $ \hat V $ as follows:
{\small
\begin{equation}\label{U}
	\begin{aligned}
		\hat U^{(k)}= \left\{ \begin{gathered}
			\bar {\tilde X}^{(k)}{\left( {\hat V^{(k)}} \right)^*}{\left( \hat V^{(k)}{{\left( {\hat V^{(k)}} \right)}^*} \right)^\dagger},\,k = 1, \ldots, \left\lceil \frac{q+1}{2} \right\rceil, \hfill \\
			conj\left( {{{\hat U}^{({q} - k + 2)}}} \right),\,k = \left\lceil \frac{q+1}{2} \right\rceil+1, \ldots ,{q}, \hfill \\
		\end{gathered}  \right.		
	\end{aligned}
\end{equation}
\begin{equation}\label{V}
	\begin{aligned}
		\hat V^{(k)}= \left\{ \begin{gathered}
			{\left( {{\left( {\hat U^{(k)}} \right)}^*}\hat U^{(k)} \right)^\dagger}{\left( {\hat U^{(k)}} \right)^*}\bar {\tilde X}^{(k)},\,k = 1, \ldots, \left\lceil \frac{q+1}{2} \right\rceil, \hfill \\
			conj\left( {{{\hat V}^{({q} - k + 2)}}} \right),\,k = \left\lceil \frac{q+1}{2} \right\rceil+1, \ldots ,{q}. \hfill \\
		\end{gathered}  \right.		
	\end{aligned}
\end{equation}}

Based on above discussions, a tensor factorization based method can be outlined as Algorithm 4.1, denoted by  DTRTC.
\begin{table}[htbp]
	\centering
	\begin{tabular}{l}
		\toprule
		\toprule
		{\bfseries Algorithm 4.1} Double Tubal Rank Tensor Completion (DTRTC)       \\
		\midrule
		{\bfseries Input:} The tensor data ${\mathcal M} \in {{\mathbb R}^{{n_1} \times {n_2} \times {n_3}}}$, the observed set
		$\Omega $, $ t_0 $\\\qquad\quad\:\! and parameters $\gamma$.                                  \\
		{\bfseries Initialization:} $\X^0,\,\hat P^0,\,\hat Q^0,\,\hat U^0,\,\hat V^0$. The initialized rank  $ \bm r^0_\X\in\R^{n_3}$\\\qquad\qquad\qquad\, and $ \bm r^0_{\tilde\X}\in\R^{q} $.                                          \\
		{\bfseries While not converge do}                \\		
		\qquad  $ \bm{1.} $ Fix $\hat Q^t$ and $\mathcal X^t $ to update $\hat P^{t+1}$ by \eqref{P}.     \\
		\qquad  $ \bm{2.} $ If $ t\le t_0 $ then \\
		\qquad \quad\qquad Fix $\hat P^{t+1}$ and $ \hat Q^{t} $ to compute $\mathcal X^{t} $ by \eqref{M-X}. \\
		\qquad  $ \bm{3.} $ Fix $\hat P^{t+1}$ and $\mathcal X^{t} $ to update $\hat Q^{t+1}$ by \eqref{Q}.     \\
		\qquad  $ \bm{4.} $ If $ t\le t_0 $ then \\
		\qquad \quad\qquad Fix $\hat P^{t+1}$ and $ \hat Q^{t+1} $ to compute $\mathcal X^{t} $ by \eqref{M-X}. \\
		\qquad  $ \bm{5.} $ Fix $\hat V^t$ and $\mathcal X^t $ to update $\hat U^{t+1}$ by \eqref{U}.     \\
		\qquad  $ \bm{6.} $ If $ t\le t_0 $ then \\
		\qquad \quad\qquad Fix $\hat U^{t+1}$ and $ \hat V^{t} $ to compute $\mathcal X^{t} $ by \eqref{M-X}. \\
		\qquad  $ \bm{7.} $ Fix $\hat U^{t+1}$ and $\mathcal X^t $ to update $\hat V^{t+1}$ by \eqref{V}.     \\
		\qquad  $ \bm{8.} $ Adopt the rank decreasing scheme to adjust $ \bm r^t_\X $ and $ \bm r^t_{\tilde\X} $,
		\\\qquad\quad\:\! adjust the sizes of $\hat P^{t+1},\,\hat Q^{t+1},$ $\hat U^{t+1} $ and $ \hat V^{t+1} $. \\
		\qquad
		$\bm {9.}$ Fix $\hat P^{t+1},\,\hat Q^{t+1},\,\hat U^{t+1},\,\hat V^{t+1}$ to compute $\mathcal X^{t+1} $ by \eqref{X}.\\
		\qquad  $ \bm{10.} $ Check the stop criterion: ${\left\| {\X^{t + 1} - {\X^t}} \right\|_F}/{\left\| {{\X^t}} \right\|_F} < \varepsilon $.  \\
		\qquad  $ \bm{11.} $ $t \leftarrow t + 1$.            \\
		{\bfseries end while}                             \\
		{\bfseries Output:}  $\mathcal X^{t + 1}  $.      \\		
		\bottomrule
		\bottomrule
	\end{tabular}
\end{table}
\begin{remark}
	Similar to TCTF-M, it does not know the true multi-tubal rank of optimal tensor $\X$ and $\tilde \X$ in advance. Hence,  we adopt the same rank  estimation and rank decreasing strategy proposed in \cite{WYZ12,XHYS15,ZLLZ18}.
	
	In our paper, we set the update rule of $ \gamma^{t+1} $ as follows
	$$ \gamma^{t+1}=\frac{\left\| P_{\Omega}(\X^t-\M)\right\|_F}{\left\|P_{\Omega}(\tilde\X^t-\M)\right\|_F}. $$

	Complexity analysis: At each iteration, the cost of updating $ \P $ and $ \Q $ by \eqref{P} and \eqref{Q} is $\mathcal{O}\left(\hat{r}_\X\left(n_{1}+n_{2}\right) n_{3} \log n_{3}+\hat{r}_\X n_{1} n_{2} \left\lceil\frac{n_3+1}{2} \right\rceil\right)$, respectively. The cost of updating $ \U $ and $ \V $ by \eqref{U} and \eqref{V} is $\mathcal{O}\left(\hat{r}_{\tilde\X}\left(n_{3}+p\right) q \log q+\hat{r}_{\tilde\X} n_{3}p\left\lceil\frac{q+1}{2}\right\rceil\right)$, where $ \hat{r}_\X $ and $ \hat{r}_{\tilde\X} $ is the estimated tubal rank of $ \X $ and $ \tilde\X $, respectively. For updating $ \X $ by \eqref{X}, the computational cost for conducting the (inverse) DFT and matrix product is $\mathcal{O}\left(\hat{r}_\X\left(n_{1}+n_{2}\right) n_{3} \log n_{3}+\hat{r}_\X n_{1} n_{2} \left\lceil\frac{n_3+1}{2} \right\rceil+\hat{r}_{\tilde\X}\left(n_{3}+p\right) \right.\\q \log q
	\left.+\hat{r}_{\tilde\X} n_{3}p\left\lceil\frac{q+1}{2} \right\rceil\right)$. In step 8, we use QR decomposition to estimate the target rank whose cost is $\mathcal{O}\left(\hat{r}_\X\left(n_{1}+n_{2}\right) n_{3} \log n_{3}+\hat{r}_\X n_{1} n_{2} \left\lceil\frac{n_3+1}{2} \right\rceil\right)$ and $\mathcal{O}\left(\hat{r}_{\tilde\X}\left(n_{3}+p\right) q \log q+\hat{r}_{\tilde\X} n_{3}p\left\lceil\frac{q+1}{2} \right\rceil\right)$.
	In summary, the total cost at each iteration is $ \mathcal{O}\left(\hat{r}_\X\left(n_{1}+n_{2}\right) n_{3} \log n_{3}+\hat{r}_{\tilde\X}\left(n_{3}+p\right) q \log q+\hat{r}_\X n_{1} n_{2}\right.\\
	\left. \left\lceil\frac{n_3+1}{2} \right\rceil +\hat{r}_{\tilde\X} n_{3}p\left\lceil\frac{q+1}{2} \right\rceil\right)$.
\end{remark}

\subsection{Convergence analysis}
In this subsection, we present the convergence of DTRTC. The following notation will be used in our analysis. In problem \eqref{lm}, $\Omega$ is an index set which locates the observed data. We use $\Omega^c$ to denote the complement of the set $\Omega$ with respect to the set $\{(i,j,k): i\in {[n_1]},j\in{[n_2]},k\in {[n_3]}\}$. To simply the notation, we denote $z^t=\left(\P^t,\Q^t,\U^t,\V^t,\X^t\right)$, $ f\left(\P,\Q,\U,\V,\X\right) :=\frac{1}{2}\left\| \P \ast\Q - \X
\right\|_F^2+\frac{\gamma}{2}\left\| \U \ast\V - \tilde\X
\right\|_F^2 $ and $ f^t:=f\left(\P^t,\Q^t,\U^t,\V^t,\X^t\right) $ in this subsection.

\begin{theorem}
	Assume that the sequence $\left\{\P^t,\Q^t,\U^t,\V^t,\X^t\right\}$ generated by Algorithm 4.1 is bounded, Then
	it satisfies the following properties:
	\begin{itemize}
		\item[(1)] $f^t$ is monotonically decreasing. Actually, it
		satisfies the following inequality:
		\begin{equation*}
			\begin{aligned}
			f^t-f^{t+1} \ge& \frac{1}{2n_3}\left\|\hat{P}^{t+1} \hat{Q}^{t+1}-\hat{P}^{t} \hat{Q}^{t}\right\|_F^2\\&+\frac{\gamma}{2q}\left\|\hat{U}^{t+1} \hat{V}^{t+1}-\hat{U}^{t} \hat{V}^{t}\right\|_F^2\ge 0.	
			\end{aligned}
		\end{equation*}	
		\item[(2)] Any accumulation point $ \left(\P_\star,\Q_\star,\U_\star,\V_\star,\X_\star\right) $ of the sequence $\left\{\P^t,\Q^t,\U^t,\V^t,\X^t\right\}$ is a KKT point of problem \eqref{lm}.
	\end{itemize}	
\end{theorem}


\section{Numerical Experiments}

In this section, we conduct some experiments on real-world dataset to compare the performance of TCTF-M and DTRTC to show their validity. We employ the peak signal-to-noise rate (PSNR) \cite{WBSS04}, the structural similarity (SSIM) \cite{WBSS04}, the feature similarity (FSIM) \cite{ZZMZ11} and the recovery computation time to measure the quality of the recovered results. We compare TCTF-M for the matrix completion problem with four existing methods, including SRMF \cite{RZWQ12}, MC-NMF \cite{XYWZ12}, FPCA \cite{MGC11} and SPG \cite{YZ21}. We compare DTRTC for the tensor completion problem with WSTNN \cite{ZHZ20}, TCTF \cite{ZLLZ18}, TNN \cite{ZEA14}, NCPC \cite{XY13} and NTD \cite{Xu15}. 
All methods are implemented on the platform of Windows 10 and Matlab (R2020b) with an Intel(R) Core(TM) i7-7700 CPU at 3.60GHz and 24 GB RAM.

\subsection{Grayscale Image Inpainting}
In this subsection, we use the USC-SIPI image database\footnote{http://sipi.usc.edu/database/.} to evaluate our proposed method TCTF-M for grayscale image inpainting. In our test, six images are randomly selected from this database, including texture images ``Plastic" and ``Bark", high altitude aerial images ``Pentagon" and ``Wash", other images ``Male" and ``Airport". Among them, only the pixels of ``Wash" is $ 2250\times 2250 $, and the others are $ 1024\times 1024 $. The data of images are normalized in the range $ \left[0,1\right] $.

For each taken image, we randomly sample by the sampling ratio $p=70\%$. The initial tubal rank is set to $ \left(50,20,\ldots,20\right) $ in TCTF-M, the initial matrix rank is set to $ 100 $ in SRMF and MC-NMF. In TCTF-M, ``Wash" data sets form a tensor of size $ 2250\times 150 \times 15 $ and the others set form a tensor of size $ 1024\times 64 \times 16 $.

In Table \ref{tab:grayimage}, we present the results of all five methods for different images, and the best results are highlighted in bold. It is easy to see that  TCTF-M outperforms the other four methods. TCTF-M is the fastest method,
about $ 3 $ times faster than the second fastest method MC-NMF. MC-NMF is only slightly longer than TCTF-M in running time, but it has no exact recovery
performance guarantee. Both SRMF and FPCA are far inferior to TCTF-M in terms of running time and inpainting results. 
Although SPG has similar PSNR, SSIM, and FSIM values as TCTF-M, its running time is almost $ 20.9 $ times that of TCTF-M. Especially for the more challenging image ``Wash" inpainting, TCTF-M is about $ 58.6 $ times faster than SPG. Since SPG has to compute SVD at each iteration, it runs slower. In summary, TCTF-M not only achieves the best inpainting results but also runs very fast.

To further demonstrate the performance, images recovered by different algorithms are shown in Figure \ref{fig:image_inpainting}. Enlarged views of the recovered images evidently show the recovery differences. It can be seen that MC-NMF fails to recover the ``Male" image. Furthermore, the recovered images of SRMF and MC-NMF still have some visible reconstruction errors, such as roads in ``Pentagon" image, river edge in ``Wash" image and lines in ``Airport image". TCTF-M and SPG recover these details with  better performance.

To further demonstrate the advantage of the proposed algorithms in terms of computational cost, we make a comparison of computation complexity for fives methods in Figure \ref{fig:grayscale_time}, which shows the PSNR, SSIM, and FSIM values over running time. We can see that the PSNR, SSIM, and FSIM values of methods based on TCTF-M optimization rapidly increase to the highest values with less running time than other methods.

\begin{table}[htbp]
	\centering
	\caption{GRAYSCALE IMAGE INPAINTING PERFORMANCE COMPARISON: PSNR, SSIM, FSIM AND RUNNING TIME}\label{tab:grayimage}
	\begin{tabular}{cccccc}
		\toprule
		Image & Methods & PSNR  & SSIM  & FSIM  & Time \\
		\midrule
		\multirow{5}{*}{Plastic} & TCTF-M & \textbf{30.762 } & \textbf{0.872 } & \textbf{0.995 } & \textbf{0.796 } \\
		& SRMF  & 27.148  & 0.708  & 0.973  & 18.867  \\
		& MC-NMF & 26.512  & 0.673  & 0.964  & 3.070  \\
		& FPCA  & 20.855  & 0.397  & 0.833  & 41.008  \\
		& SPG   & 29.709  & 0.841  & 0.984  & 15.725  \\
		\midrule
		\multirow{5}{*}{Bark} & TCTF-M & \textbf{29.590 } & \textbf{0.890 } & \textbf{0.996 } & \textbf{0.765 } \\
		& SRMF  & 25.651  & 0.727  & 0.975  & 18.524  \\
		& MC-NMF & 24.413  & 0.663  & 0.960  & 3.497  \\
		& FPCA  & 19.219  & 0.400  & 0.847  & 40.227  \\
		& SPG   & 29.306  & 0.881  & 0.990  & 17.890  \\
		\midrule
		\multirow{5}{*}{Pentagon} & TCTF-M & \textbf{29.018 } & \textbf{0.792 } & \textbf{0.991 } & \textbf{0.875 } \\
		& SRMF  & 26.704  & 0.628  & 0.972  & 19.030  \\
		& MC-NMF & 26.518  & 0.619  & 0.968  & 2.561  \\
		& FPCA  & 22.600  & 0.412  & 0.835  & 39.701  \\
		& SPG   & 28.540  & 0.779  & 0.973  & 13.255  \\
		\midrule
		\multirow{5}{*}{Male} & TCTF-M & \textbf{30.961 } & 0.847  & \textbf{0.993 } & \textbf{0.838 } \\
		& SRMF  & 27.994  & 0.695  & 0.966  & 18.595  \\
		& MC-NMF & 12.479  & 0.437  & 0.841  & 2.170  \\
		& FPCA  & 21.664  & 0.412  & 0.811  & 40.036  \\
		& SPG   & 30.842  & \textbf{0.853 } & 0.984  & 18.296  \\
		\midrule
		\multirow{5}{*}{Airport} & TCTF-M & 28.692  & 0.799  & \textbf{0.987 } & \textbf{0.903 } \\
		& SRMF  & 26.236  & 0.648  & 0.961  & 18.621  \\
		& MC-NMF & 25.430  & 0.630  & 0.953  & 3.181  \\
		& FPCA  & 21.638  & 0.422  & 0.831  & 39.100  \\
		& SPG   & \textbf{29.111 } & \textbf{0.824 } & 0.981  & 17.229  \\
		\midrule
		\multirow{5}{*}{Wash} & TCTF-M & \textbf{24.207 } & \textbf{0.816 } & \textbf{0.996 } & \textbf{3.157 } \\
		& SRMF  & 19.383  & 0.364  & 0.965  & 108.759  \\
		& MC-NMF & 19.013  & 0.312  & 0.946  & 12.736  \\
		& FPCA  & 17.210  & 0.200  & 0.825  & 268.344  \\
		& SPG   & 24.046  & 0.783  & 0.990  & 184.925  \\
		\bottomrule
	\end{tabular}%
	\label{tab:addlabel}%
\end{table}%

\begin{figure*}[htbp]
	\centering
	\begin{subfigure}[b]{1\linewidth}
		\begin{subfigure}[b]{0.138\linewidth}
			\centering
			\includegraphics[width=\linewidth]{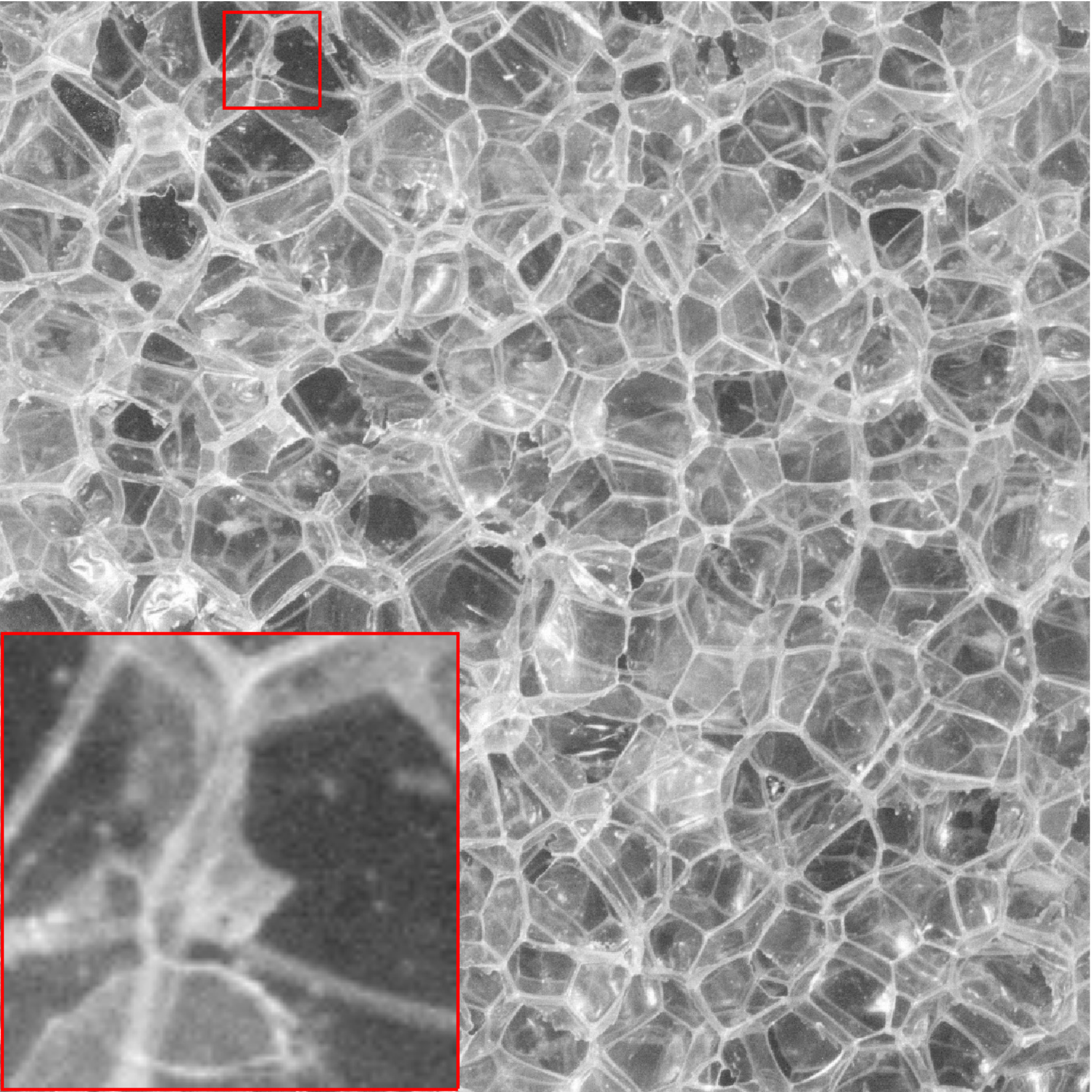}\vspace{0pt}
			\includegraphics[width=\linewidth]{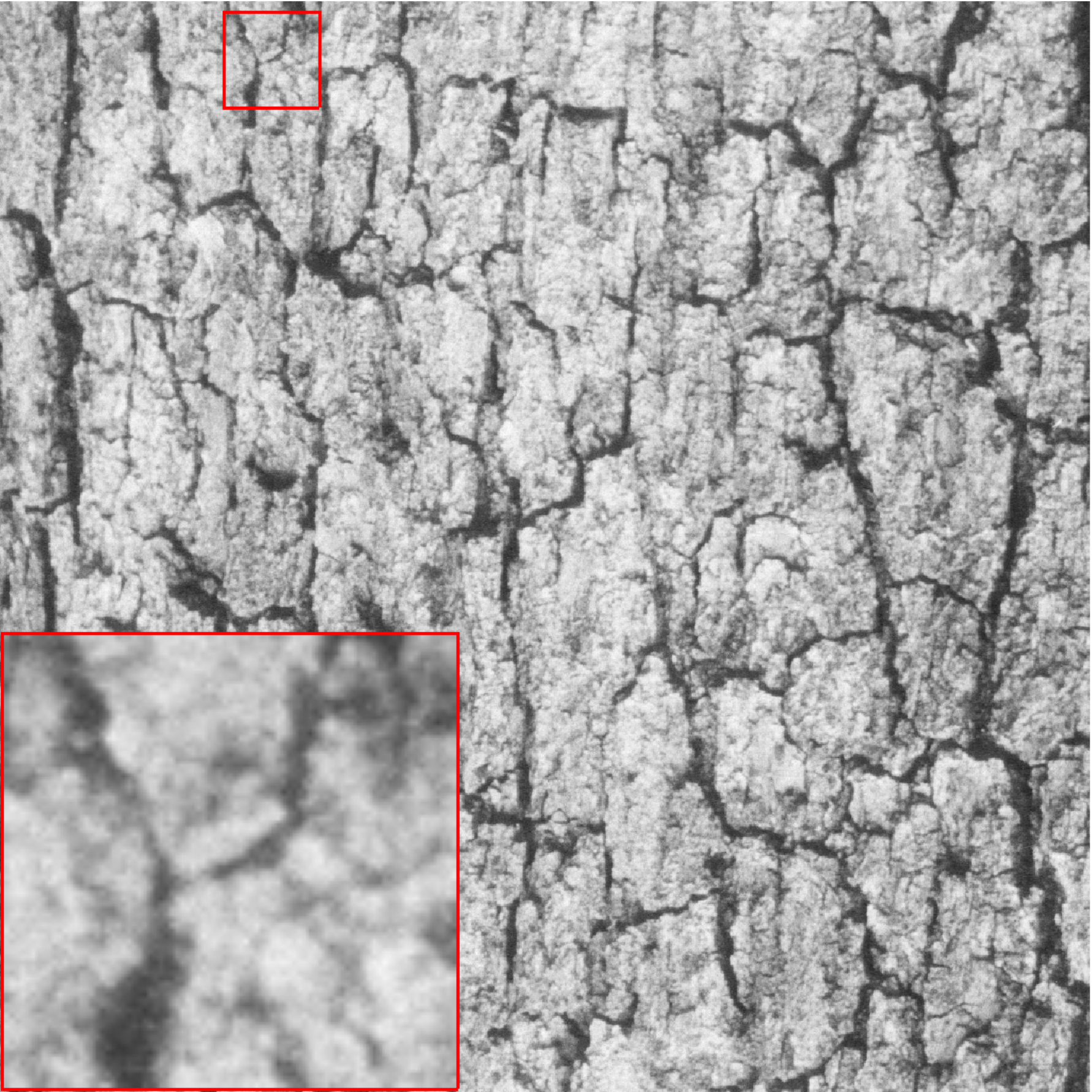}\vspace{0pt}
			\includegraphics[width=\linewidth]{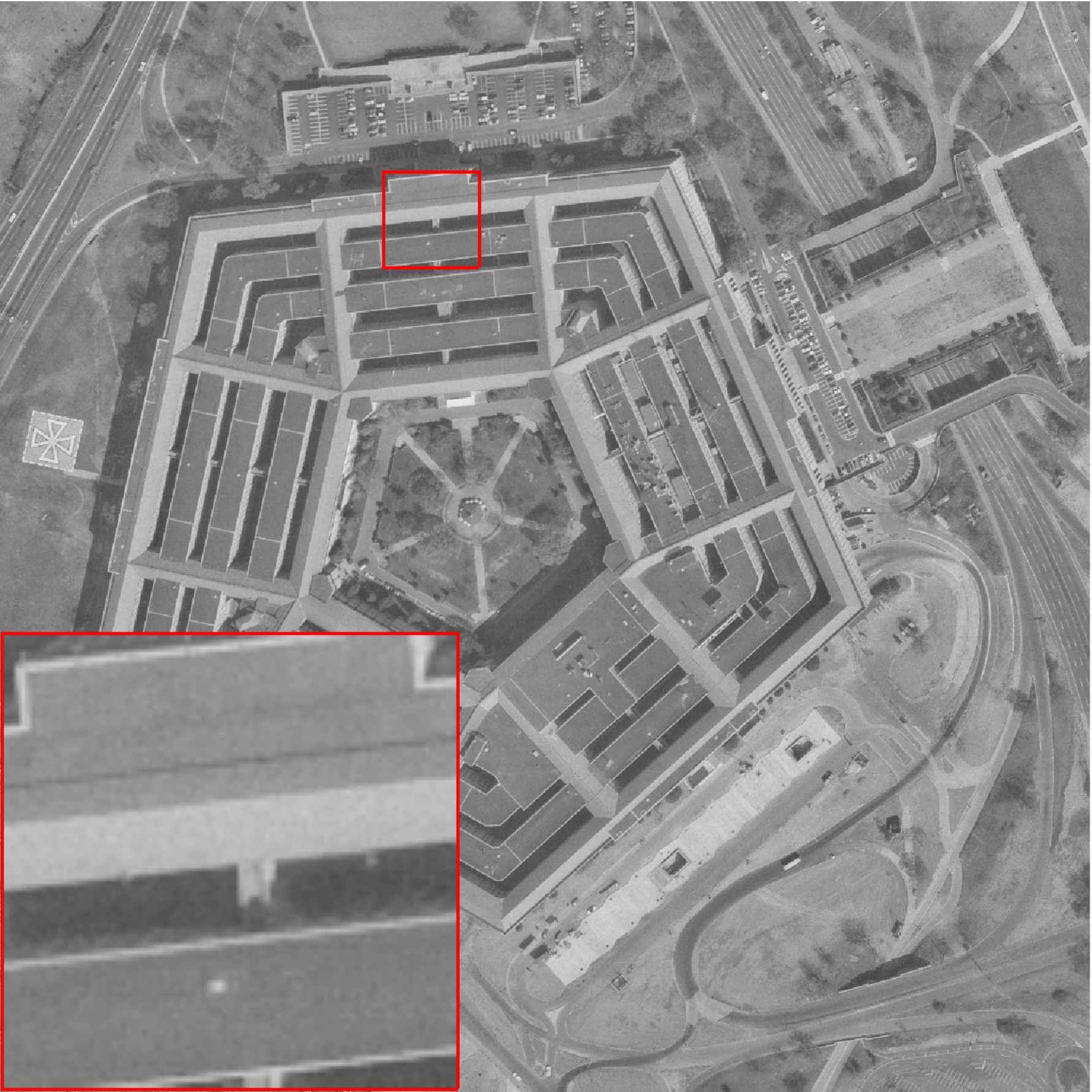}\vspace{0pt}
			\includegraphics[width=\linewidth]{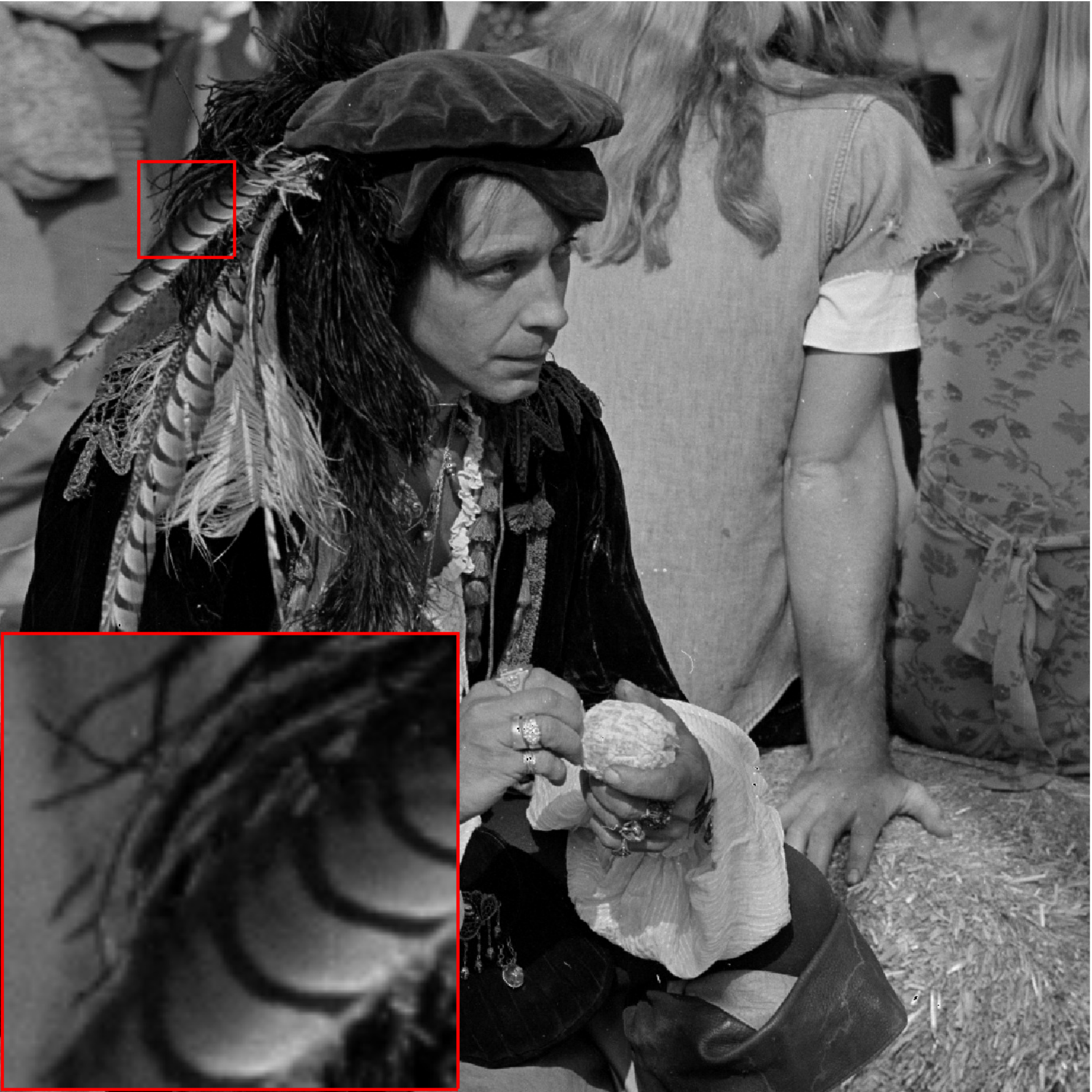}\vspace{0pt}
			\includegraphics[width=\linewidth]{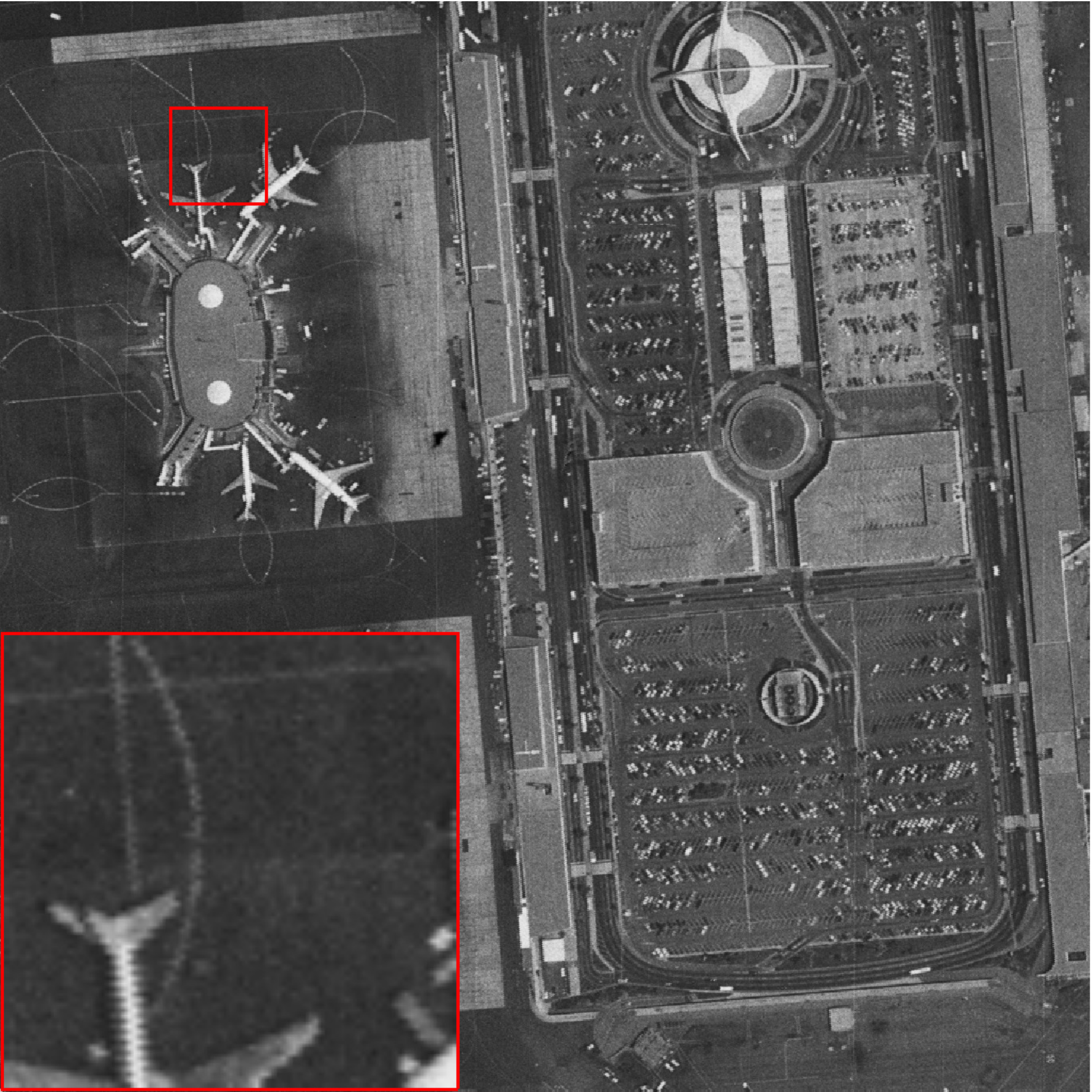}\vspace{0pt}
			\includegraphics[width=\linewidth]{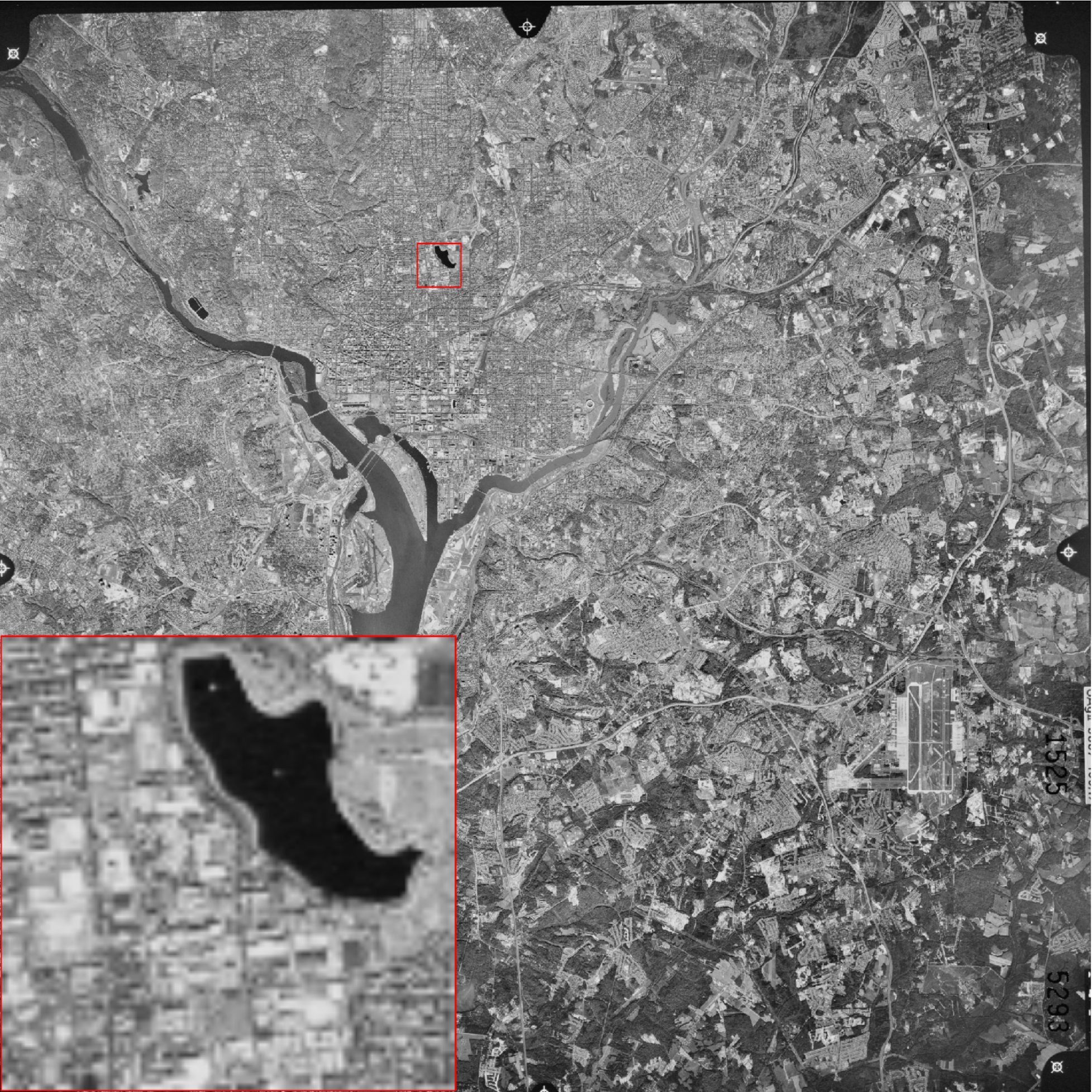}
			\caption{Original}
		\end{subfigure}   	
		\begin{subfigure}[b]{0.138\linewidth}
			\centering
			\includegraphics[width=\linewidth]{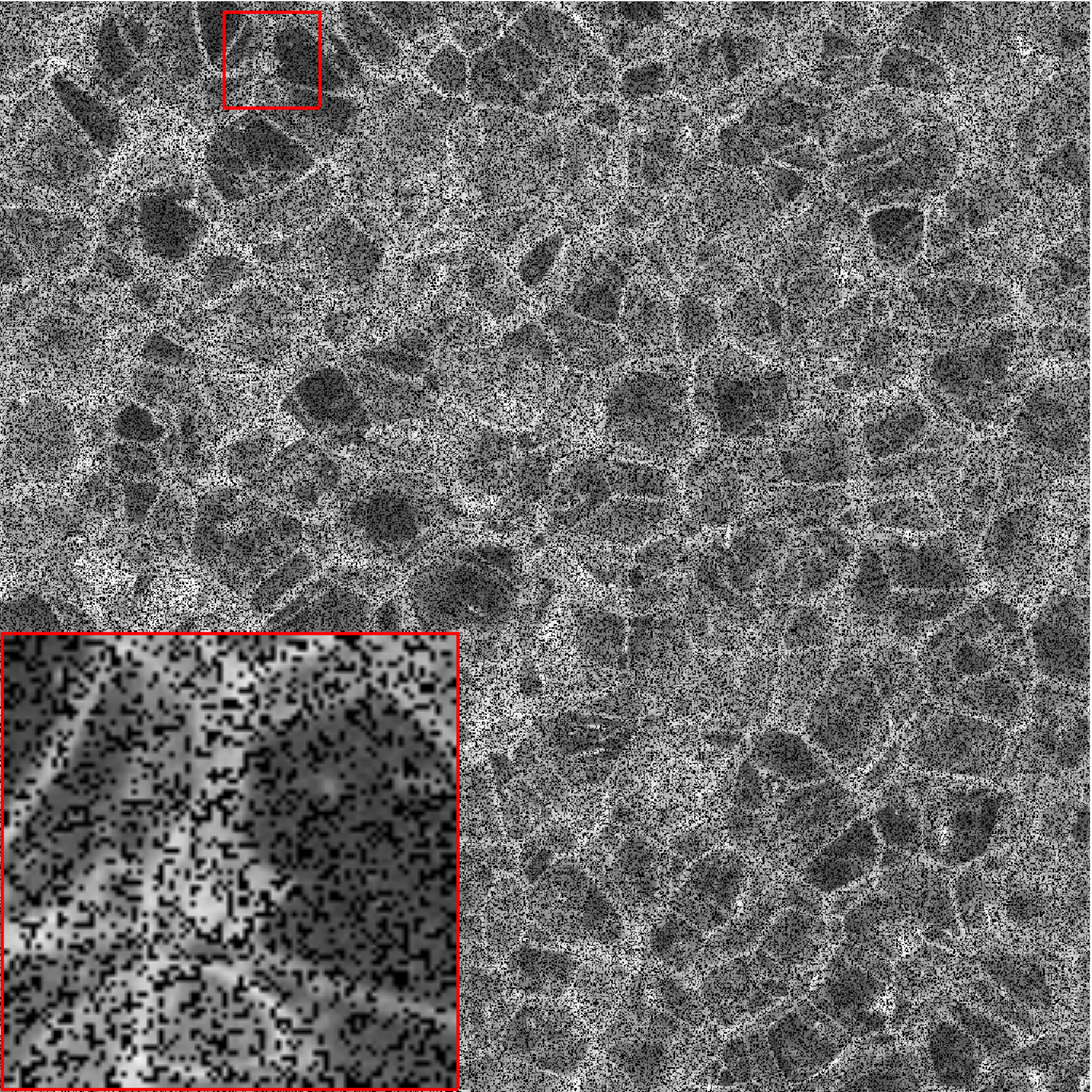}\vspace{0pt}
			\includegraphics[width=\linewidth]{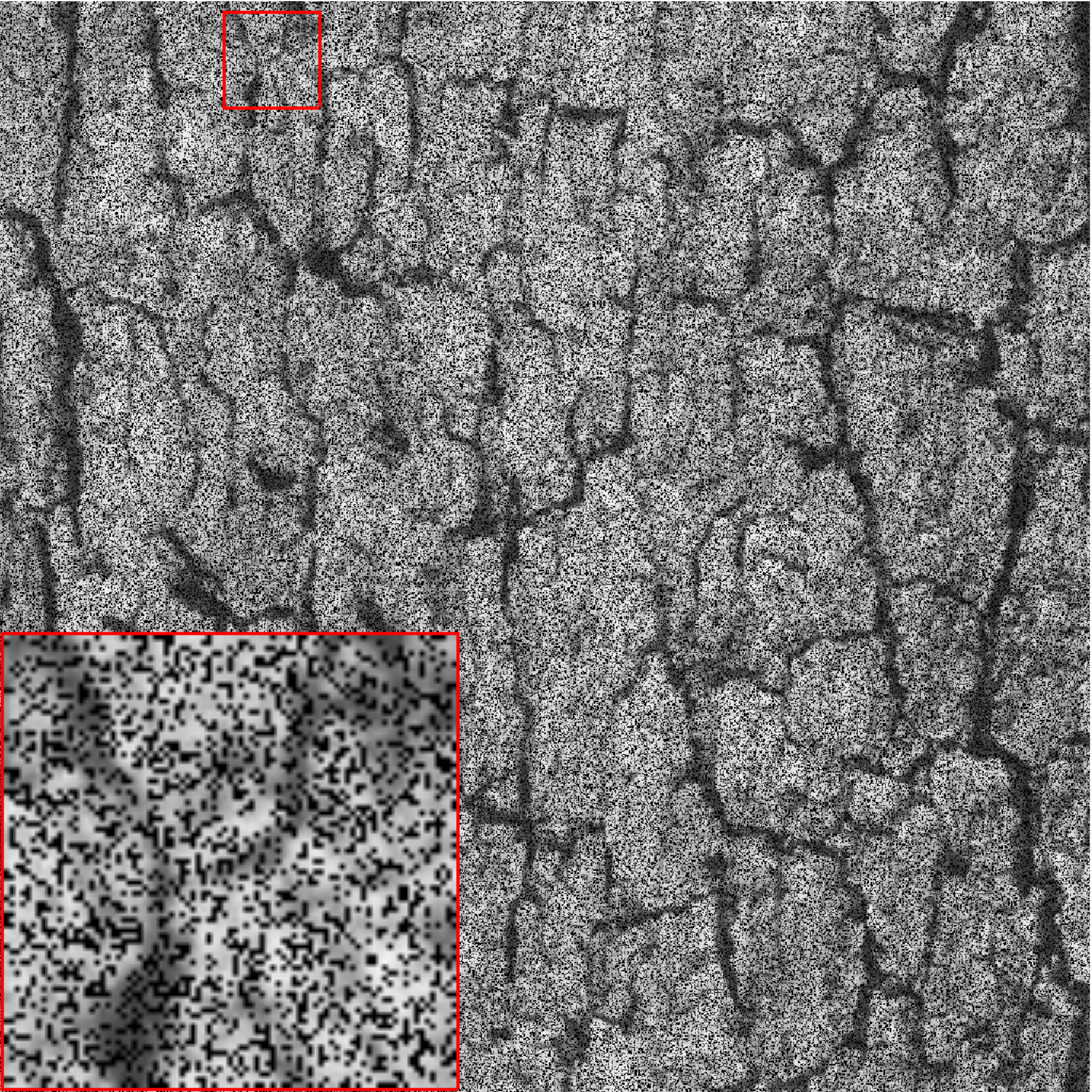}\vspace{0pt}
			\includegraphics[width=\linewidth]{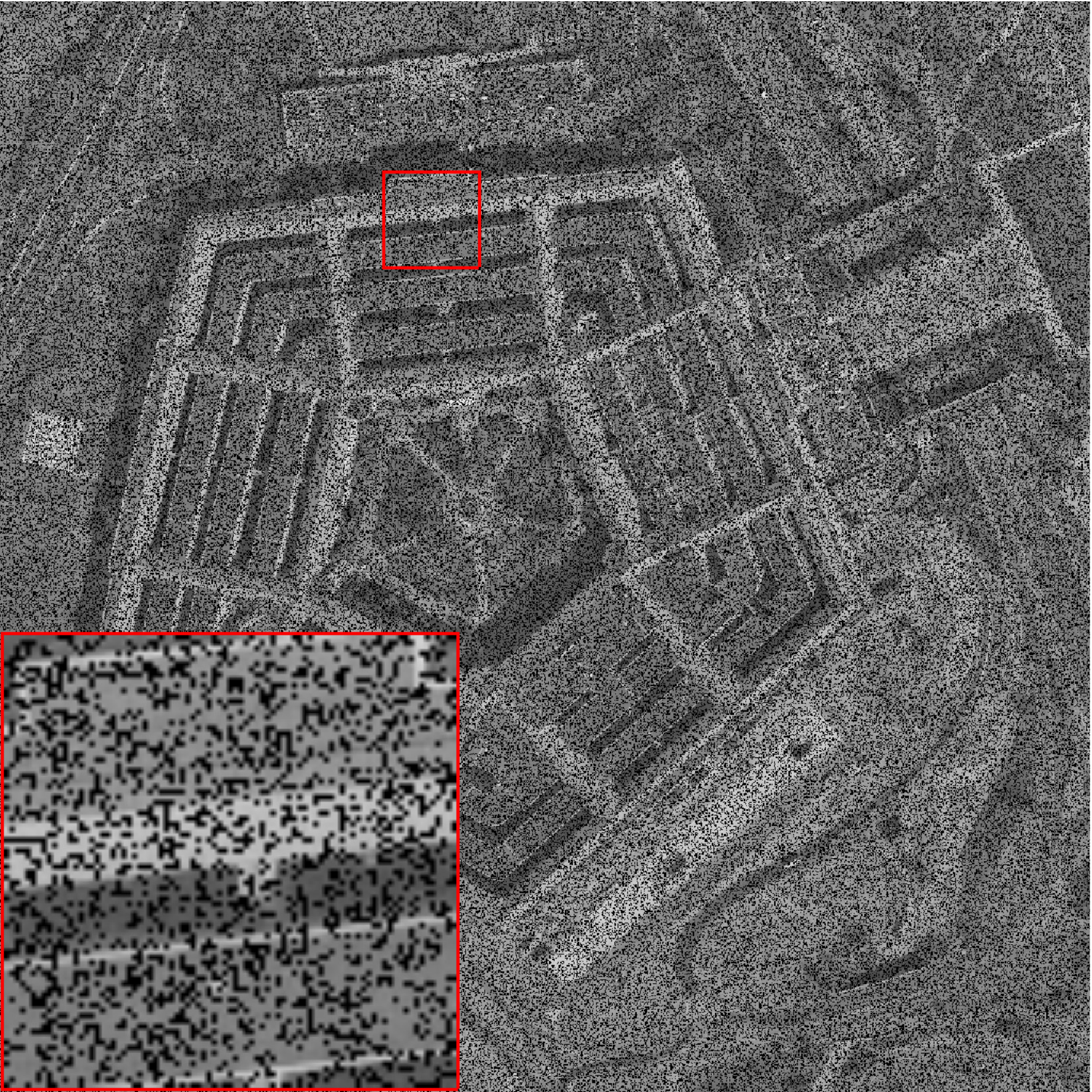}\vspace{0pt}
			\includegraphics[width=\linewidth]{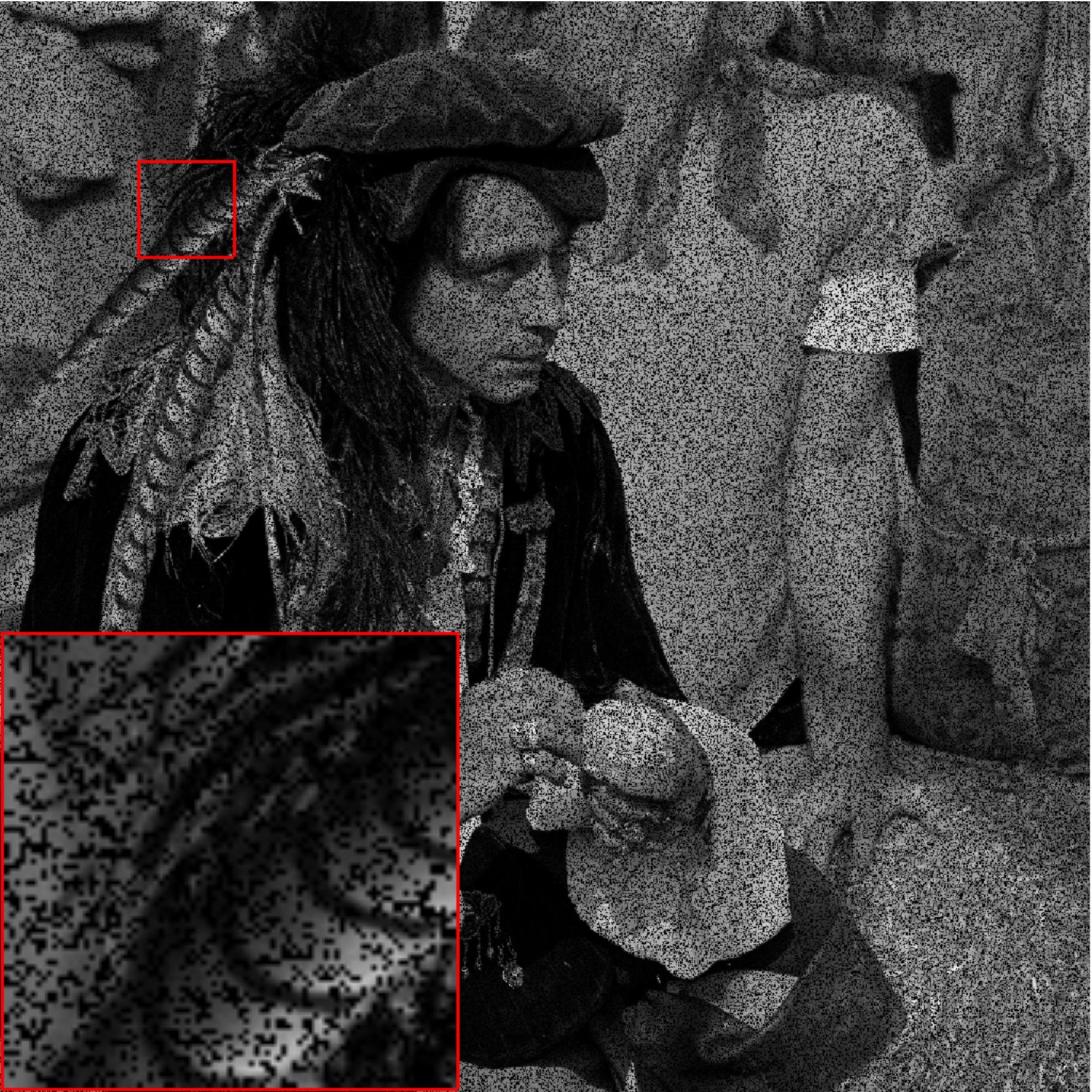}\vspace{0pt}
			\includegraphics[width=\linewidth]{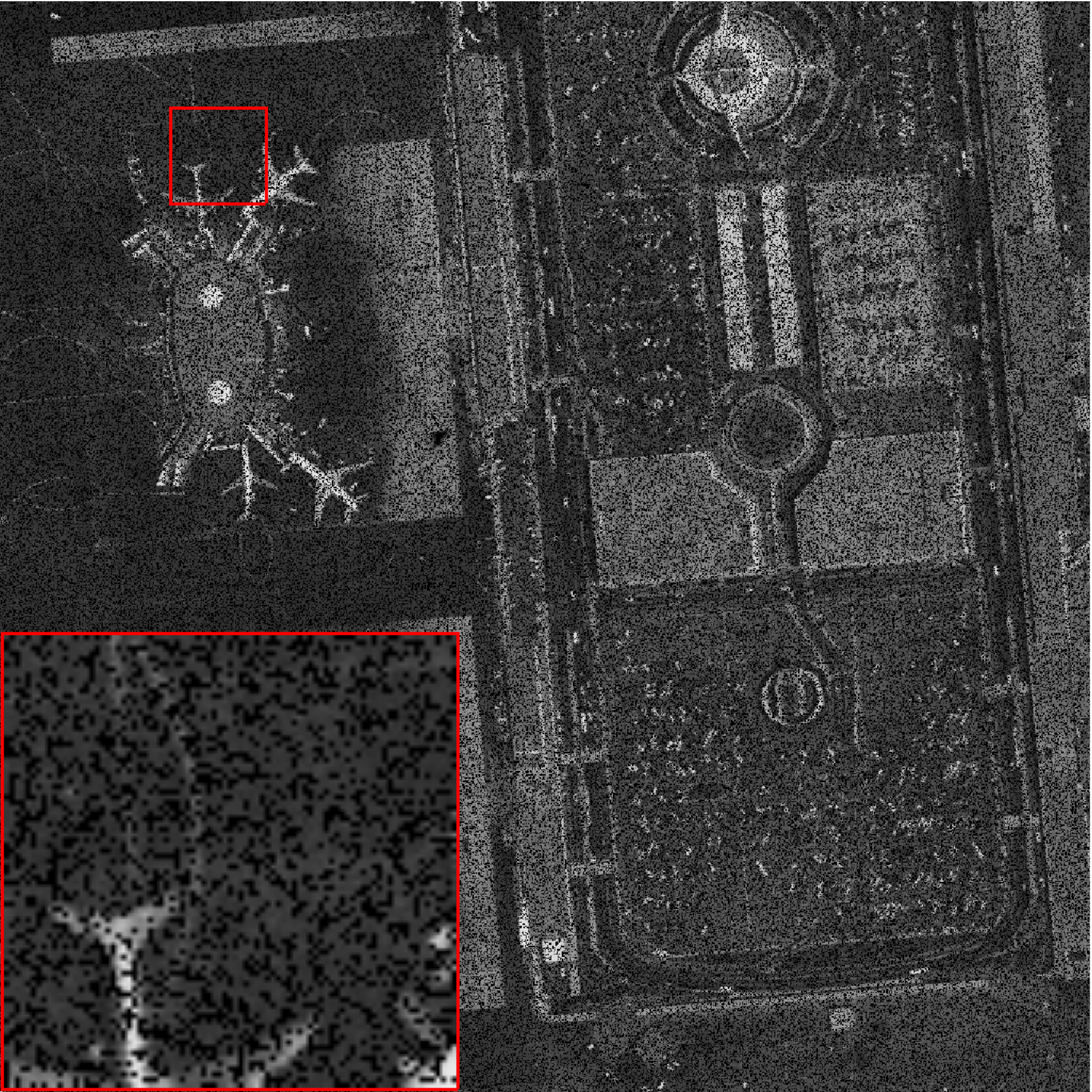}\vspace{0pt}
			\includegraphics[width=\linewidth]{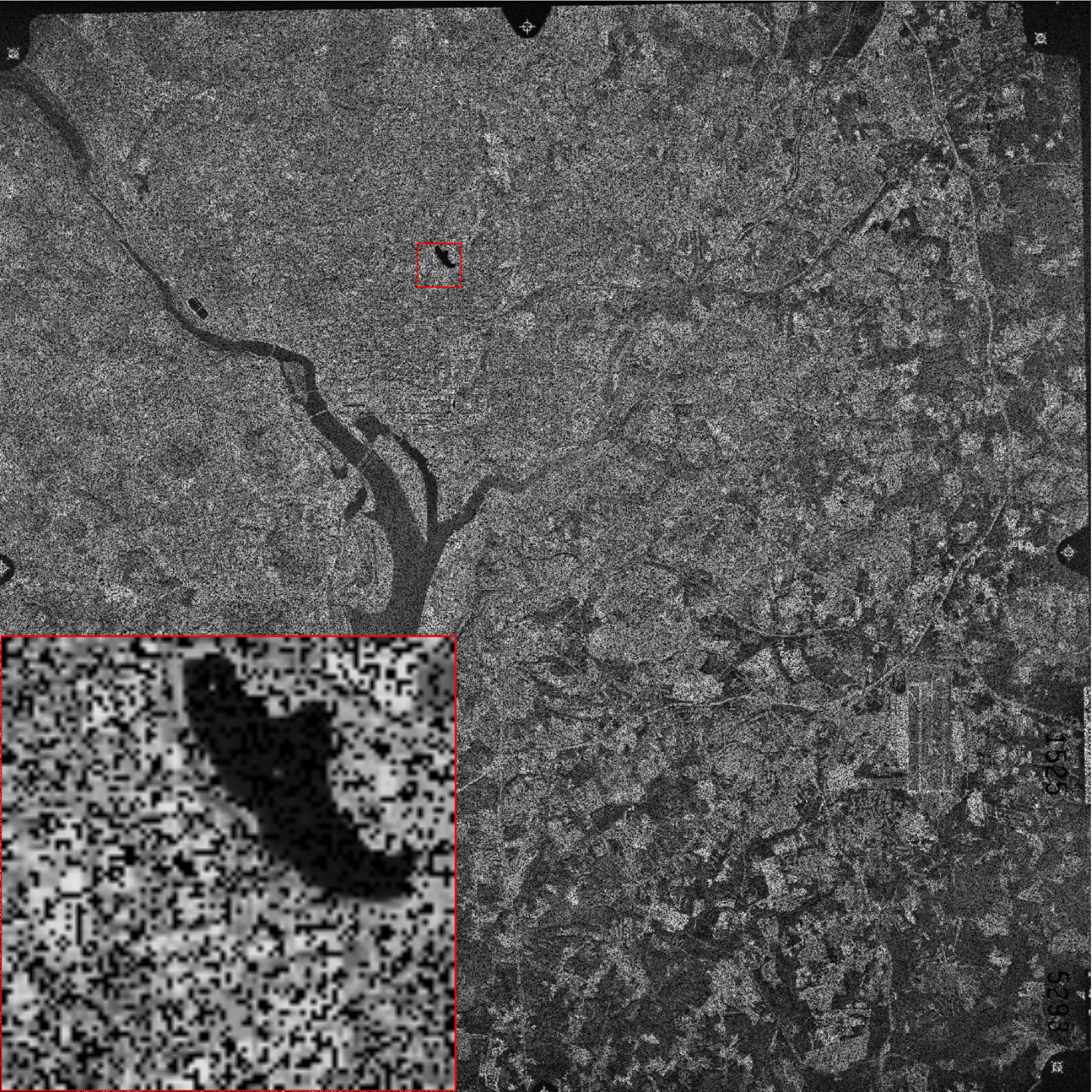}
			\caption{Observed}
		\end{subfigure}
		\begin{subfigure}[b]{0.138\linewidth}
			\centering
			\includegraphics[width=\linewidth]{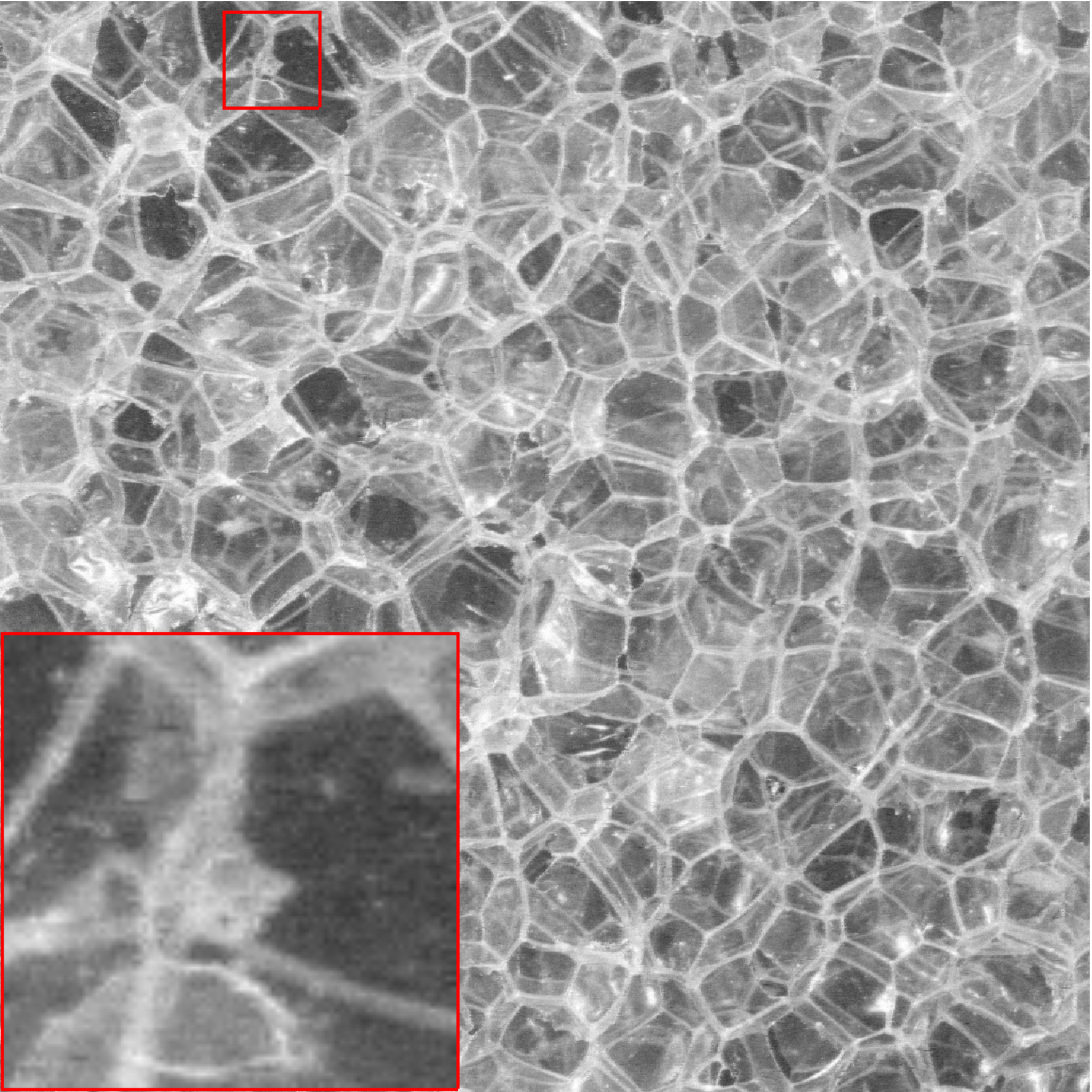}\vspace{0pt}
			\includegraphics[width=\linewidth]{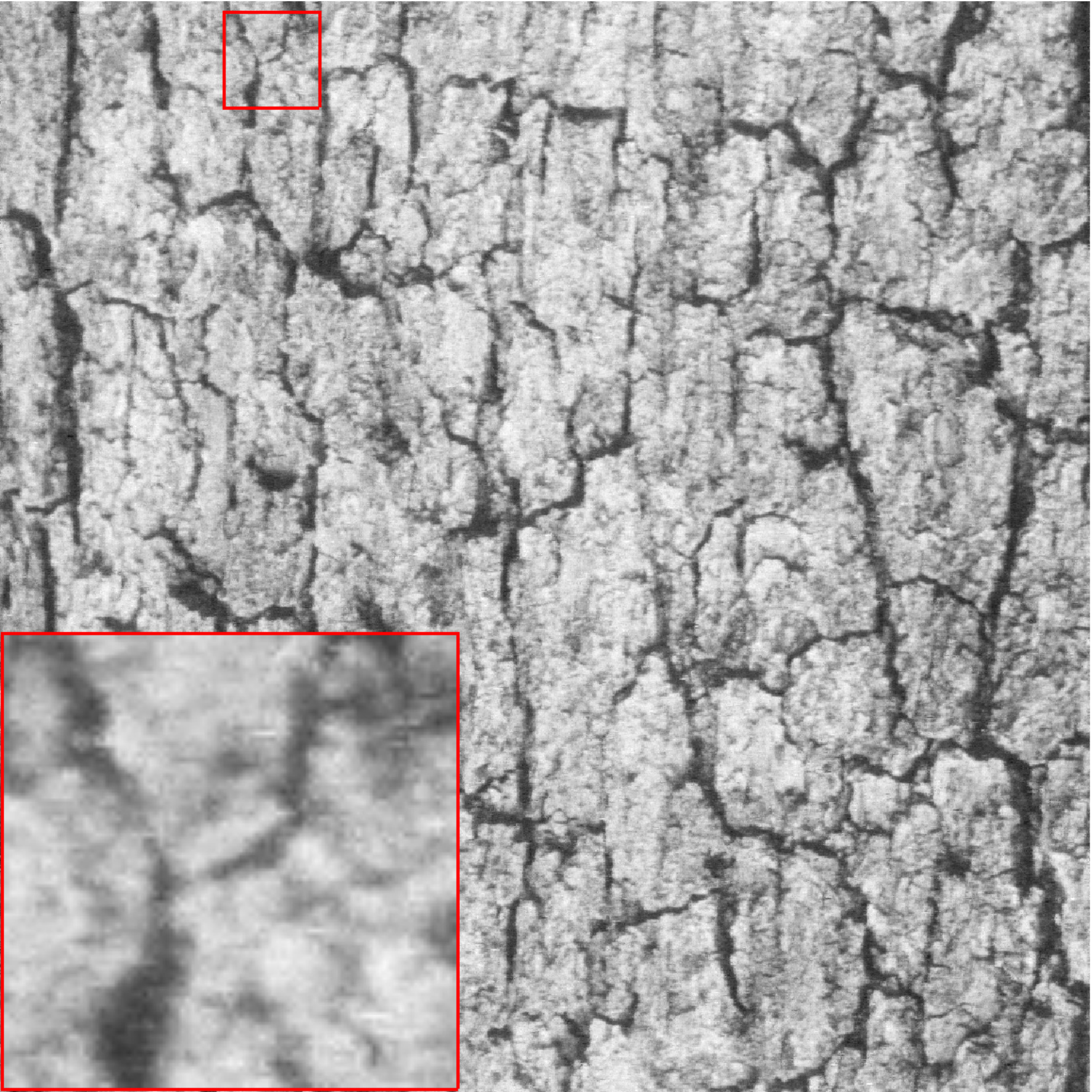}\vspace{0pt}
			\includegraphics[width=\linewidth]{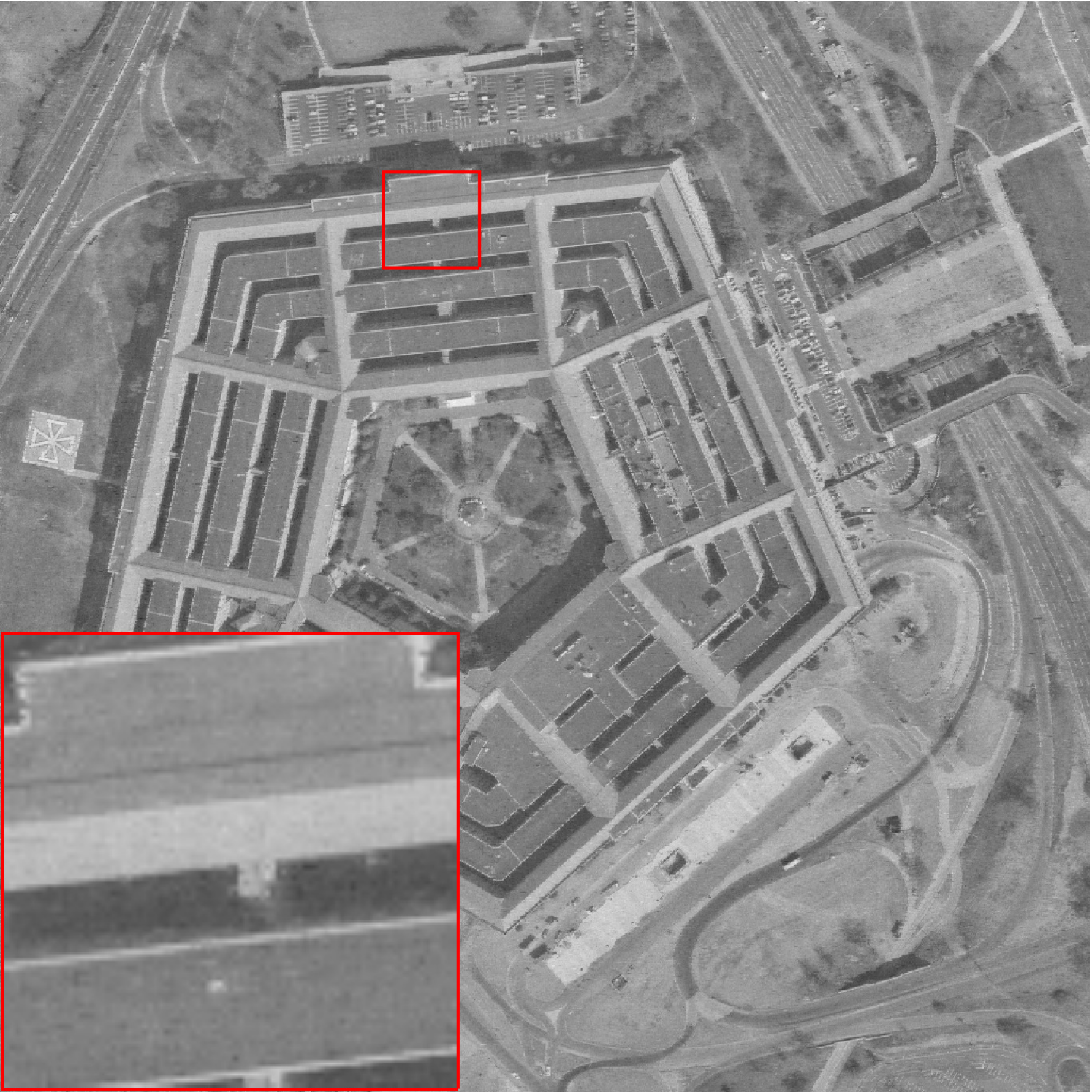}\vspace{0pt}
			\includegraphics[width=\linewidth]{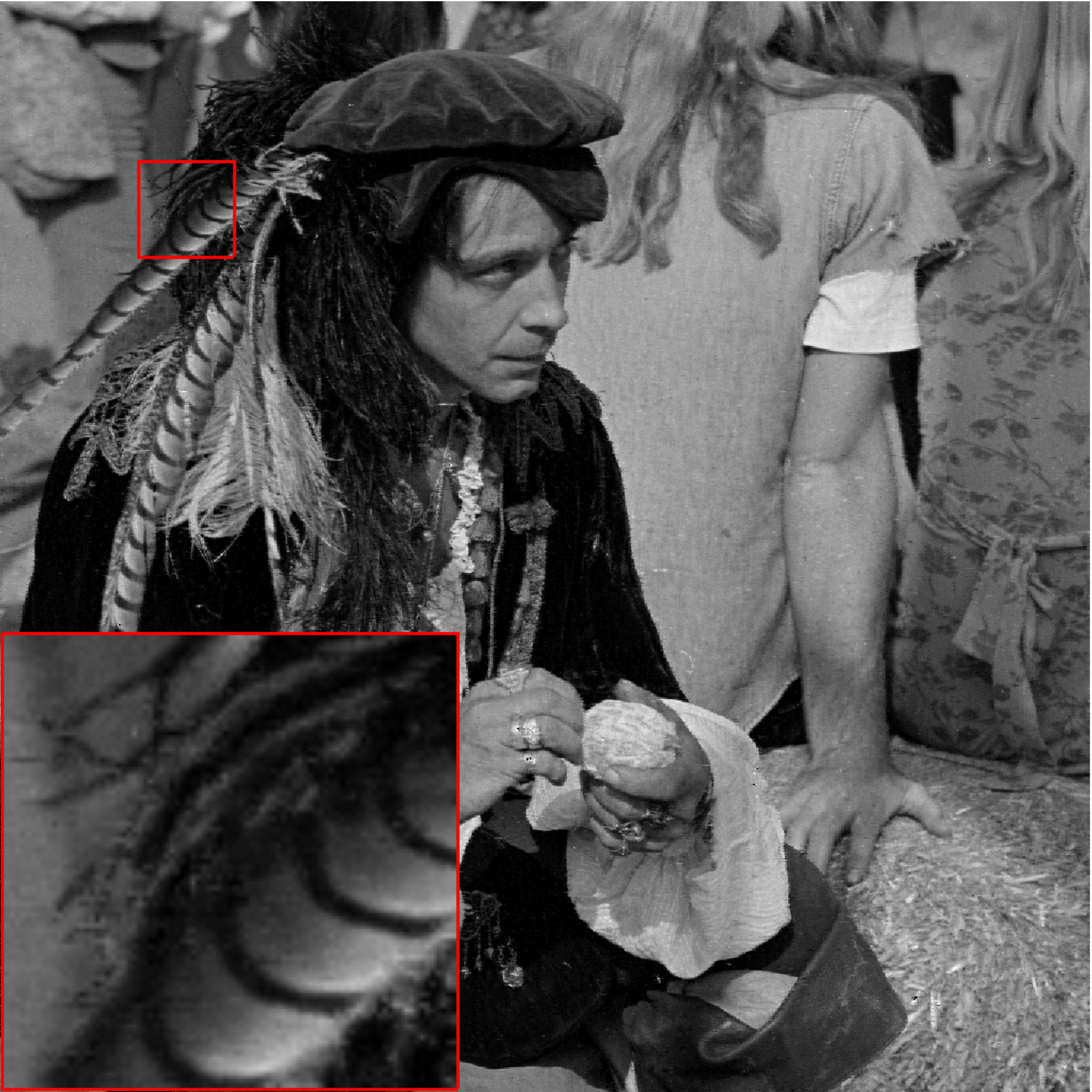}\vspace{0pt}
			\includegraphics[width=\linewidth]{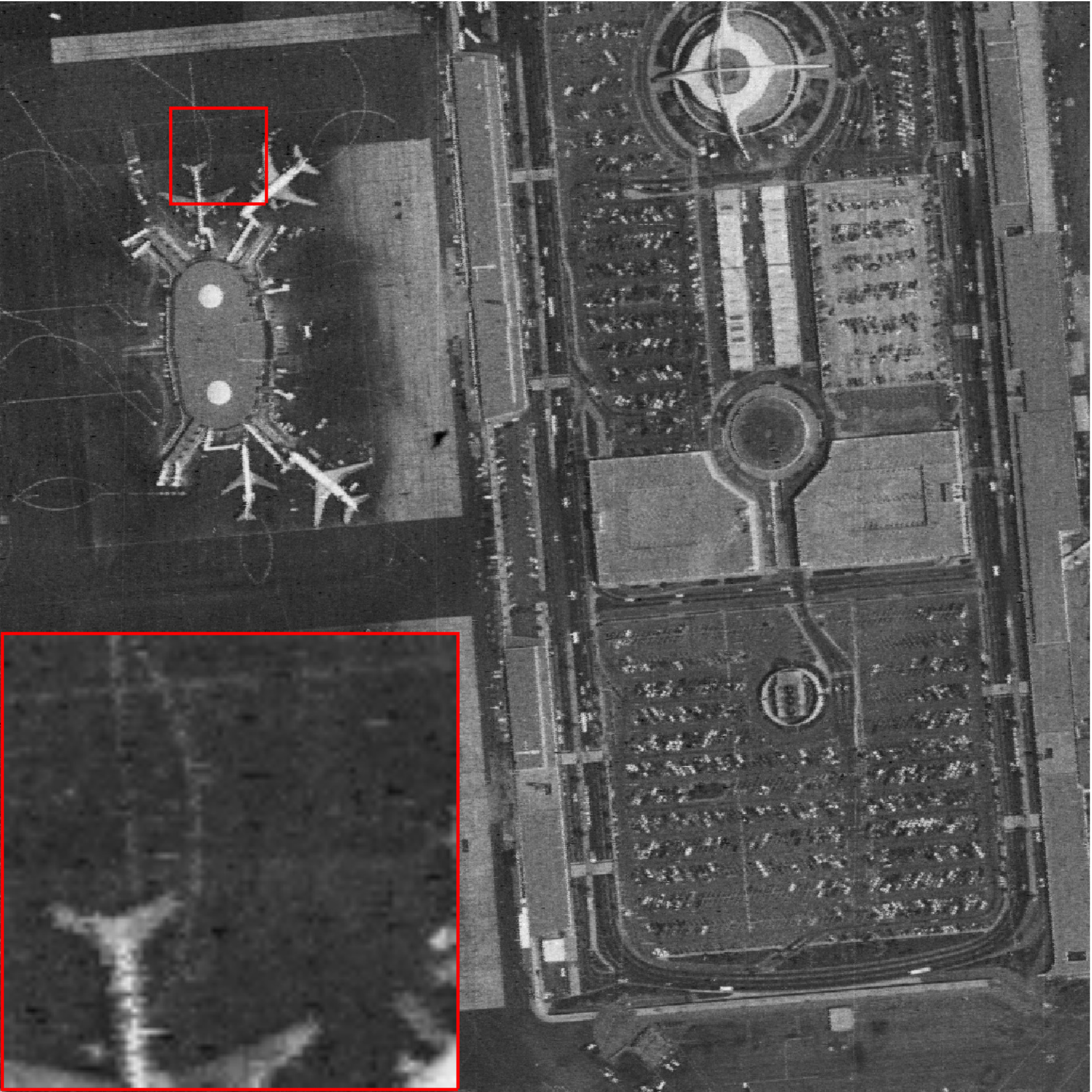}\vspace{0pt}
			\includegraphics[width=\linewidth]{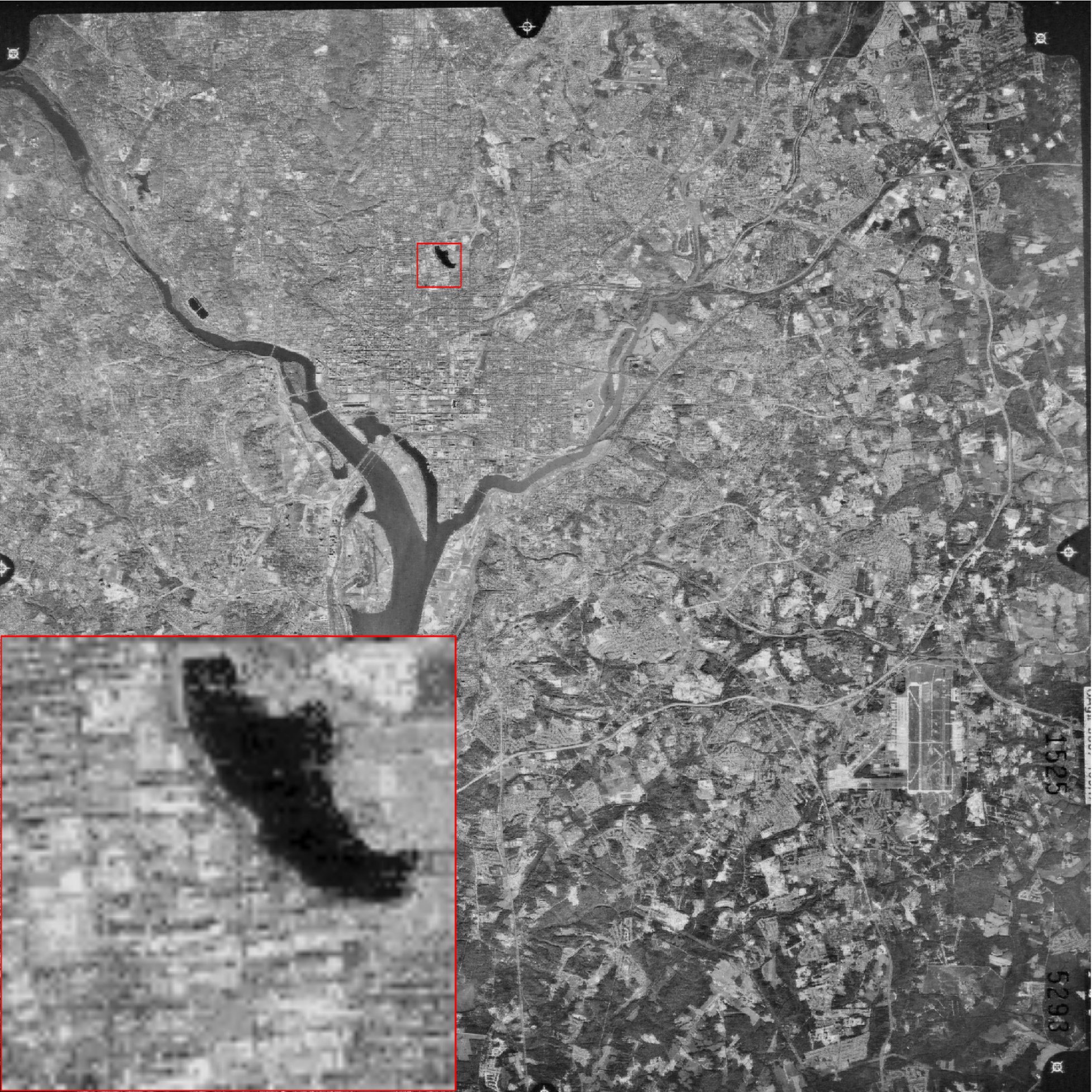}
			\caption{TCTF-M}
		\end{subfigure}
		\begin{subfigure}[b]{0.138\linewidth}
			\centering
			\includegraphics[width=\linewidth]{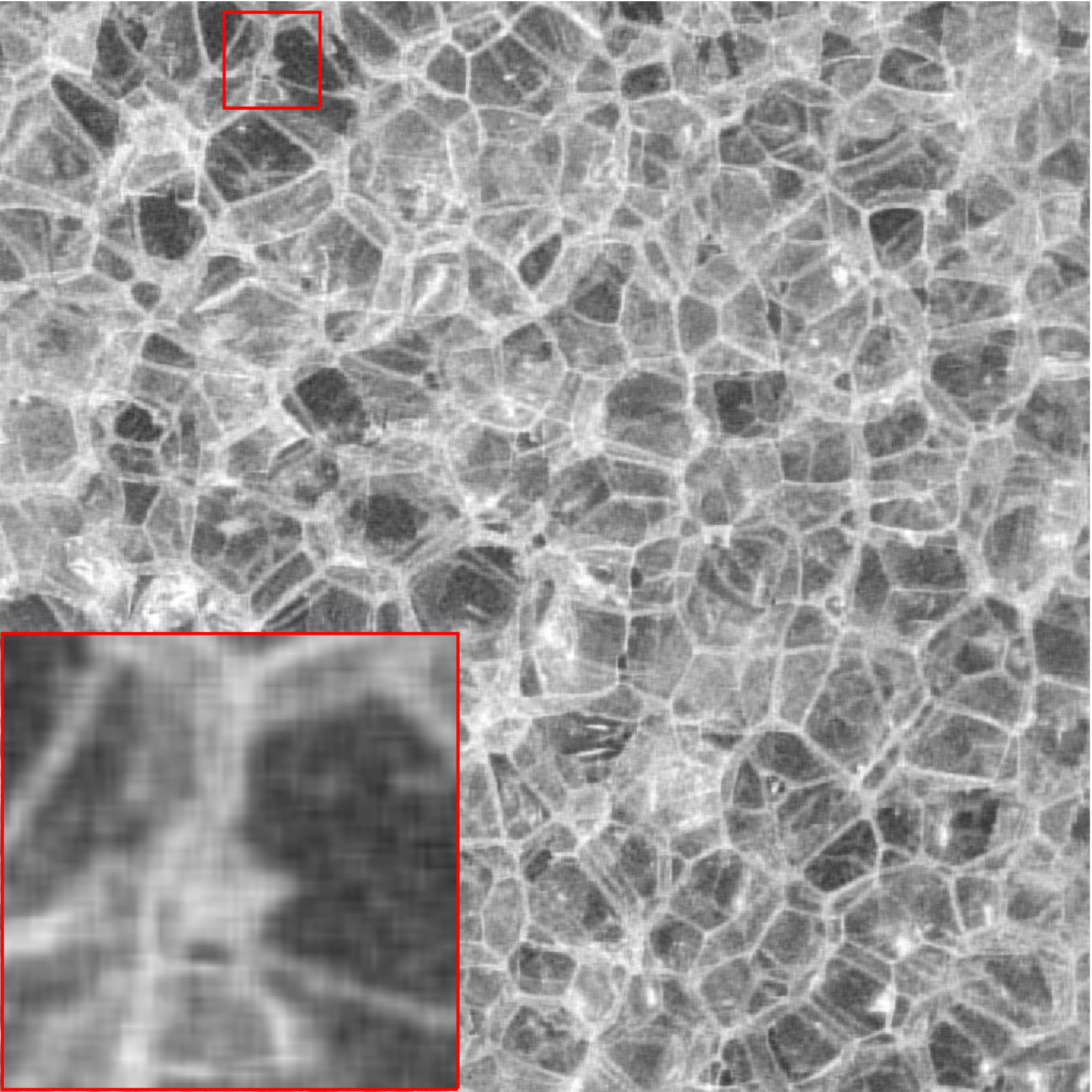}\vspace{0pt}
			\includegraphics[width=\linewidth]{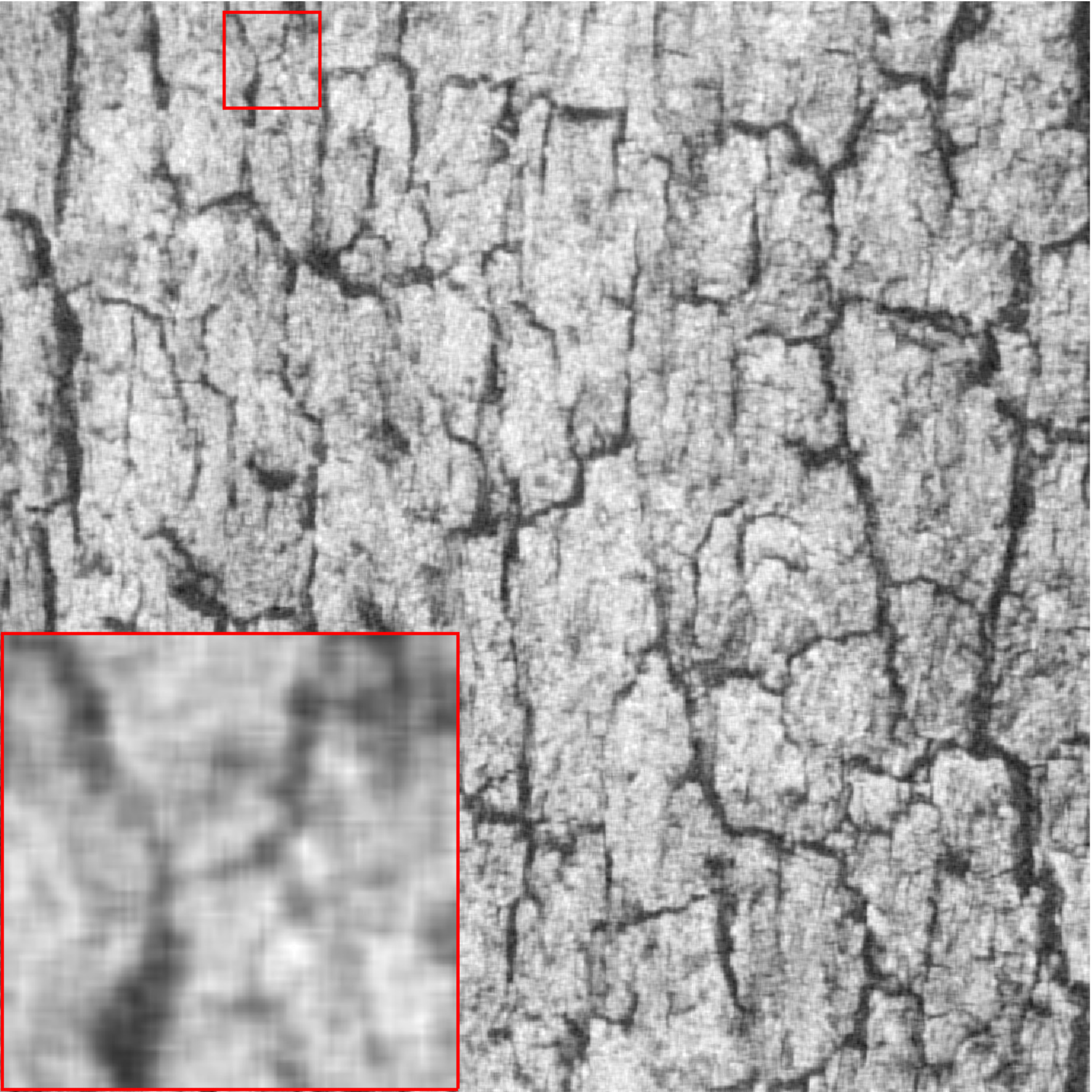}\vspace{0pt}
			\includegraphics[width=\linewidth]{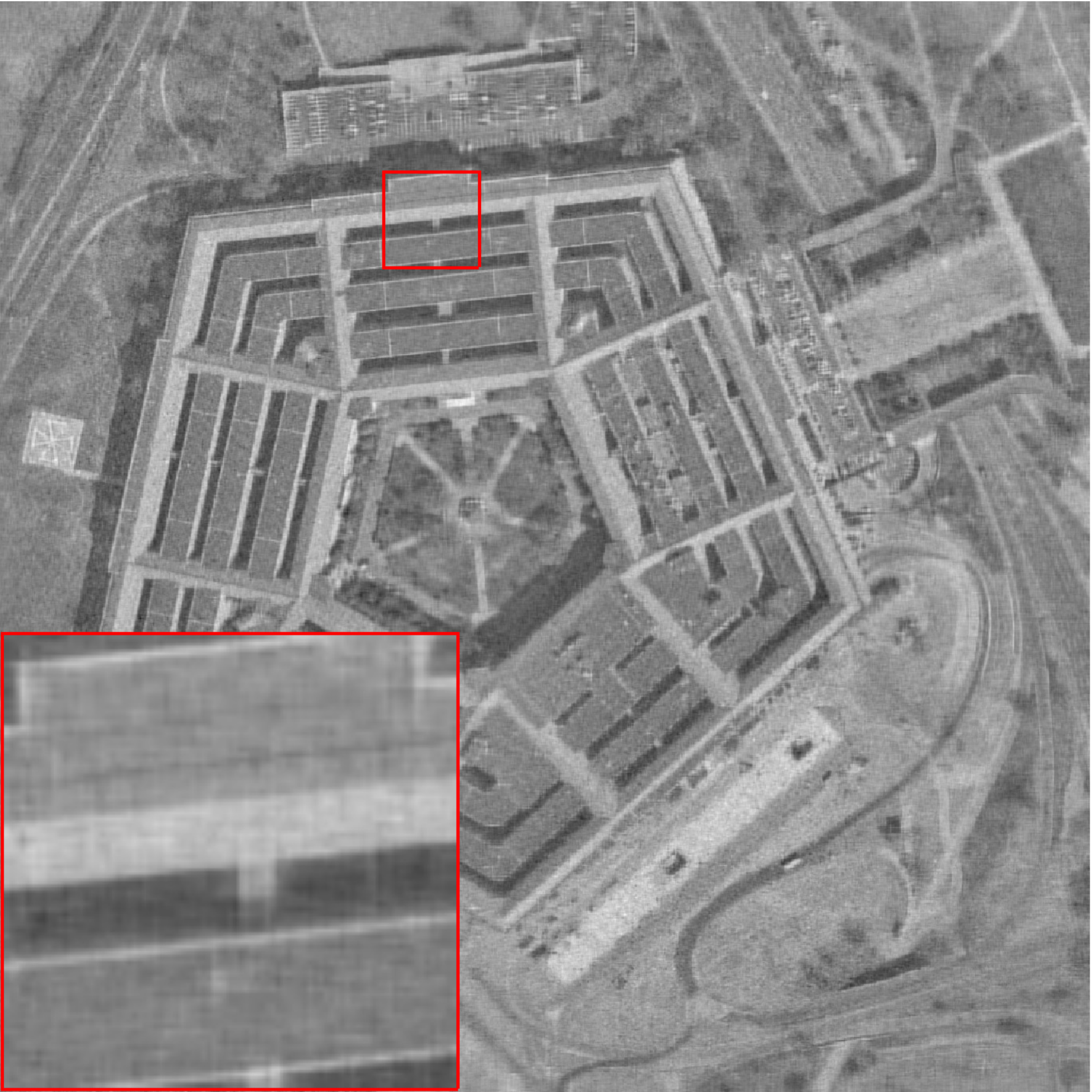}\vspace{0pt}
			\includegraphics[width=\linewidth]{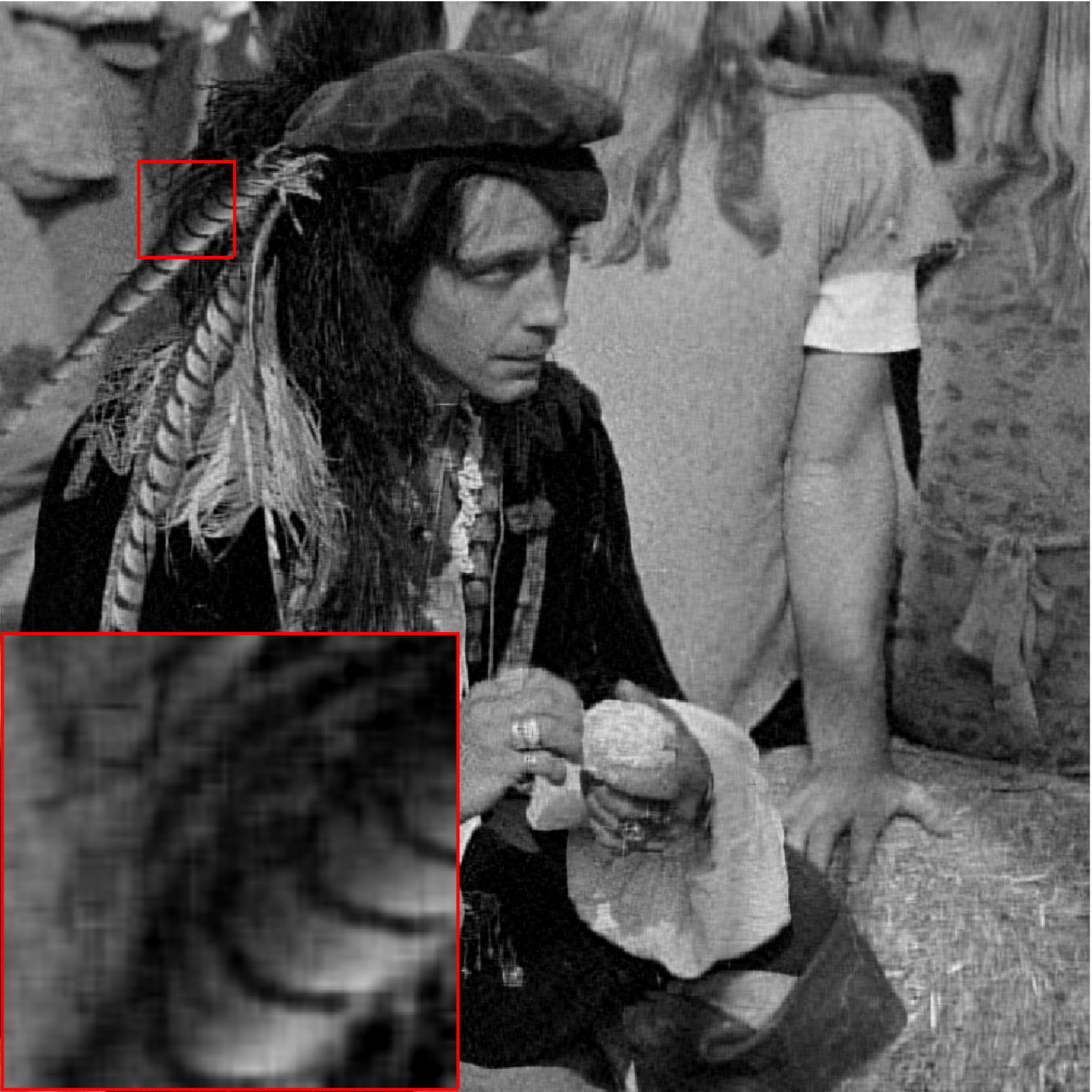}\vspace{0pt}
			\includegraphics[width=\linewidth]{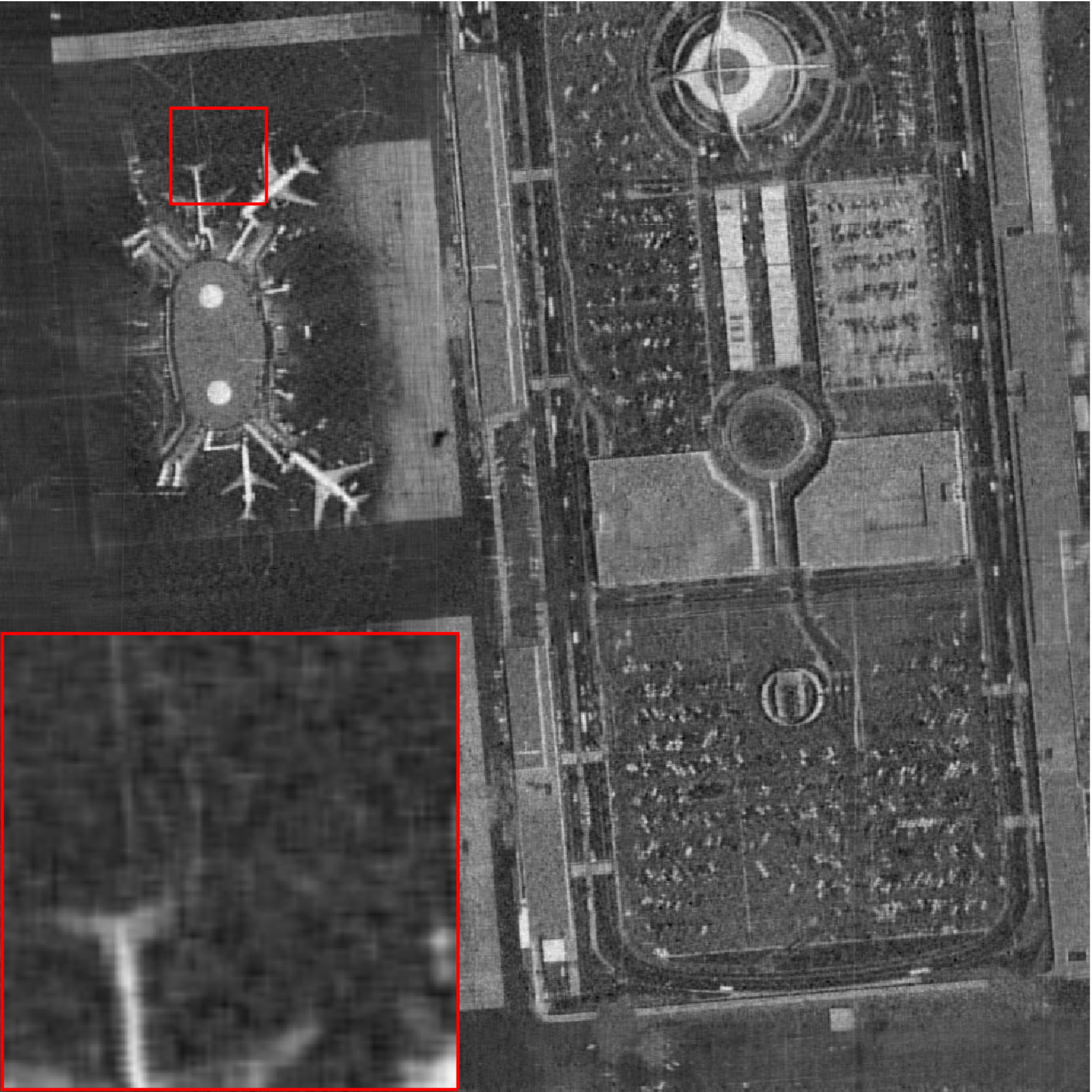}\vspace{0pt}
			\includegraphics[width=\linewidth]{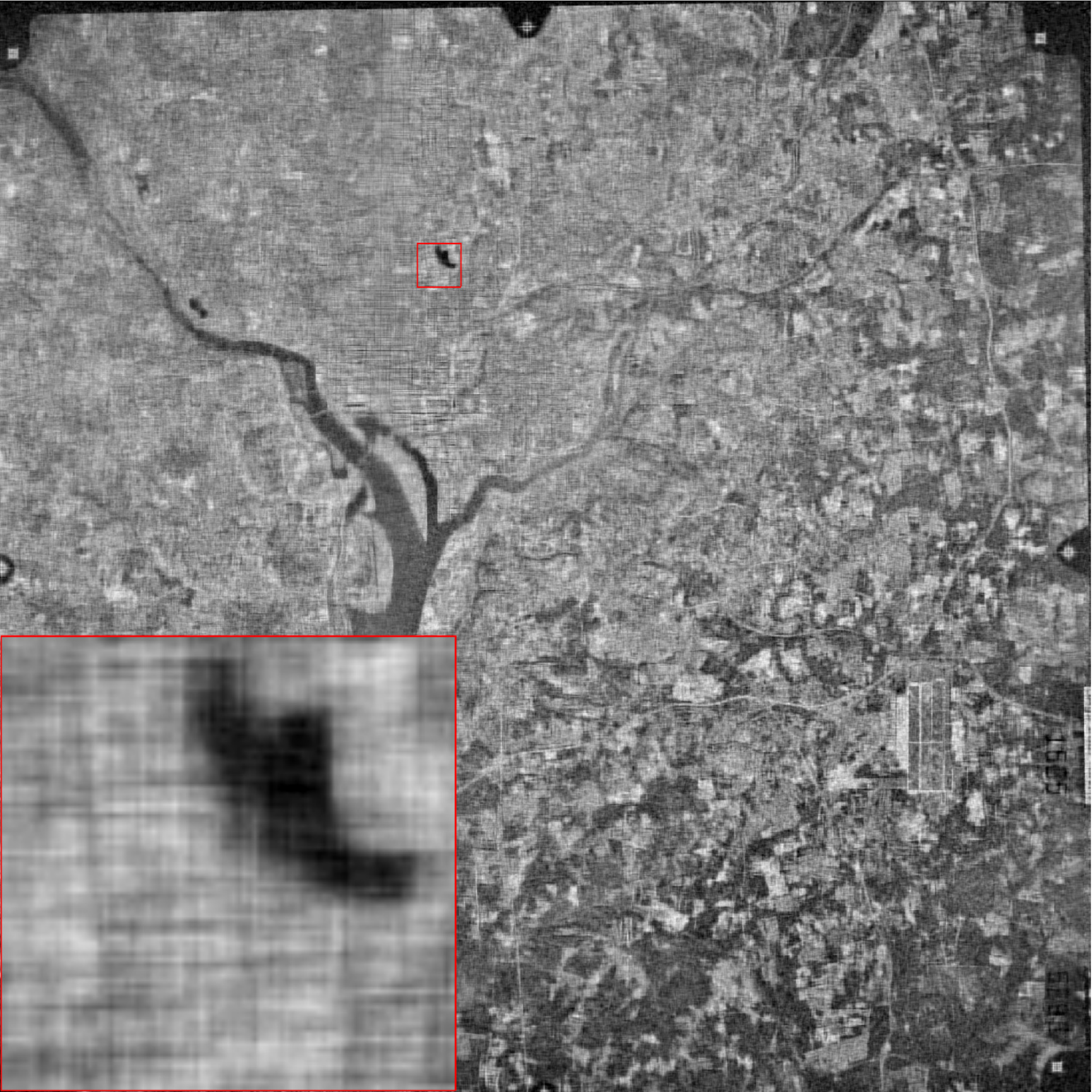}
			\caption{SRMF}
		\end{subfigure}	
		\begin{subfigure}[b]{0.138\linewidth}
			\centering
			\includegraphics[width=\linewidth]{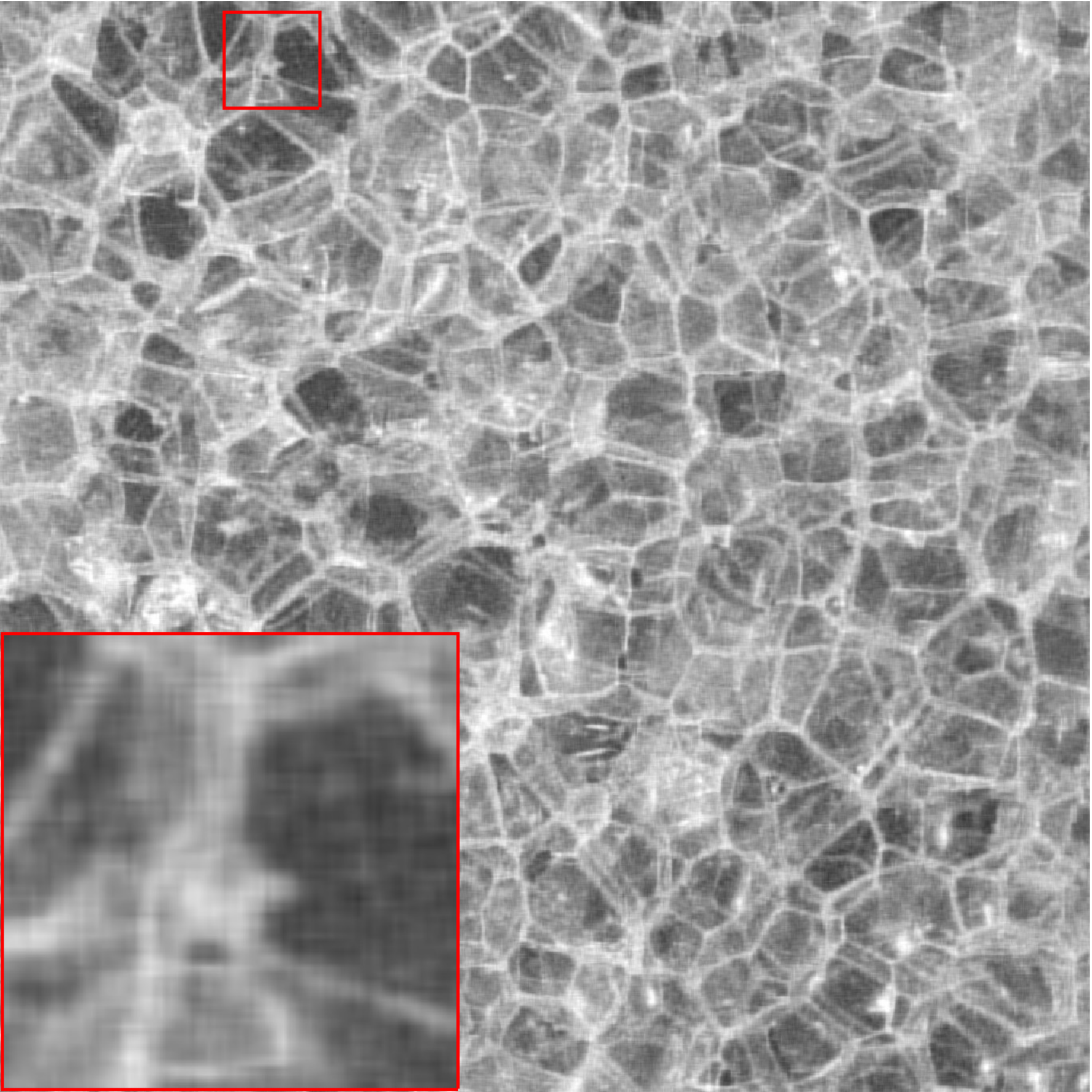}\vspace{0pt}
			\includegraphics[width=\linewidth]{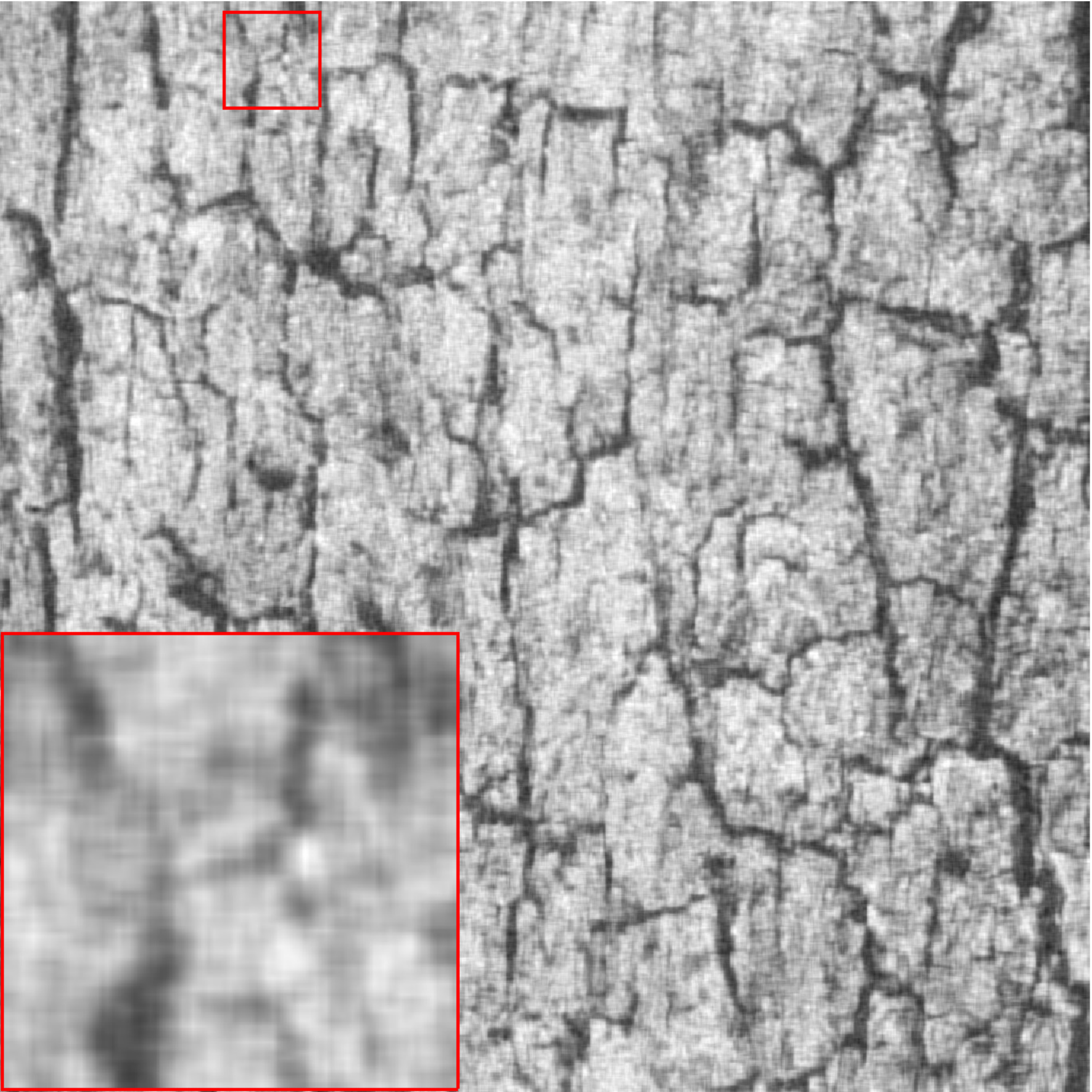}\vspace{0pt}		
			\includegraphics[width=\linewidth]{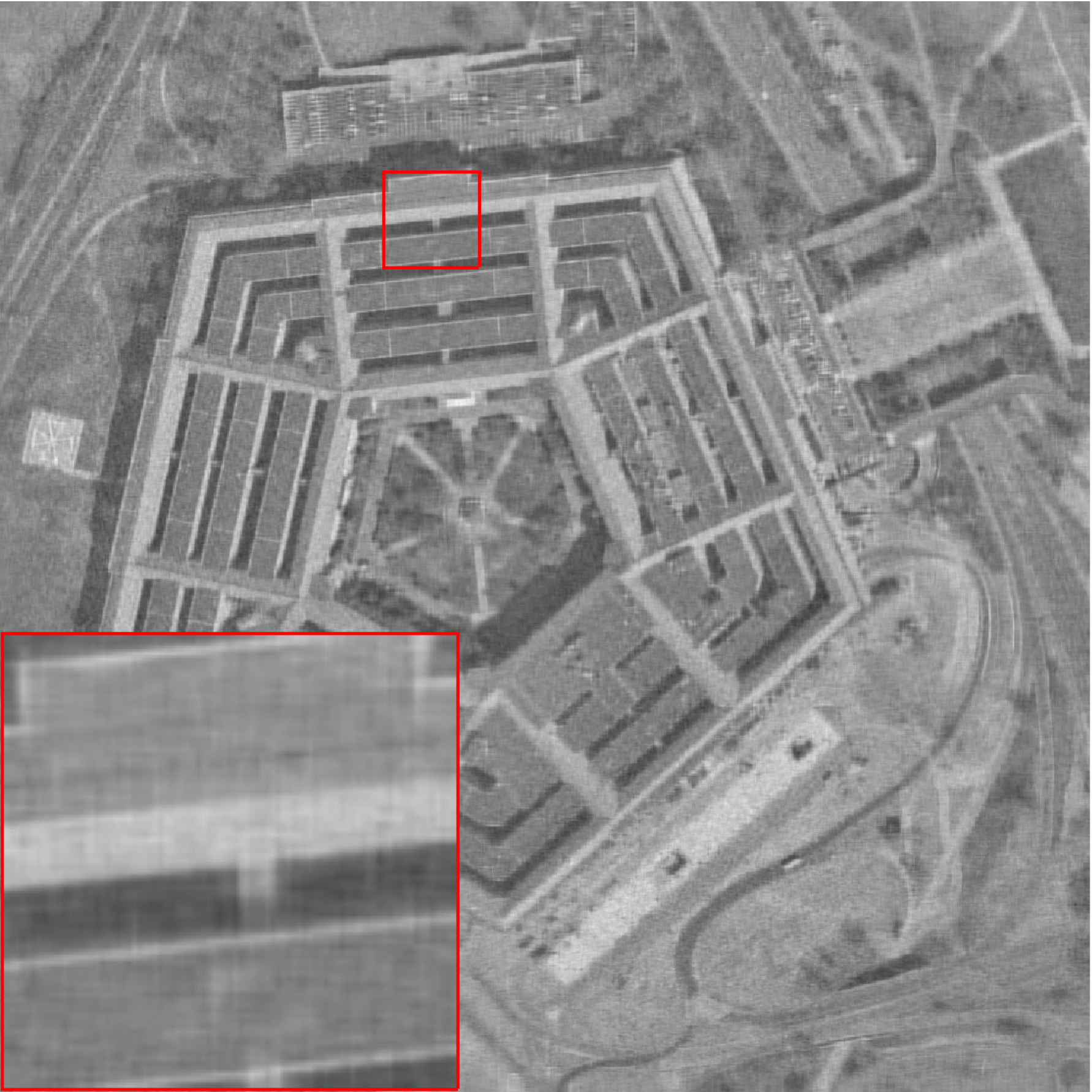}\vspace{0pt}
			\includegraphics[width=\linewidth]{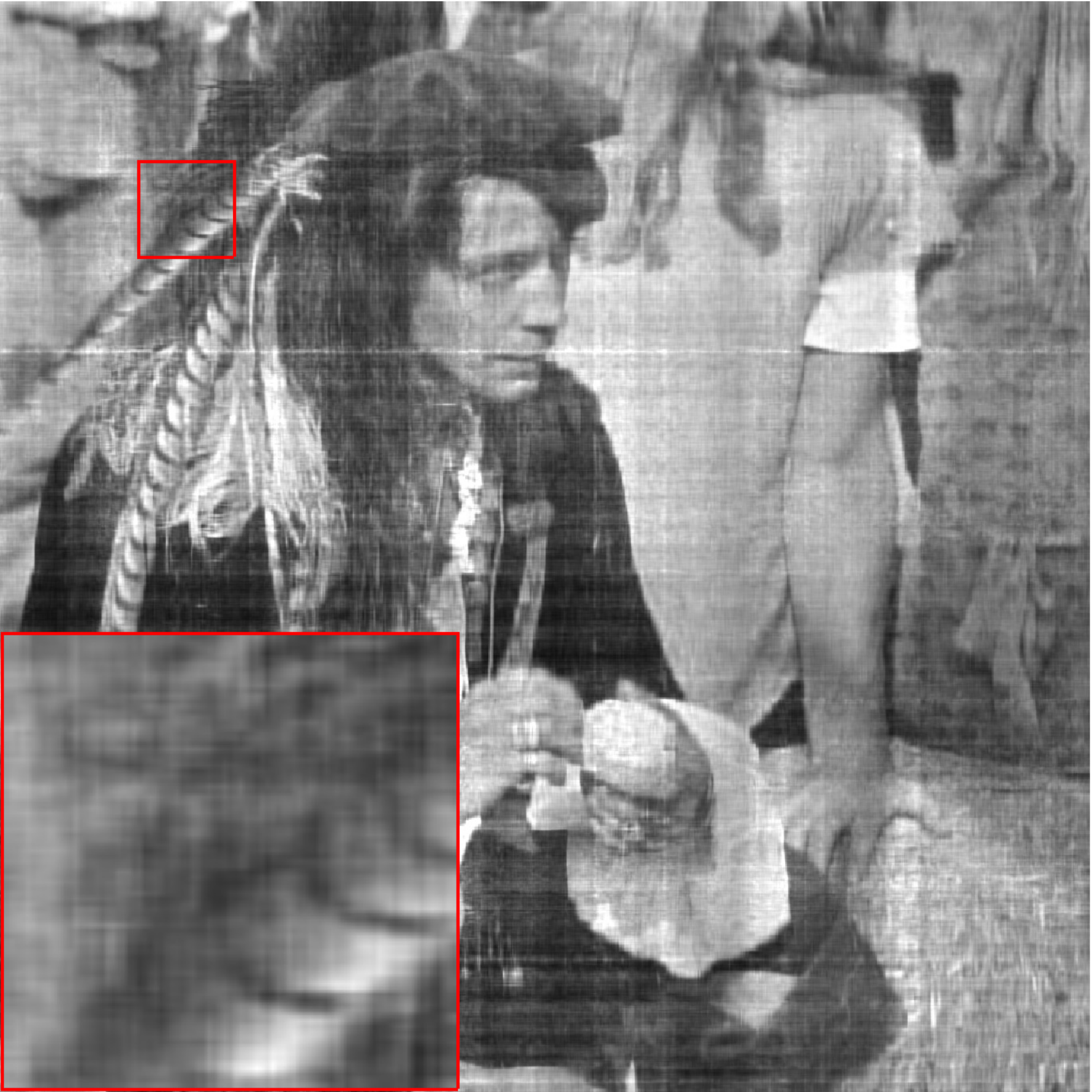}\vspace{0pt}
			\includegraphics[width=\linewidth]{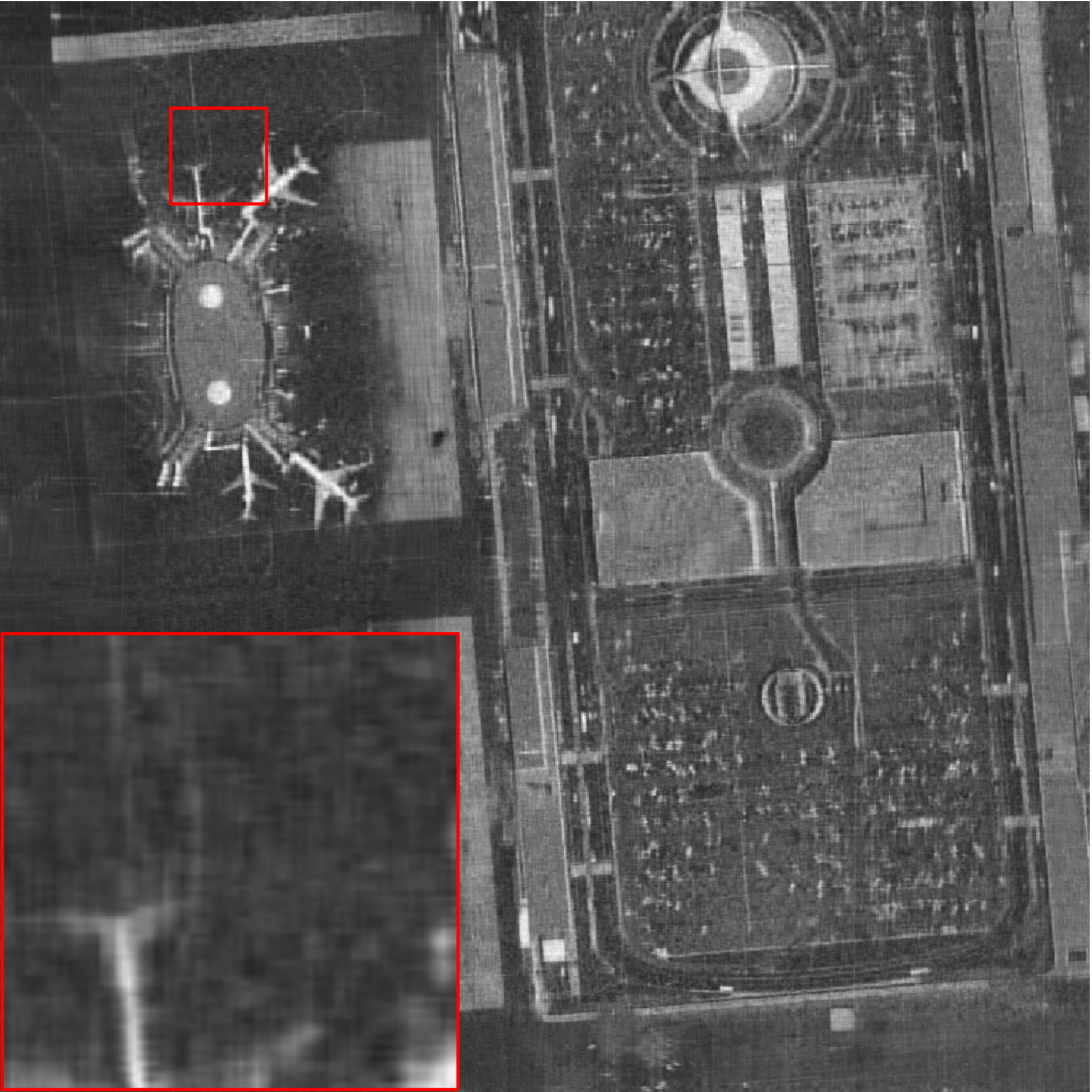}\vspace{0pt}
			\includegraphics[width=\linewidth]{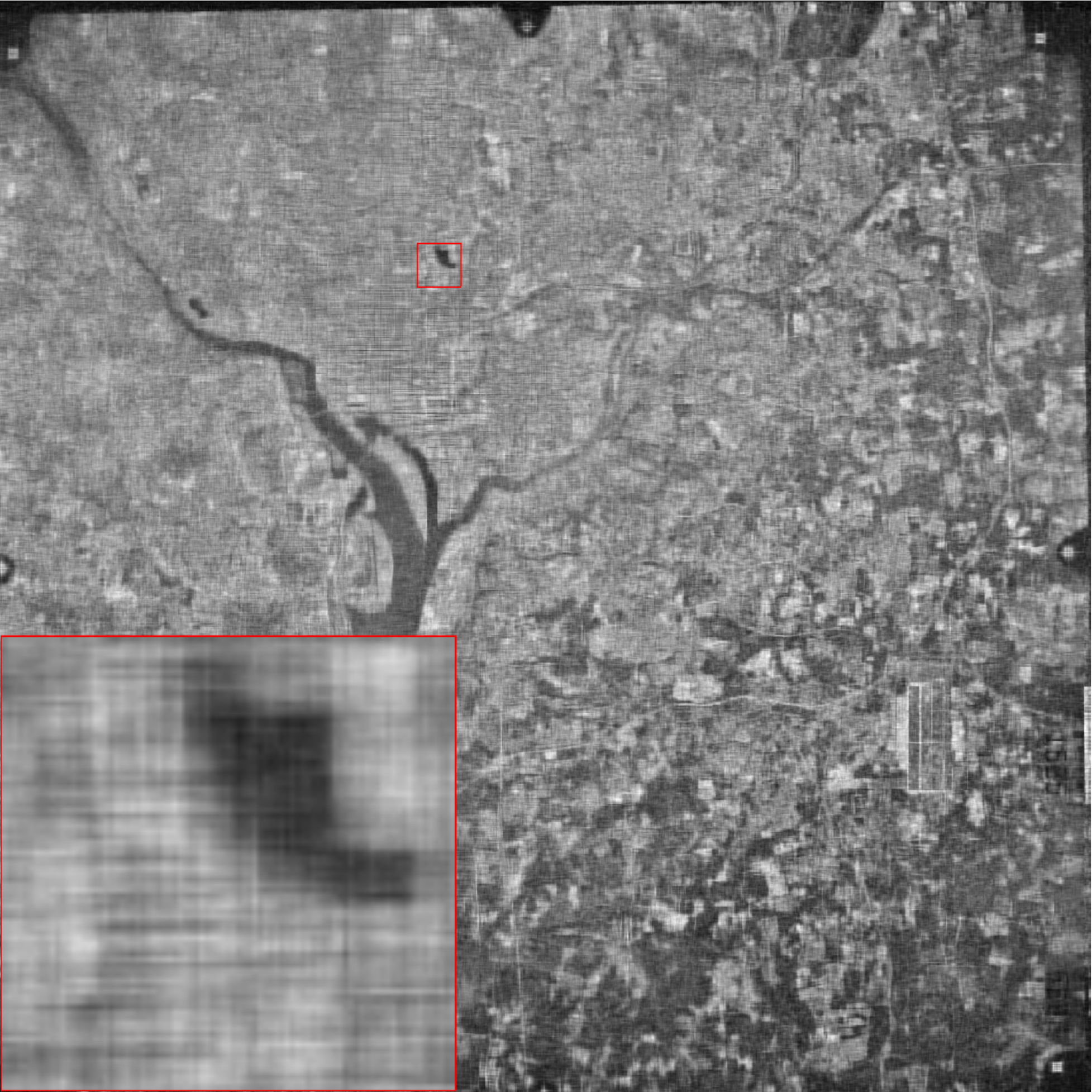}
			\caption{MC-NMF}
		\end{subfigure}
		\begin{subfigure}[b]{0.138\linewidth}
		   \centering
		   \includegraphics[width=\linewidth]{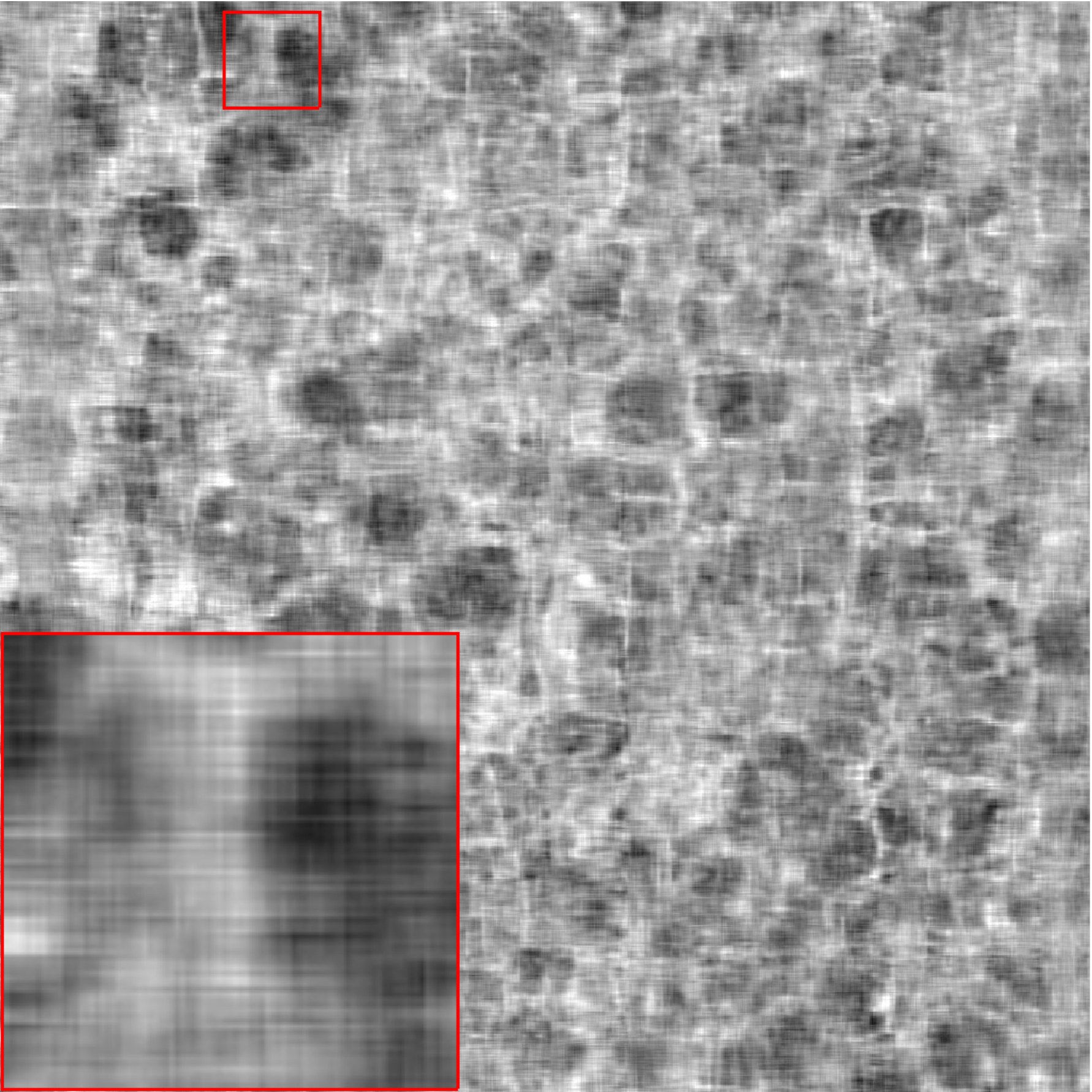}\vspace{0pt}
		   \includegraphics[width=\linewidth]{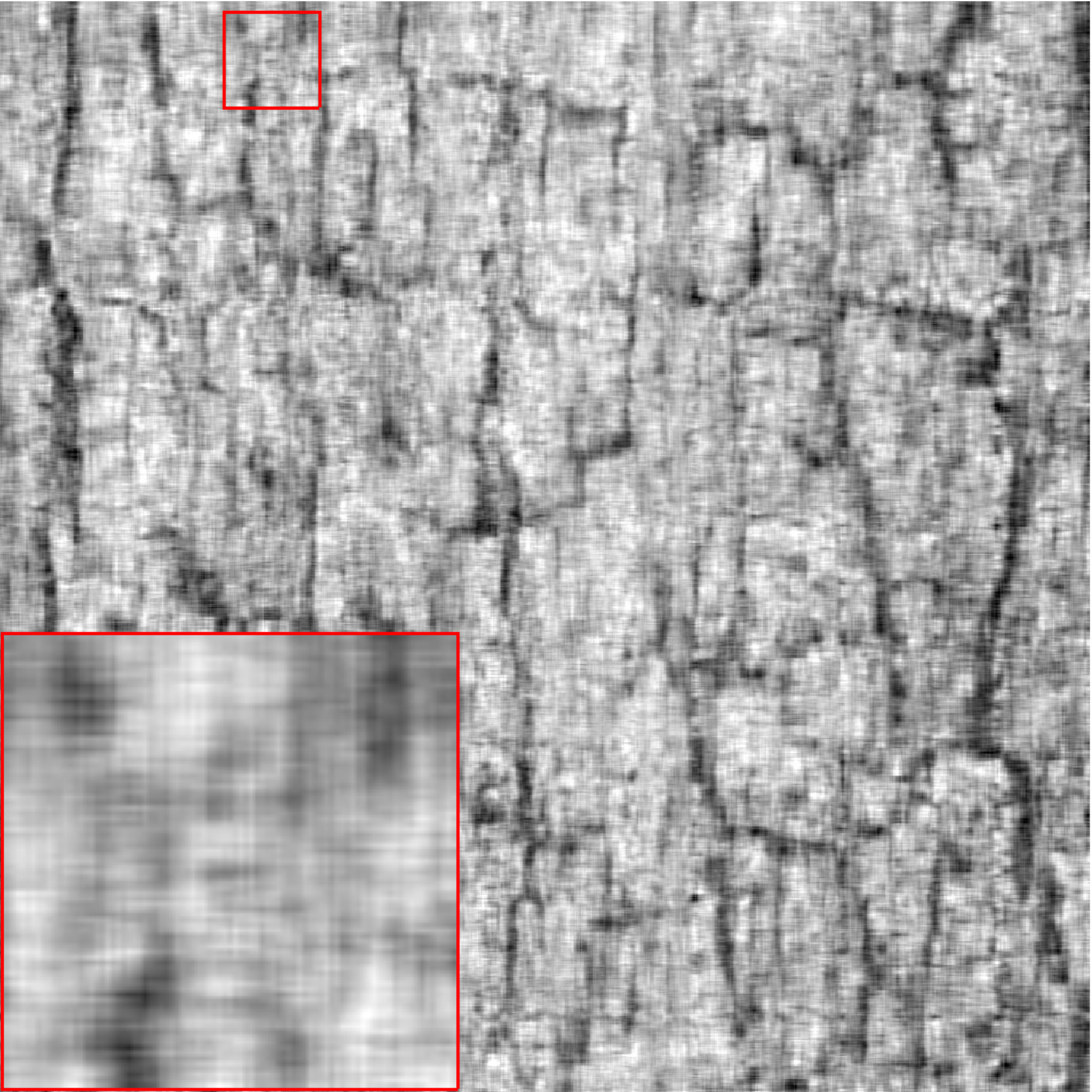}\vspace{0pt}		
		   \includegraphics[width=\linewidth]{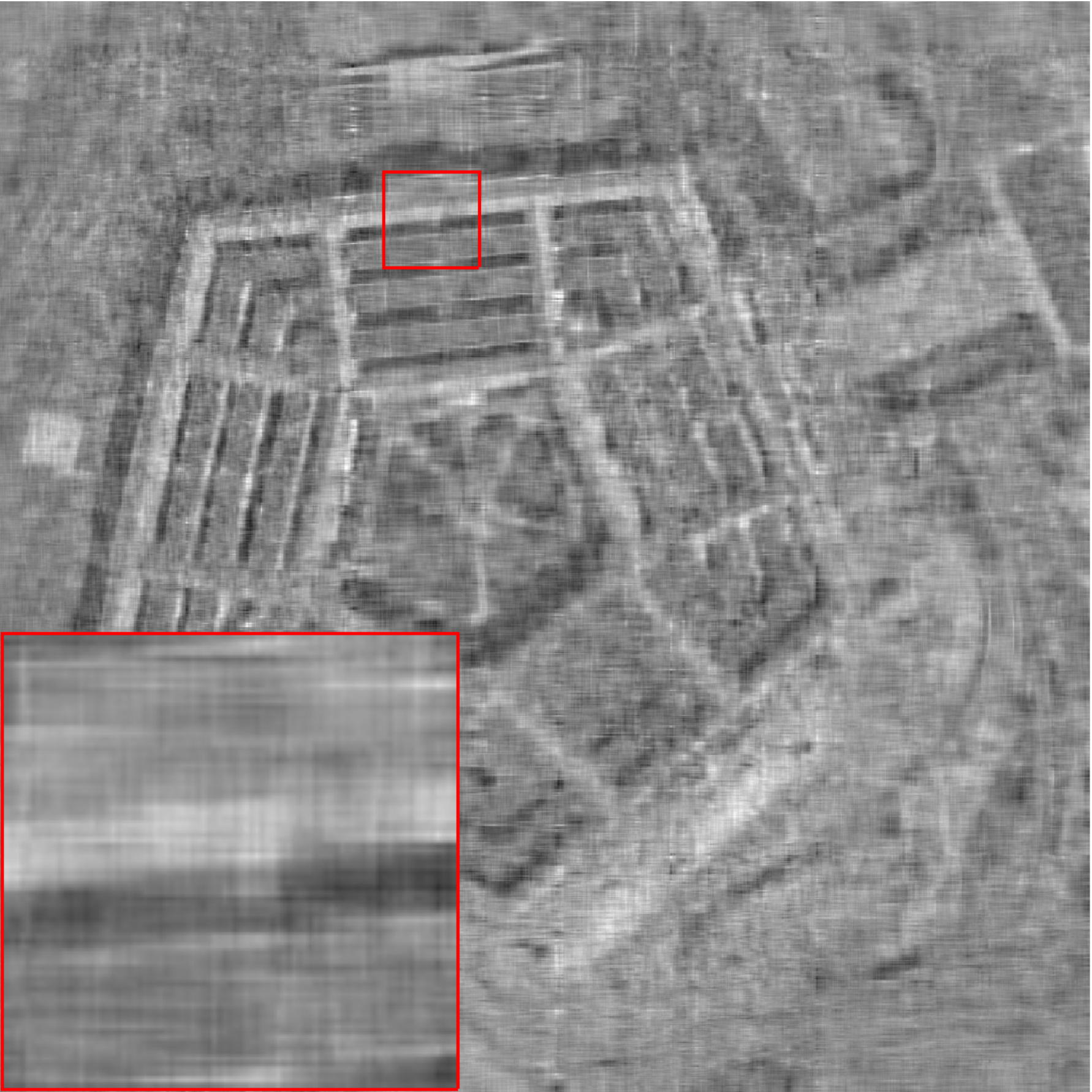}\vspace{0pt}
		   \includegraphics[width=\linewidth]{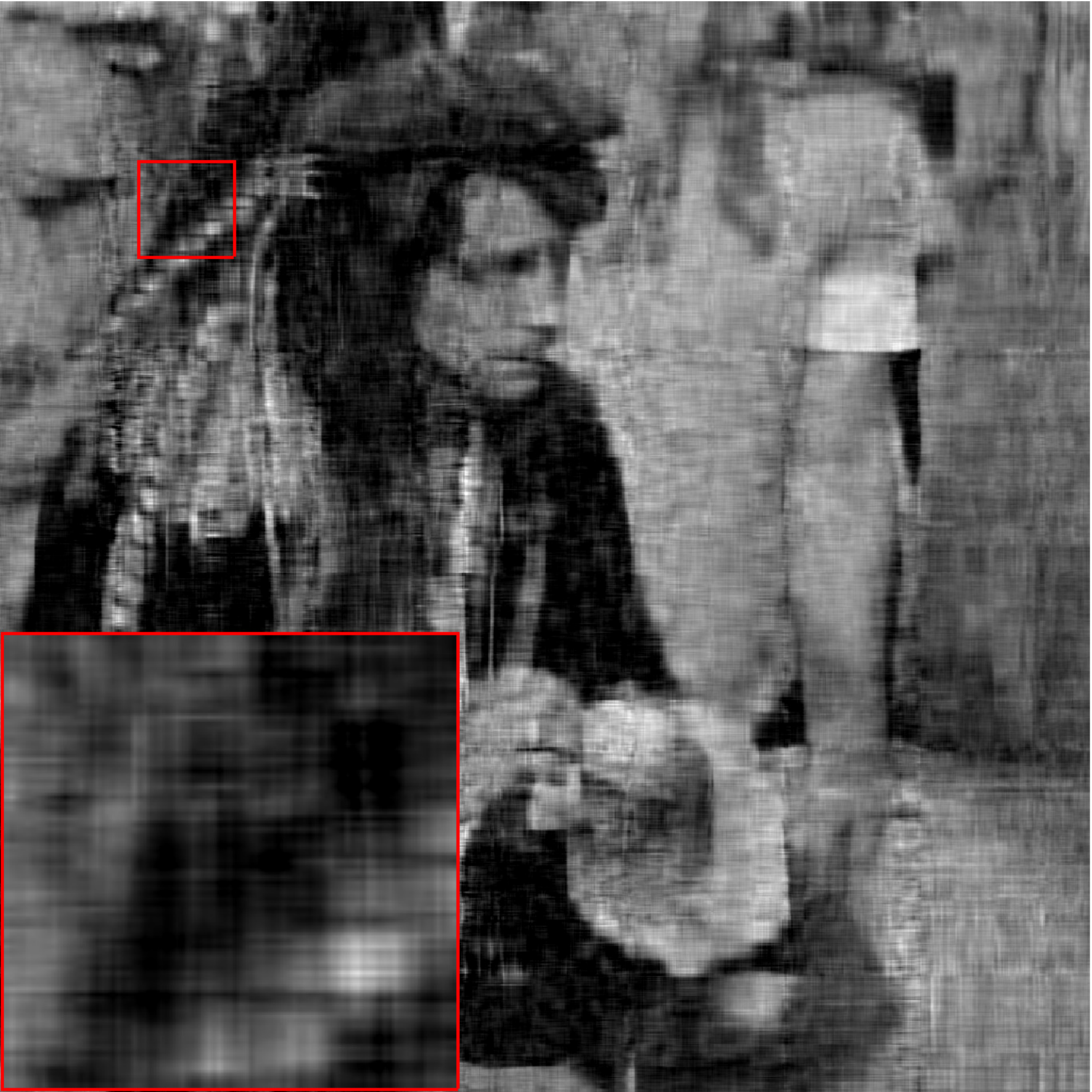}\vspace{0pt}
		   \includegraphics[width=\linewidth]{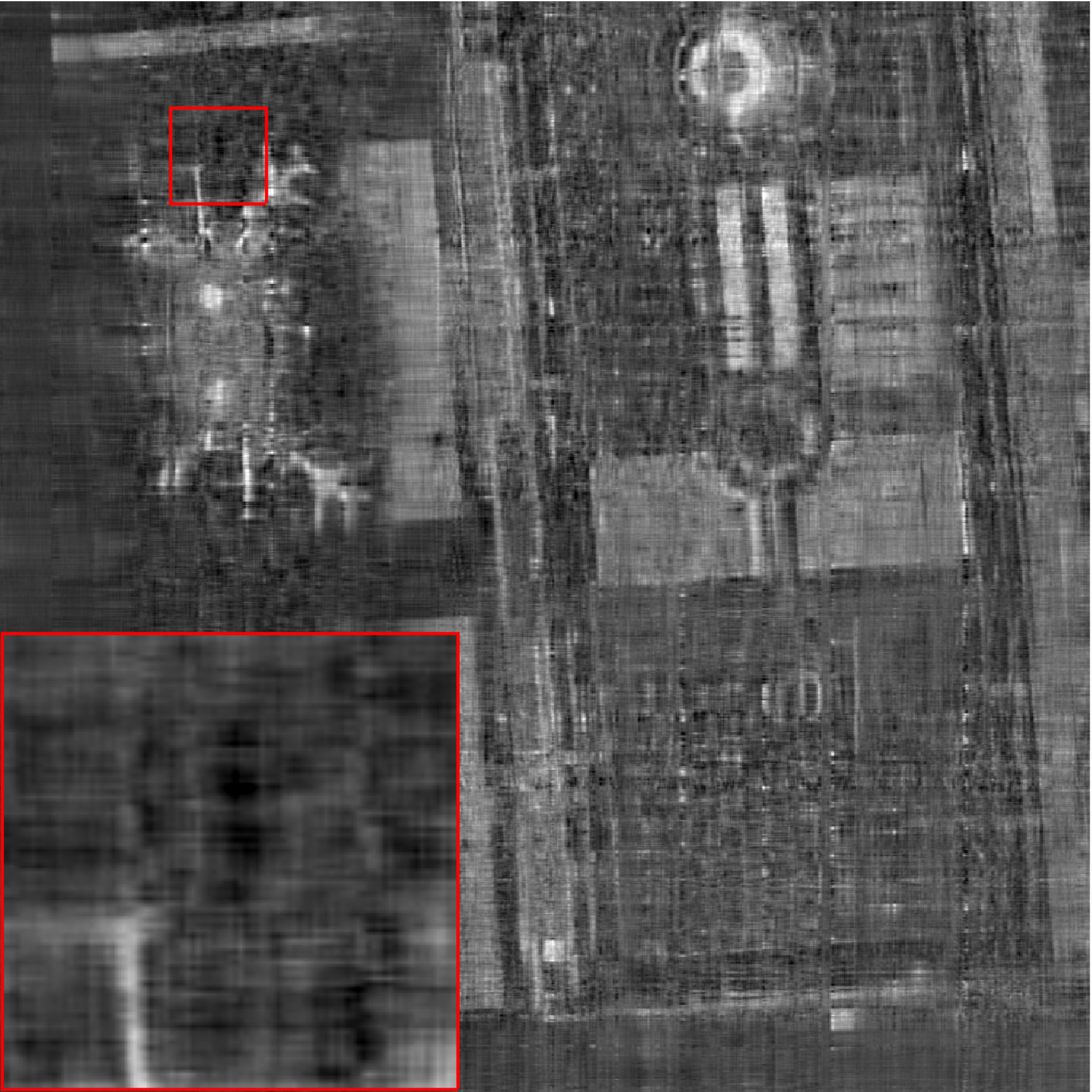}\vspace{0pt}
		   \includegraphics[width=\linewidth]{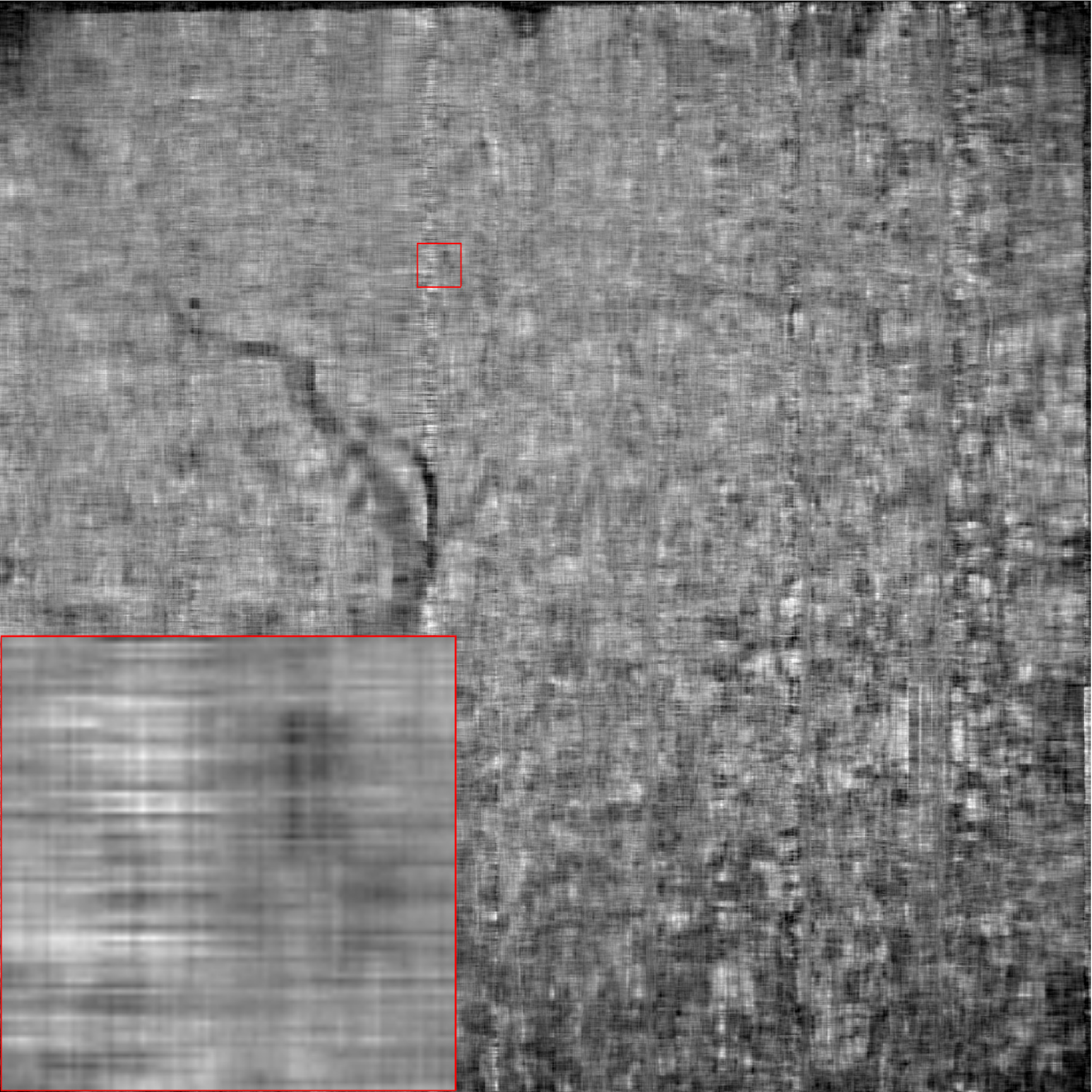}
		   \caption{FPCA}
	   \end{subfigure}
   		\begin{subfigure}[b]{0.138\linewidth}
   	       \centering
   	       \includegraphics[width=\linewidth]{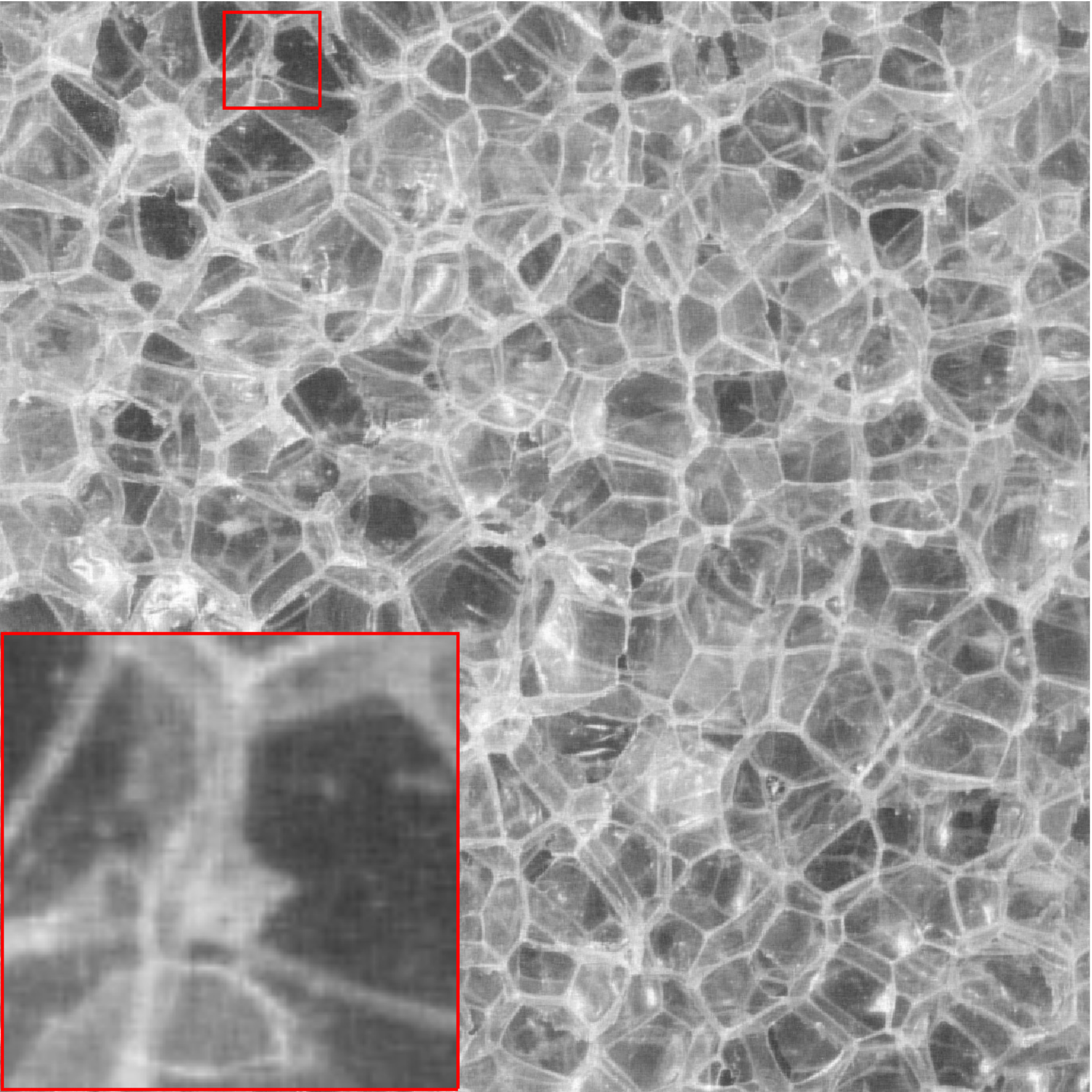}\vspace{0pt}
   	       \includegraphics[width=\linewidth]{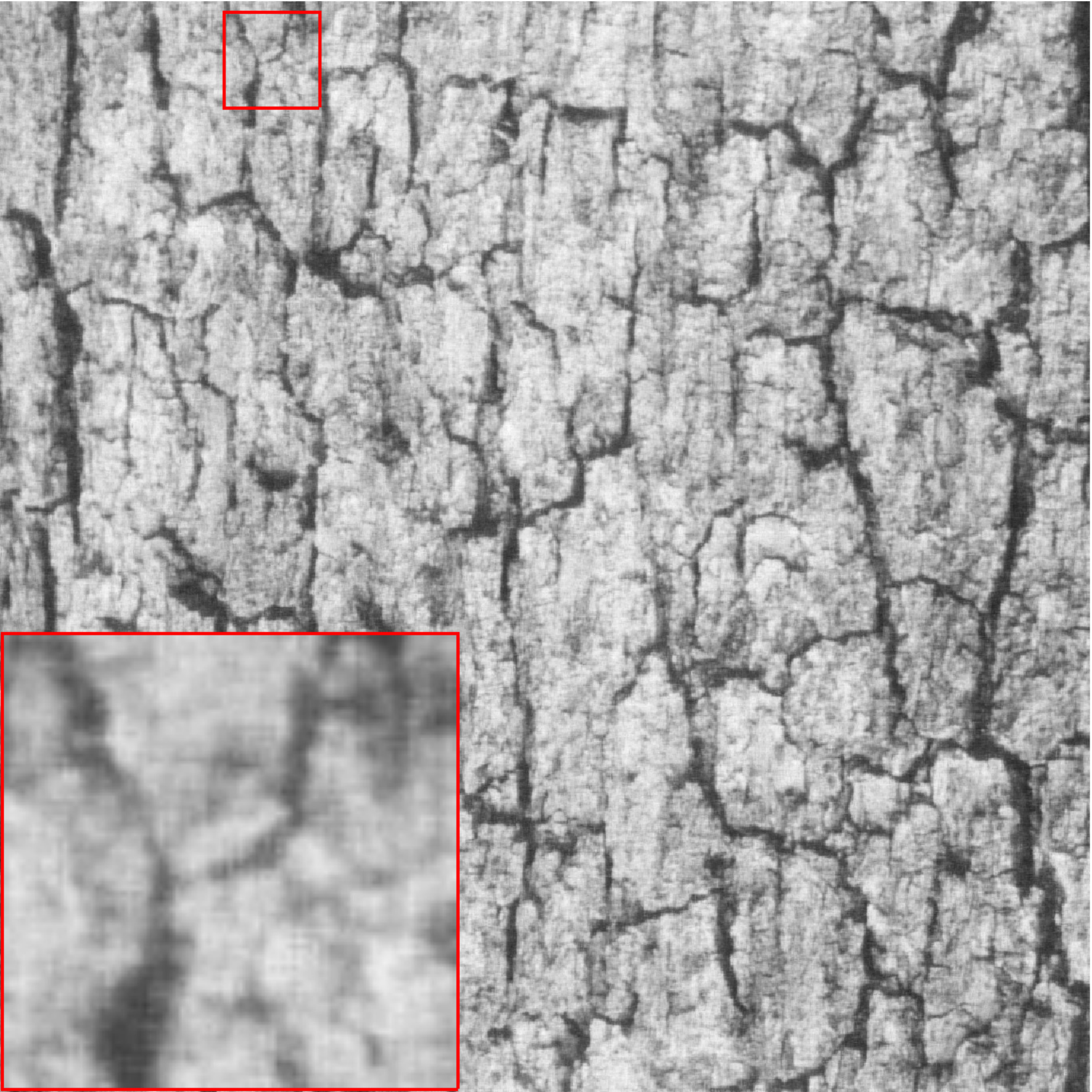}\vspace{0pt}			
   	       \includegraphics[width=\linewidth]{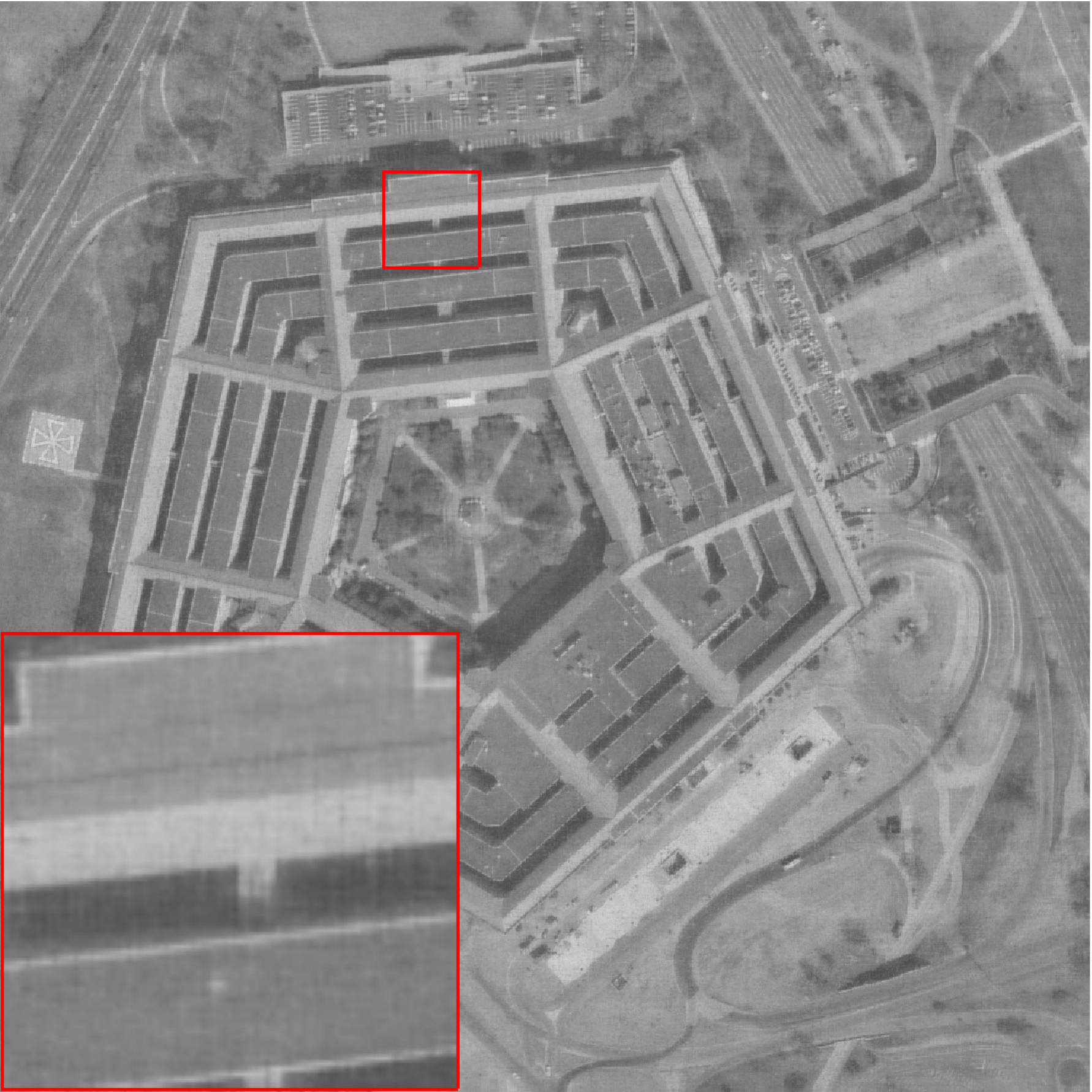}\vspace{0pt}
   	       \includegraphics[width=\linewidth]{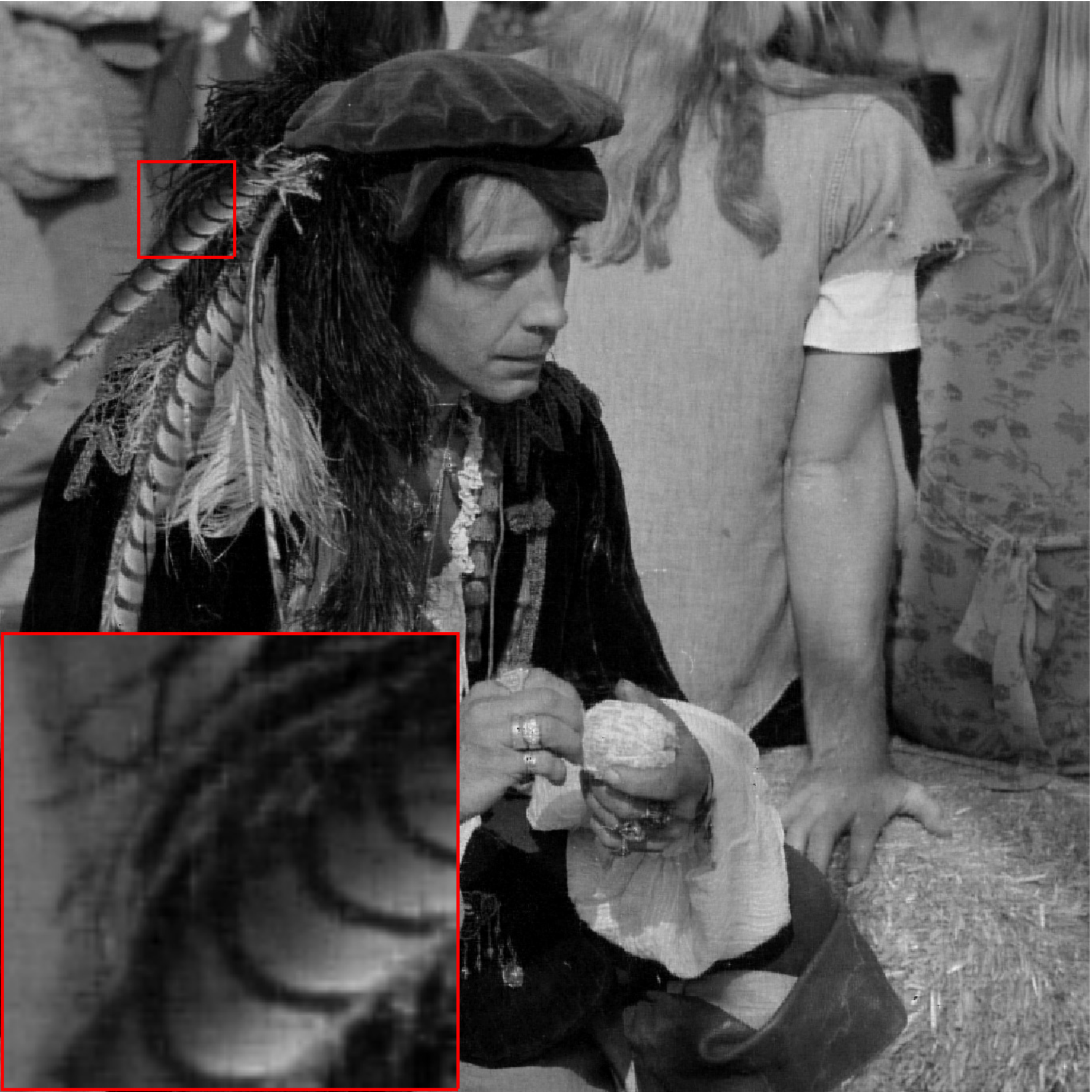}\vspace{0pt}
   	       \includegraphics[width=\linewidth]{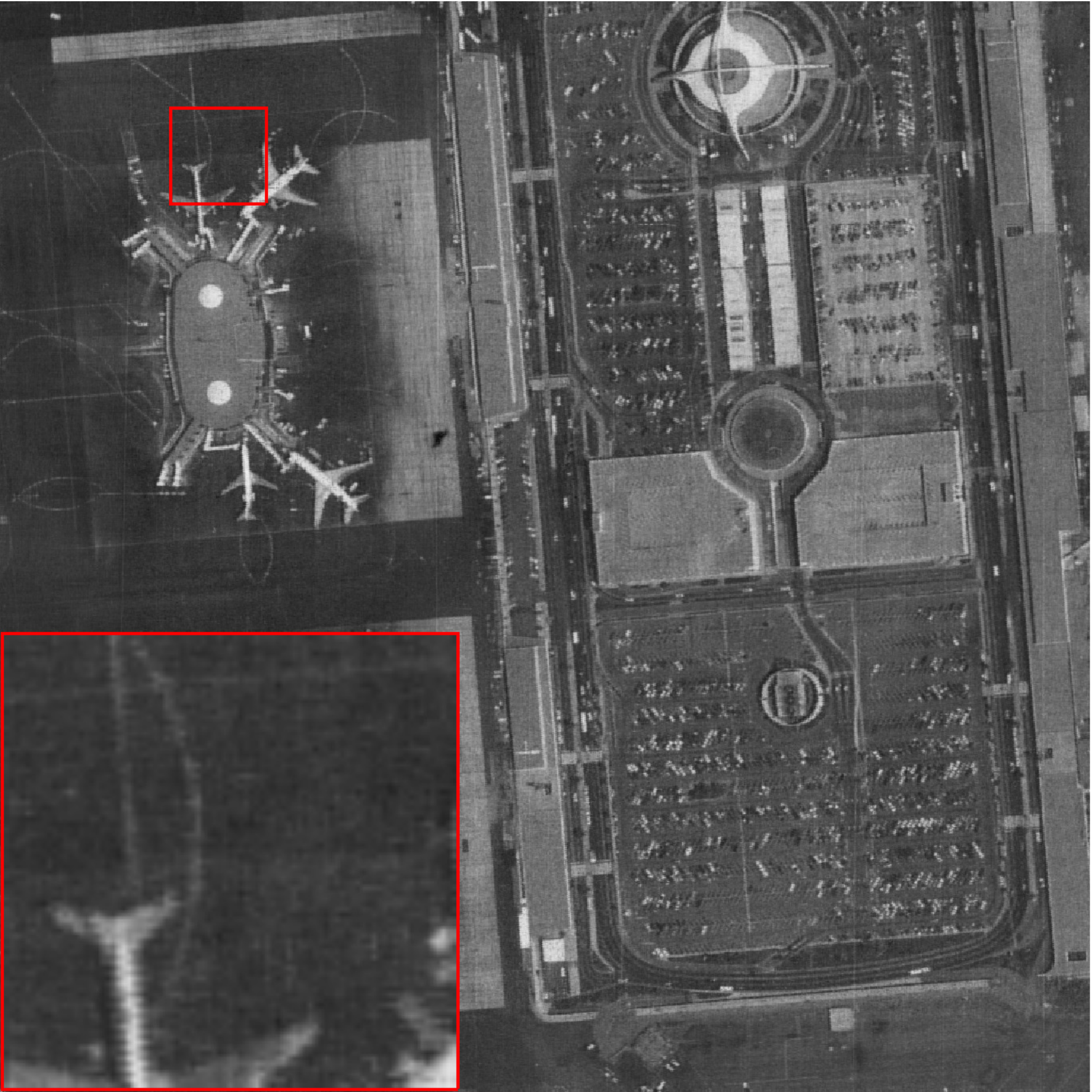}\vspace{0pt}
   	       \includegraphics[width=\linewidth]{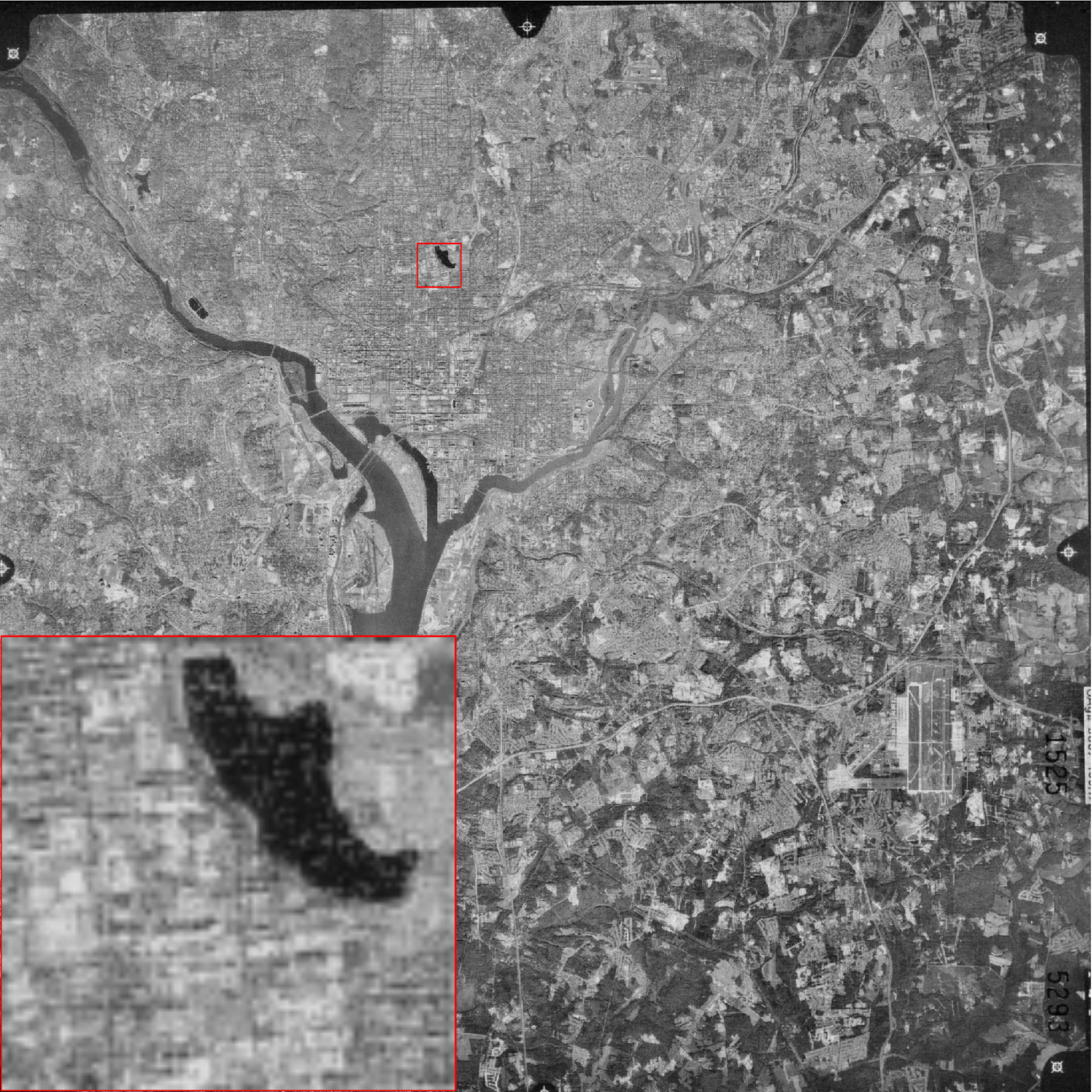}
   	\caption{SPG}
      \end{subfigure}				
	\end{subfigure}
	\vfill
	\caption{Examples of grayscale image inpainting. From top to bottom are respectively corresponding to ``Plastic", ``Bark", ``Pentagon", ``Male", ``Airport" and ``Wash".}
	\label{fig:image_inpainting}
\end{figure*}

\begin{figure*}[htbp]
	\centering
	\begin{subfigure}[b]{0.93\linewidth}
		\begin{subfigure}[b]{0.31\linewidth}
			\centering
			\includegraphics[width=\linewidth]{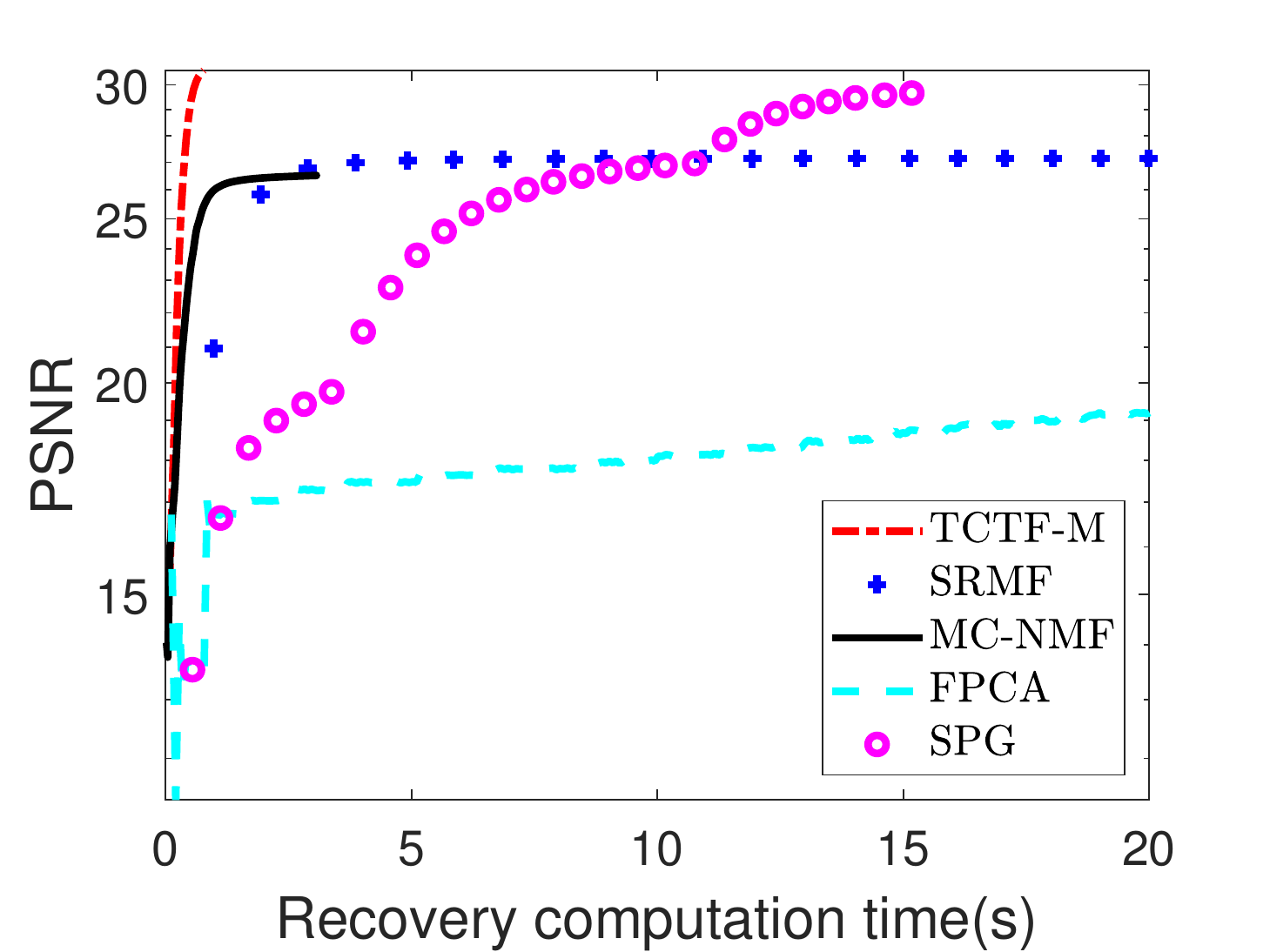}\vspace{0pt}
			\includegraphics[width=\linewidth]{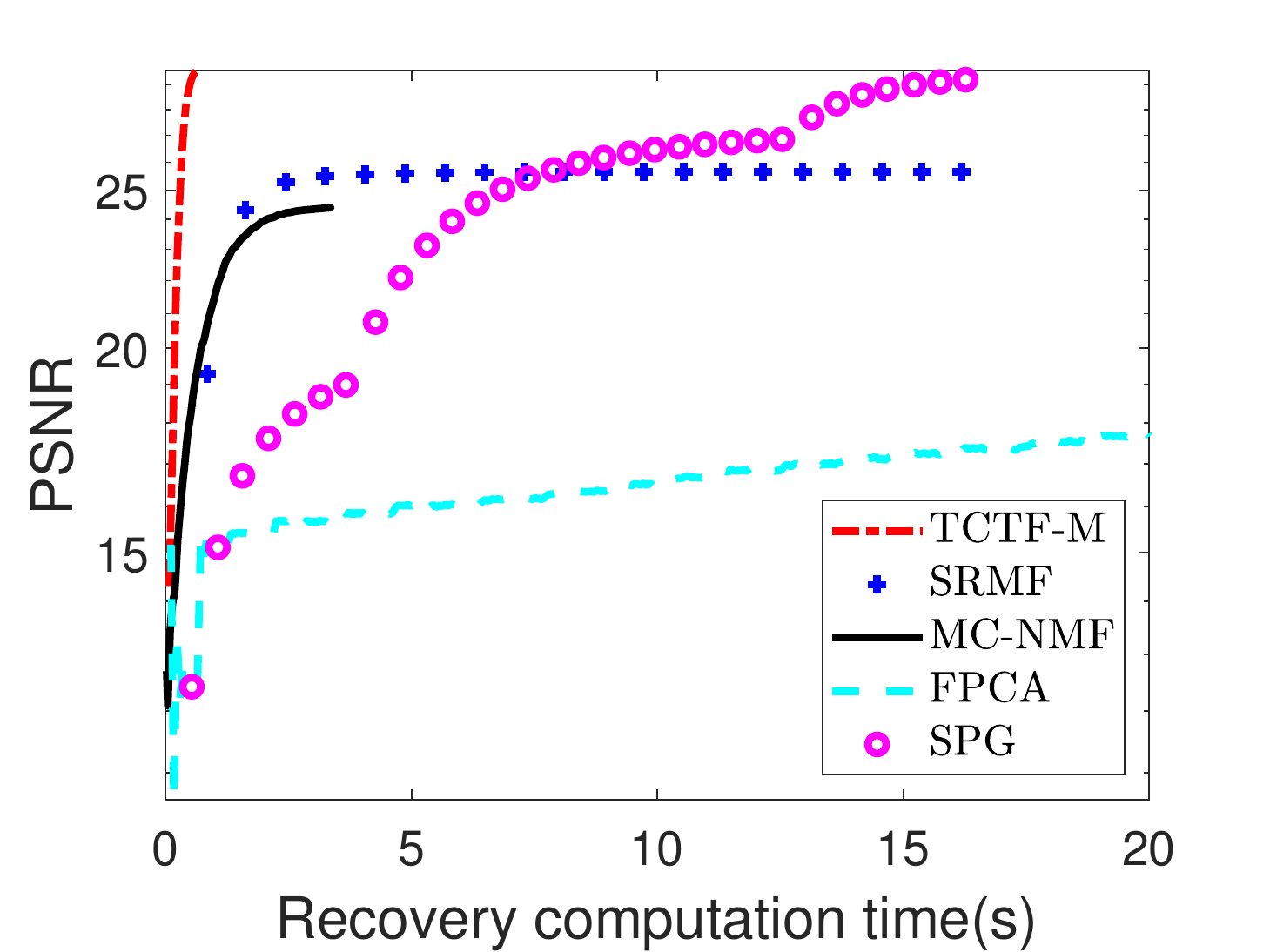}\vspace{0pt}
			\includegraphics[width=\linewidth]{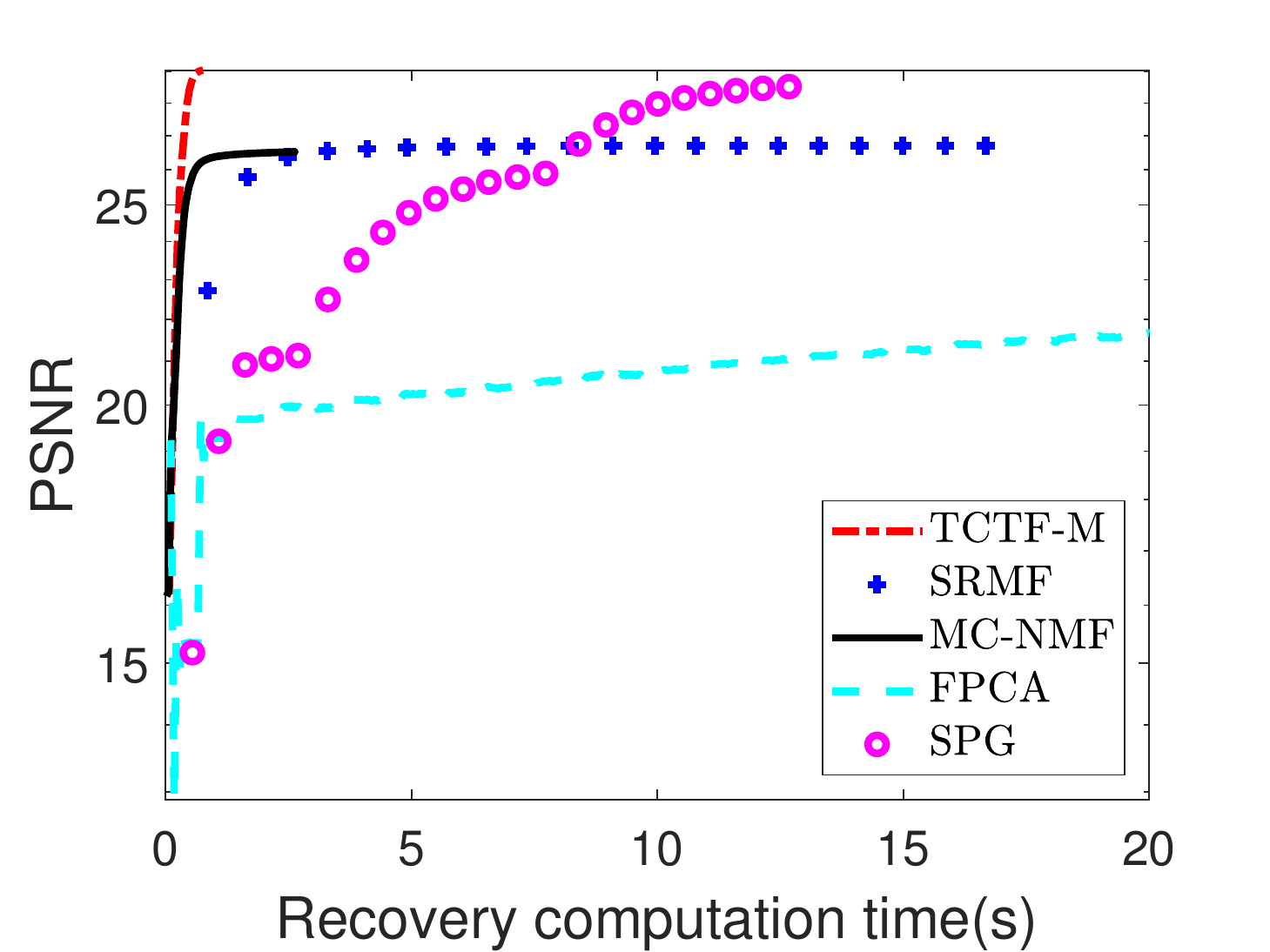}\vspace{0pt}
			\includegraphics[width=\linewidth]{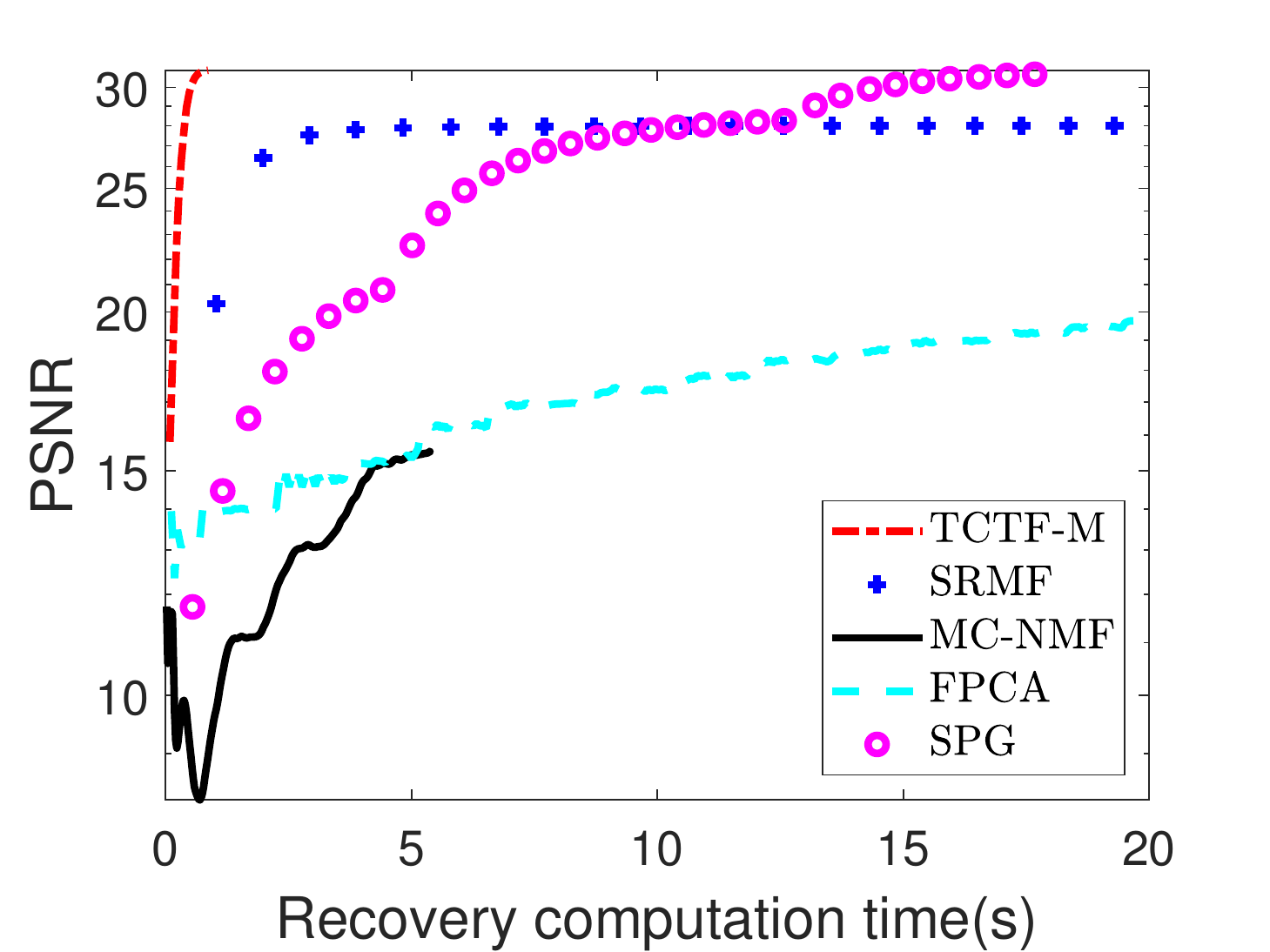}\vspace{0pt}
			\includegraphics[width=\linewidth]{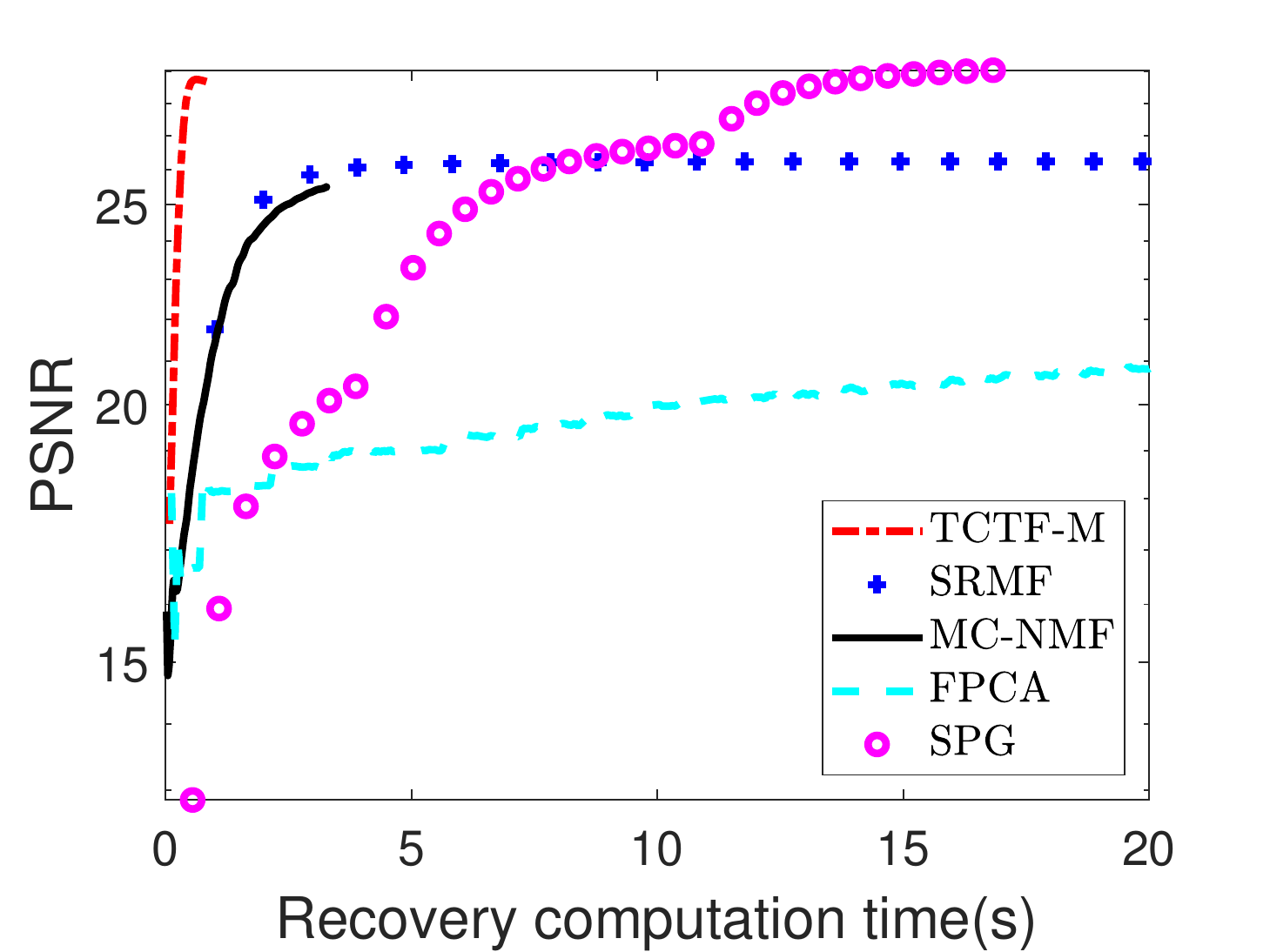}\vspace{0pt}
			\includegraphics[width=\linewidth]{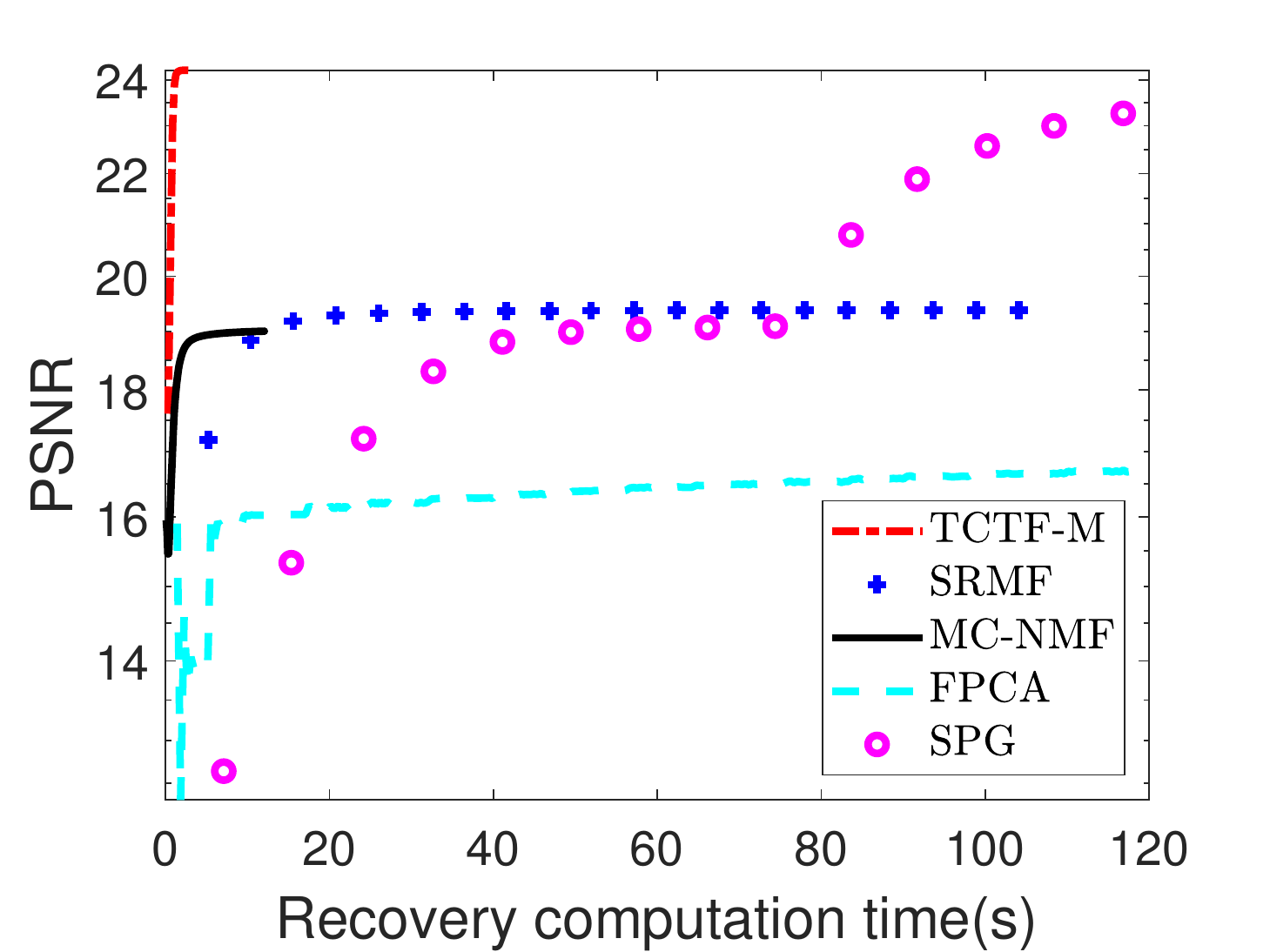}
		\end{subfigure}   	
		\begin{subfigure}[b]{0.31\linewidth}
			\centering
			\includegraphics[width=\linewidth]{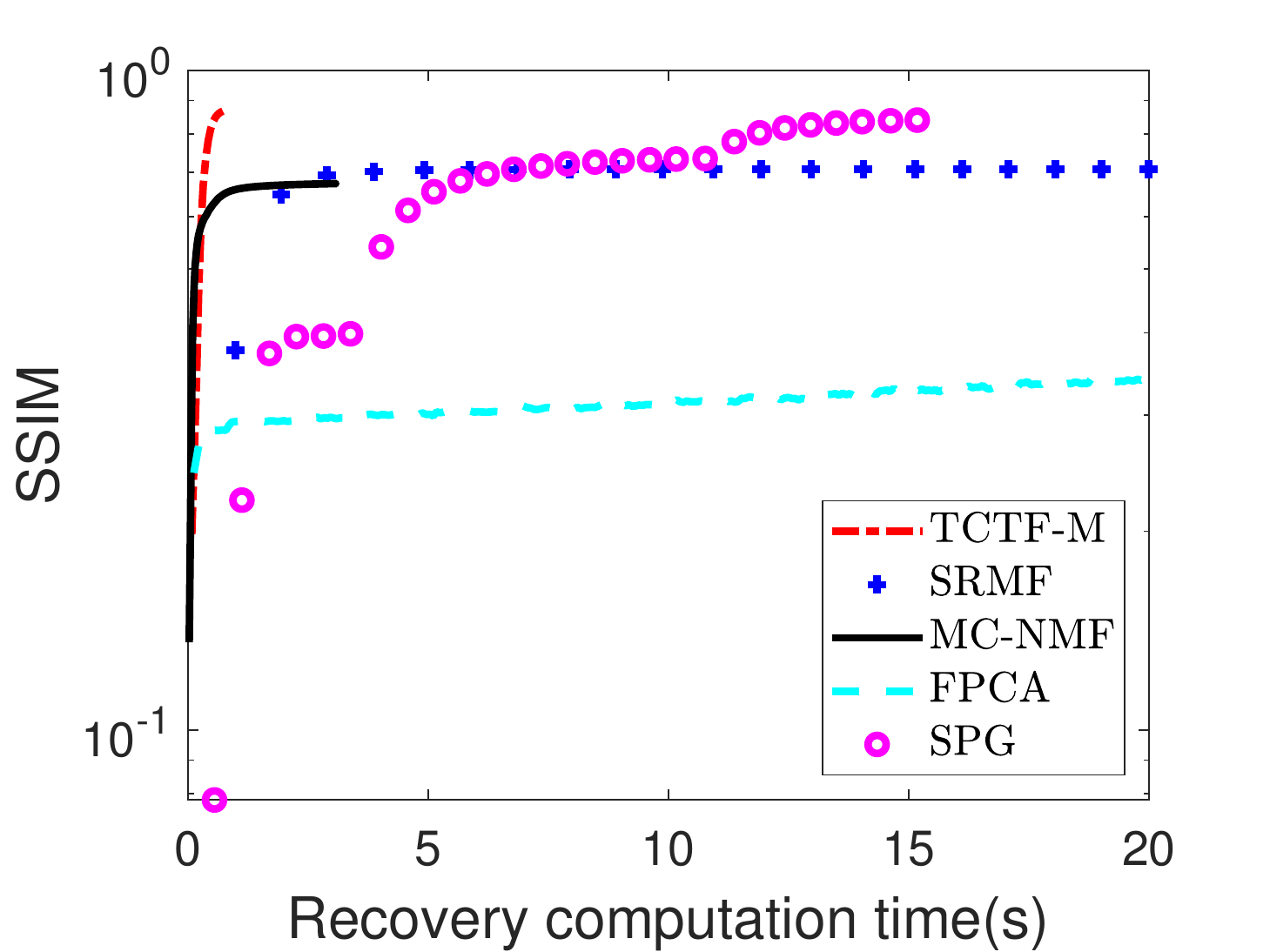}\vspace{0pt}
			\includegraphics[width=\linewidth]{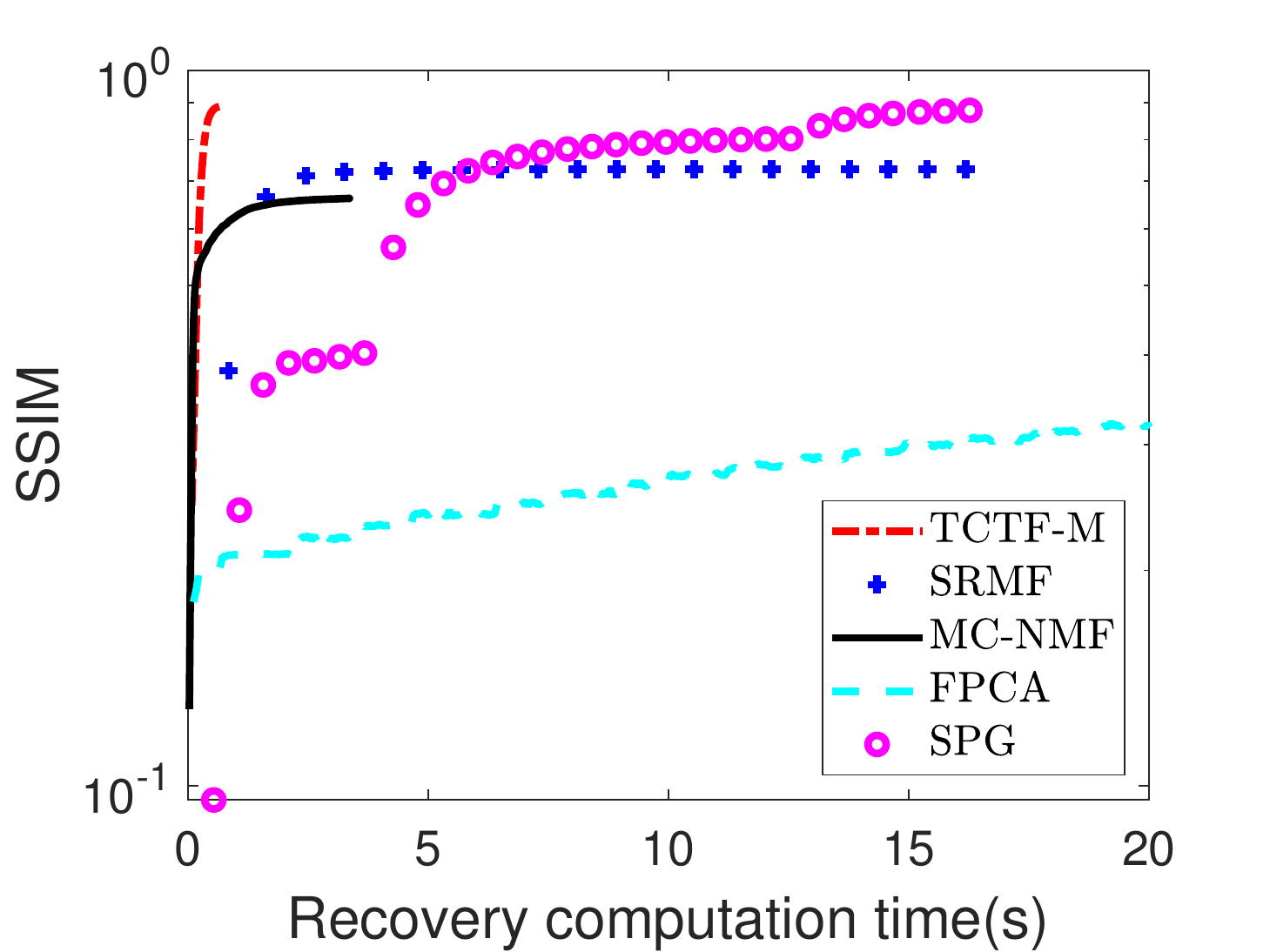}\vspace{0pt}
			\includegraphics[width=\linewidth]{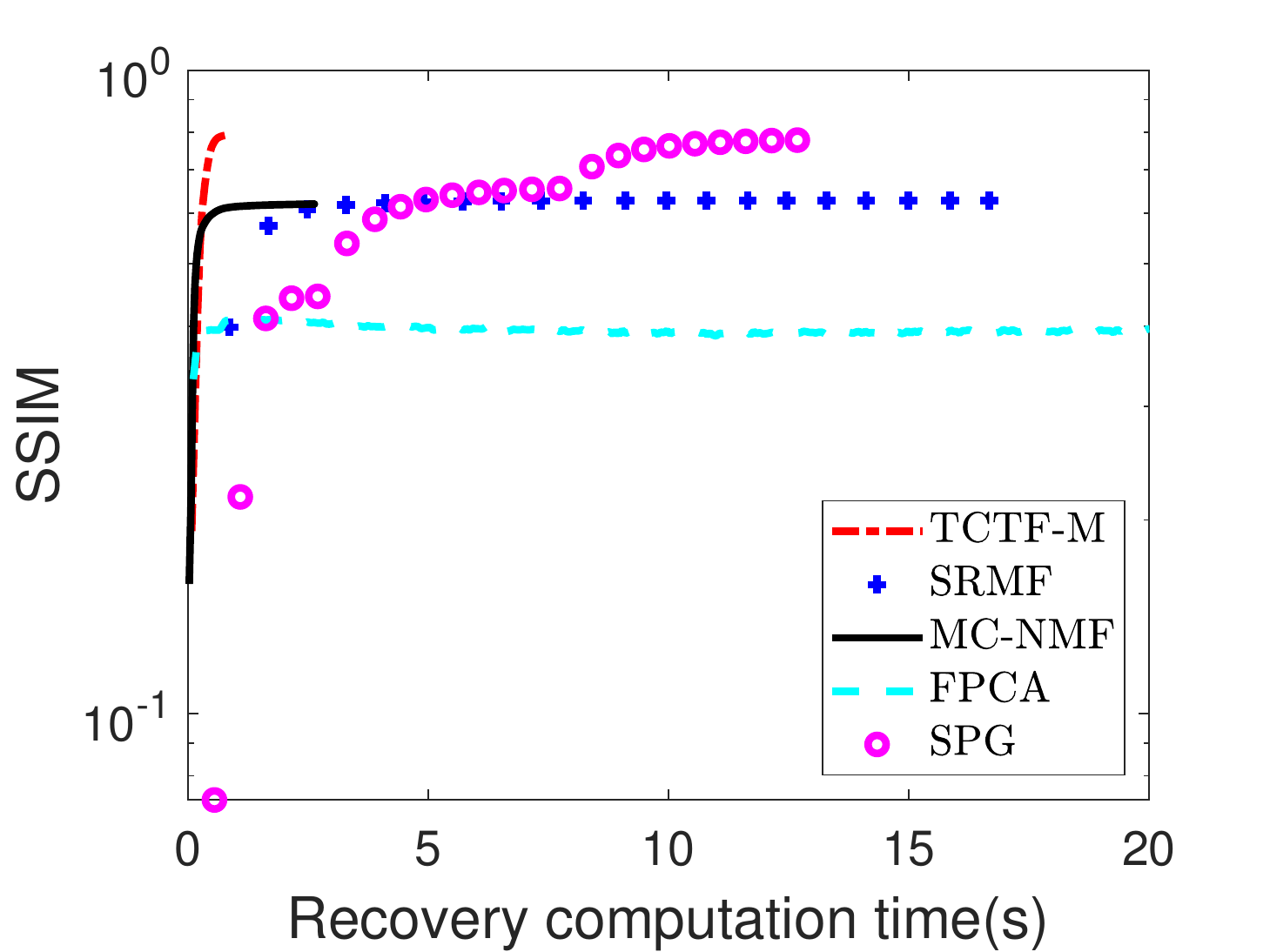}\vspace{0pt}
			\includegraphics[width=\linewidth]{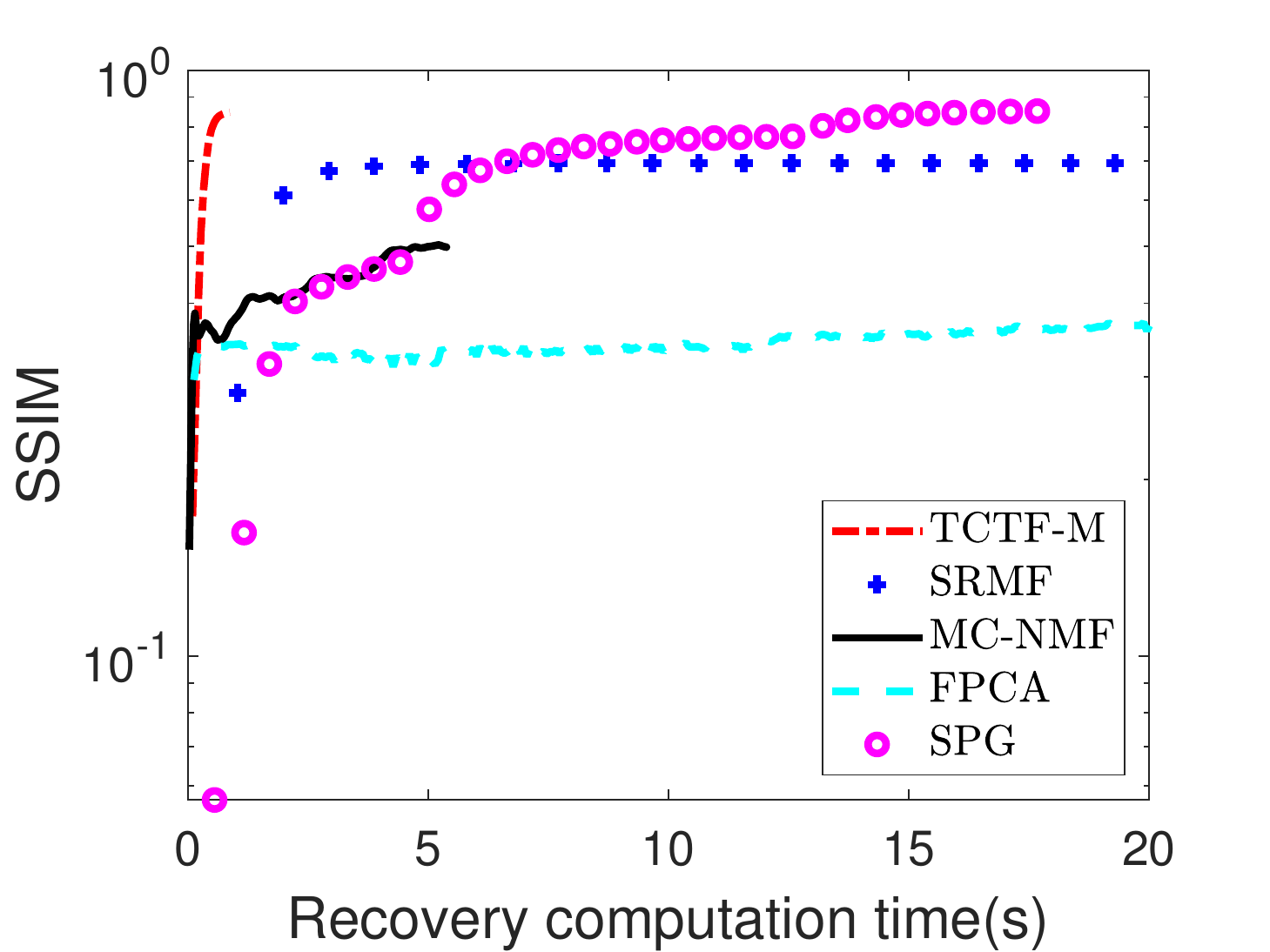}\vspace{0pt}
			\includegraphics[width=\linewidth]{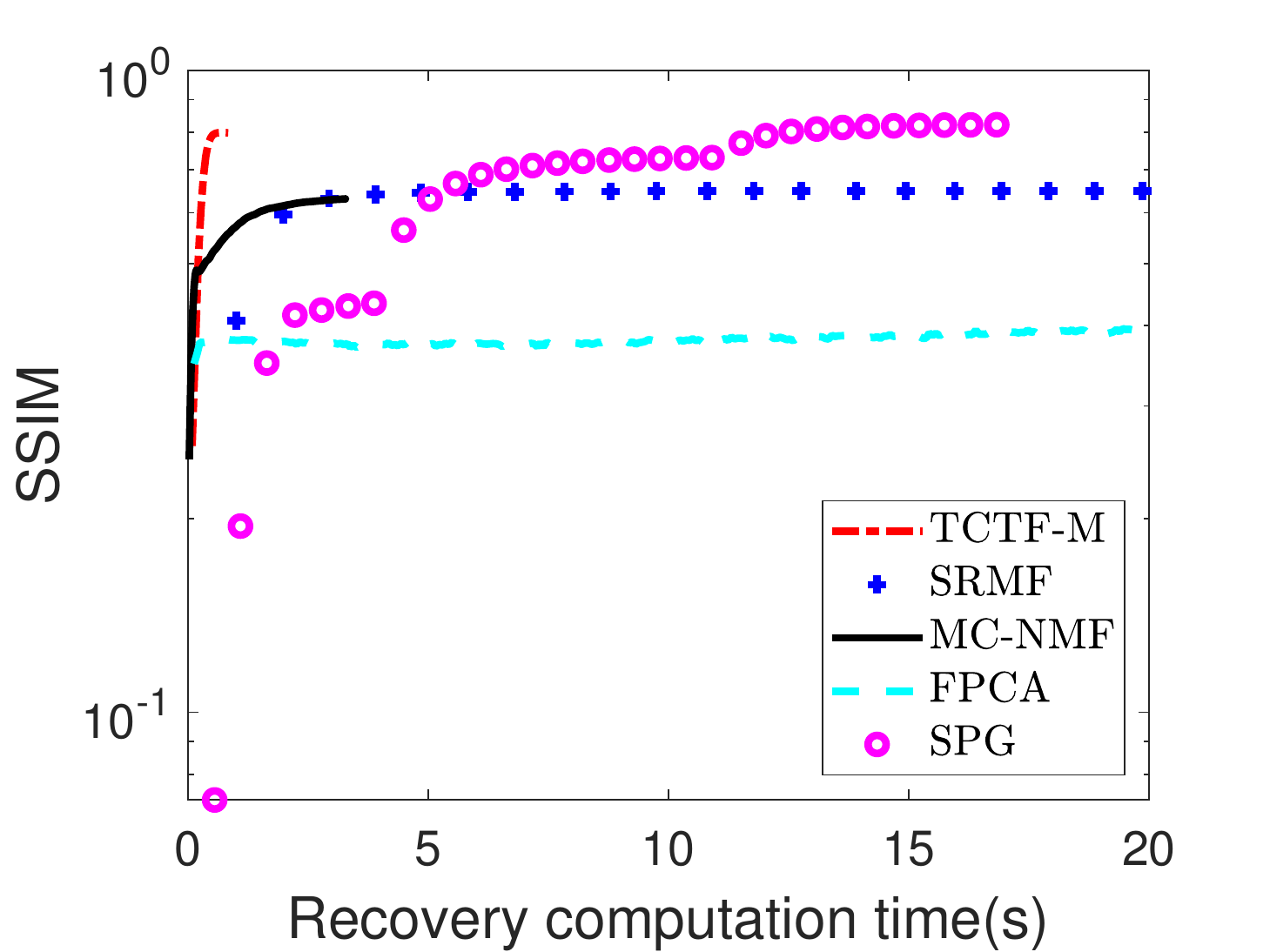}\vspace{0pt}
			\includegraphics[width=\linewidth]{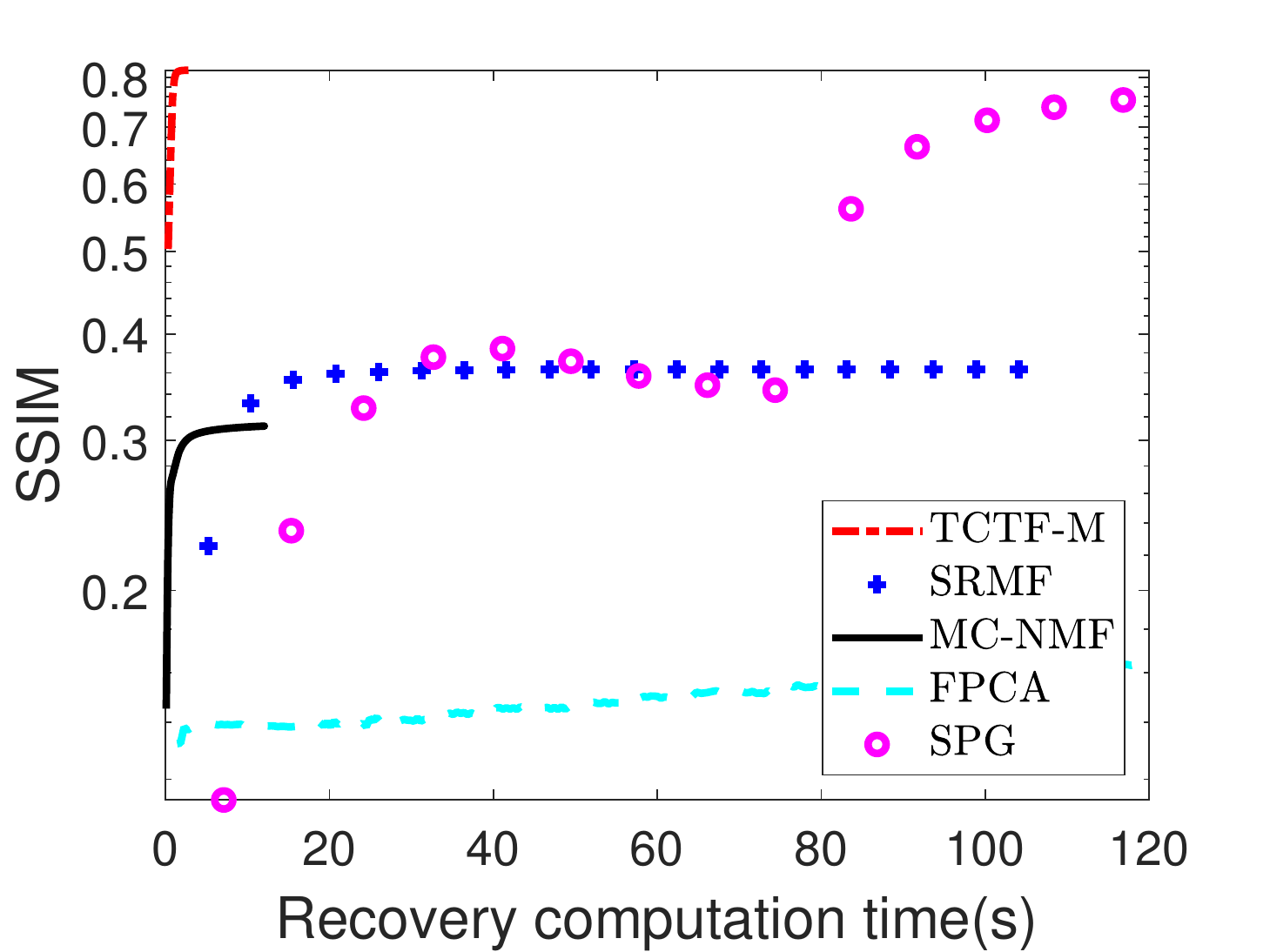}
		\end{subfigure}
		\begin{subfigure}[b]{0.31\linewidth}
			\centering
			\includegraphics[width=\linewidth]{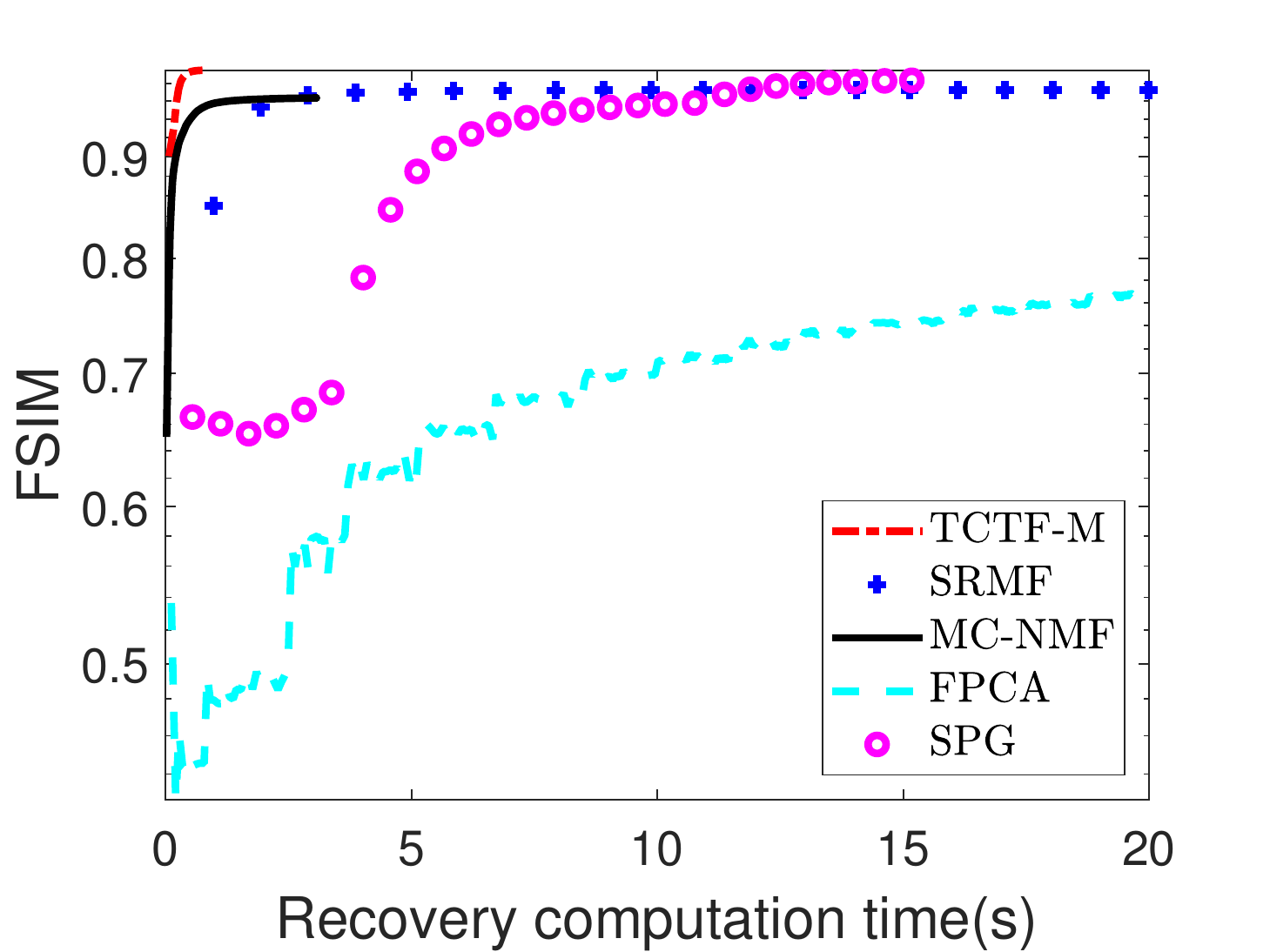}\vspace{0pt}
			\includegraphics[width=\linewidth]{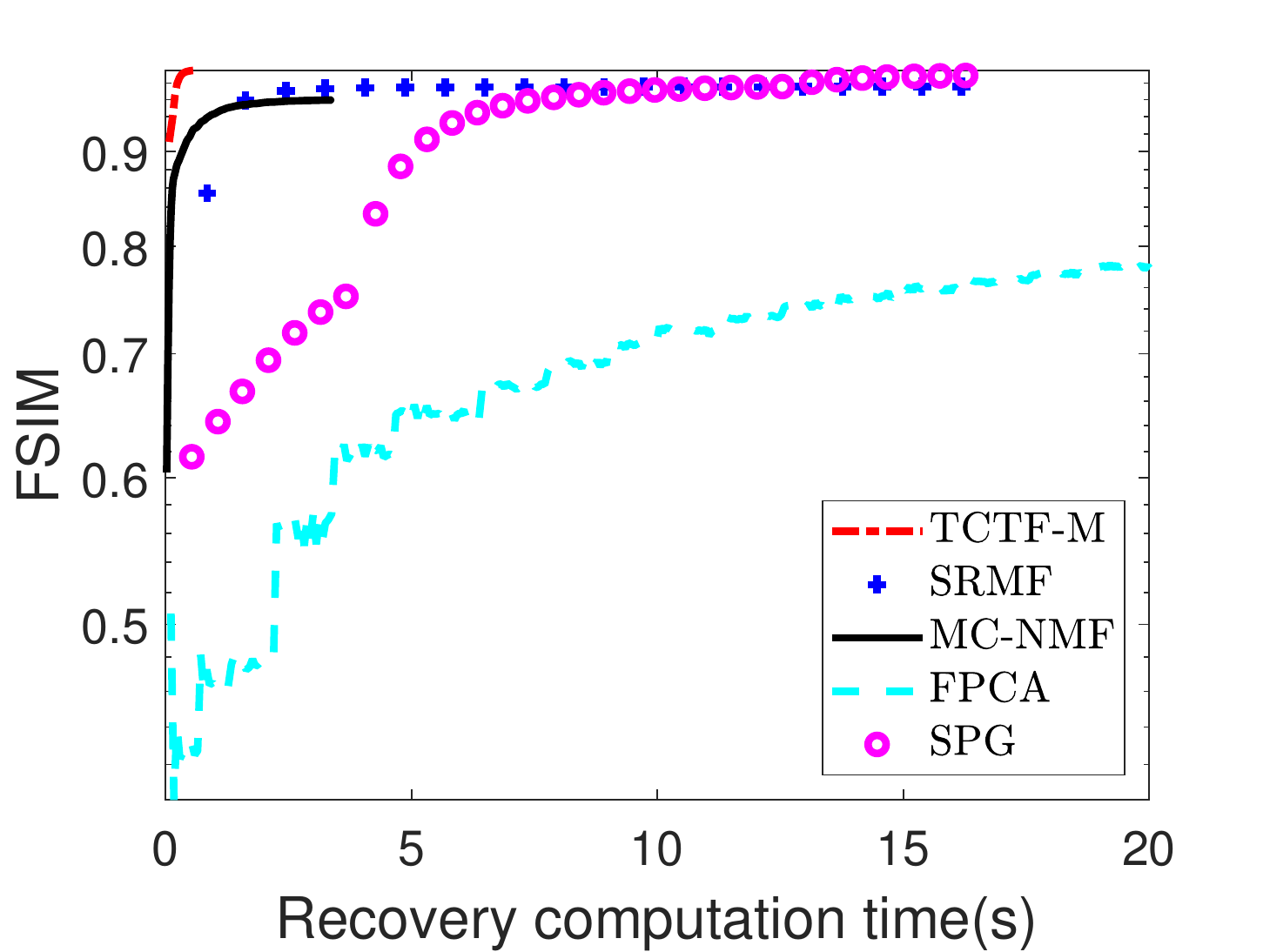}\vspace{0pt}
			\includegraphics[width=\linewidth]{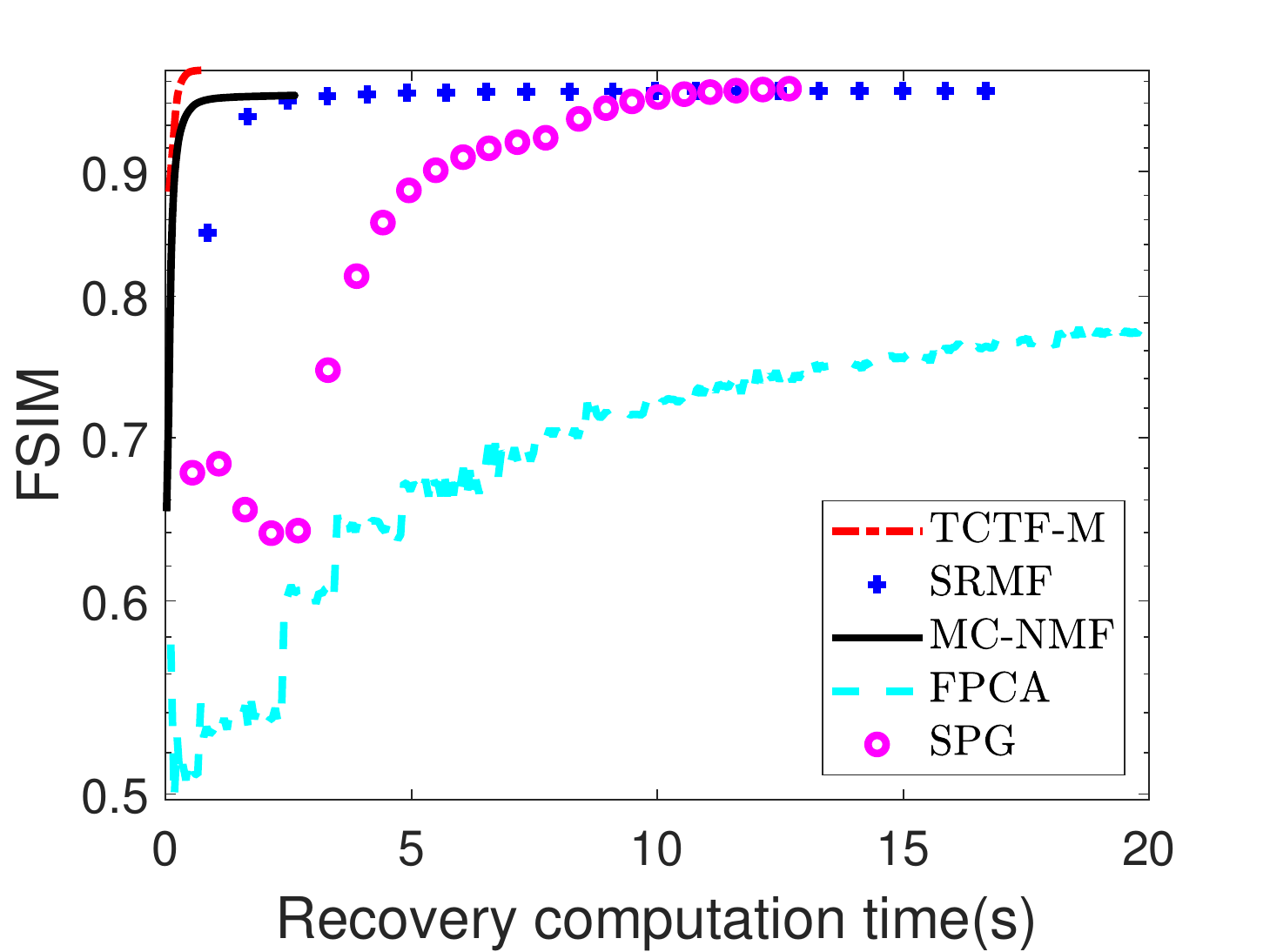}\vspace{0pt}
			\includegraphics[width=\linewidth]{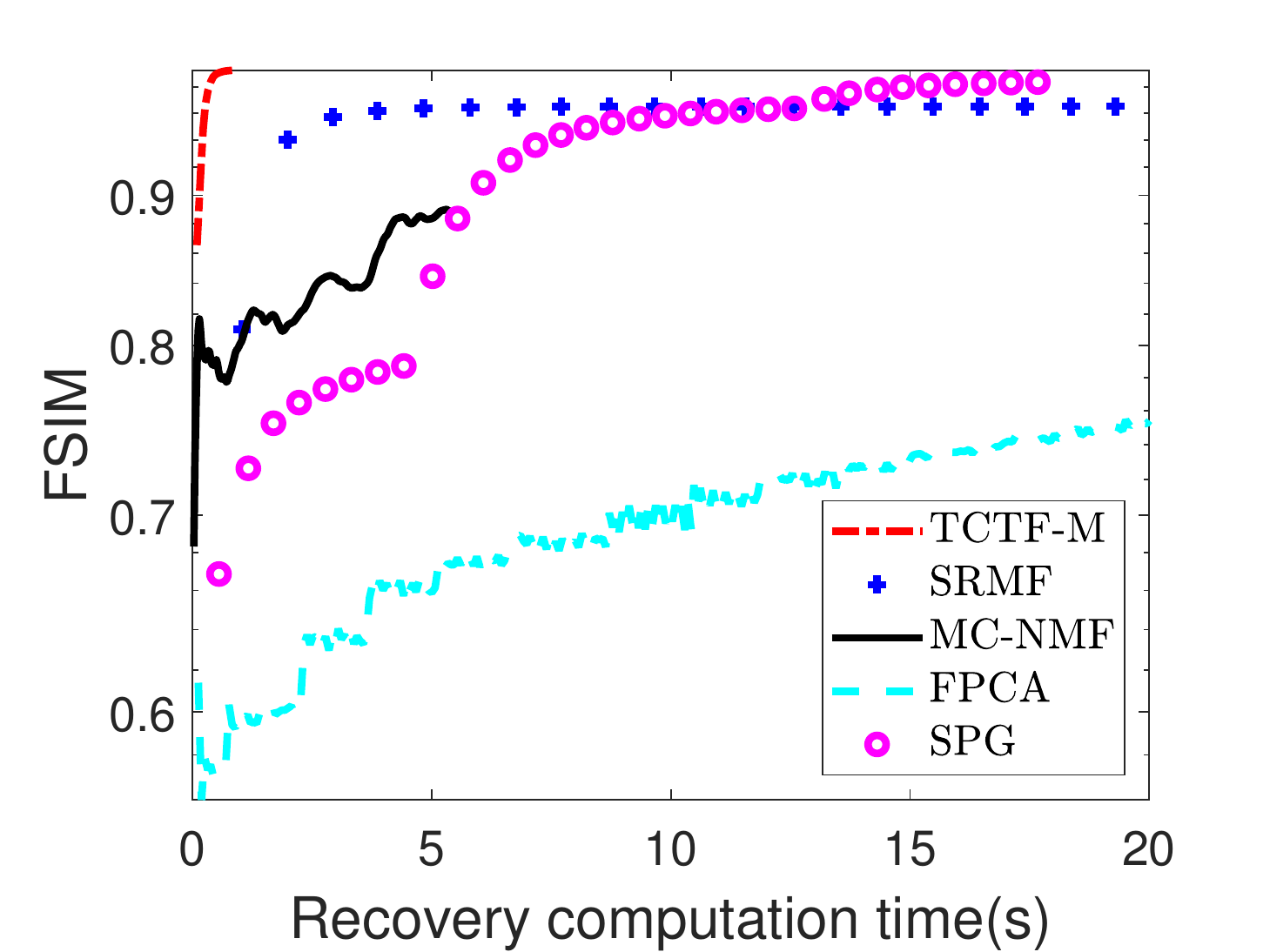}\vspace{0pt}
			\includegraphics[width=\linewidth]{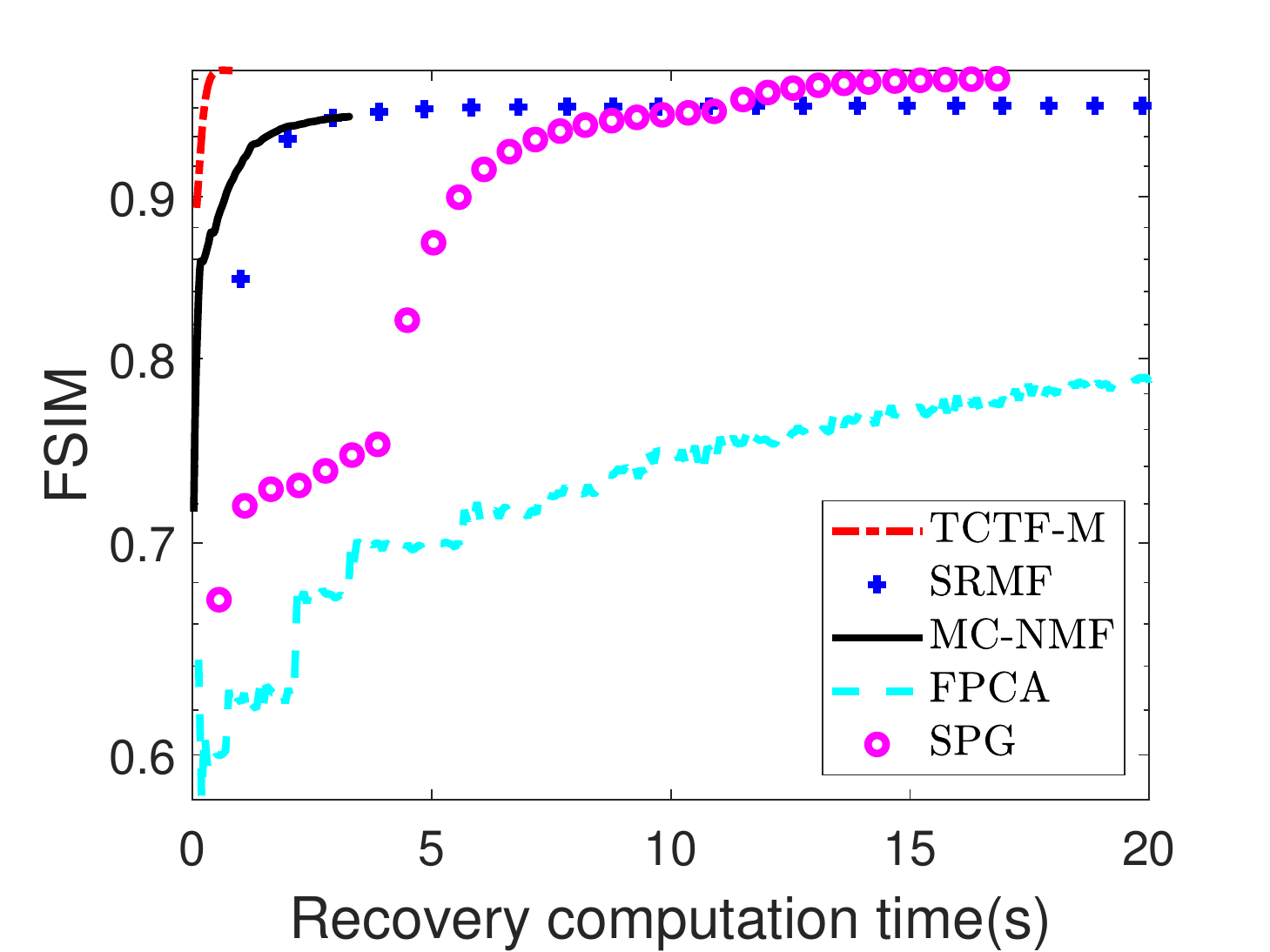}\vspace{0pt}
			\includegraphics[width=\linewidth]{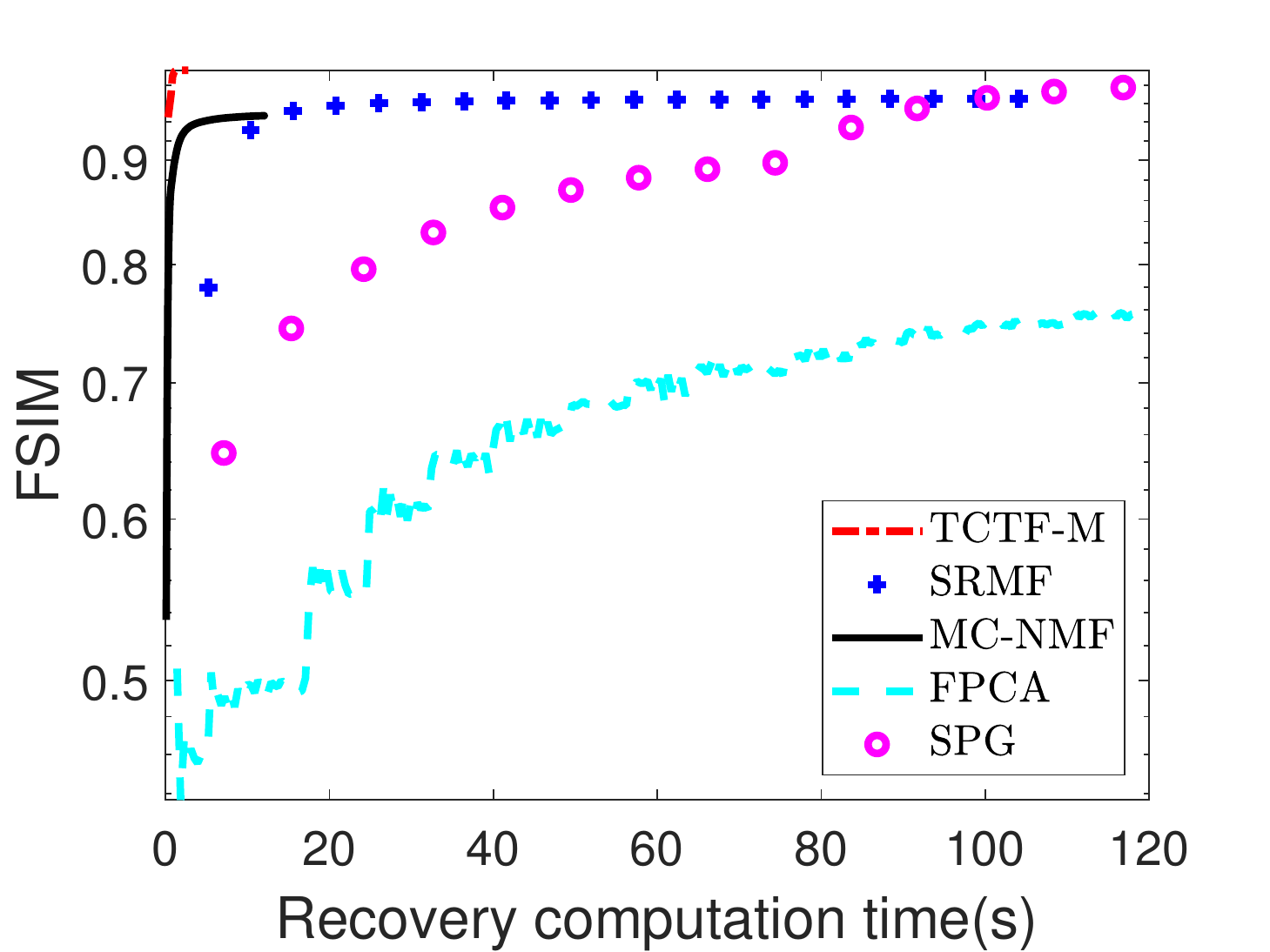}
		\end{subfigure}			
	\end{subfigure}
	\vfill
	\caption{Grayscale image inpainting:The PSNR, SSIM and FSIM values with respect to the recovery computation time. From top to bottom are respectively corresponding to ``Plastic", ``Bark", ``Pentagon", ``Male", ``Airport" and ``Wash". In order to better display the effect, we only selected the first 20 (120) seconds for comparison.}
	\label{fig:grayscale_time}
\end{figure*}

\subsection{High Altitude Aerial Image Inpainting}
This subsection applies DTRTC to high altitude aerial image inpainting. We also use the USC-SIPI image database to evaluate our proposed method for high altitude aerial image inpainting. In our test, four high altitude aerial images are randomly selected from this database. The first three images both are $ 1024\times 1024\times 3 $ pixels and that of the last one is $ 2250\times 2250\times 3 $ pixels. The data of images are normalized in the range $ \left[0,1\right] $.

For each chosen image, we randomly sample by the sampling ratio $p=40\%, 45\%, 50\%$.  We set the initial double tubal rank $\bm{r}_\X^0=\left(200,30,30\right),\,\bm{r}_{\tilde\X}^0=\left(3,\ldots,3\right)\,  $ in DTRTC, the initial tubal rank $ \left(200,30,30\right) $ in TCTF,
the initial CP rank $ 100 $ in NCPC and the initial Tucker rank $ \left(100,100,3\right) $ in NTD. In DTRTC, ``Wash" data sets form a tensor of size $ 3\times 101250 \times 50 $ and the others set form a tensor of size $ 3\times 16384 \times 64 $. In experiments, the maximum iterative number is set to be $100$ and precision $ \varepsilon $ is set to be 1e-4.

\begin{figure*}[htbp]
	\centering
	\begin{subfigure}[b]{1\linewidth}
	\begin{subfigure}[b]{0.118\linewidth}
		\centering
		\includegraphics[width=\linewidth]{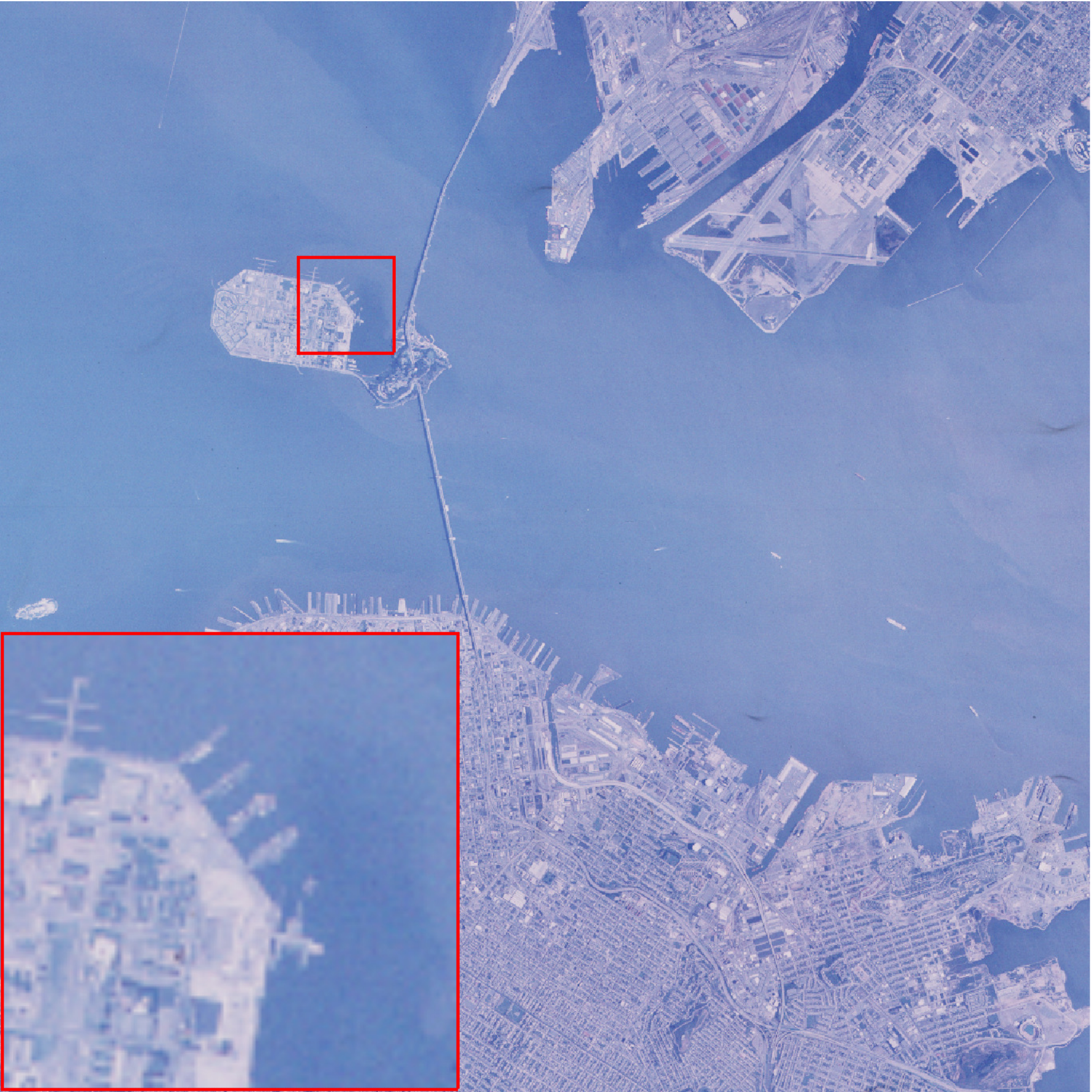}\vspace{0pt}
		\includegraphics[width=\linewidth]{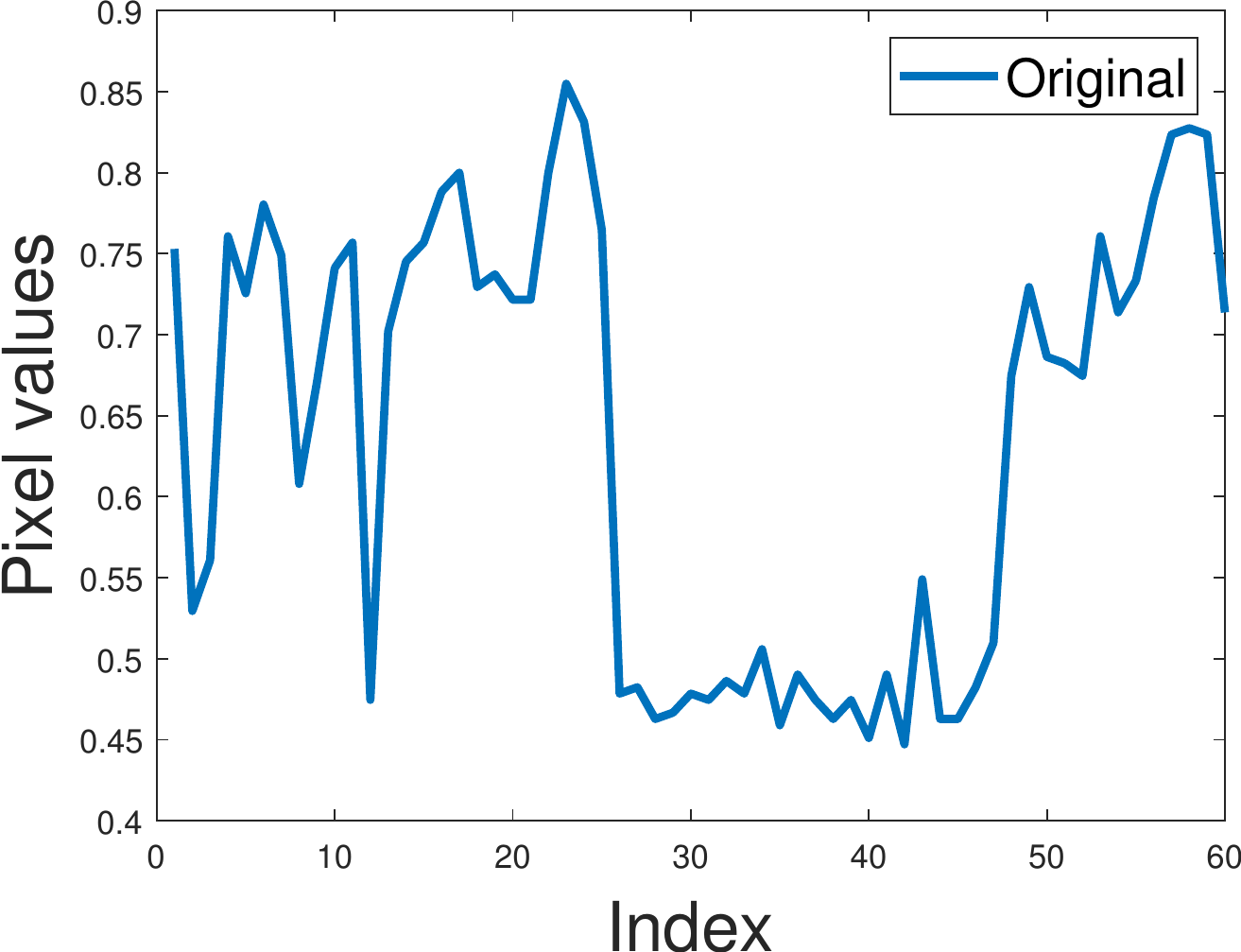}\vspace{0pt}
		\includegraphics[width=\linewidth]{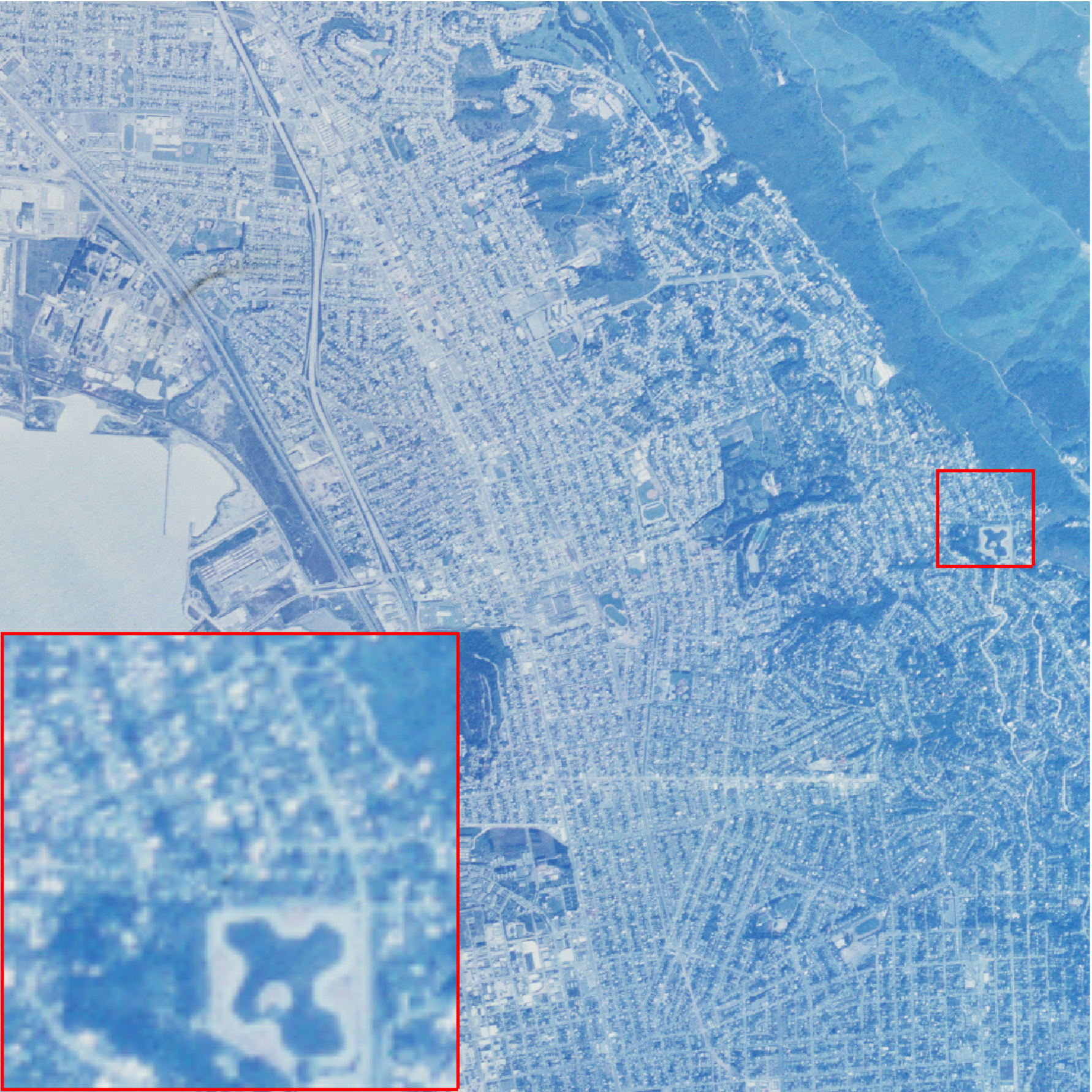}\vspace{0pt}
		\includegraphics[width=\linewidth]{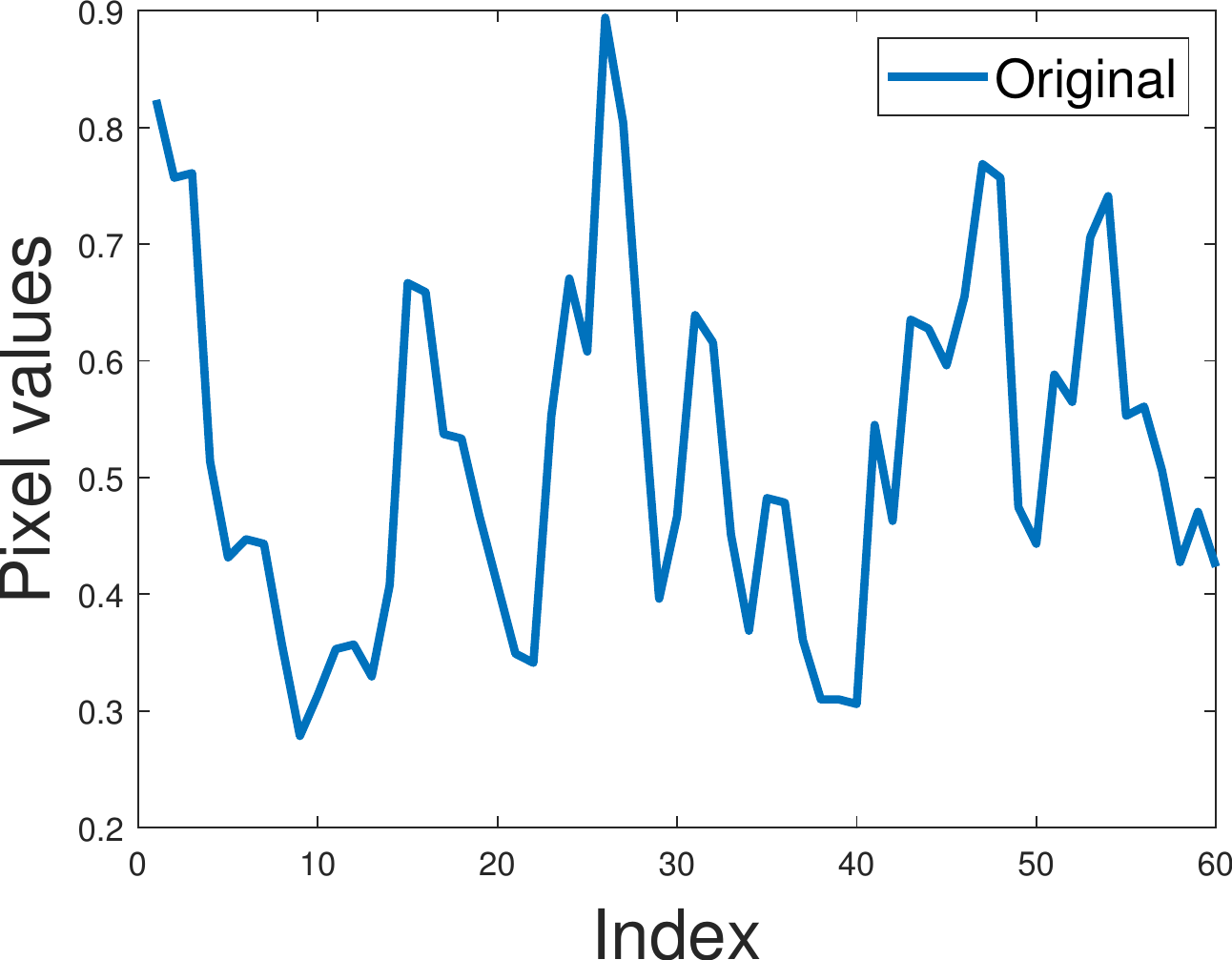}\vspace{0pt}
		\includegraphics[width=\linewidth]{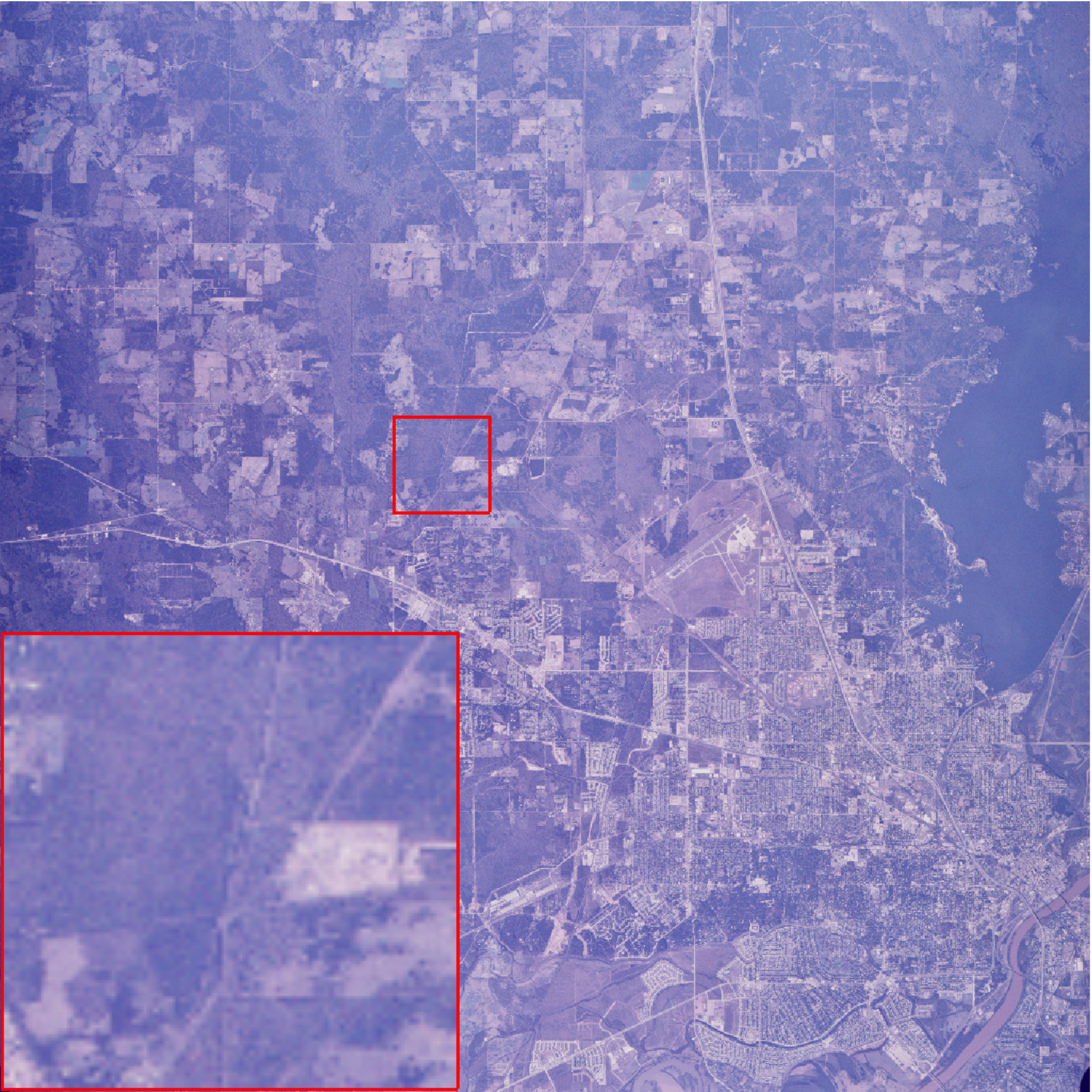}\vspace{0pt}
		\includegraphics[width=\linewidth]{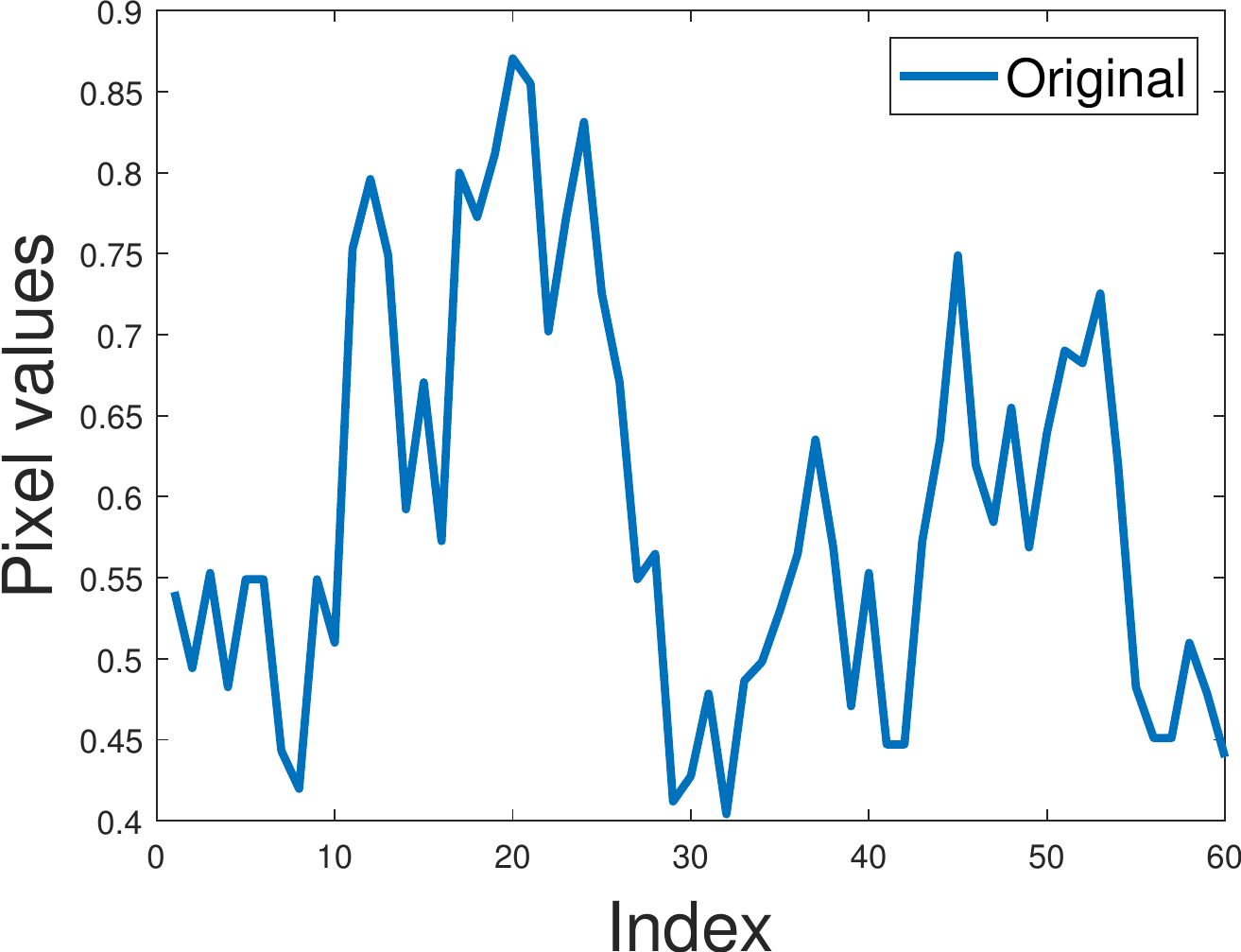}\vspace{0pt}
		\includegraphics[width=\linewidth]{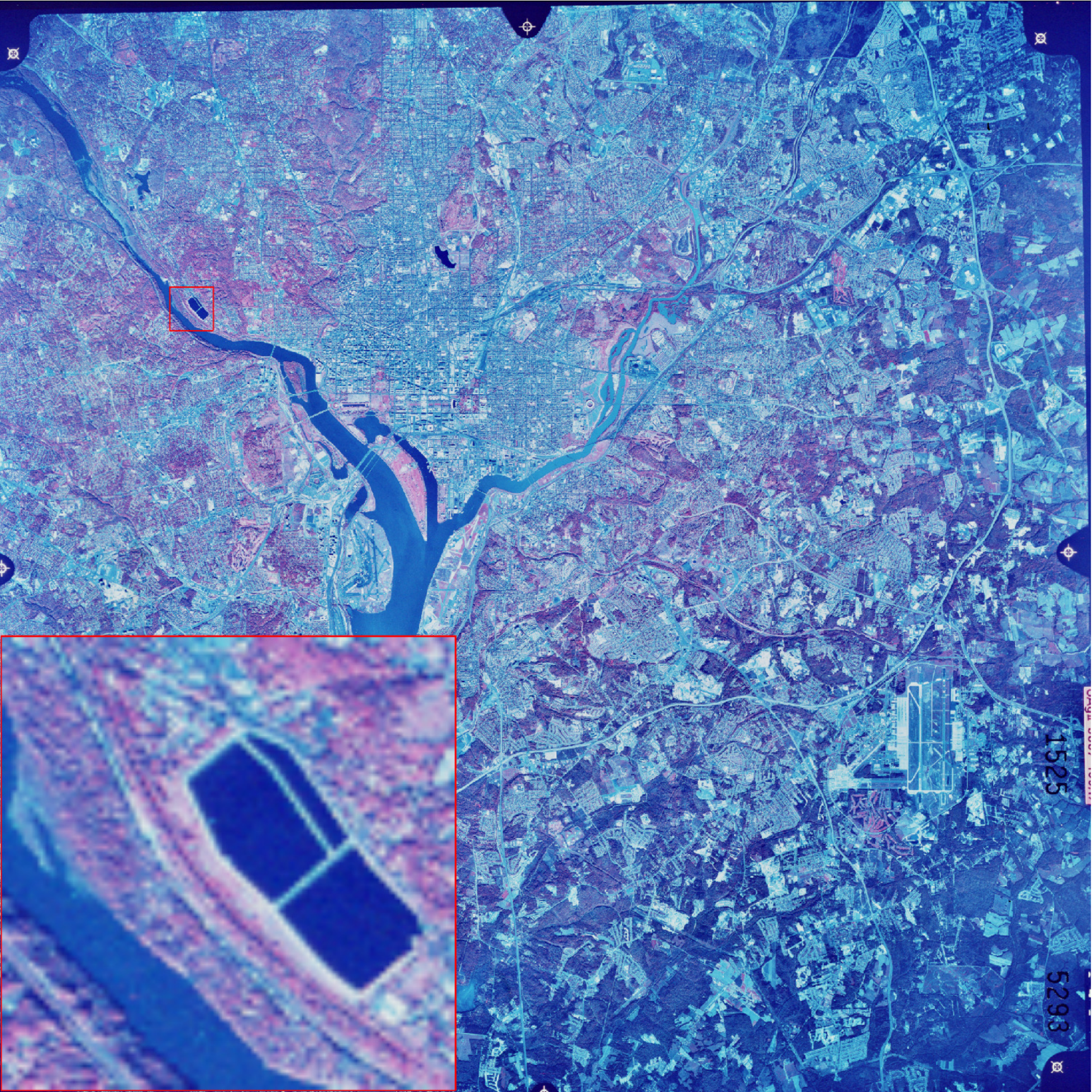}\vspace{0pt}
		\includegraphics[width=\linewidth]{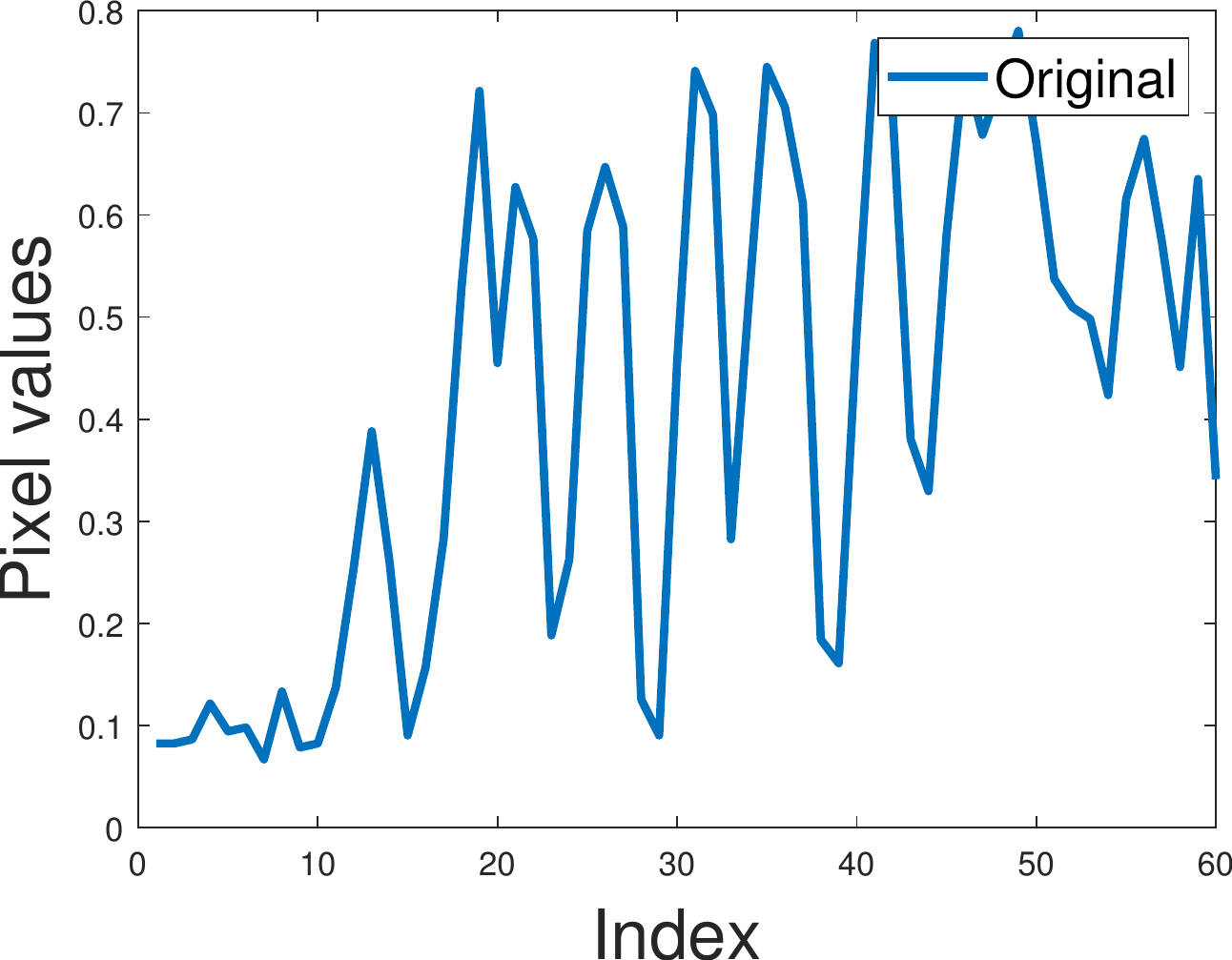}
		\caption{Original}
	\end{subfigure}   	
	\begin{subfigure}[b]{0.1175\linewidth}
		\centering
		\includegraphics[width=\linewidth]{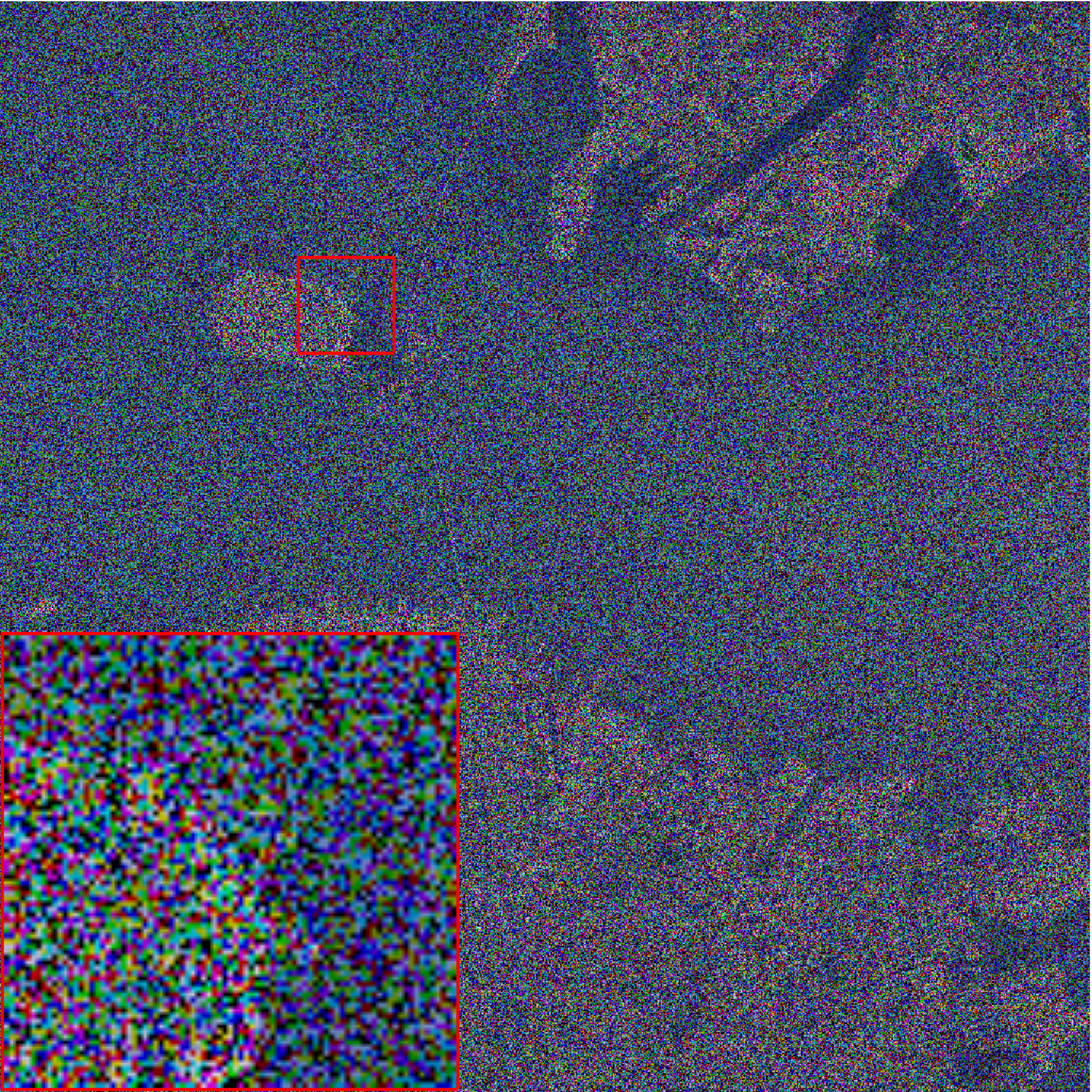}\vspace{0pt}
		\includegraphics[width=\linewidth]{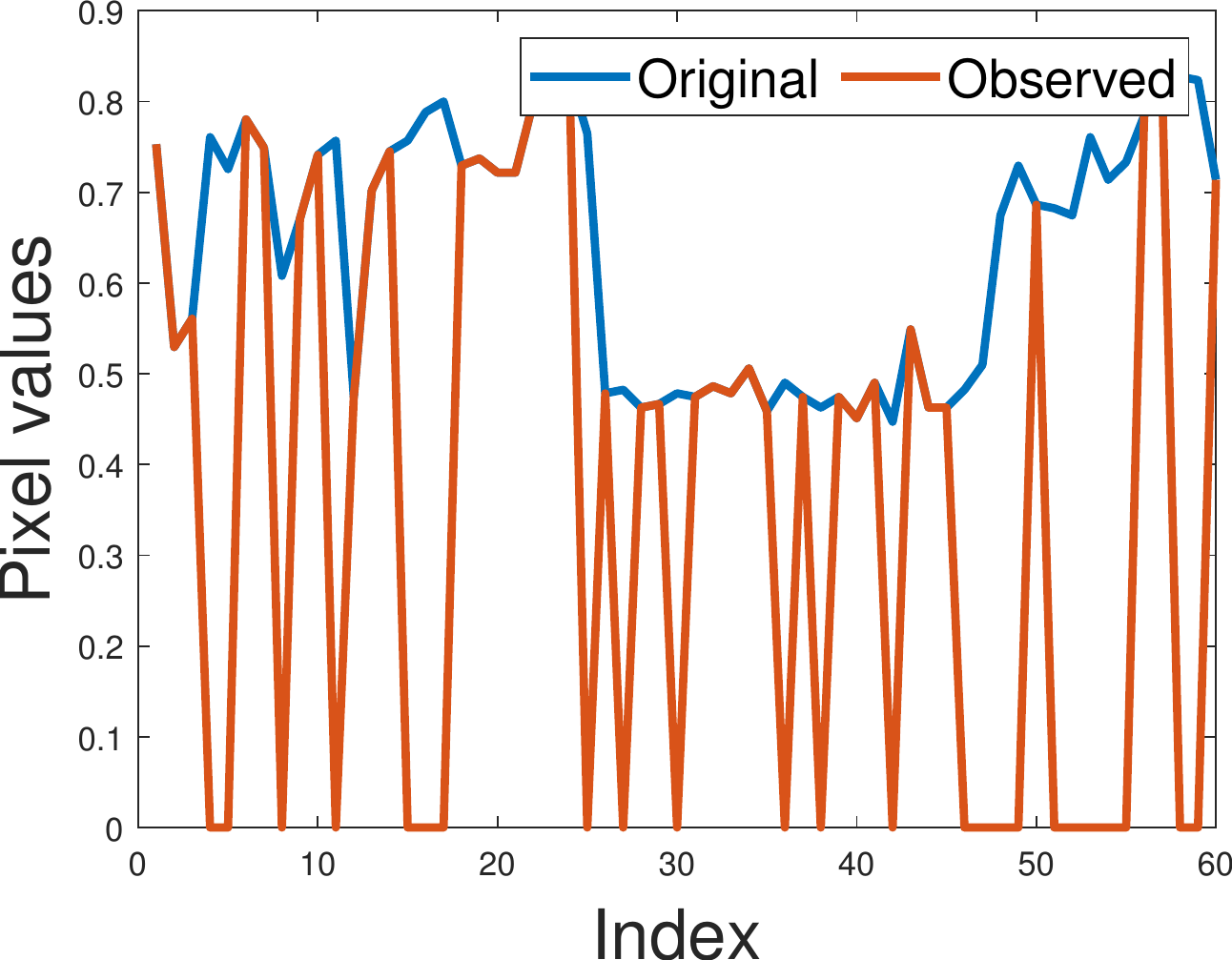}\vspace{0pt}
		\includegraphics[width=\linewidth]{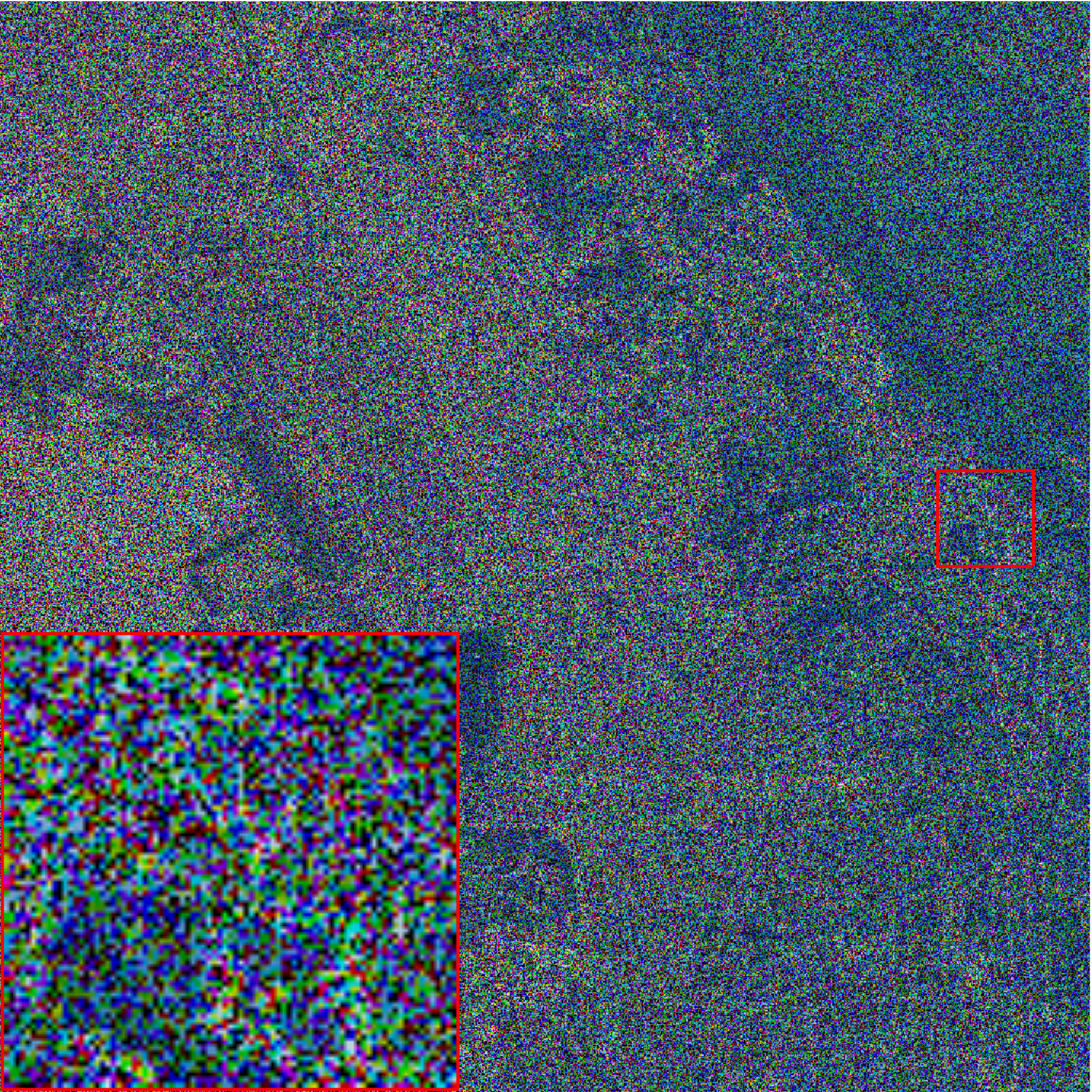}\vspace{0pt}
		\includegraphics[width=\linewidth]{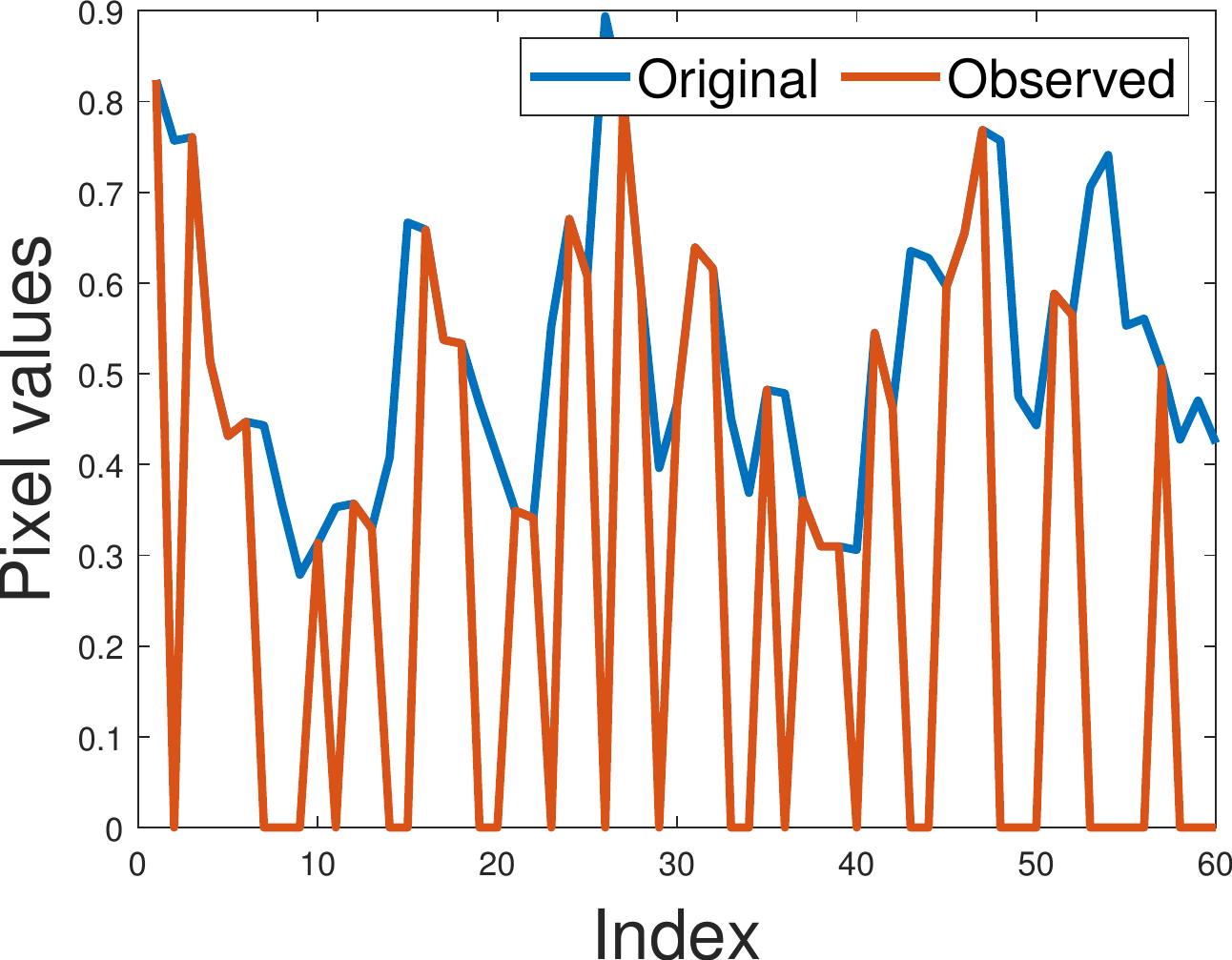}\vspace{0pt}
		\includegraphics[width=\linewidth]{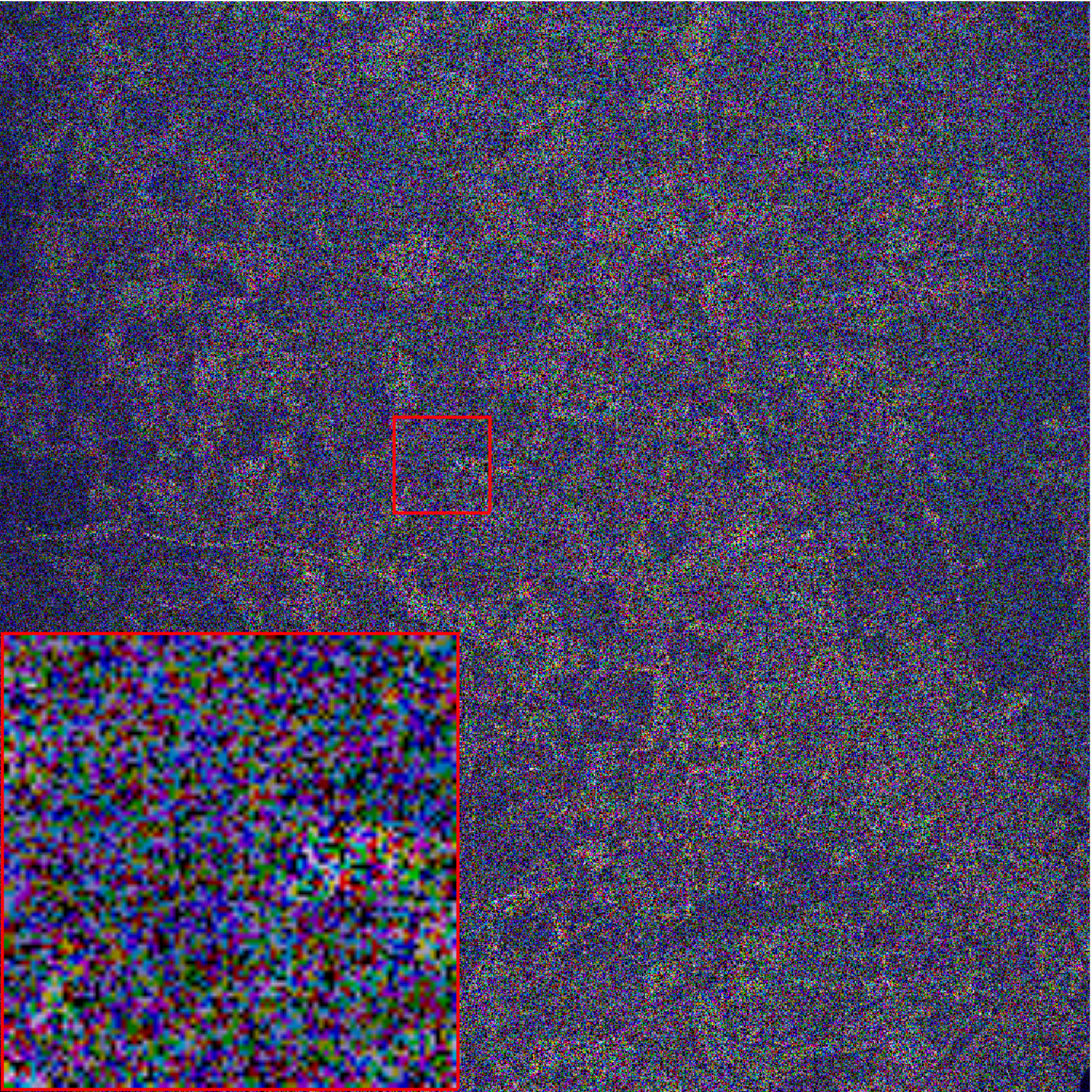}\vspace{0pt}
		\includegraphics[width=\linewidth]{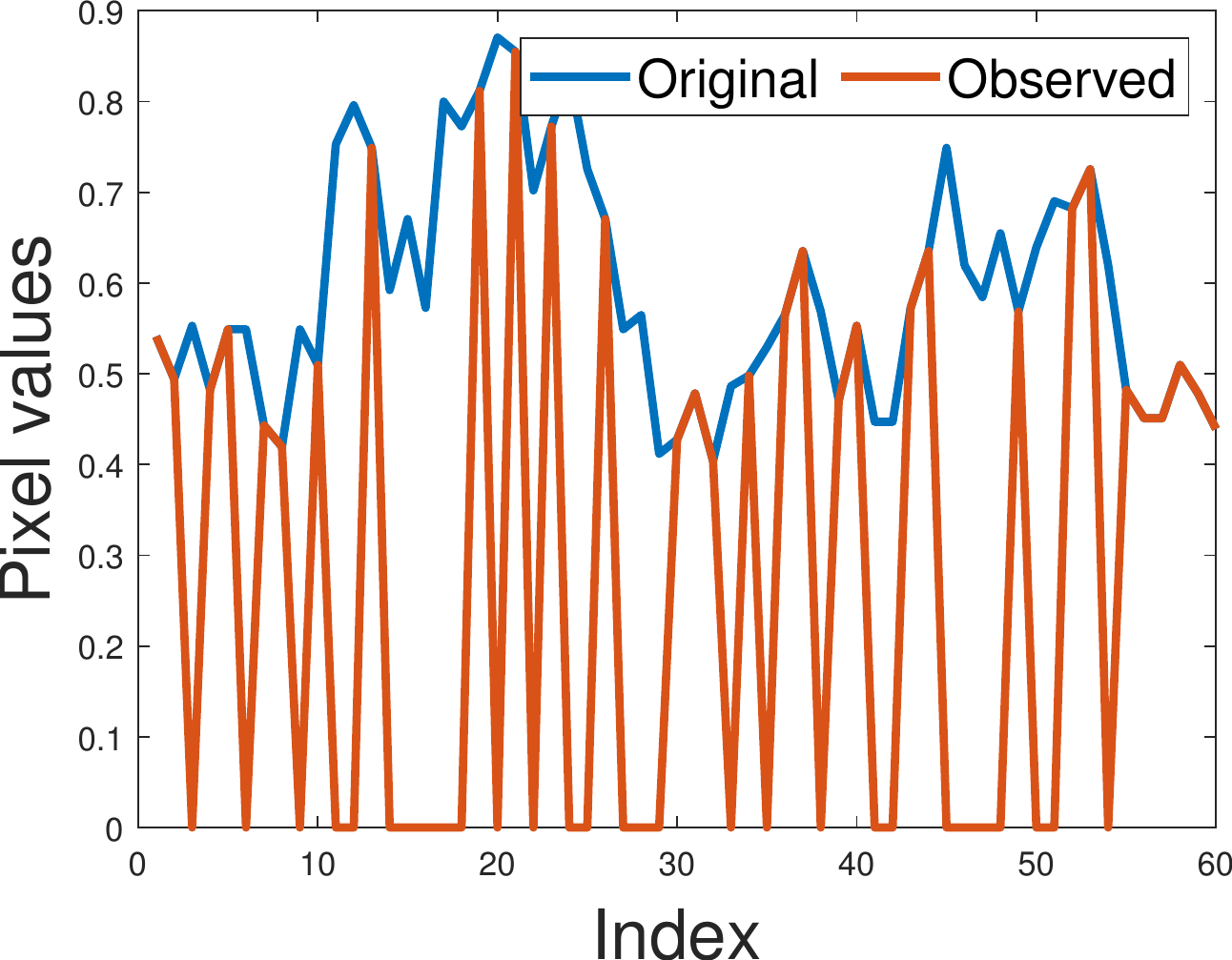}\vspace{0pt}
		\includegraphics[width=\linewidth]{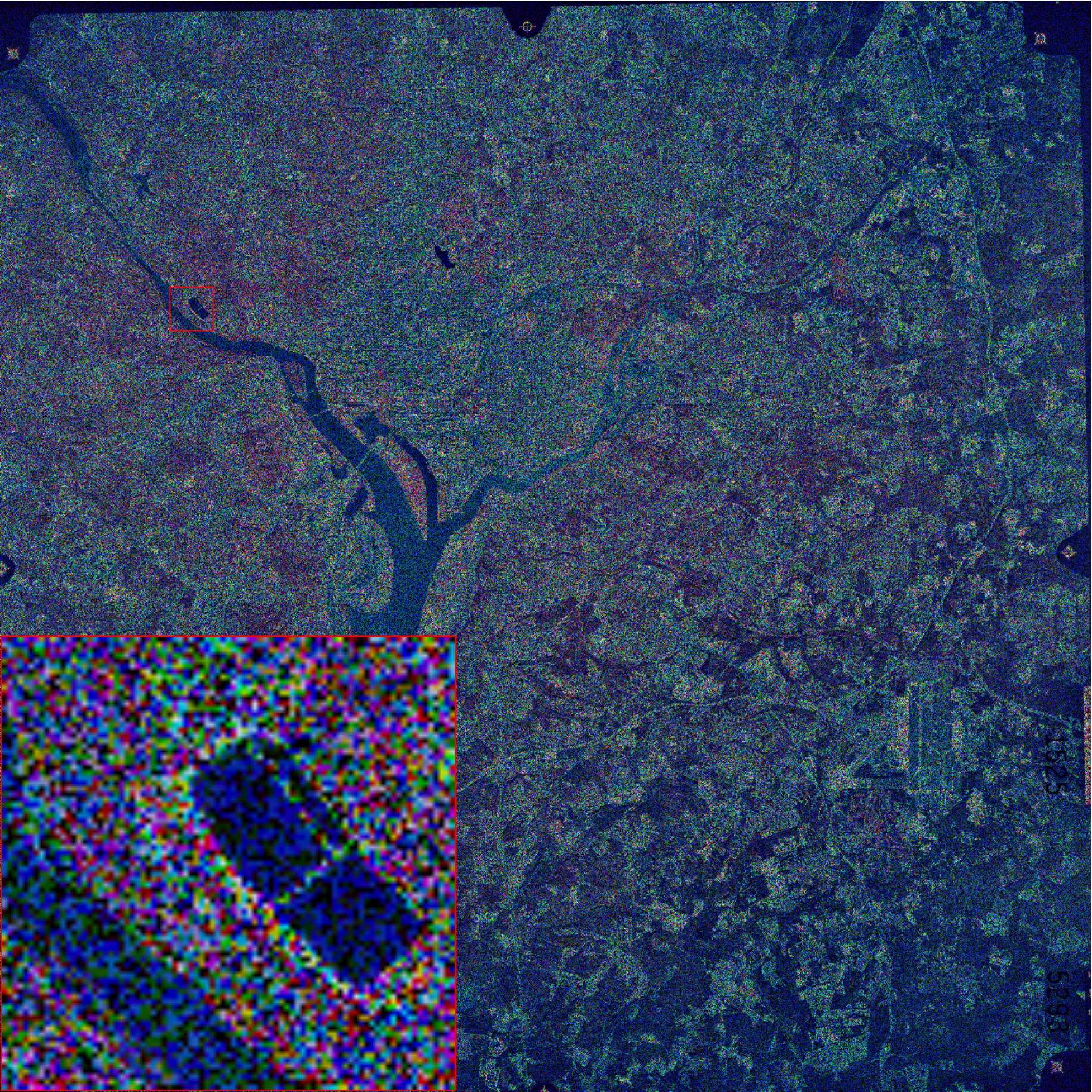}\vspace{0pt}
		\includegraphics[width=\linewidth]{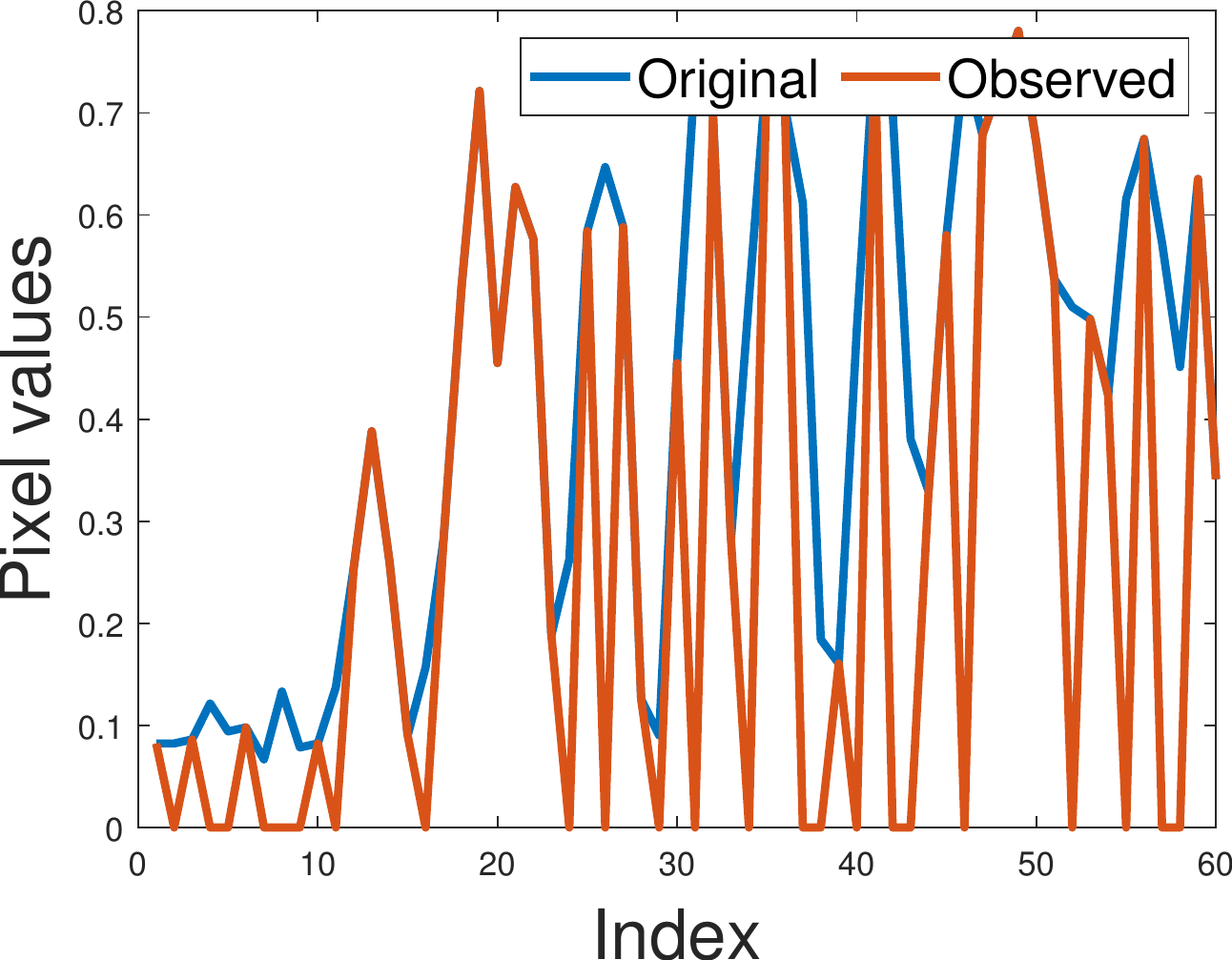}
		\caption{Observed}
	\end{subfigure}
	\begin{subfigure}[b]{0.118\linewidth}
		\centering
		\includegraphics[width=\linewidth]{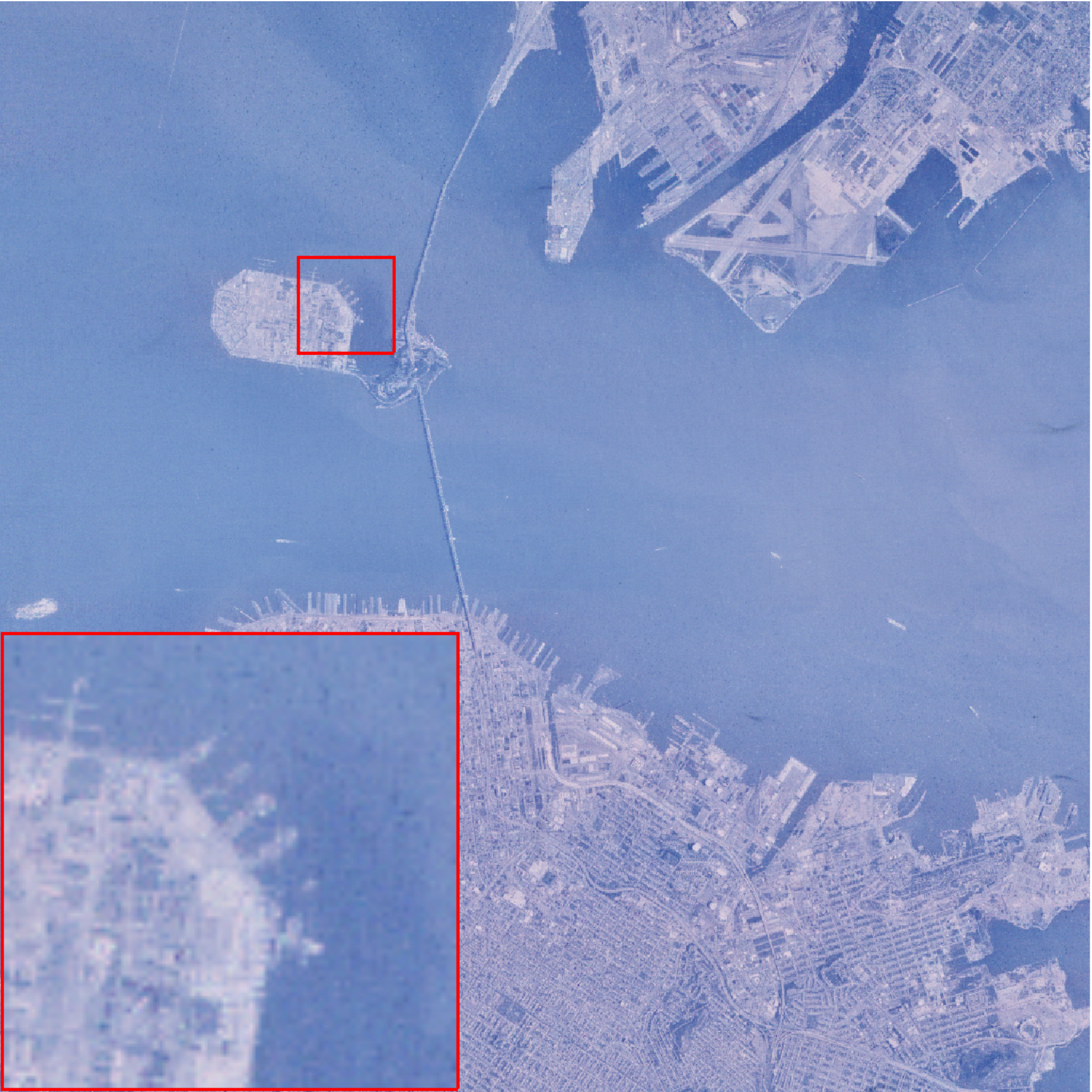}\vspace{0pt}
		\includegraphics[width=\linewidth]{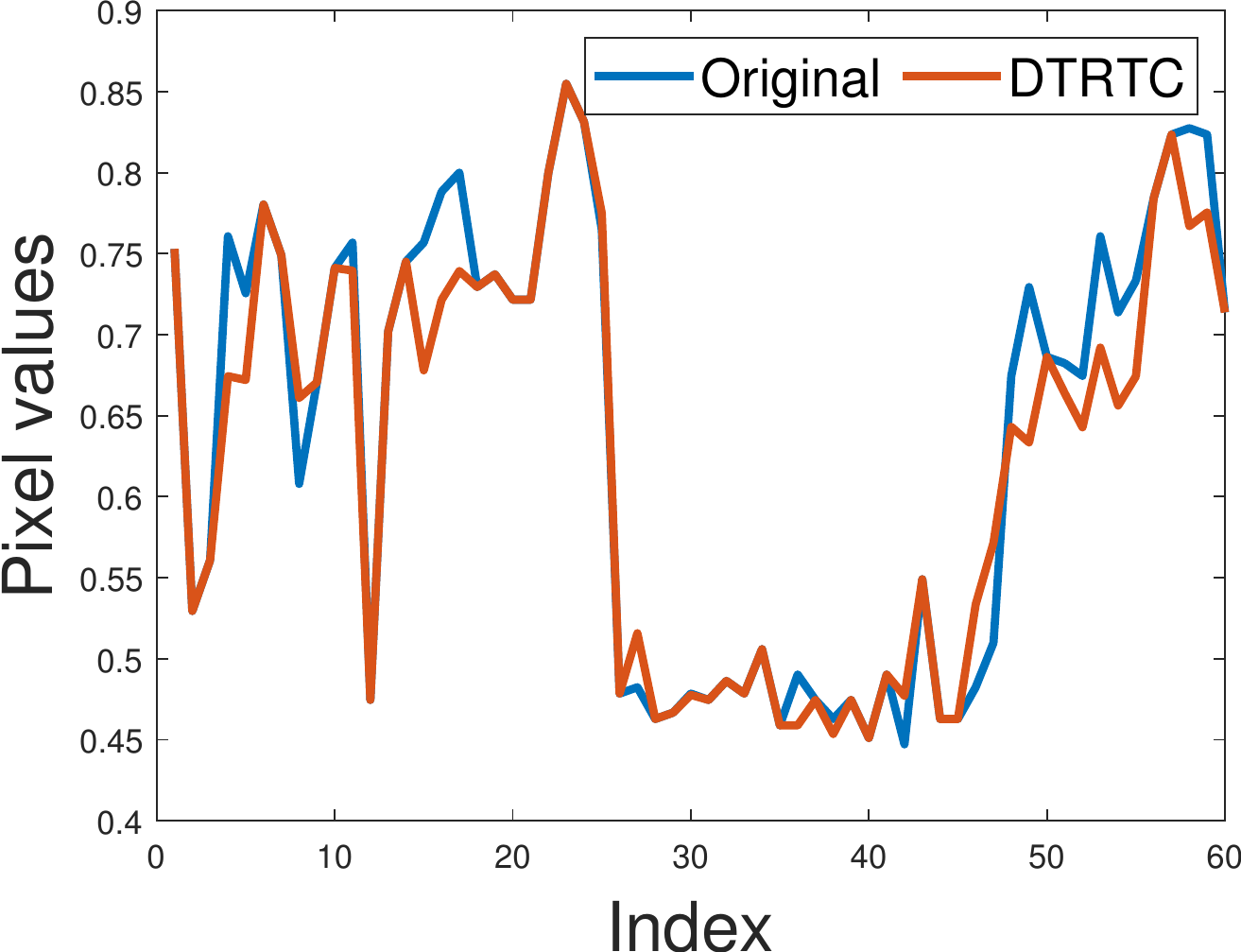}\vspace{0pt}
		\includegraphics[width=\linewidth]{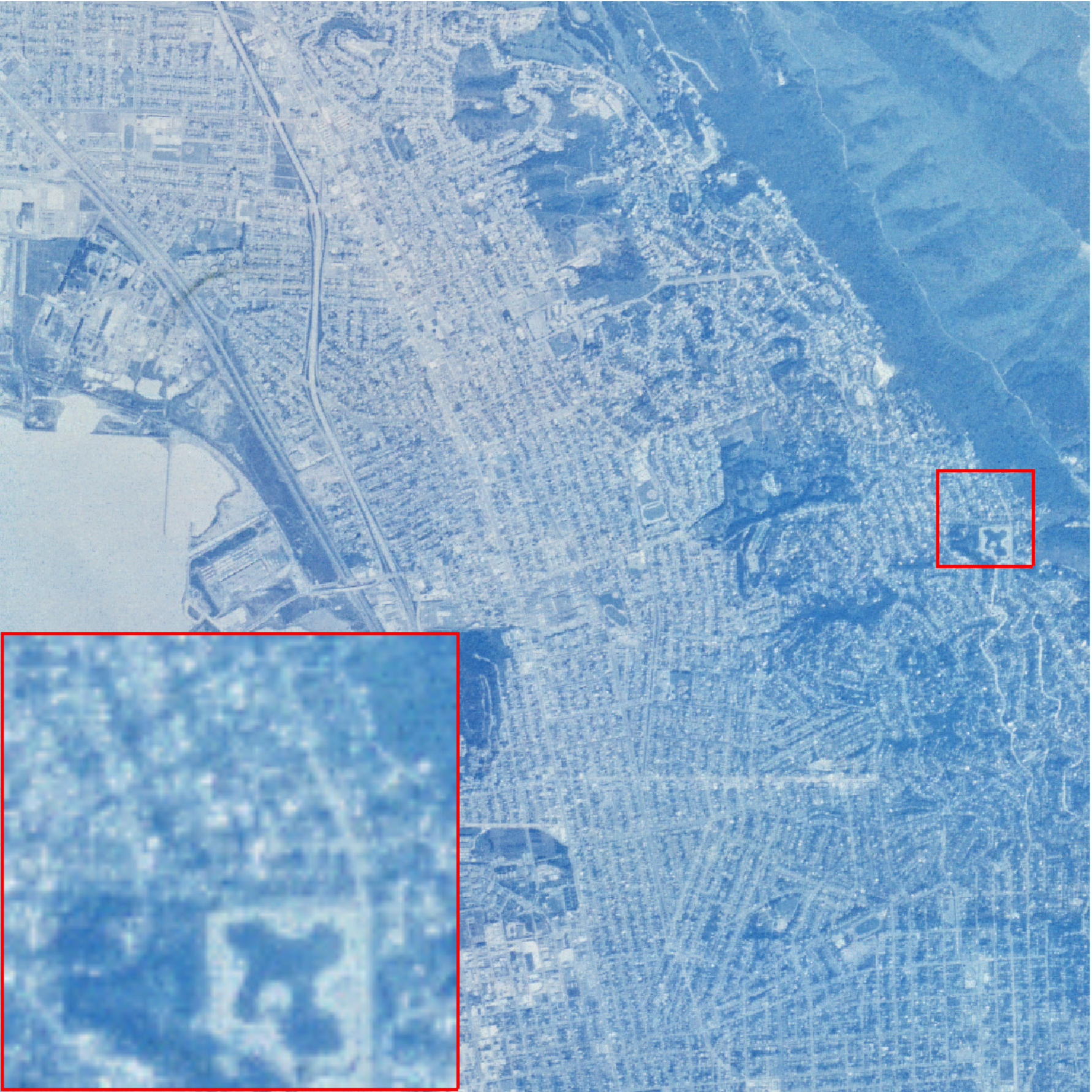}\vspace{0pt}
		\includegraphics[width=\linewidth]{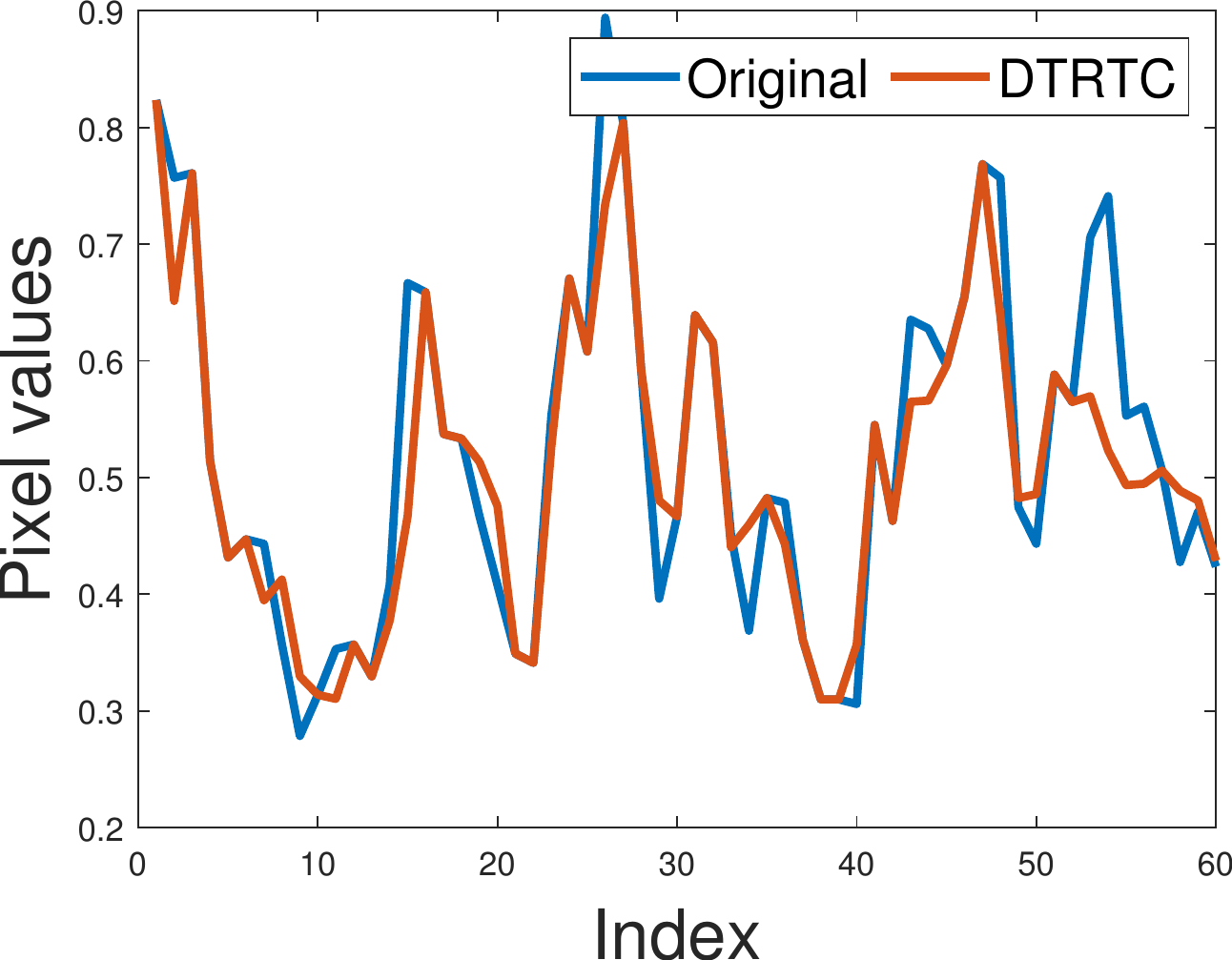}\vspace{0pt}
		\includegraphics[width=\linewidth]{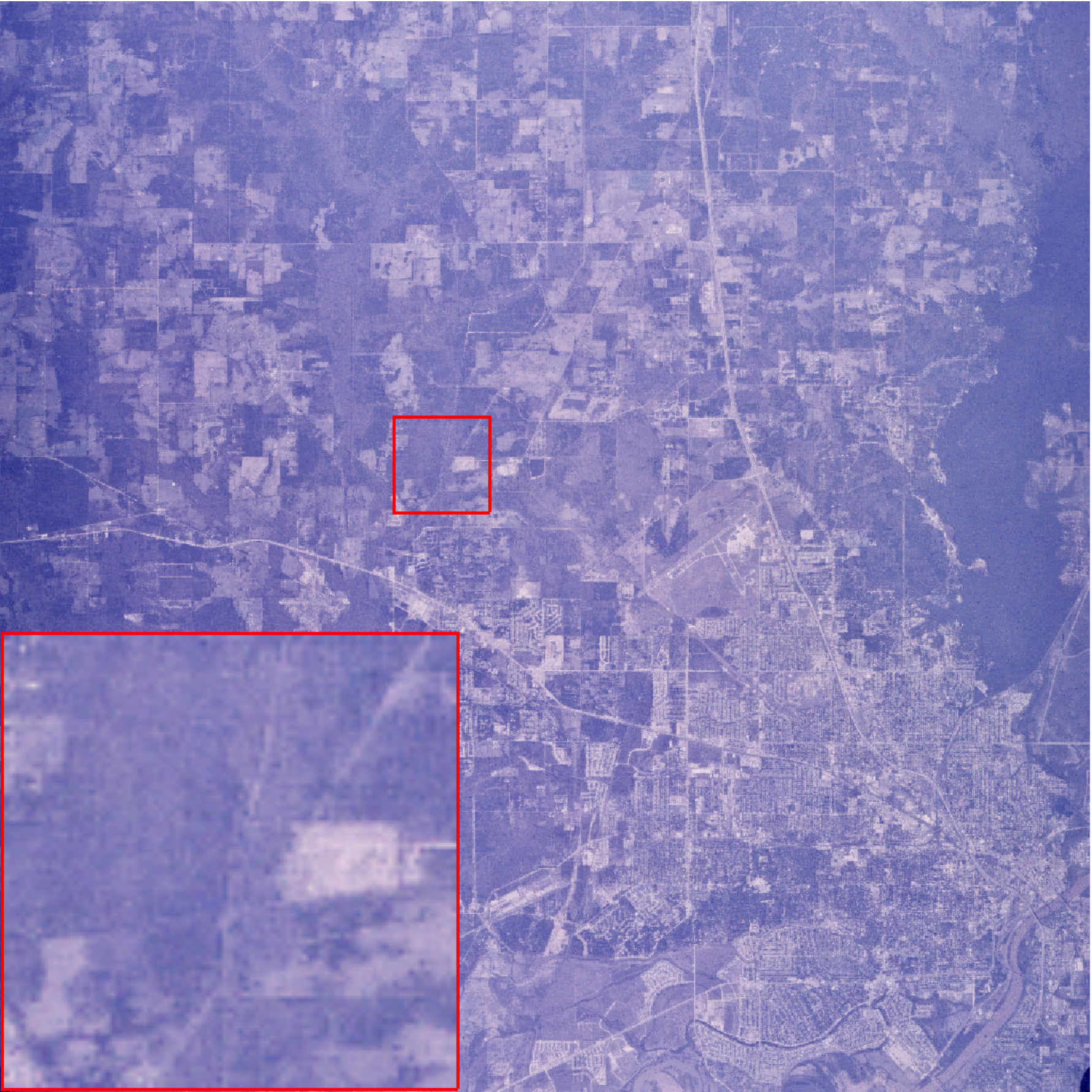}\vspace{0pt}
		\includegraphics[width=\linewidth]{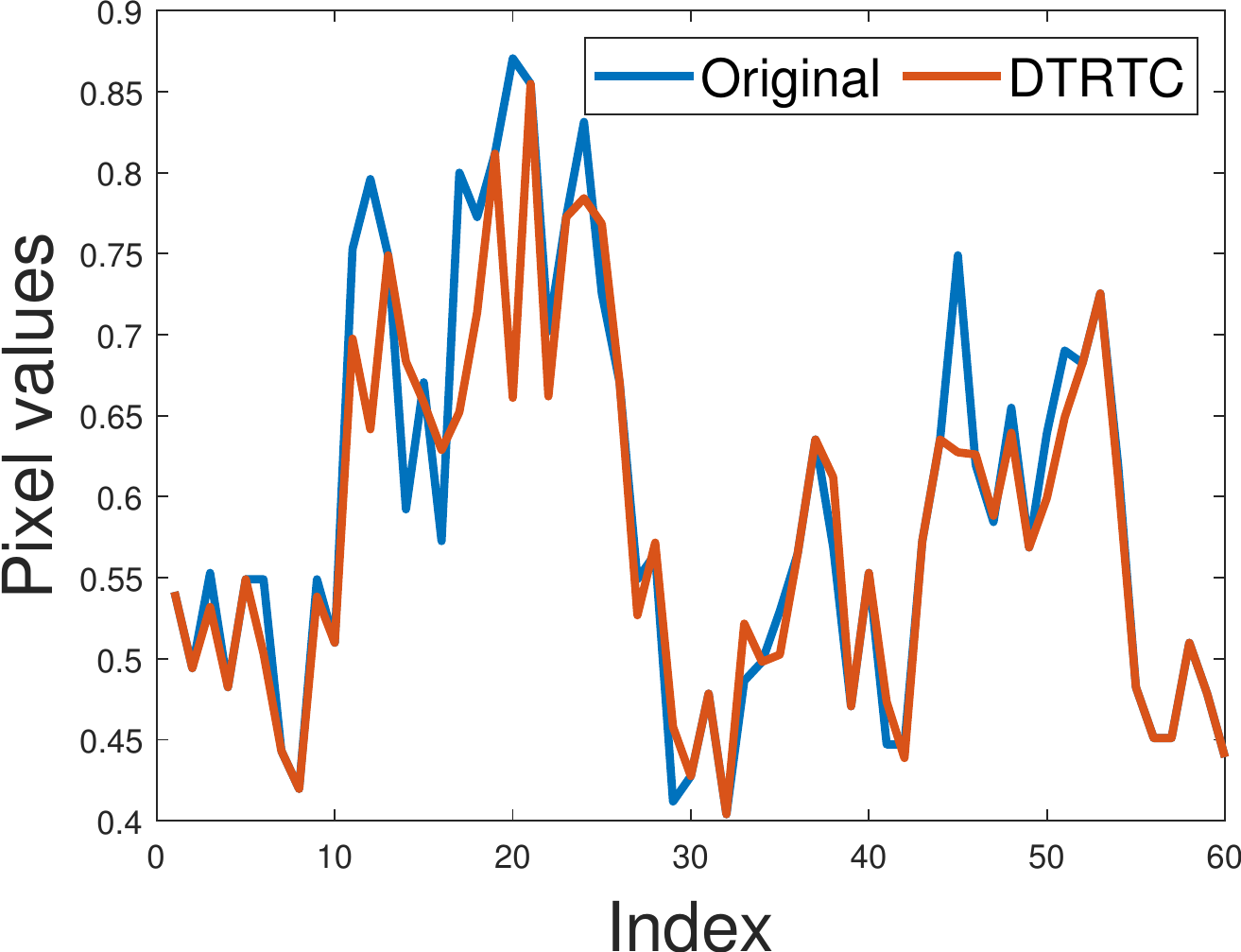}\vspace{0pt}
		\includegraphics[width=\linewidth]{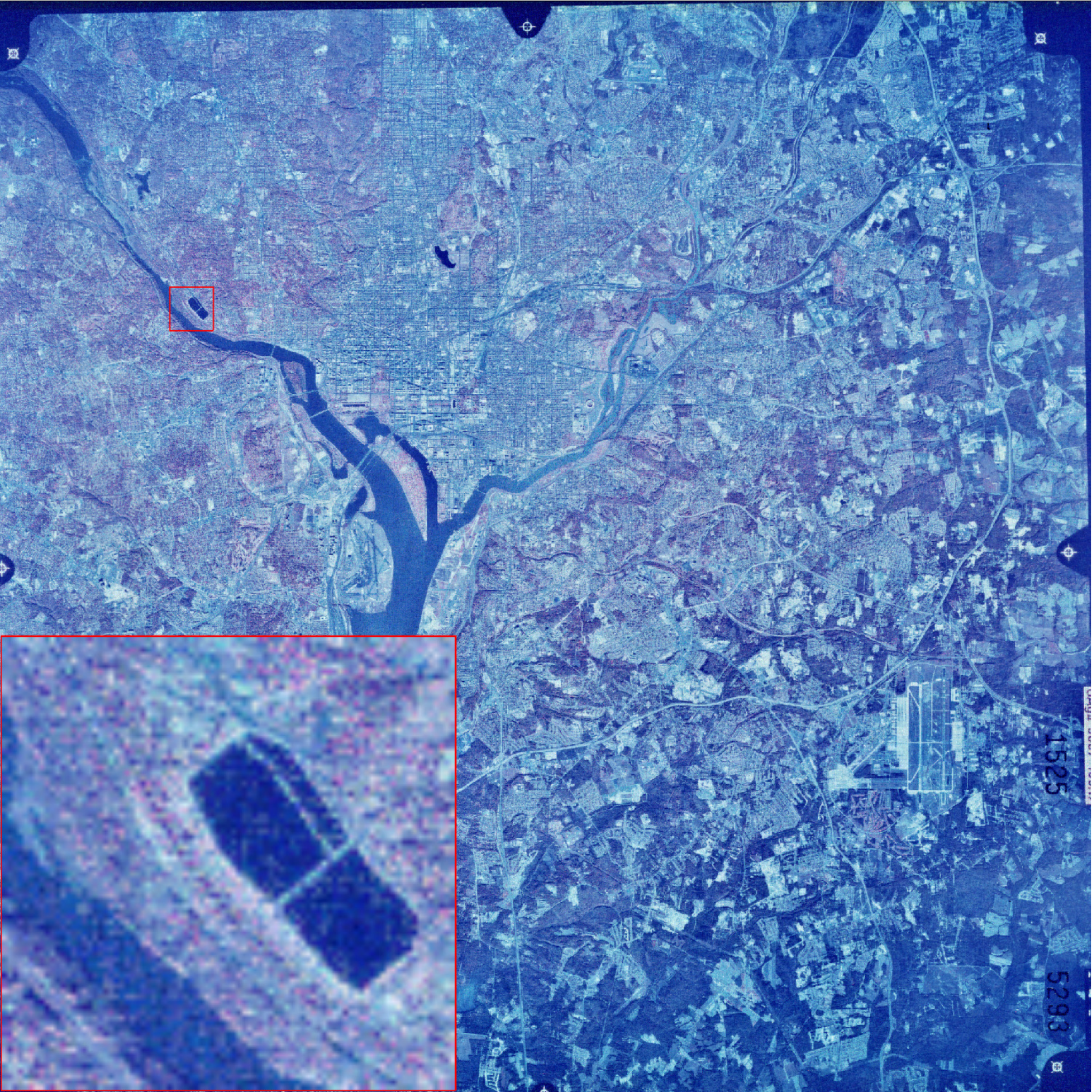}\vspace{0pt}
		\includegraphics[width=\linewidth]{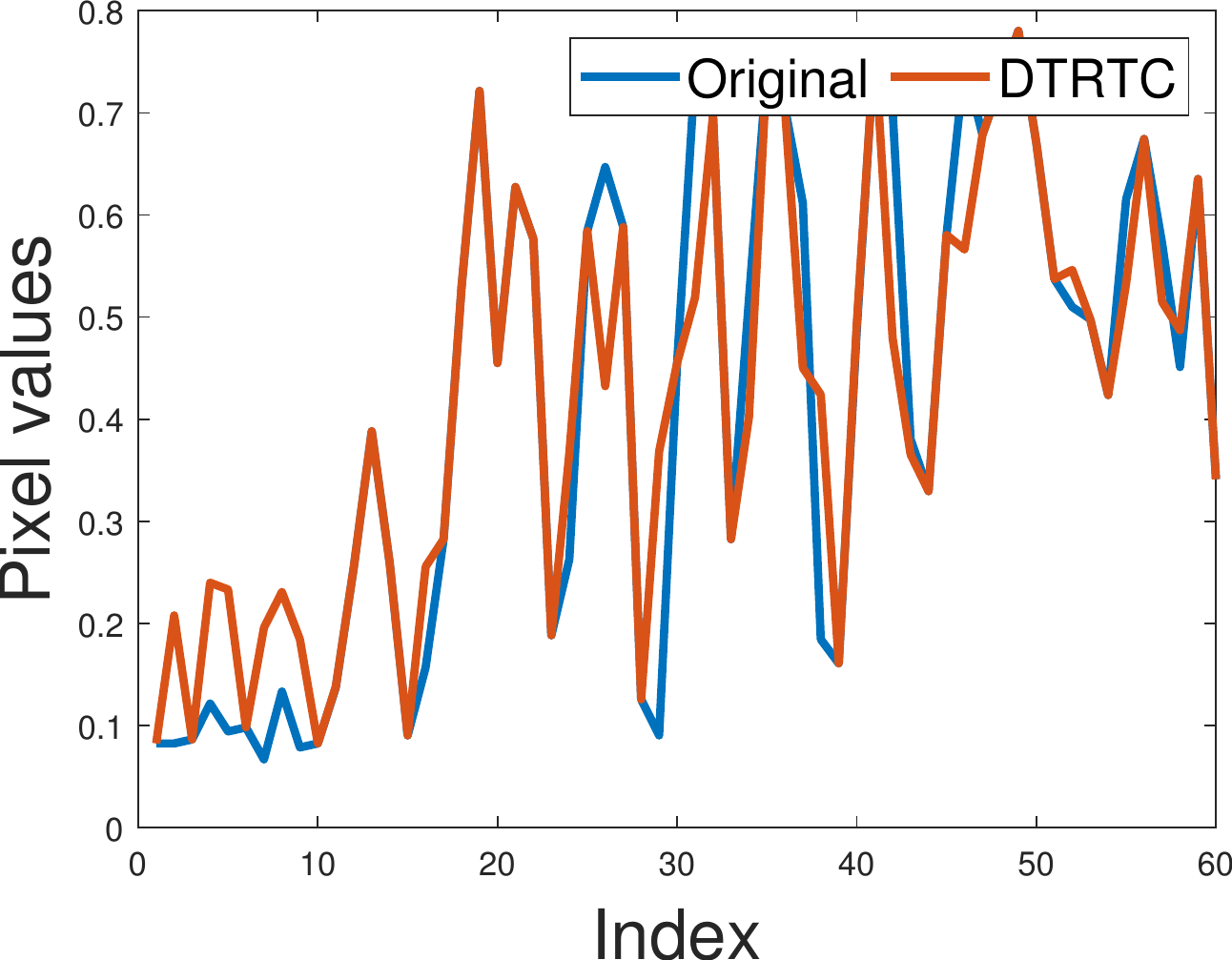}
		\caption{DTRTC}
	\end{subfigure}
	\begin{subfigure}[b]{0.118\linewidth}
		\centering
		\includegraphics[width=\linewidth]{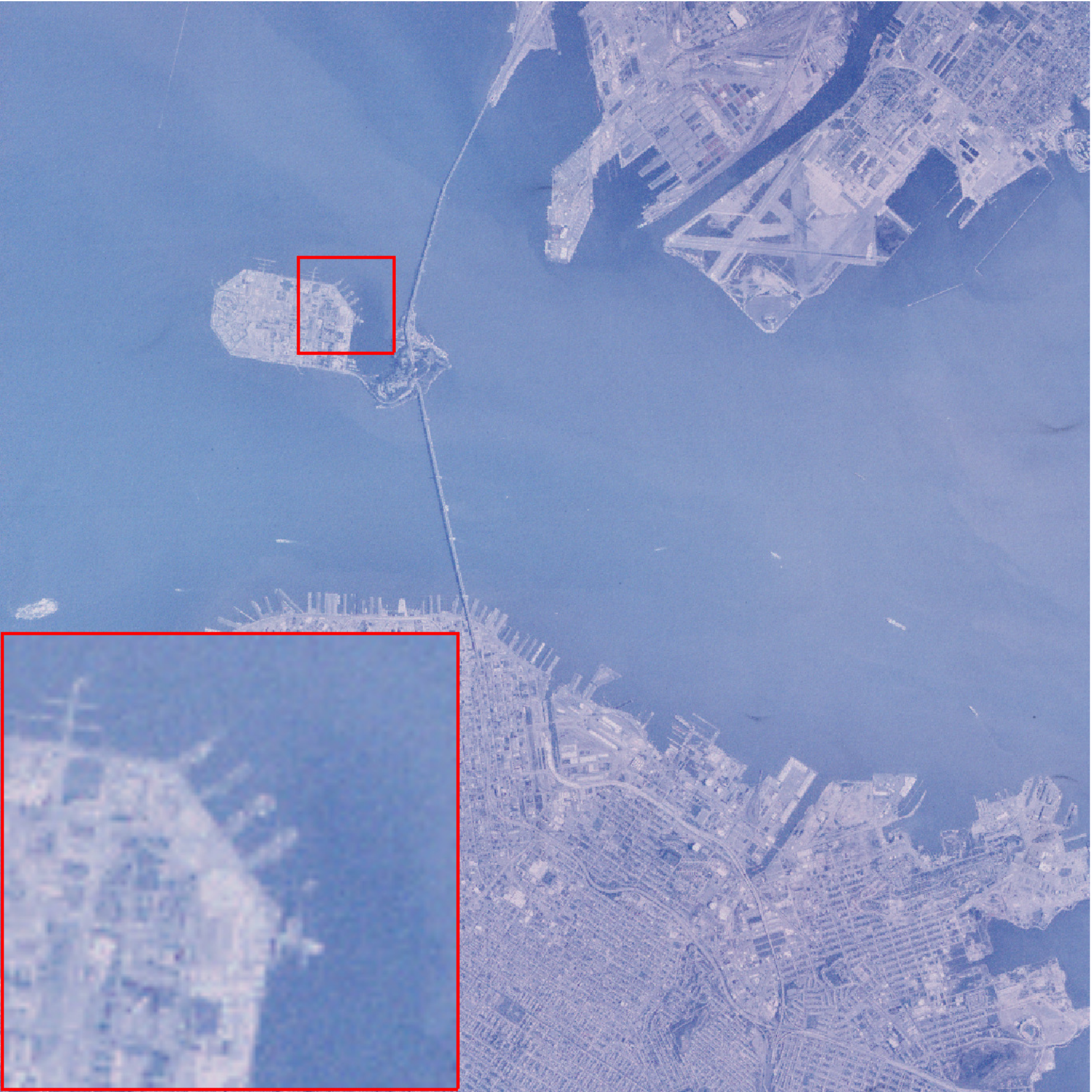}\vspace{0pt}
		\includegraphics[width=\linewidth]{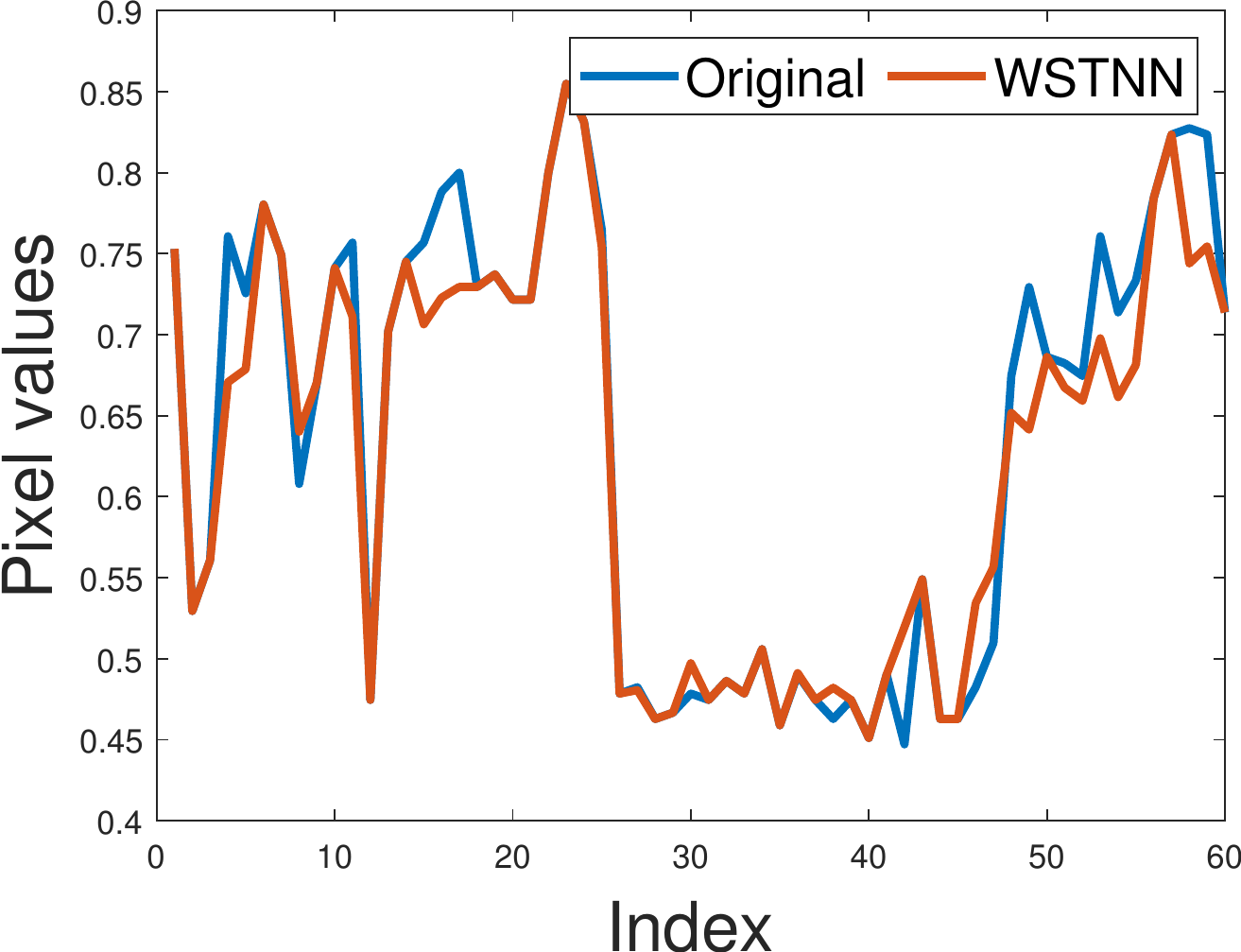}\vspace{0pt}
		\includegraphics[width=\linewidth]{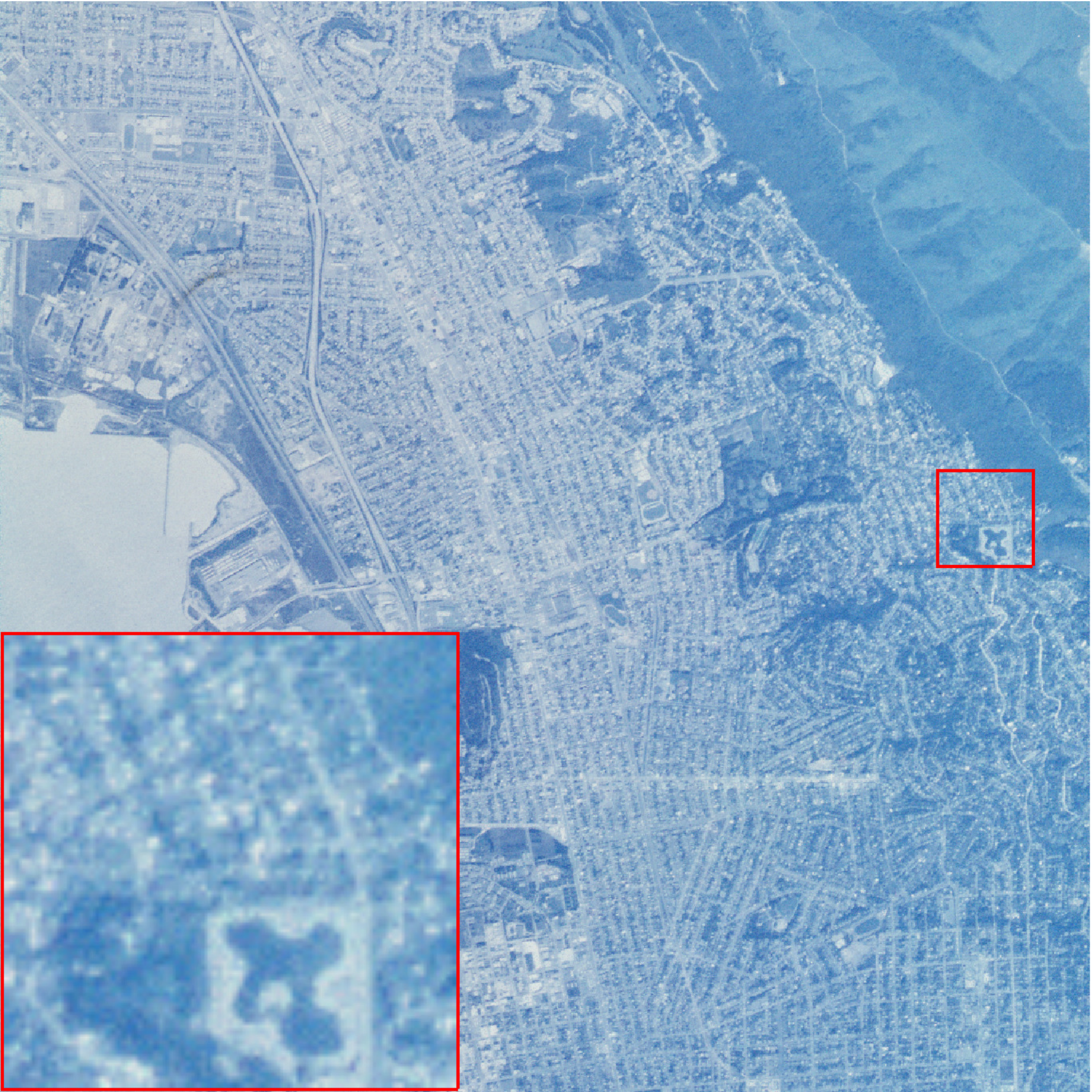}\vspace{0pt}
		\includegraphics[width=\linewidth]{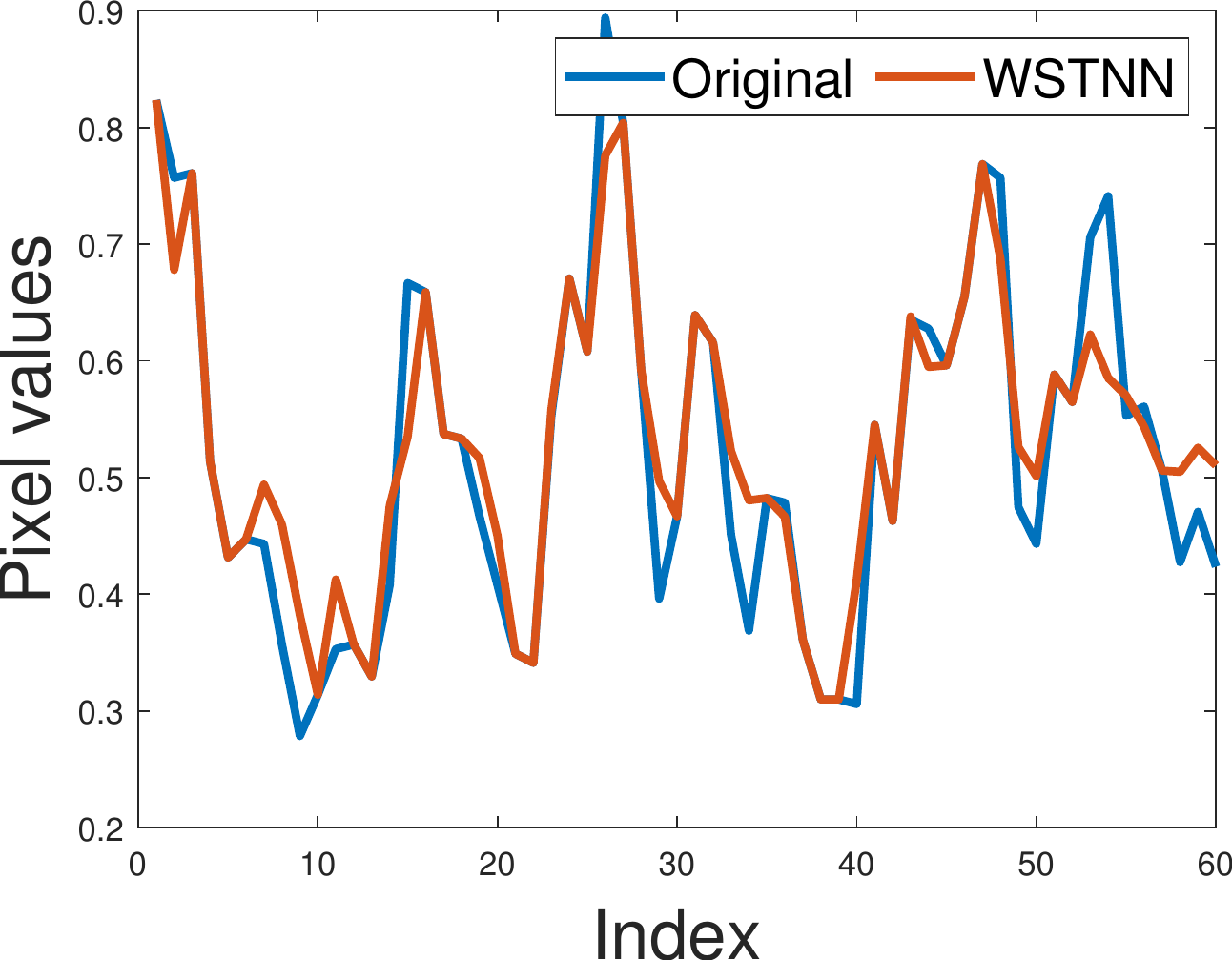}\vspace{0pt}
		\includegraphics[width=\linewidth]{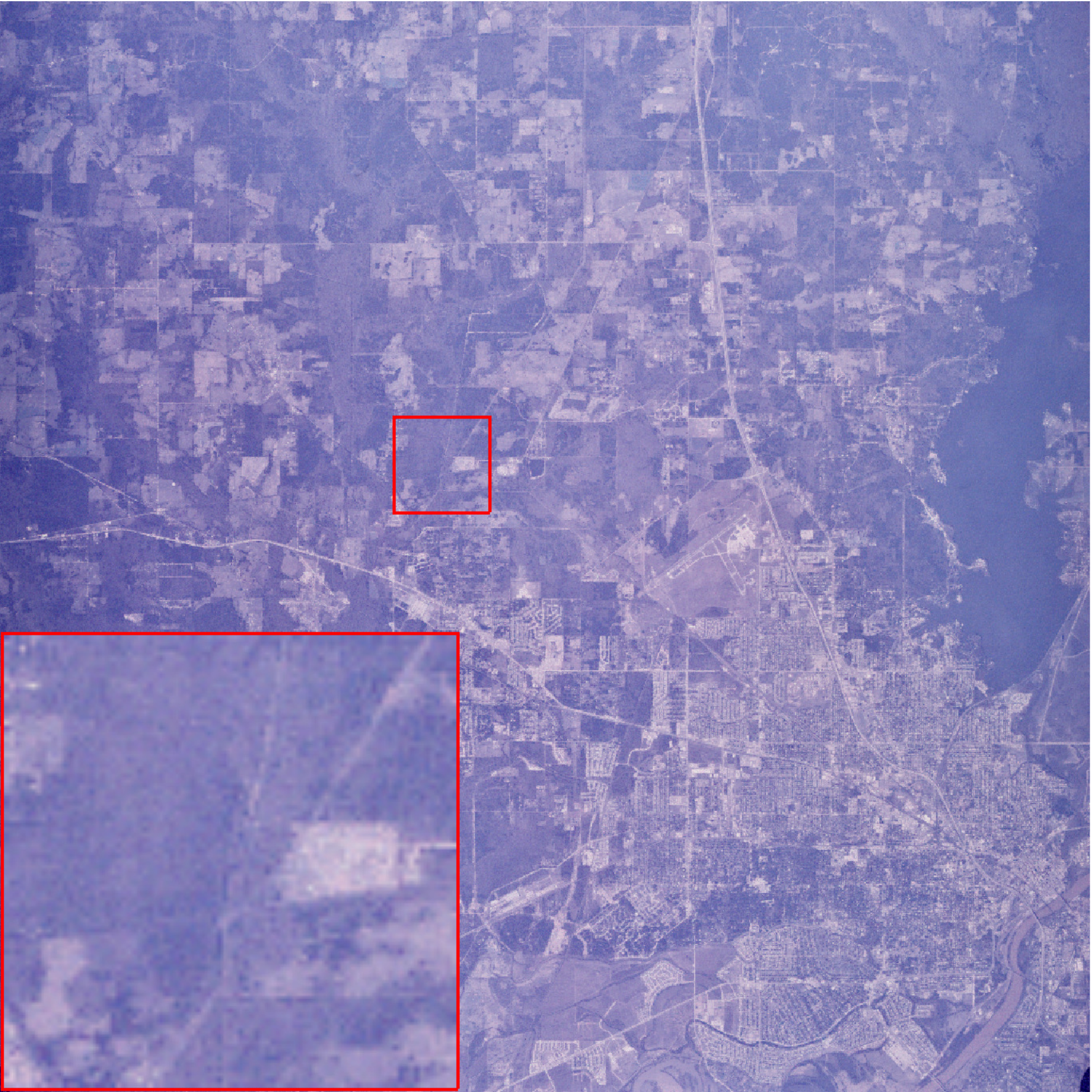}\vspace{0pt}
		\includegraphics[width=\linewidth]{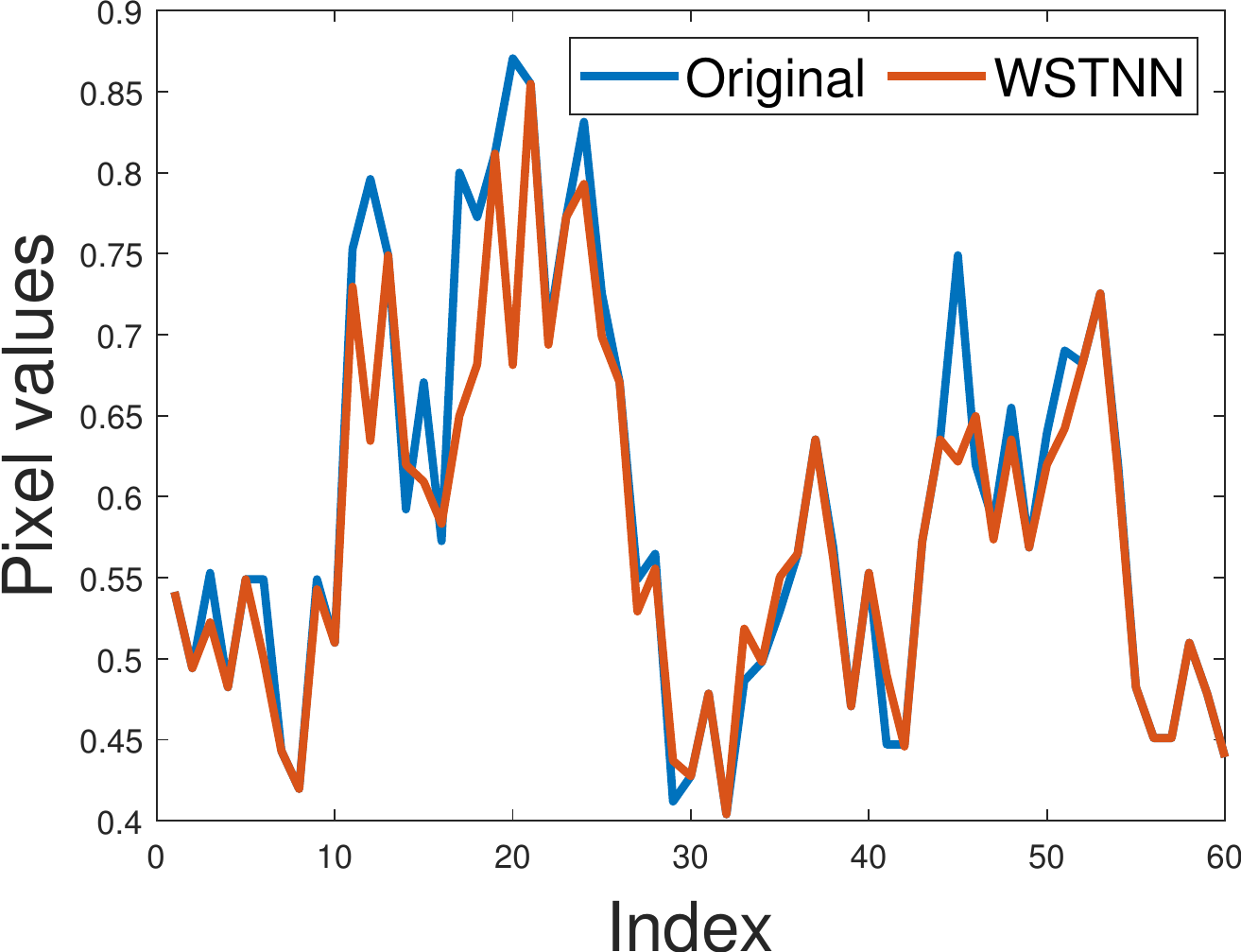}\vspace{0pt}
		\includegraphics[width=\linewidth]{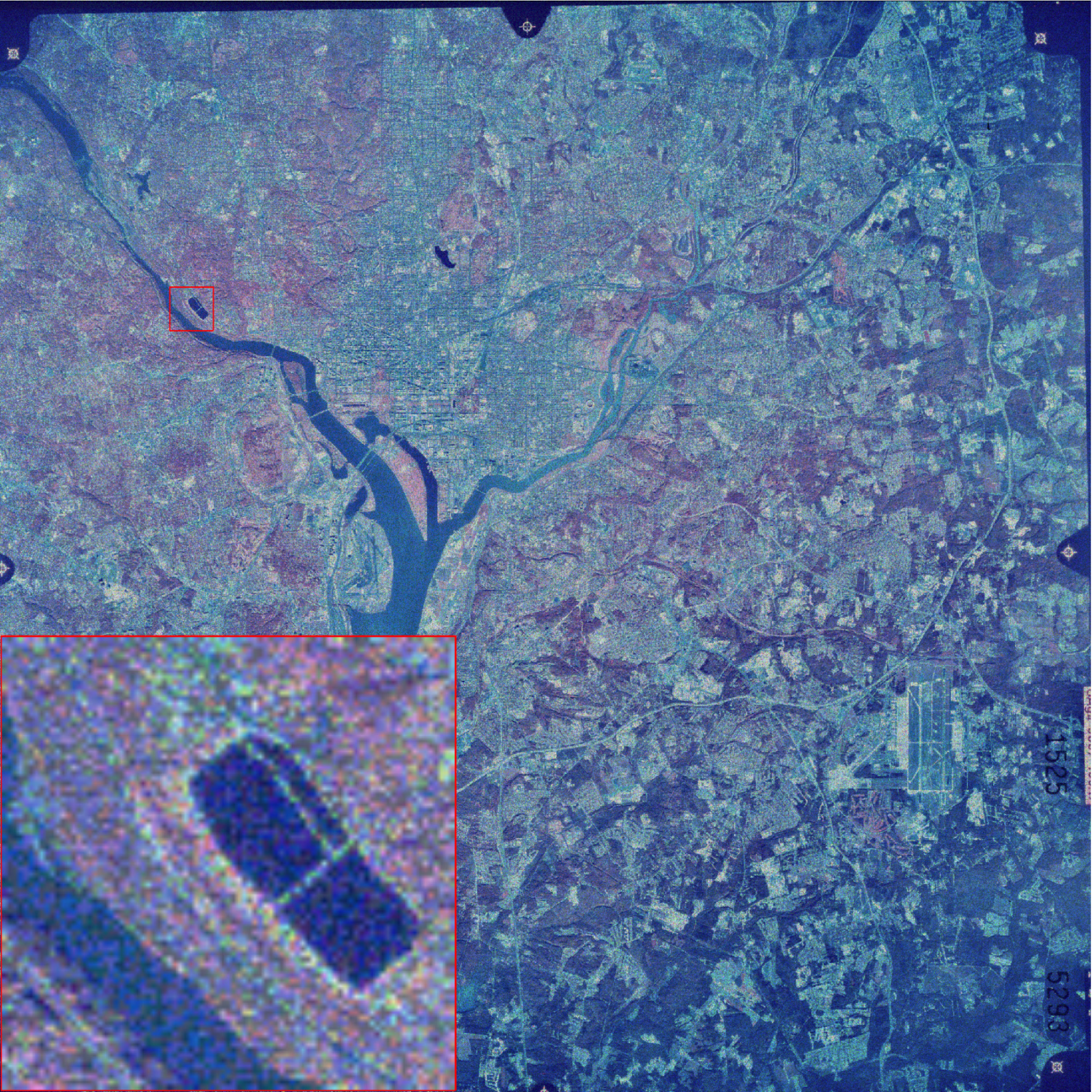}\vspace{0pt}
		\includegraphics[width=\linewidth]{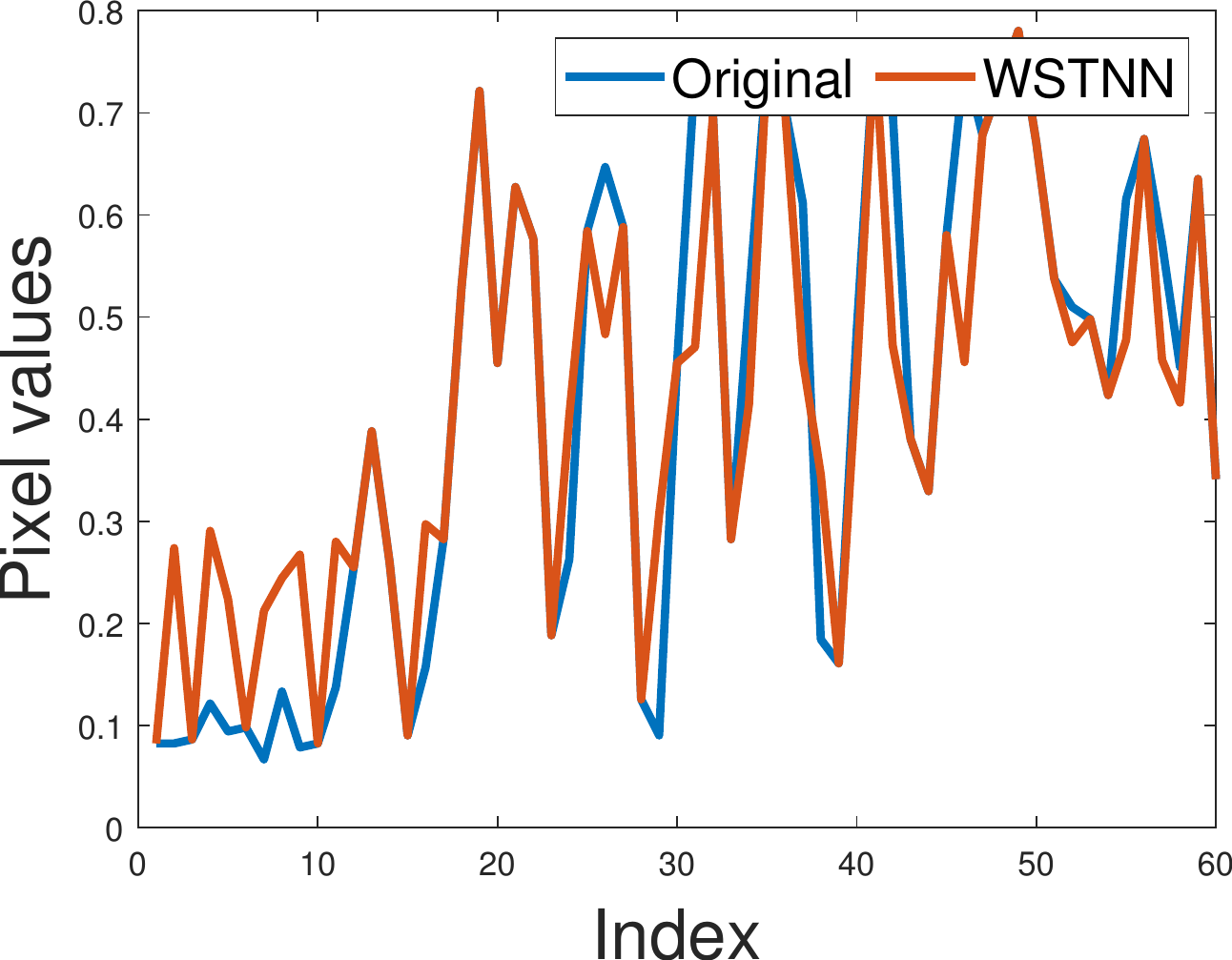}
		\caption{WSTNN}
	\end{subfigure}	
	\begin{subfigure}[b]{0.118\linewidth}
		\centering			
		\includegraphics[width=\linewidth]{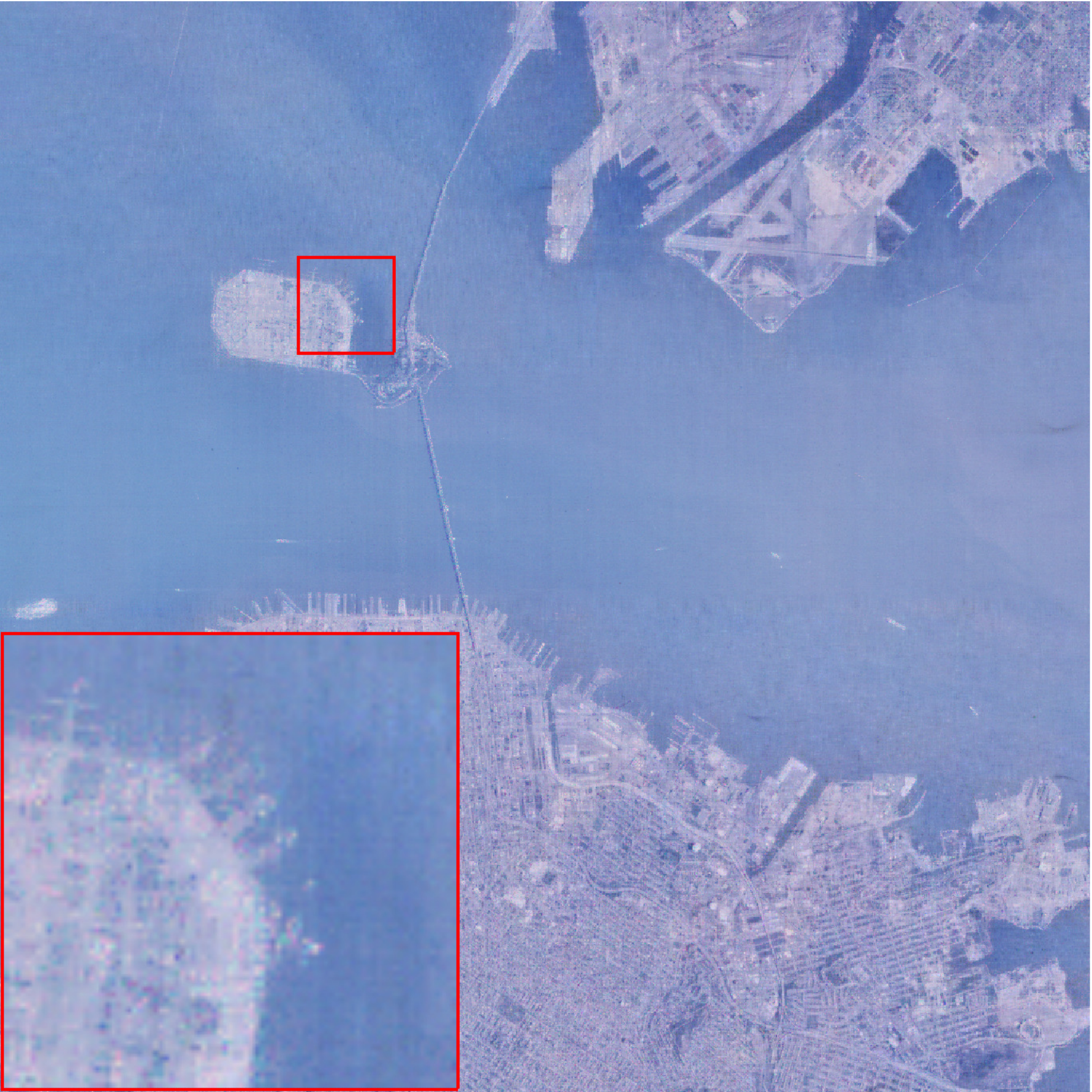}\vspace{0pt}
		\includegraphics[width=\linewidth]{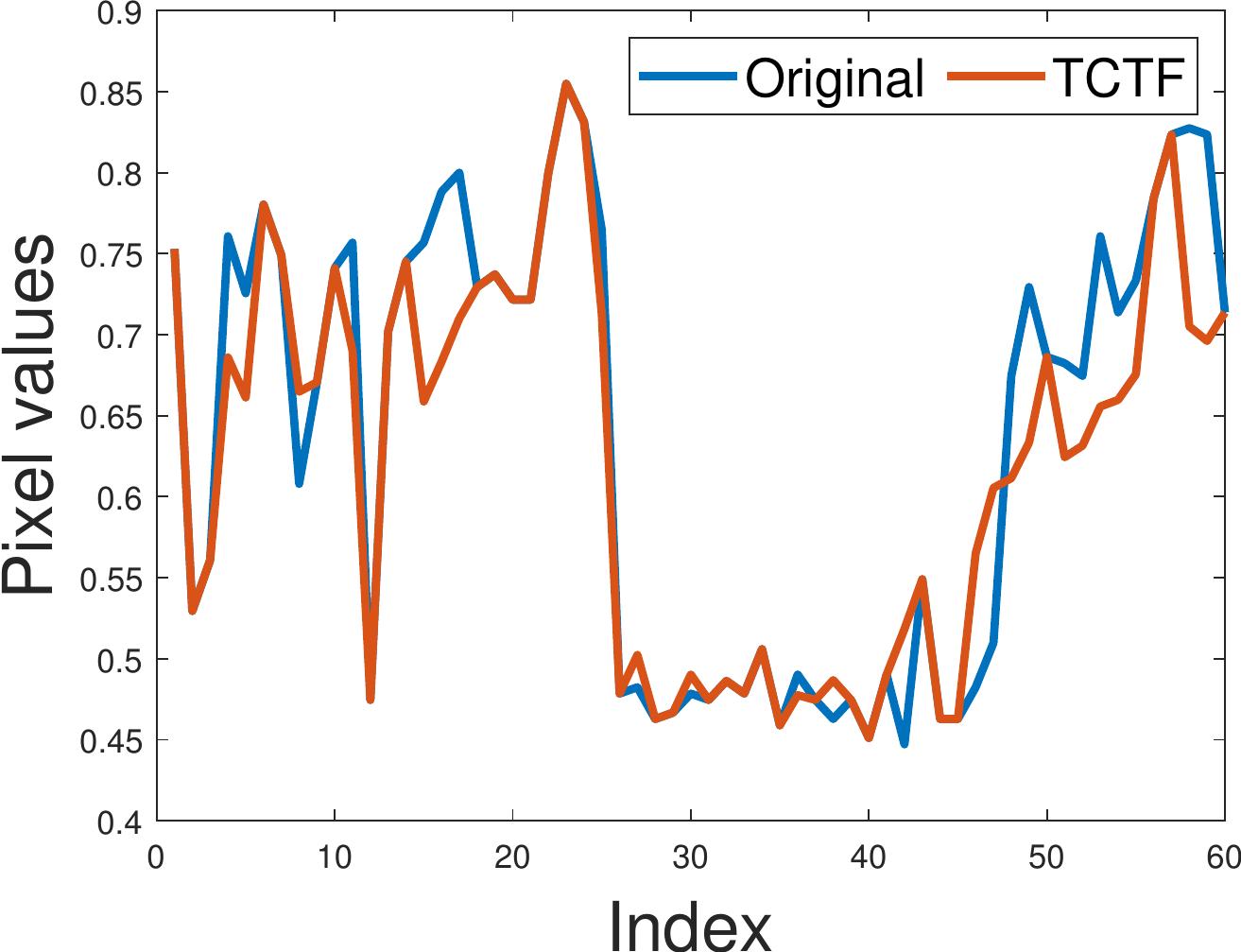}\vspace{0pt}
		\includegraphics[width=\linewidth]{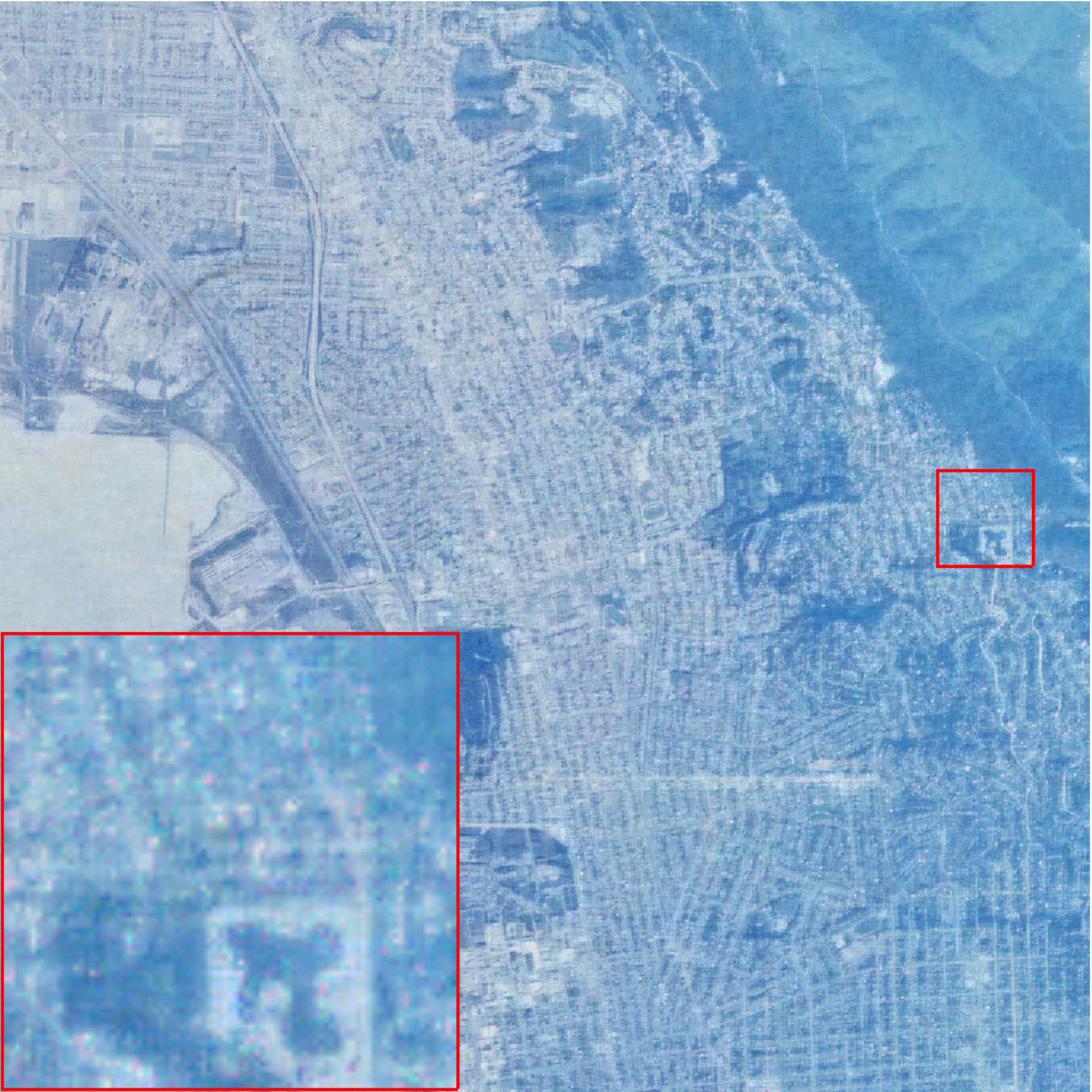}\vspace{0pt}
		\includegraphics[width=\linewidth]{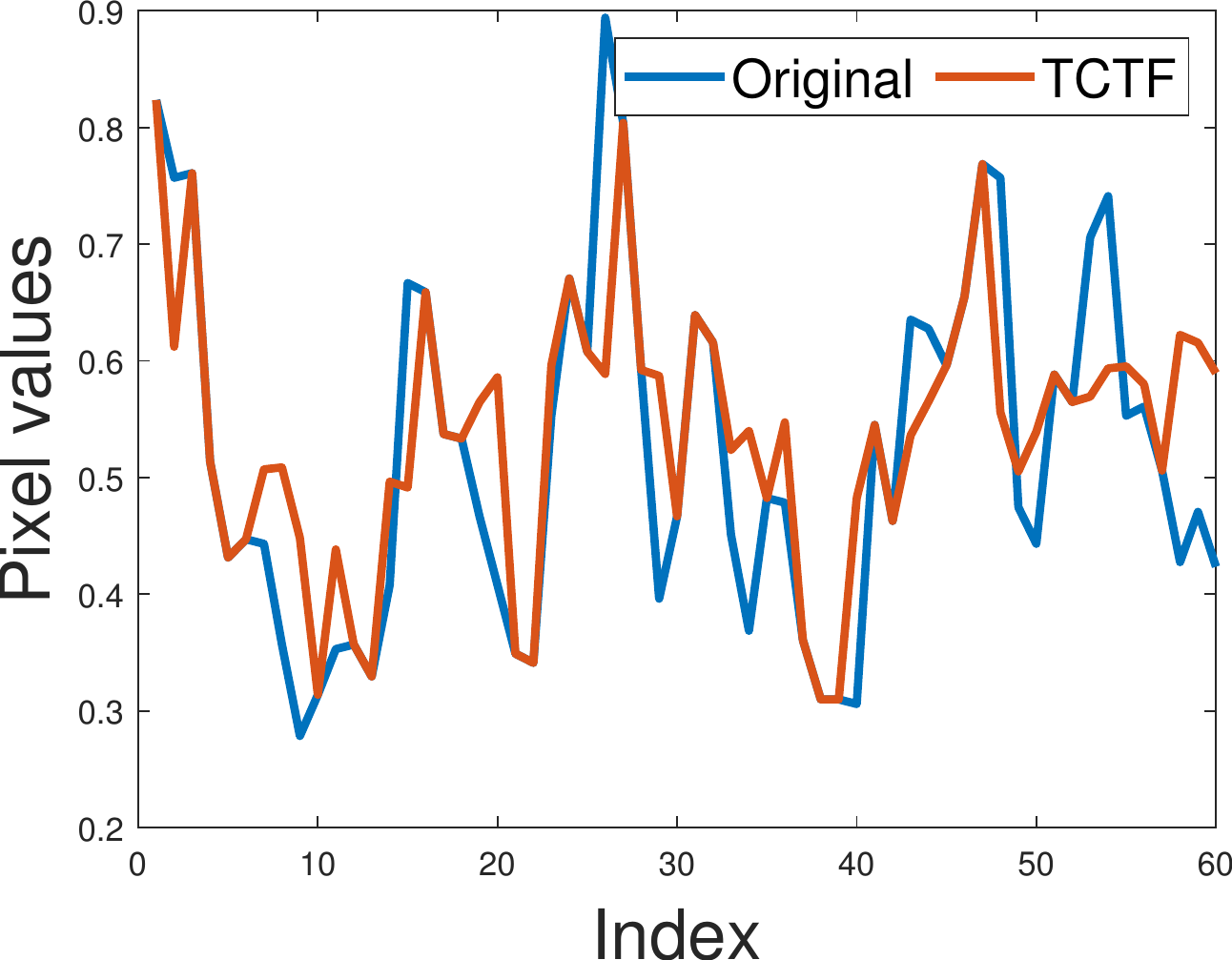}\vspace{0pt}
		\includegraphics[width=\linewidth]{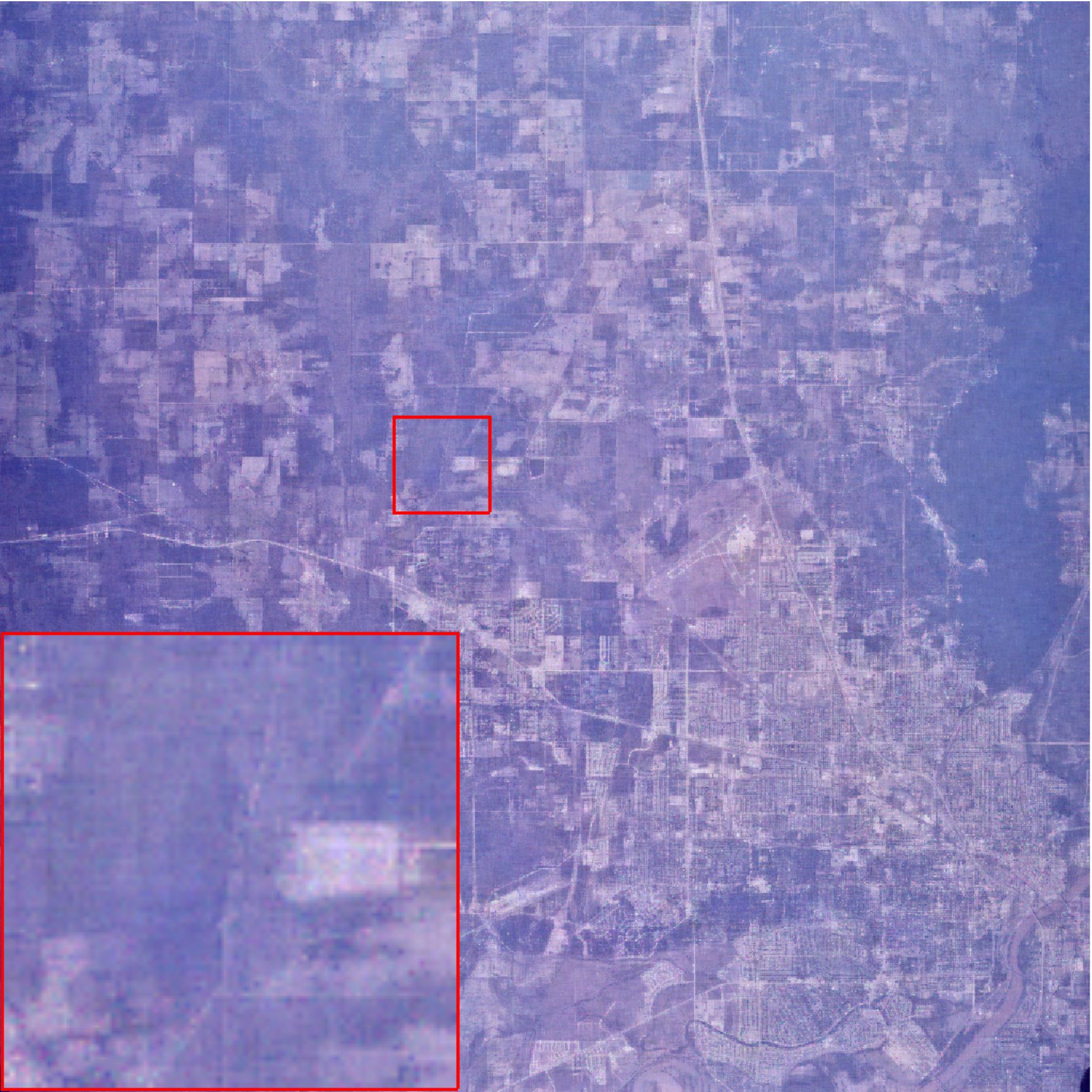}\vspace{0pt}
		\includegraphics[width=\linewidth]{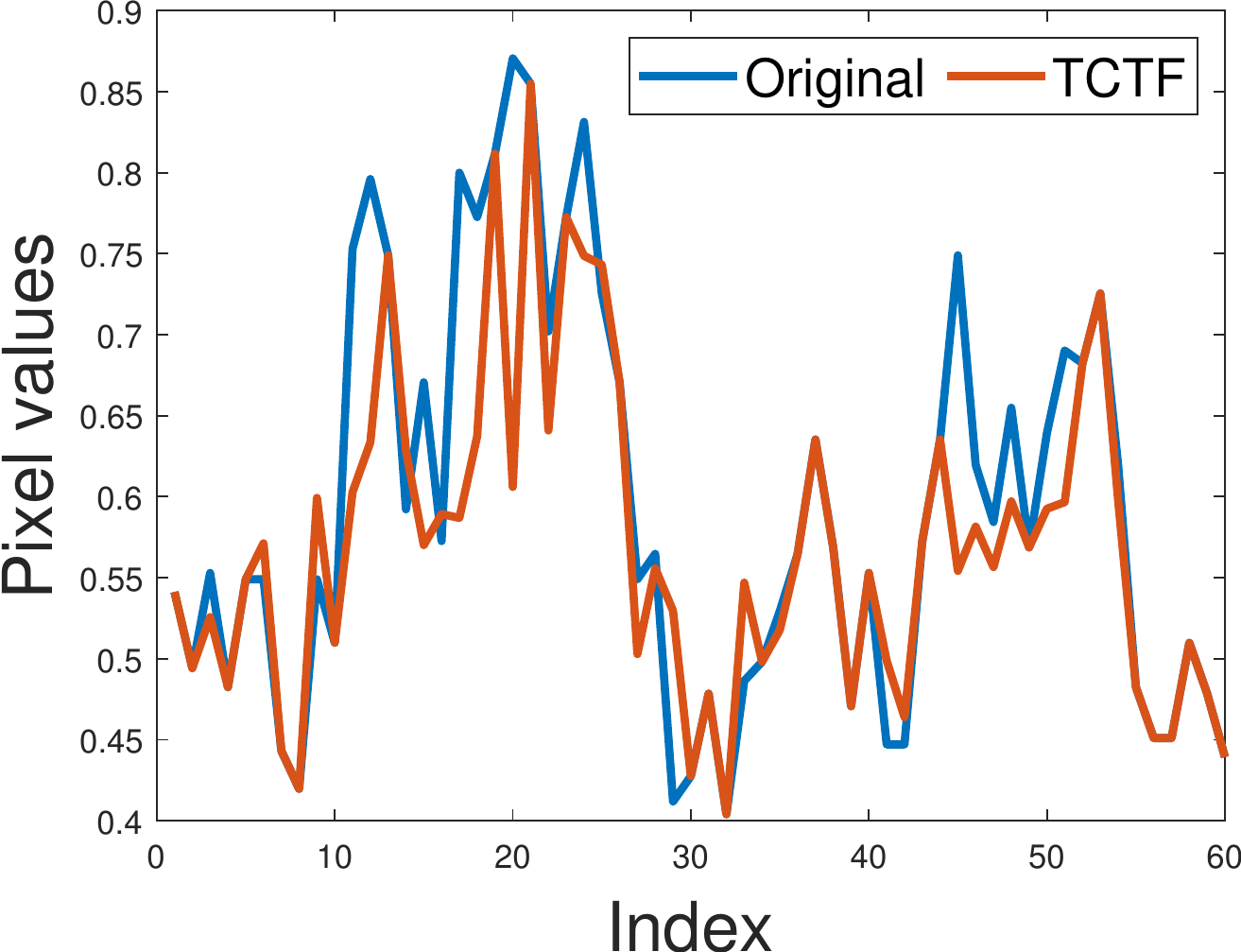}\vspace{0pt}
		\includegraphics[width=\linewidth]{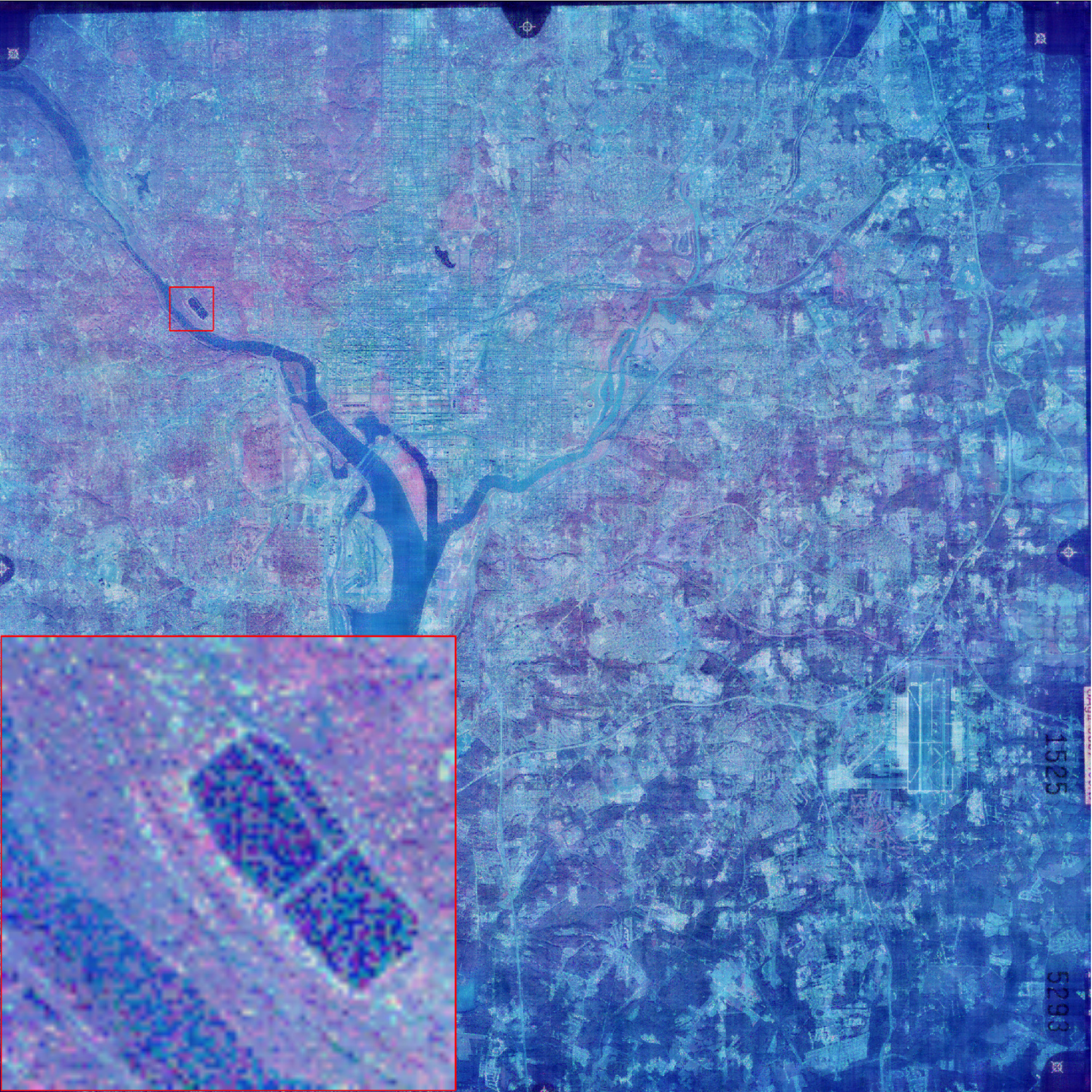}\vspace{0pt}
		\includegraphics[width=\linewidth]{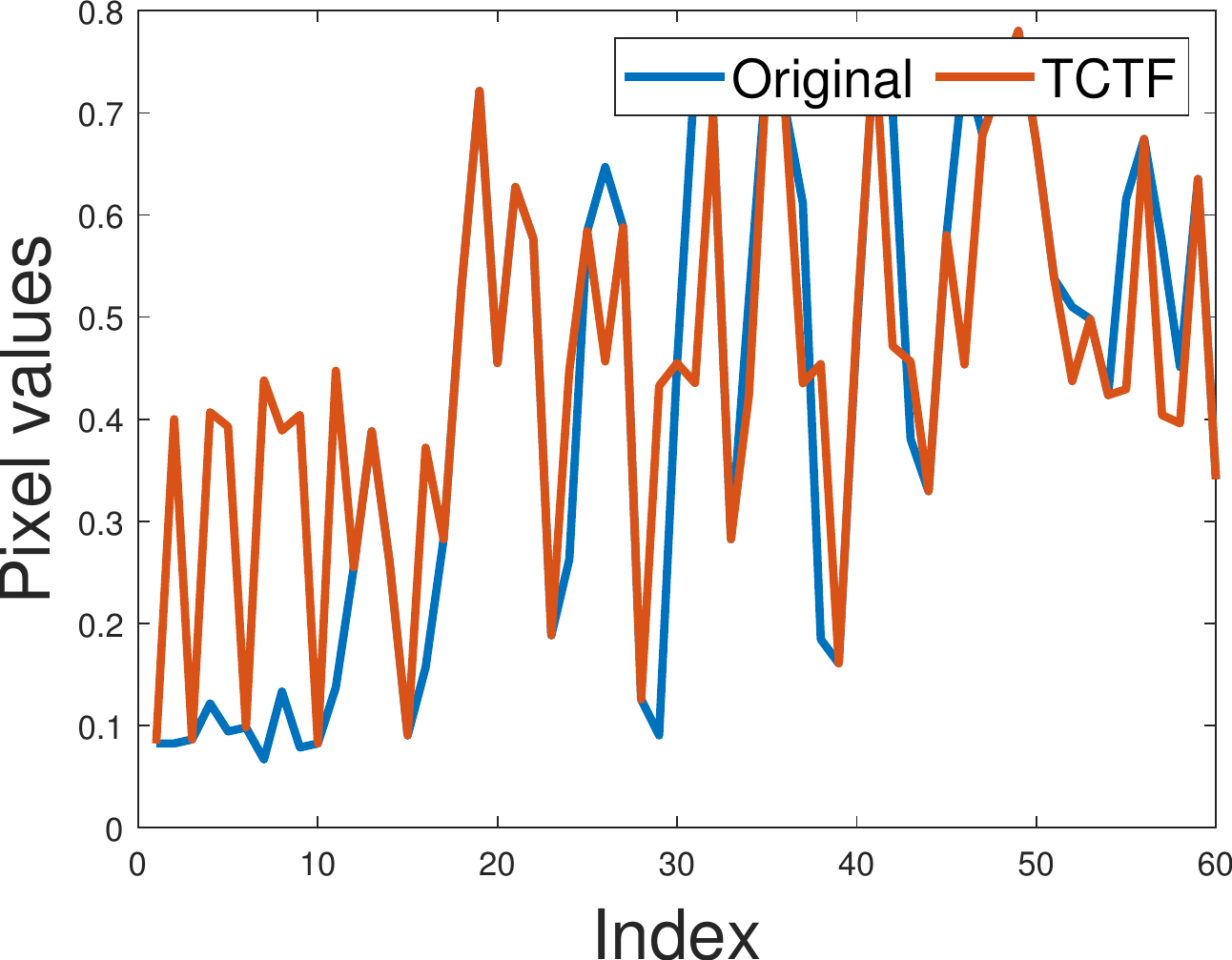}
		\caption{TCTF}
	\end{subfigure}
	\begin{subfigure}[b]{0.118\linewidth}
		\centering		
		\includegraphics[width=\linewidth]{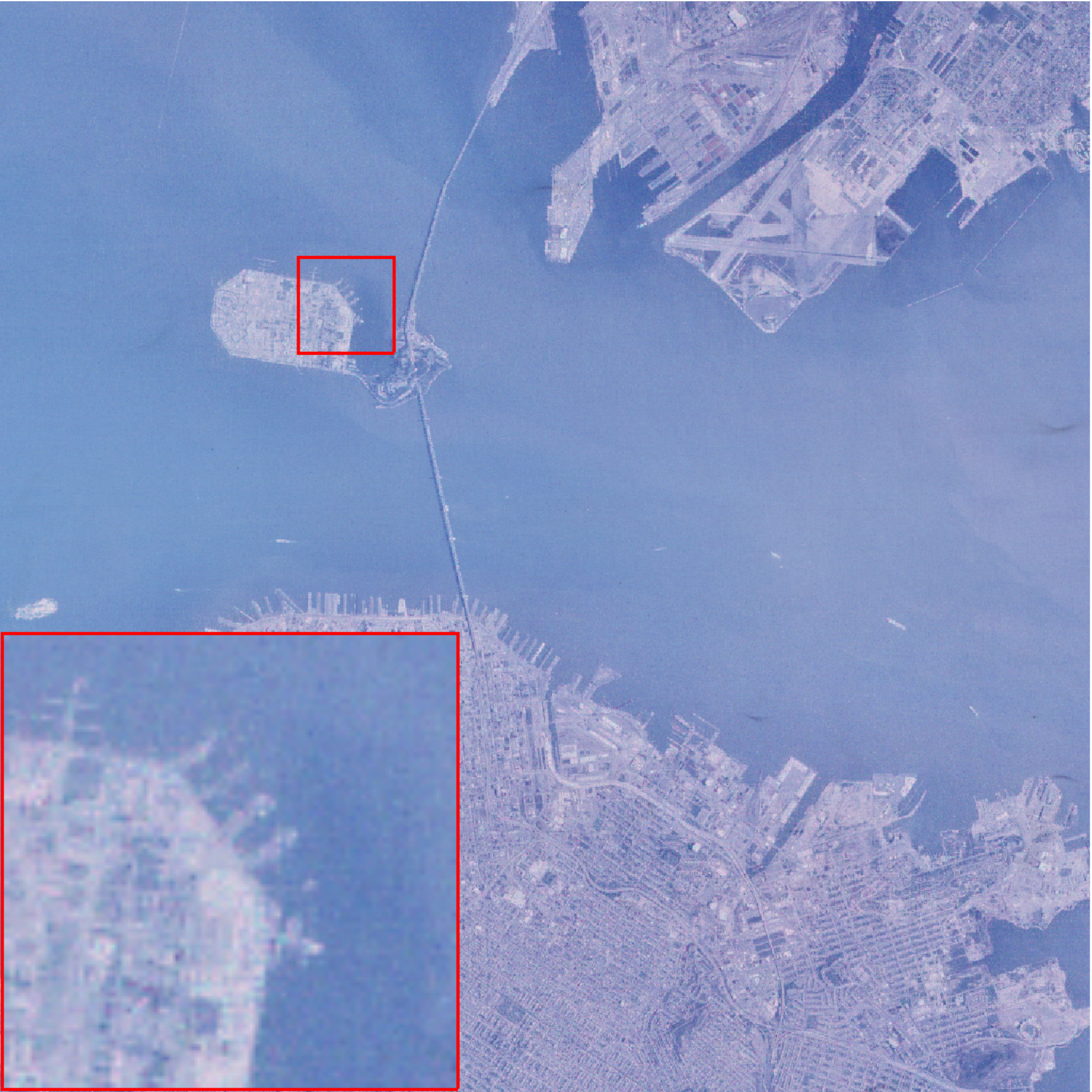}\vspace{0pt}
		\includegraphics[width=\linewidth]{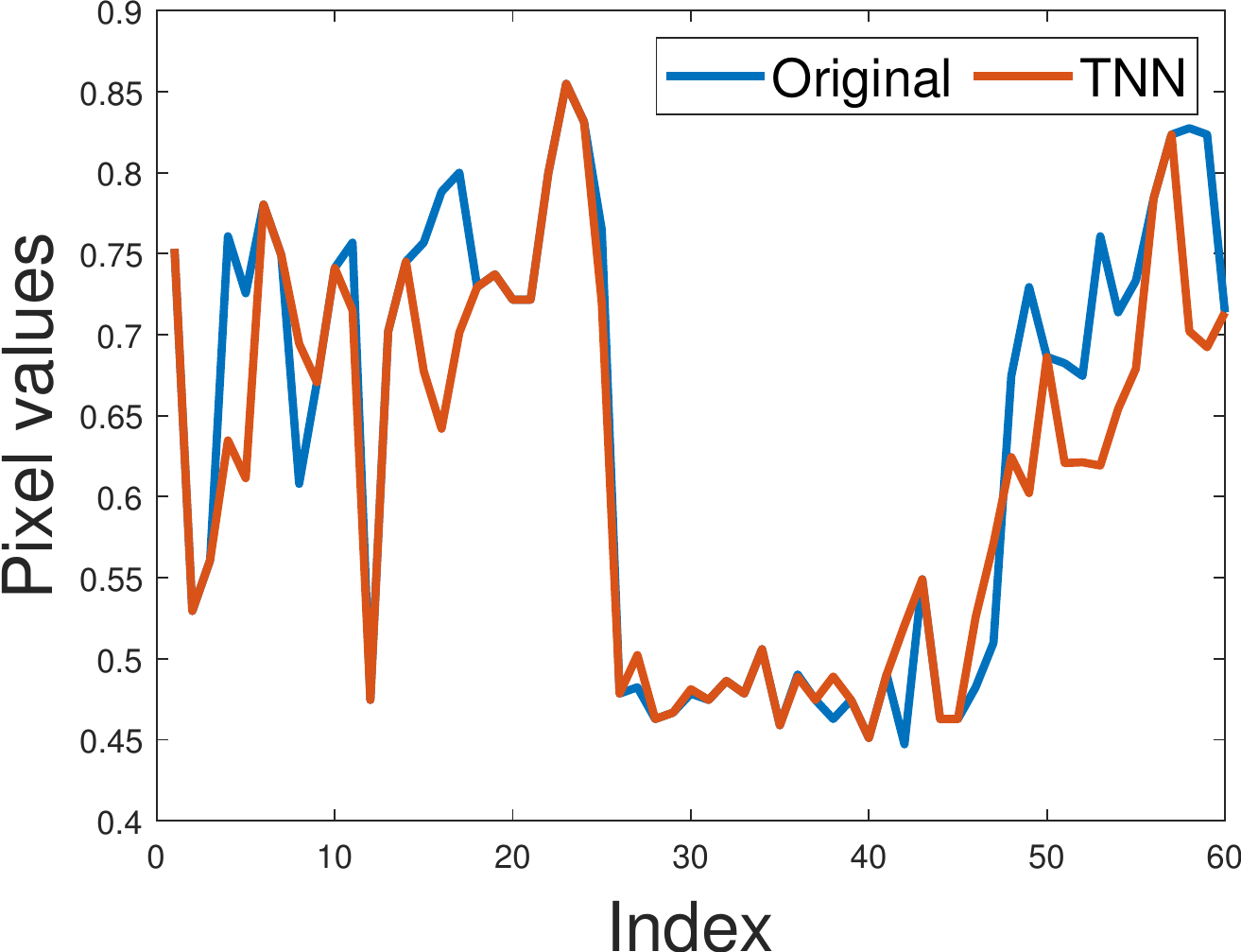}\vspace{0pt}
		\includegraphics[width=\linewidth]{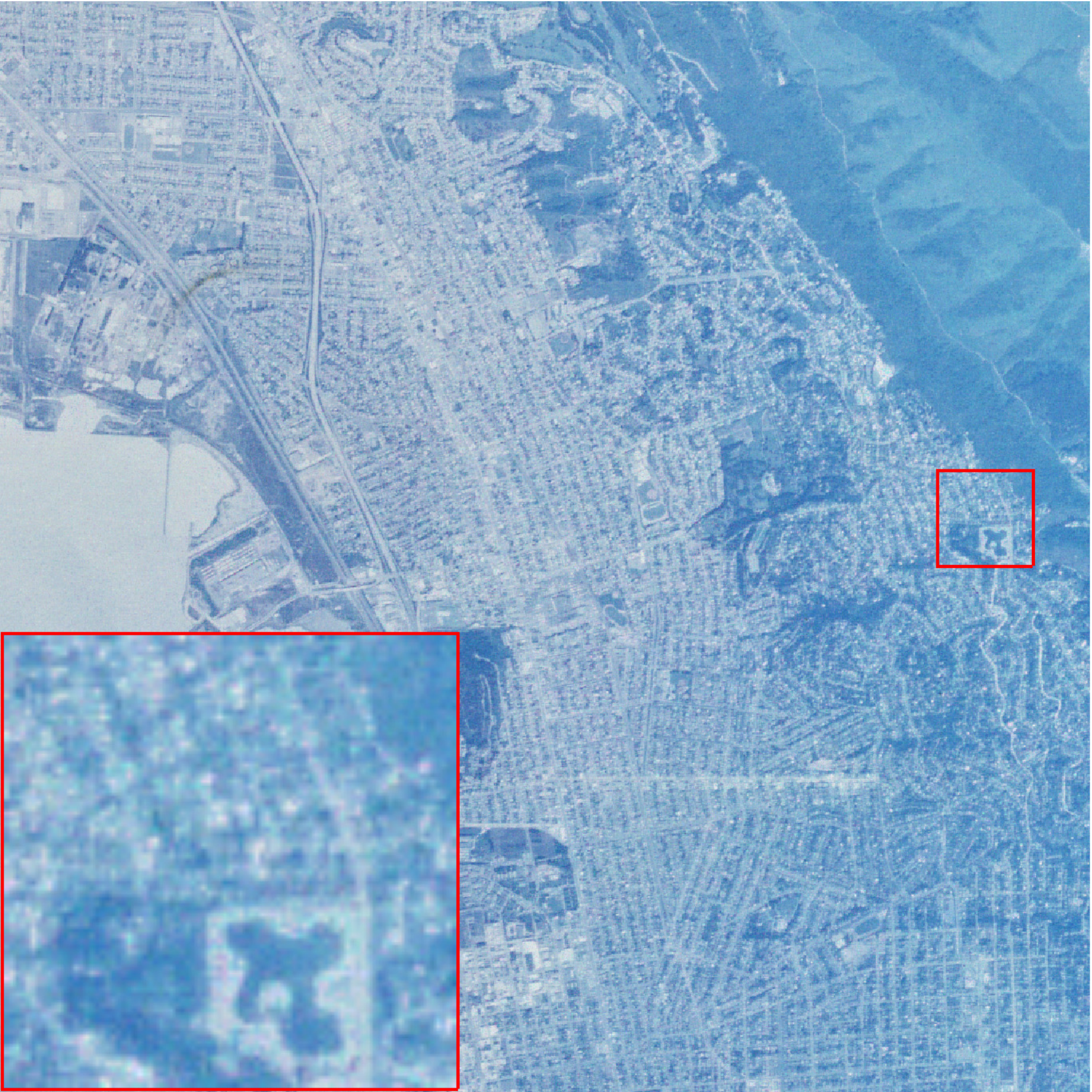}\vspace{0pt}
		\includegraphics[width=\linewidth]{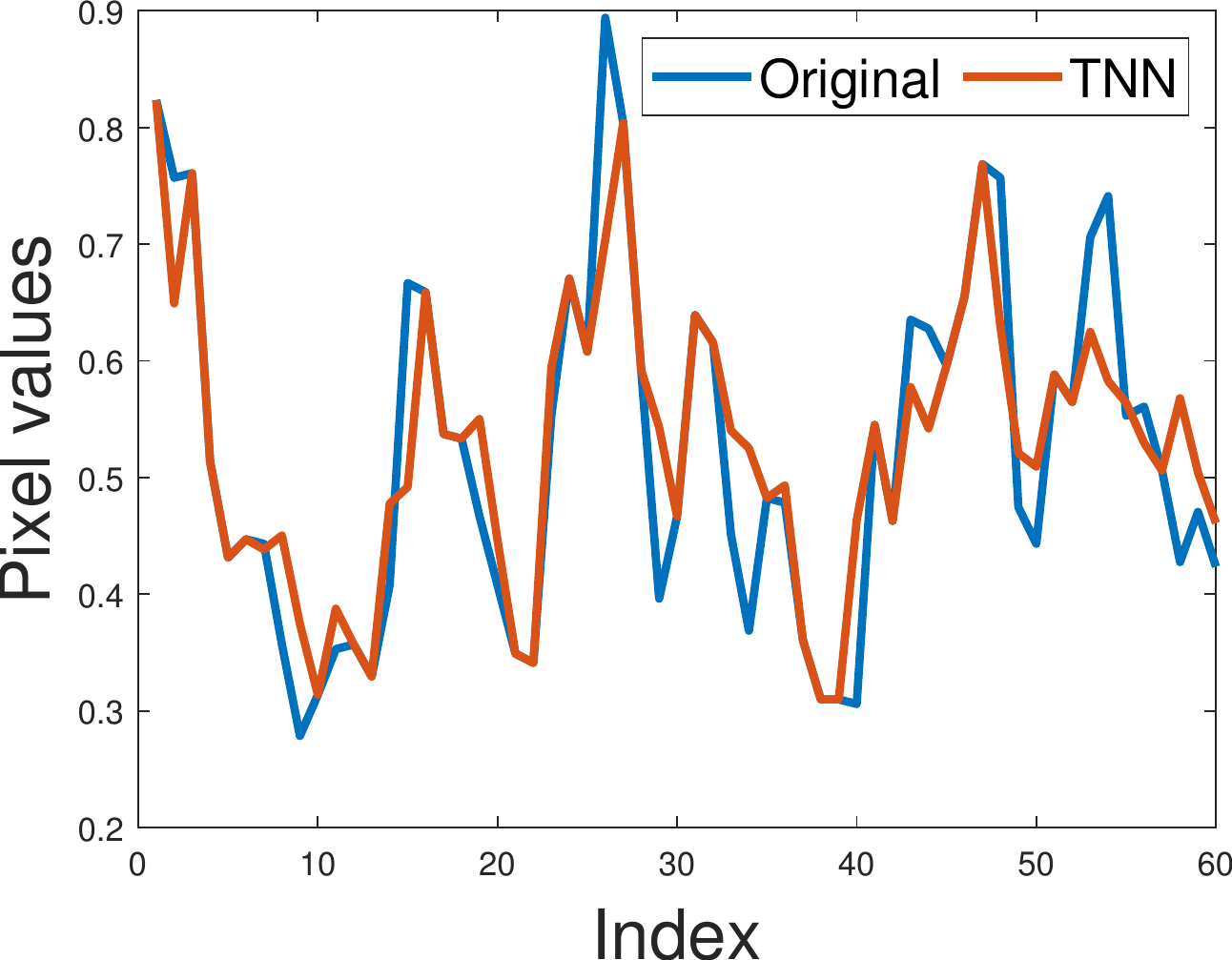}\vspace{0pt}
		\includegraphics[width=\linewidth]{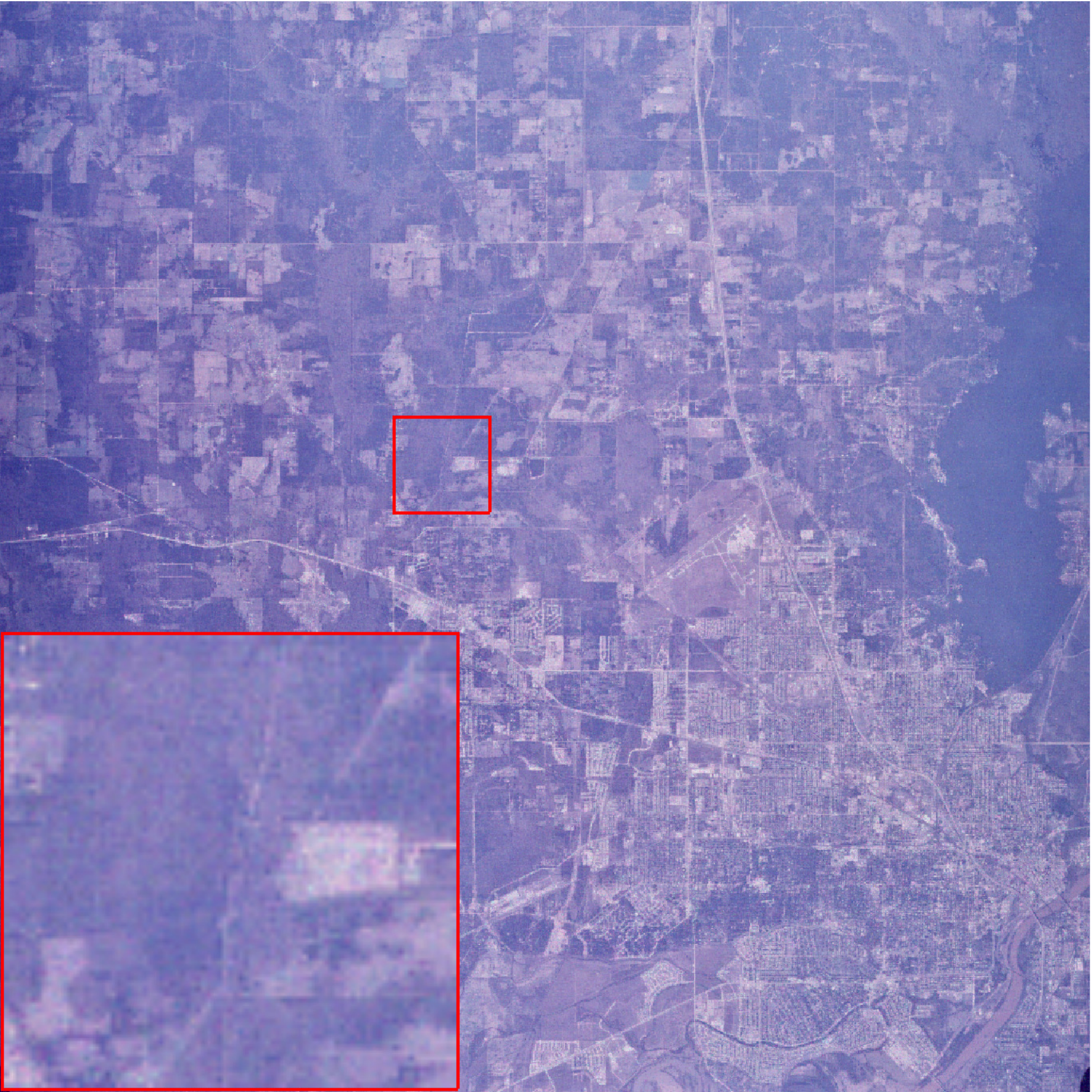}\vspace{0pt}
		\includegraphics[width=\linewidth]{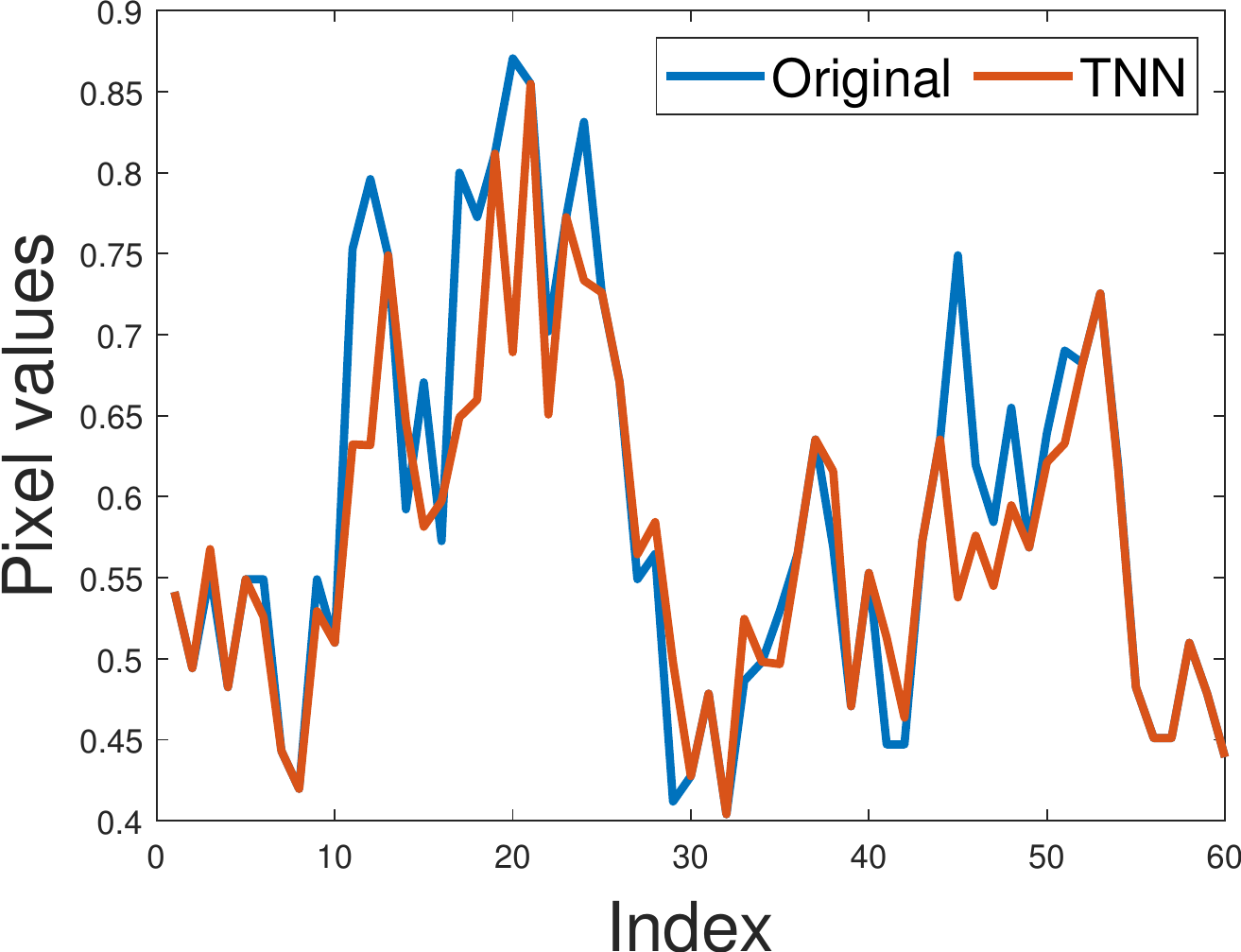}\vspace{0pt}
		\includegraphics[width=\linewidth]{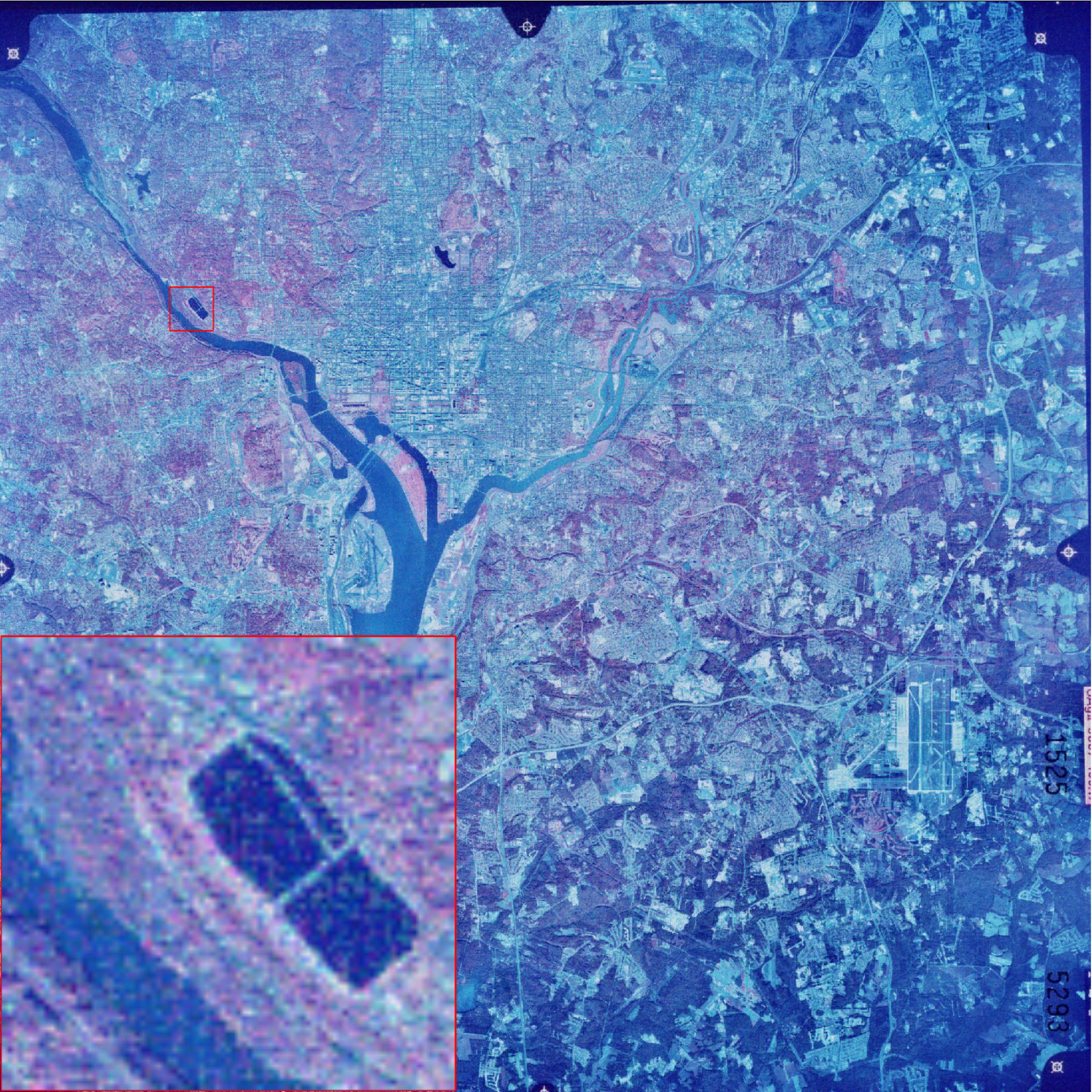}\vspace{0pt}
		\includegraphics[width=\linewidth]{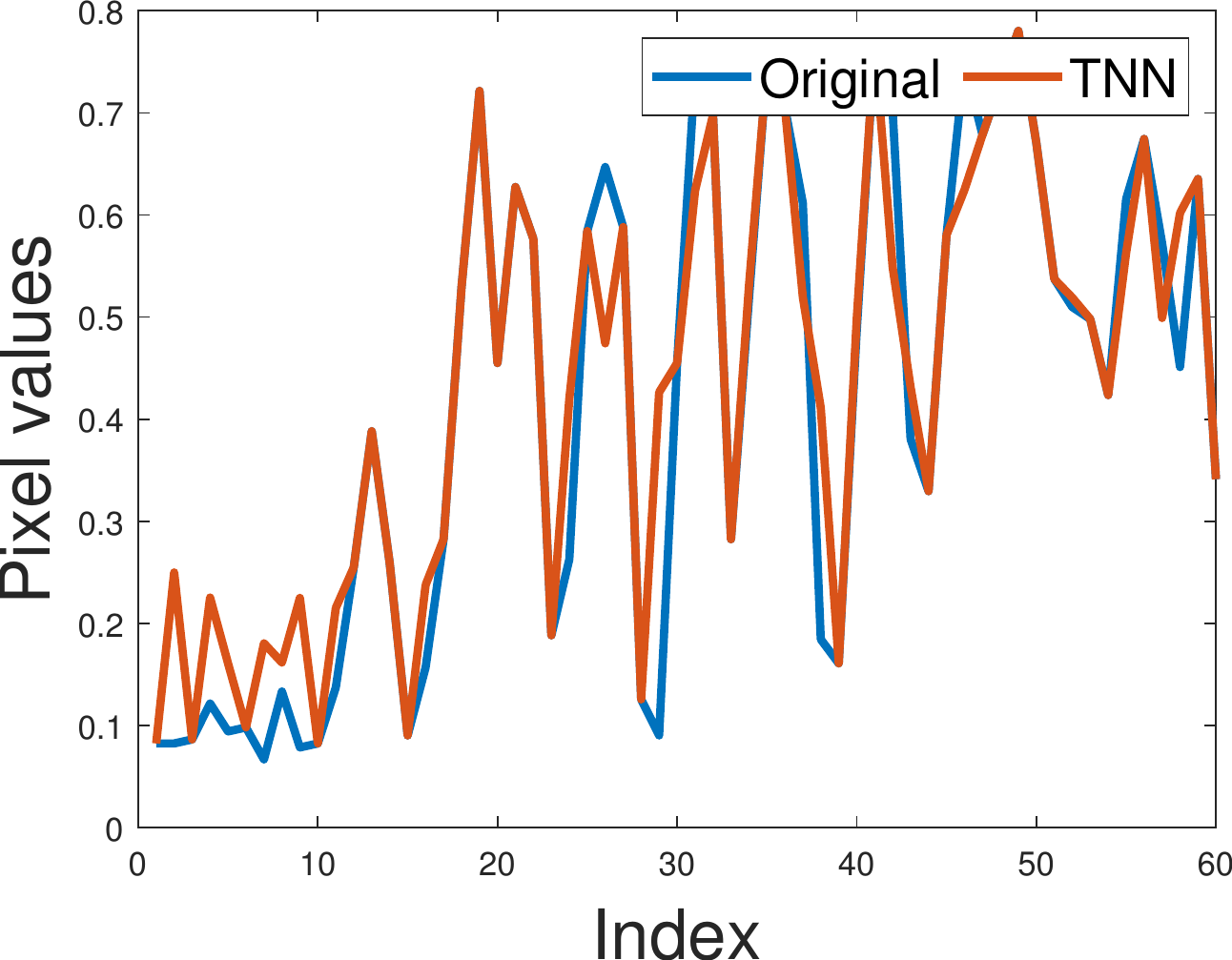}
		\caption{TNN}
	\end{subfigure}
	\begin{subfigure}[b]{0.118\linewidth}
		\centering		
		\includegraphics[width=\linewidth]{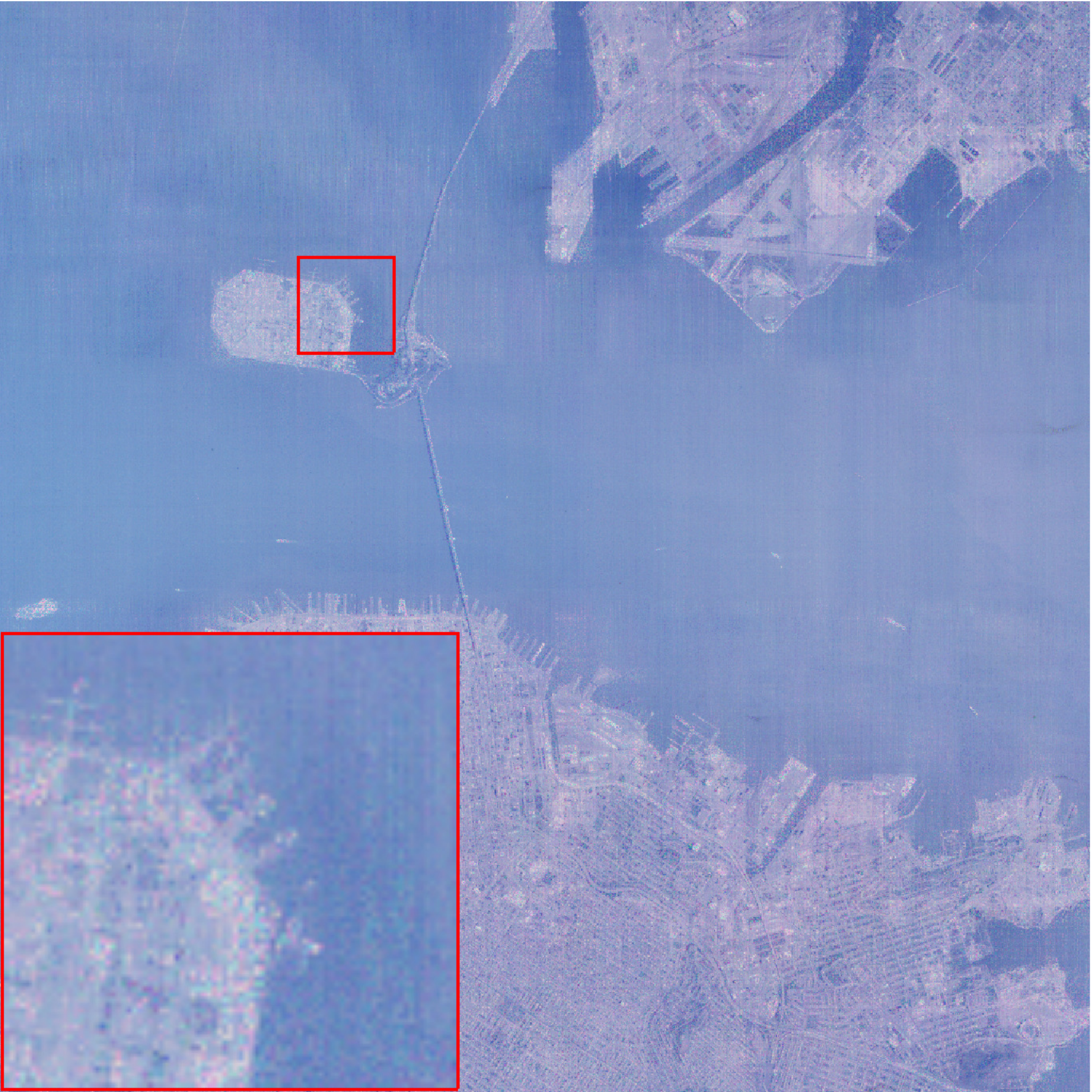}\vspace{0pt}
		\includegraphics[width=\linewidth]{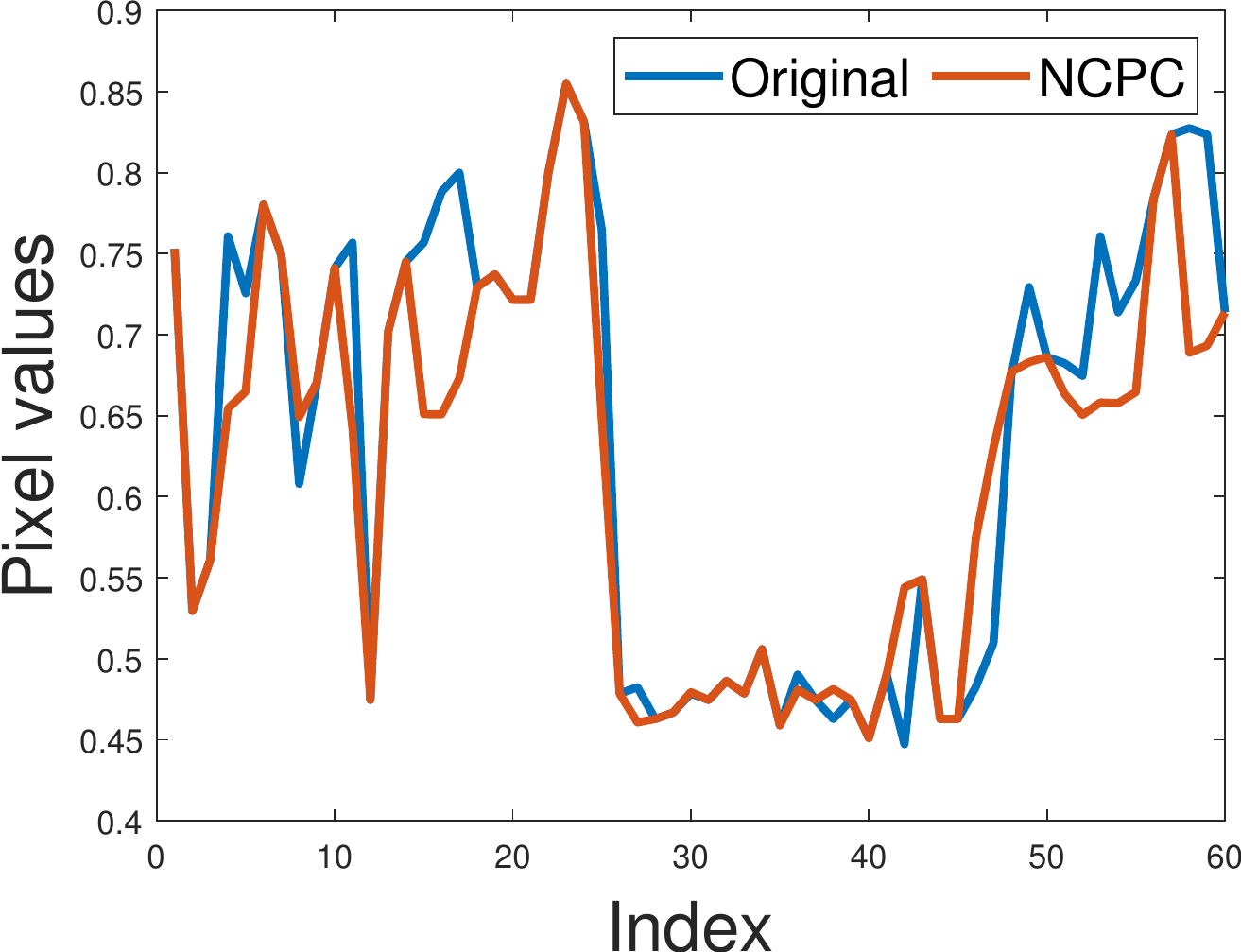}\vspace{0pt}
		\includegraphics[width=\linewidth]{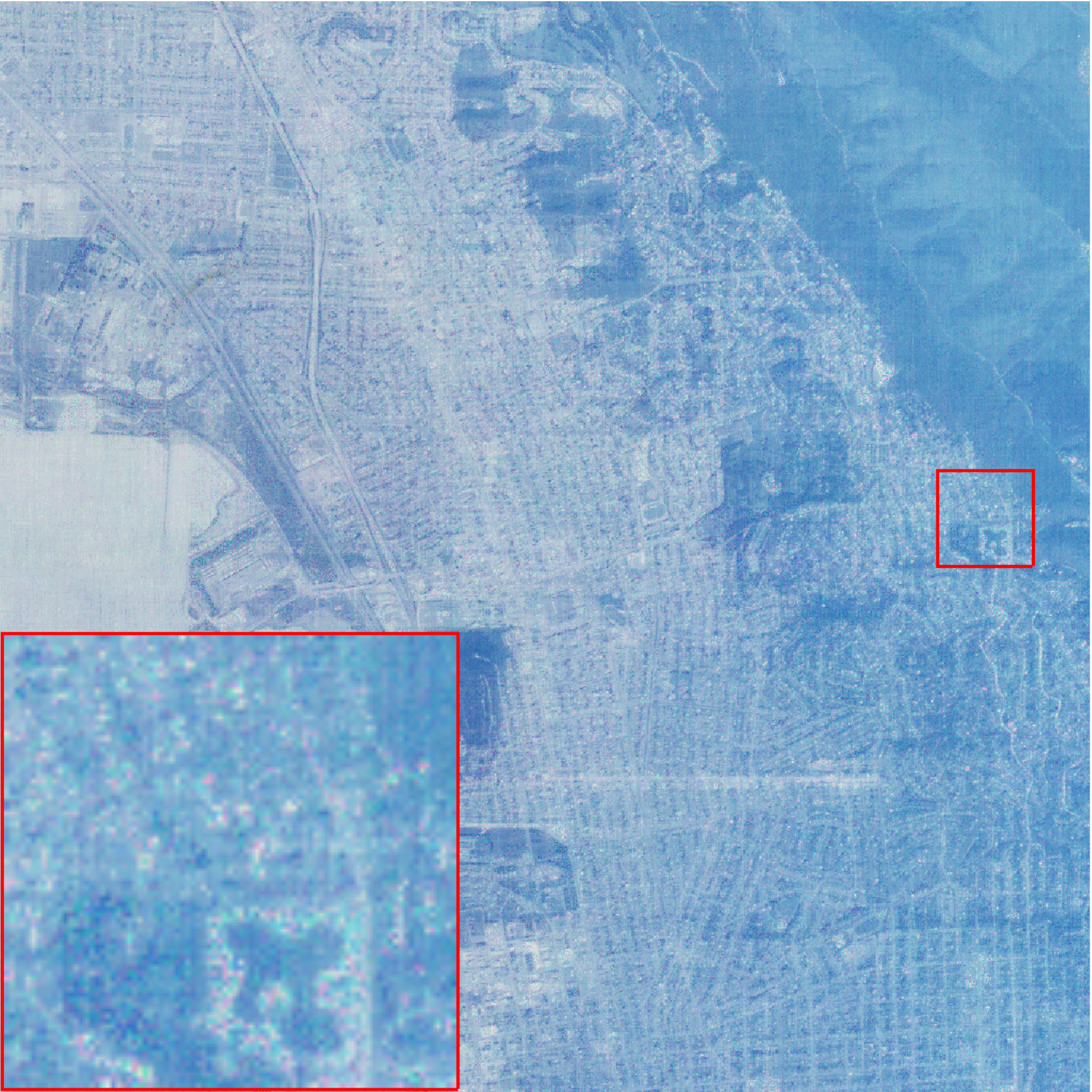}\vspace{0pt}
		\includegraphics[width=\linewidth]{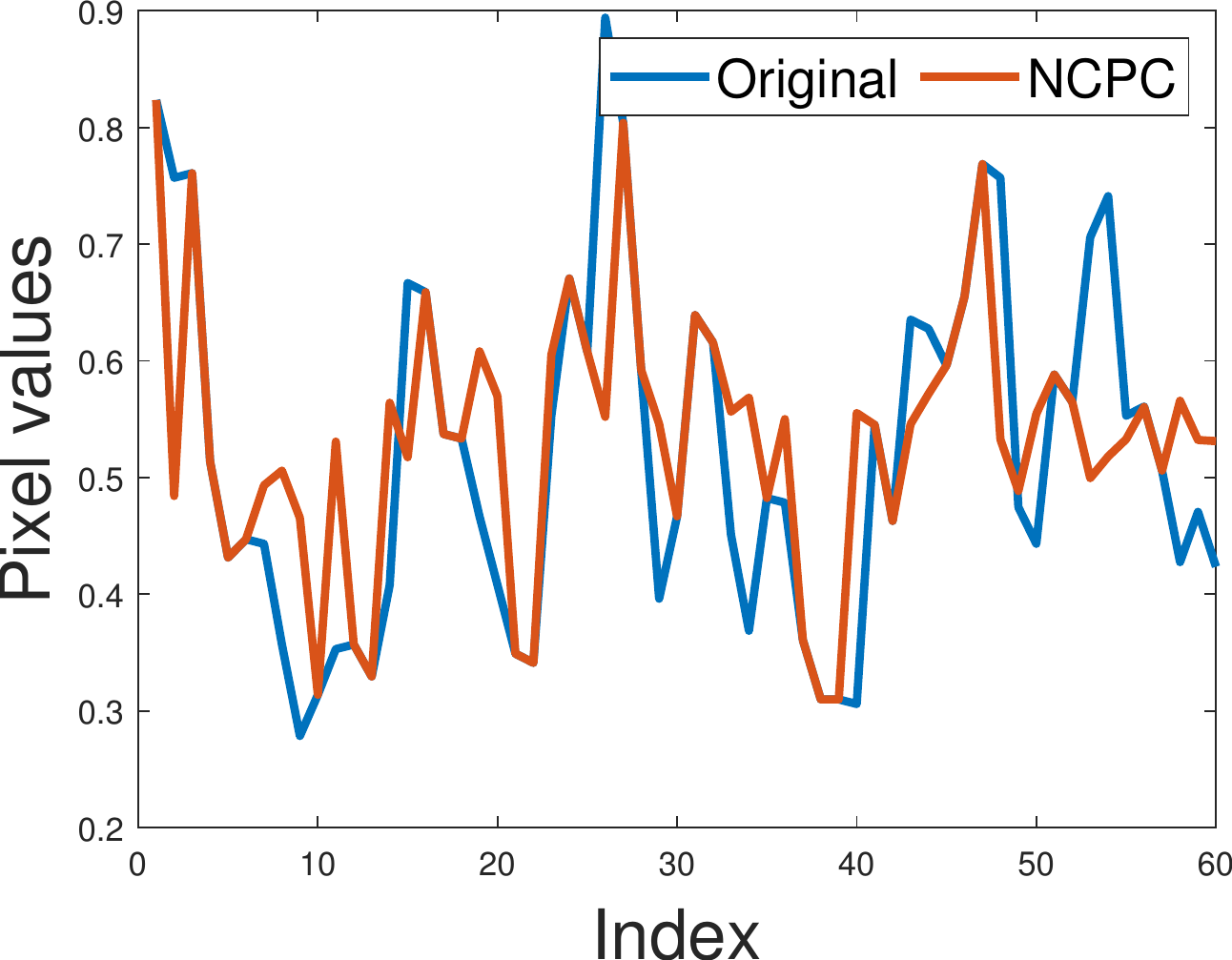}\vspace{0pt}
		\includegraphics[width=\linewidth]{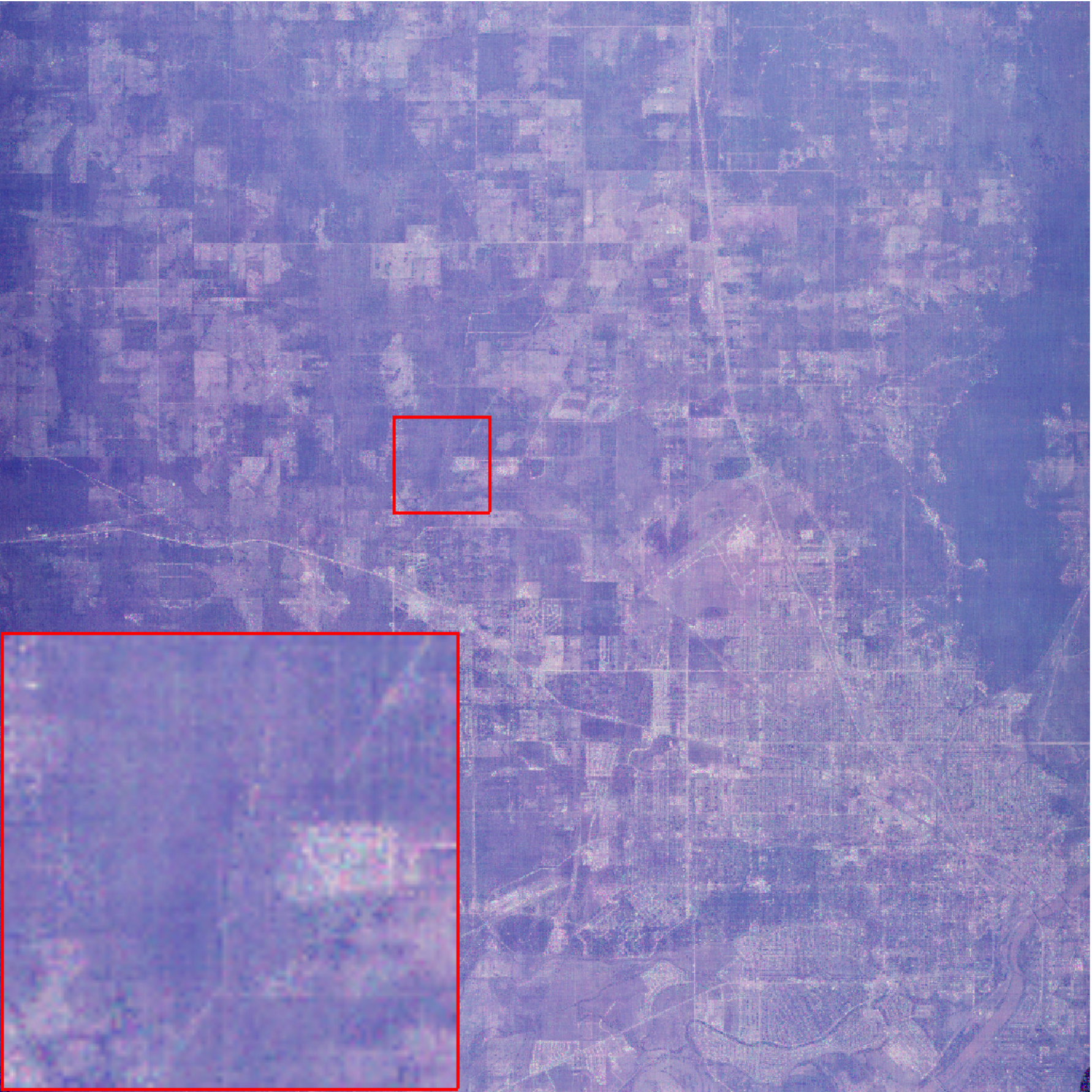}\vspace{0pt}
		\includegraphics[width=\linewidth]{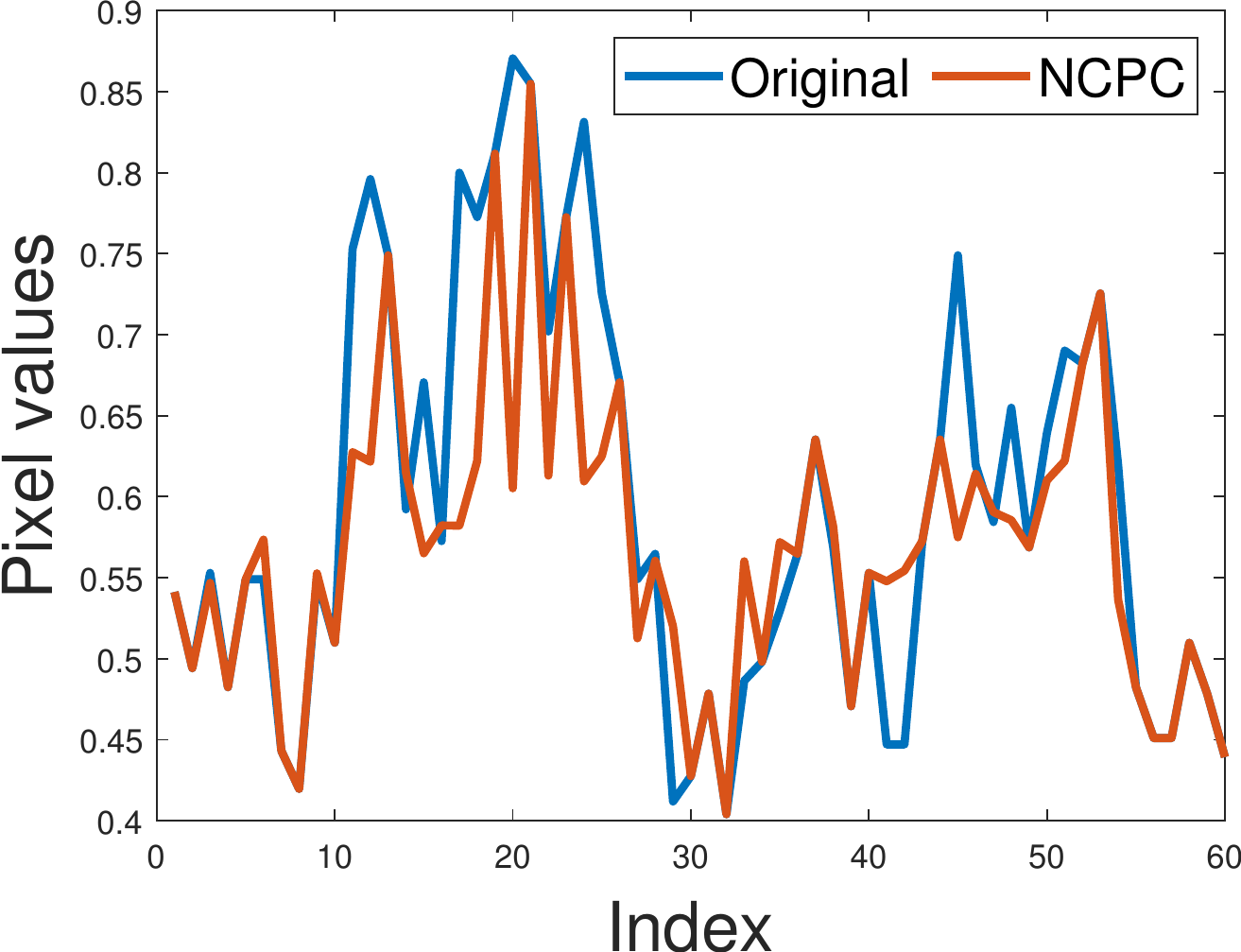}\vspace{0pt}
		\includegraphics[width=\linewidth]{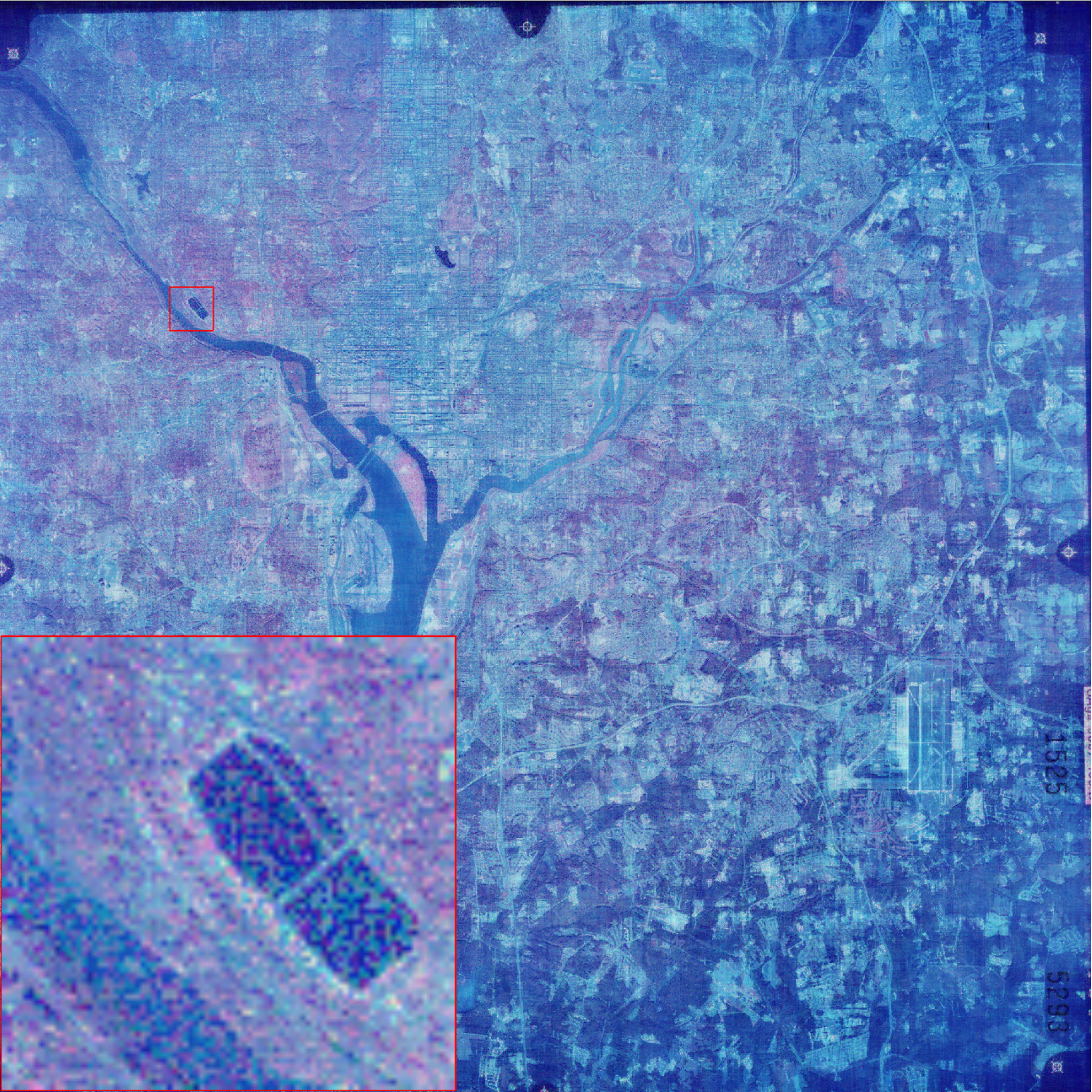}\vspace{0pt}
		\includegraphics[width=\linewidth]{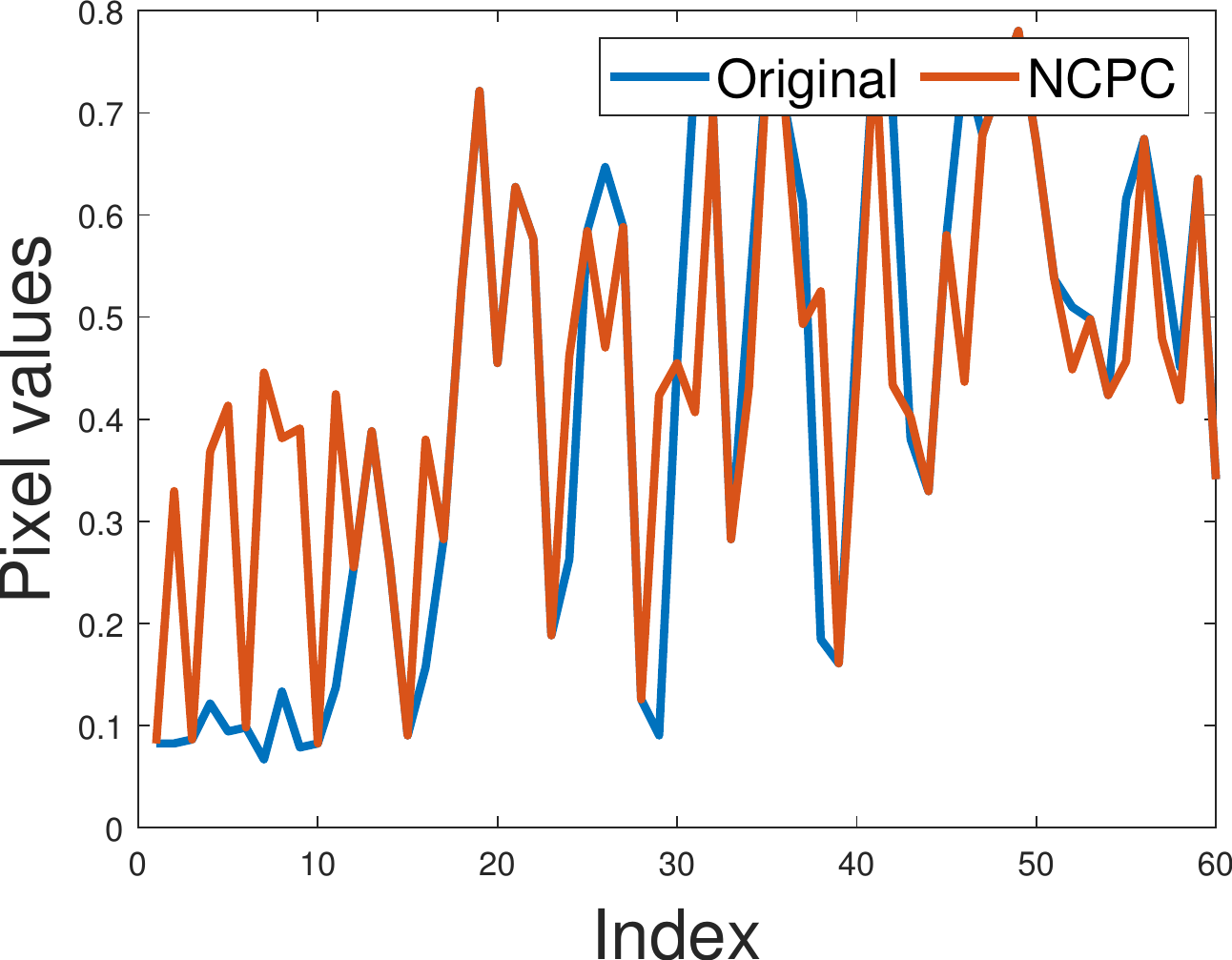}
		\caption{NCPC}
	\end{subfigure}
	\begin{subfigure}[b]{0.118\linewidth}
	\centering			
	\includegraphics[width=\linewidth]{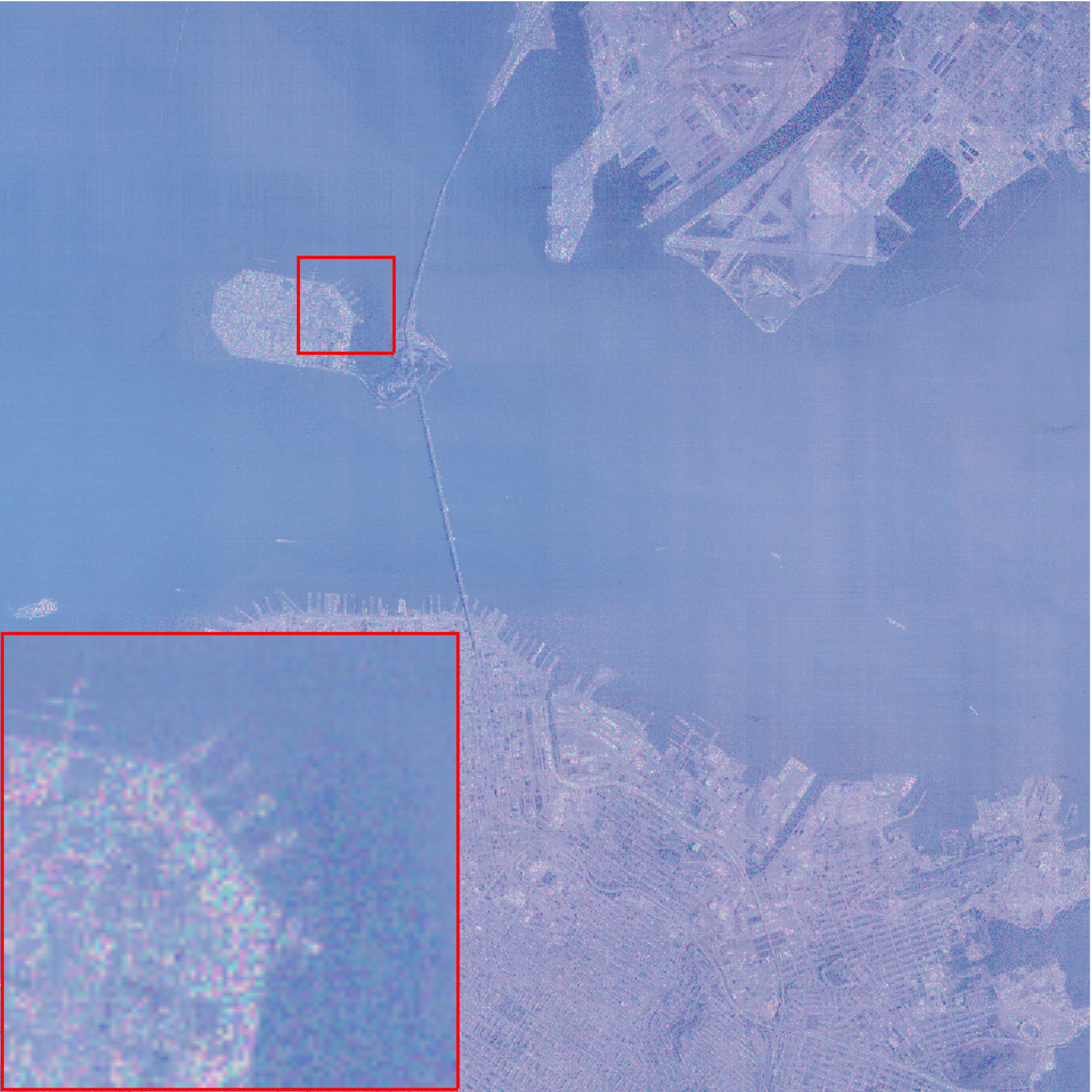}\vspace{0pt}
	\includegraphics[width=\linewidth]{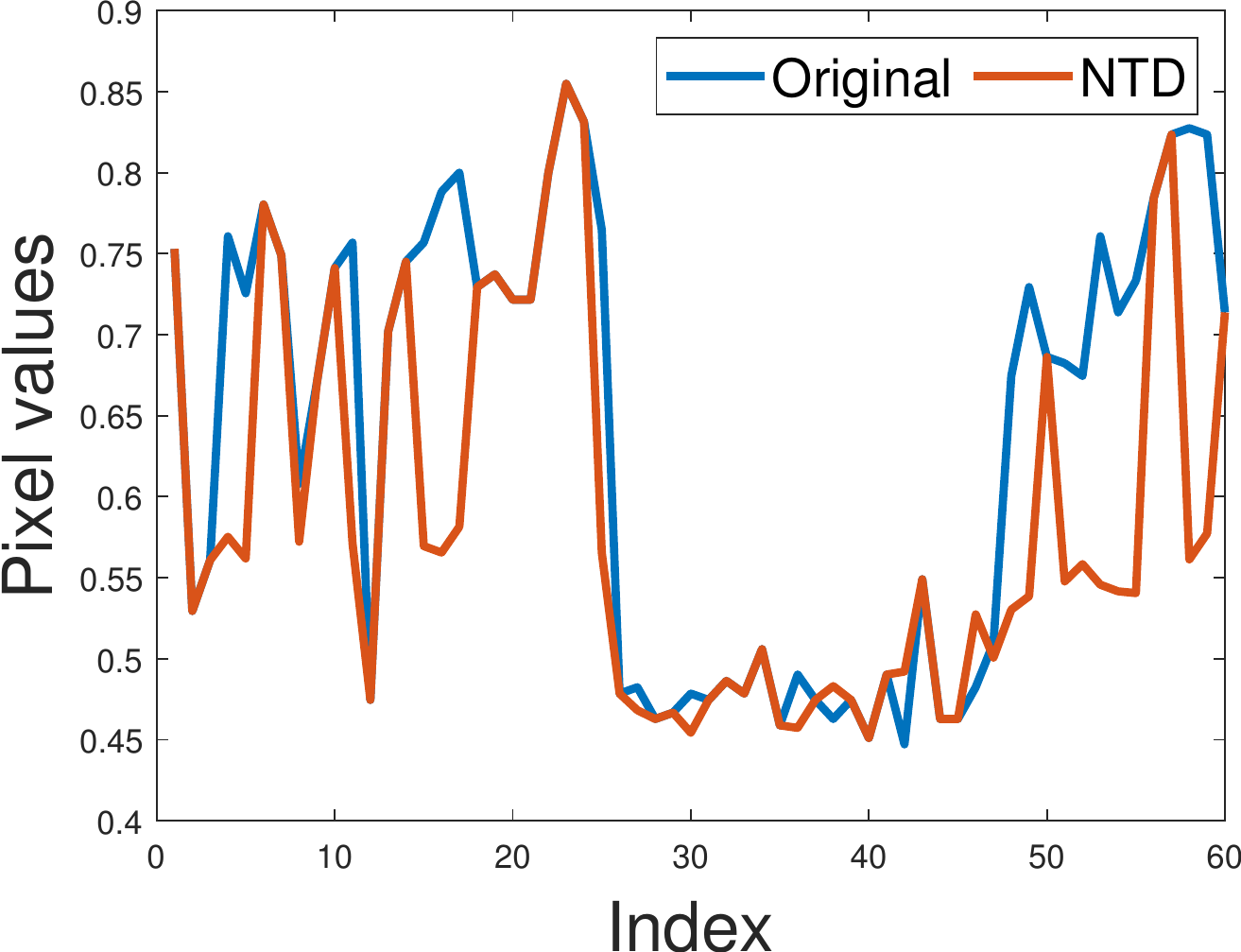}\vspace{0pt}
	\includegraphics[width=\linewidth]{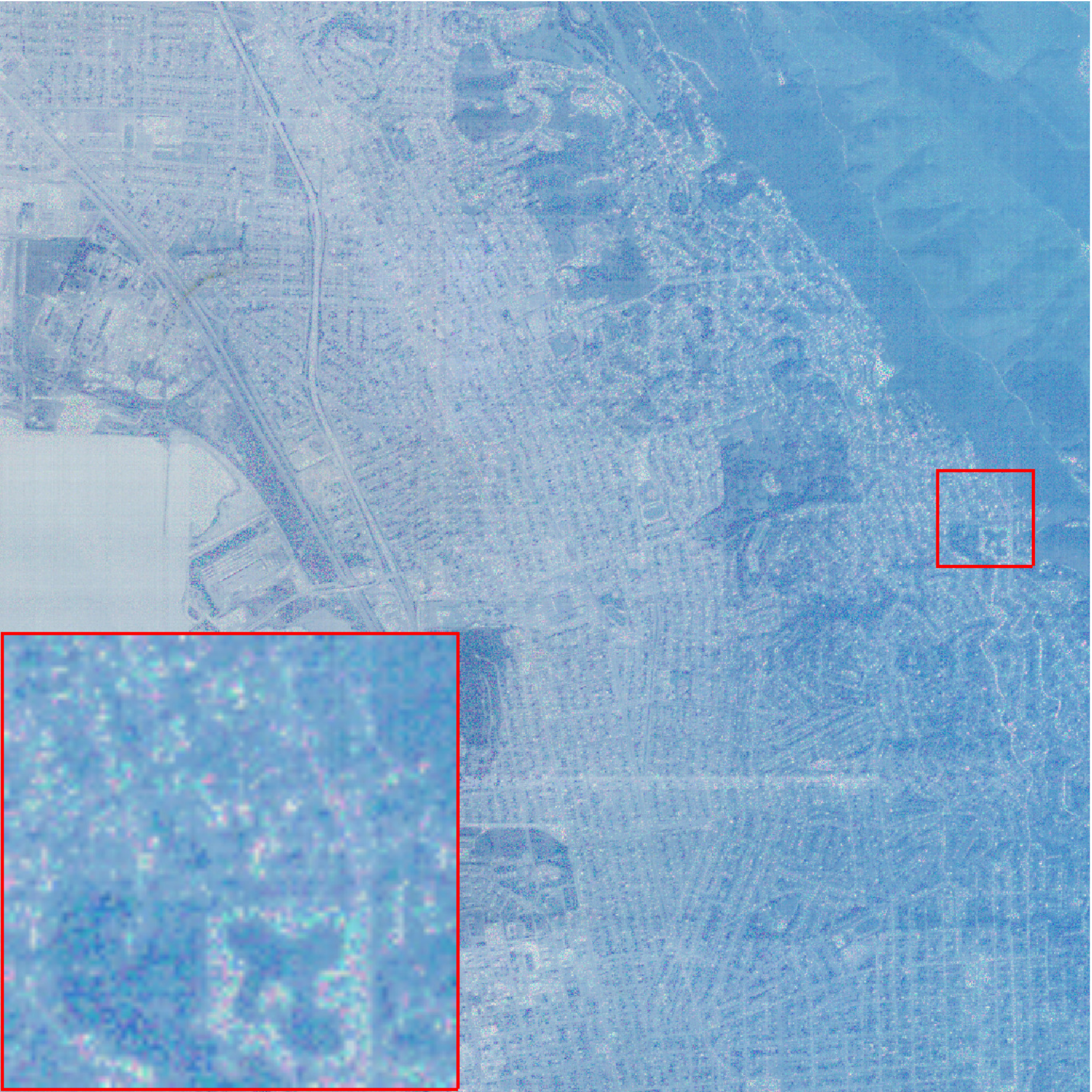}\vspace{0pt}
	\includegraphics[width=\linewidth]{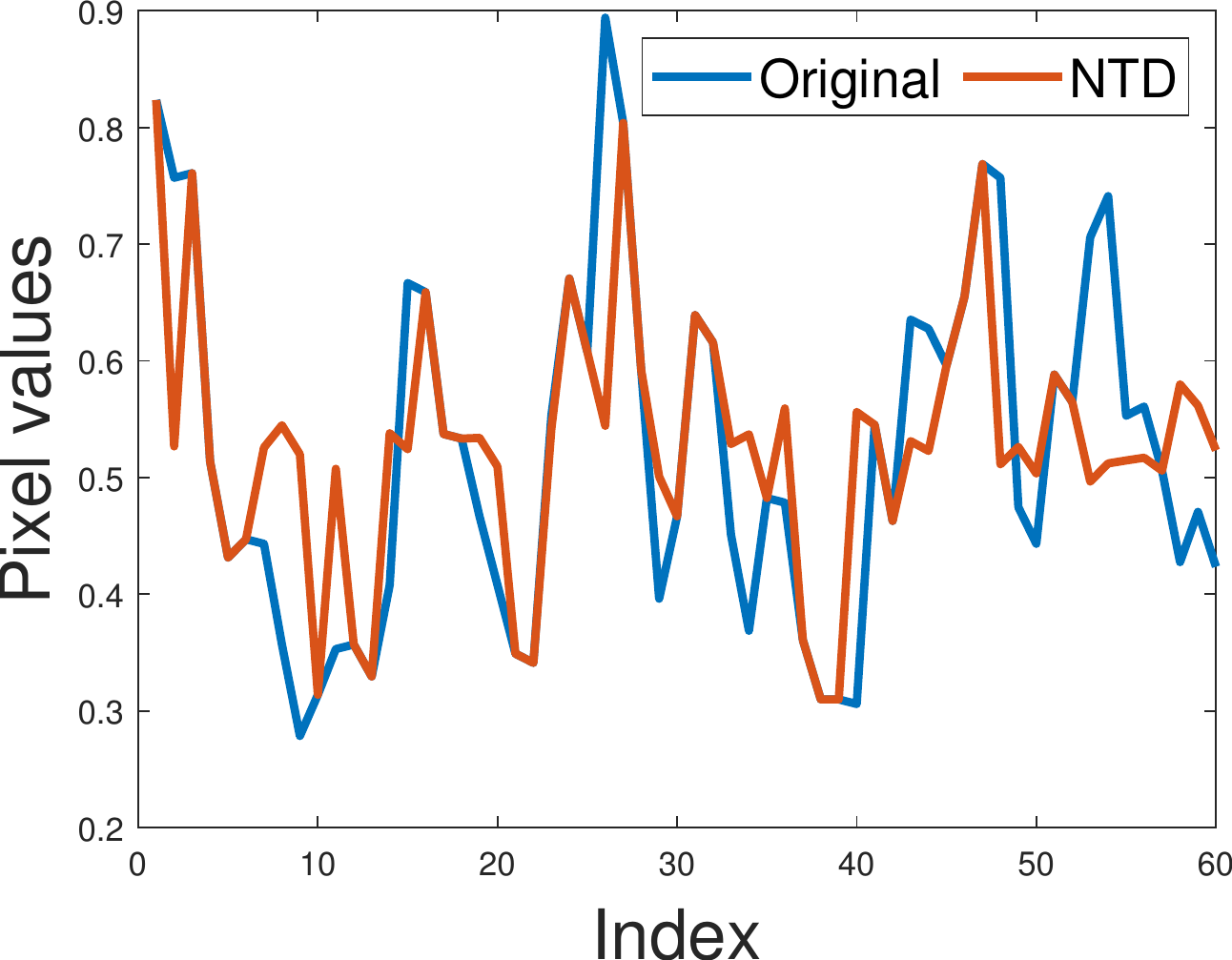}\vspace{0pt}
	\includegraphics[width=\linewidth]{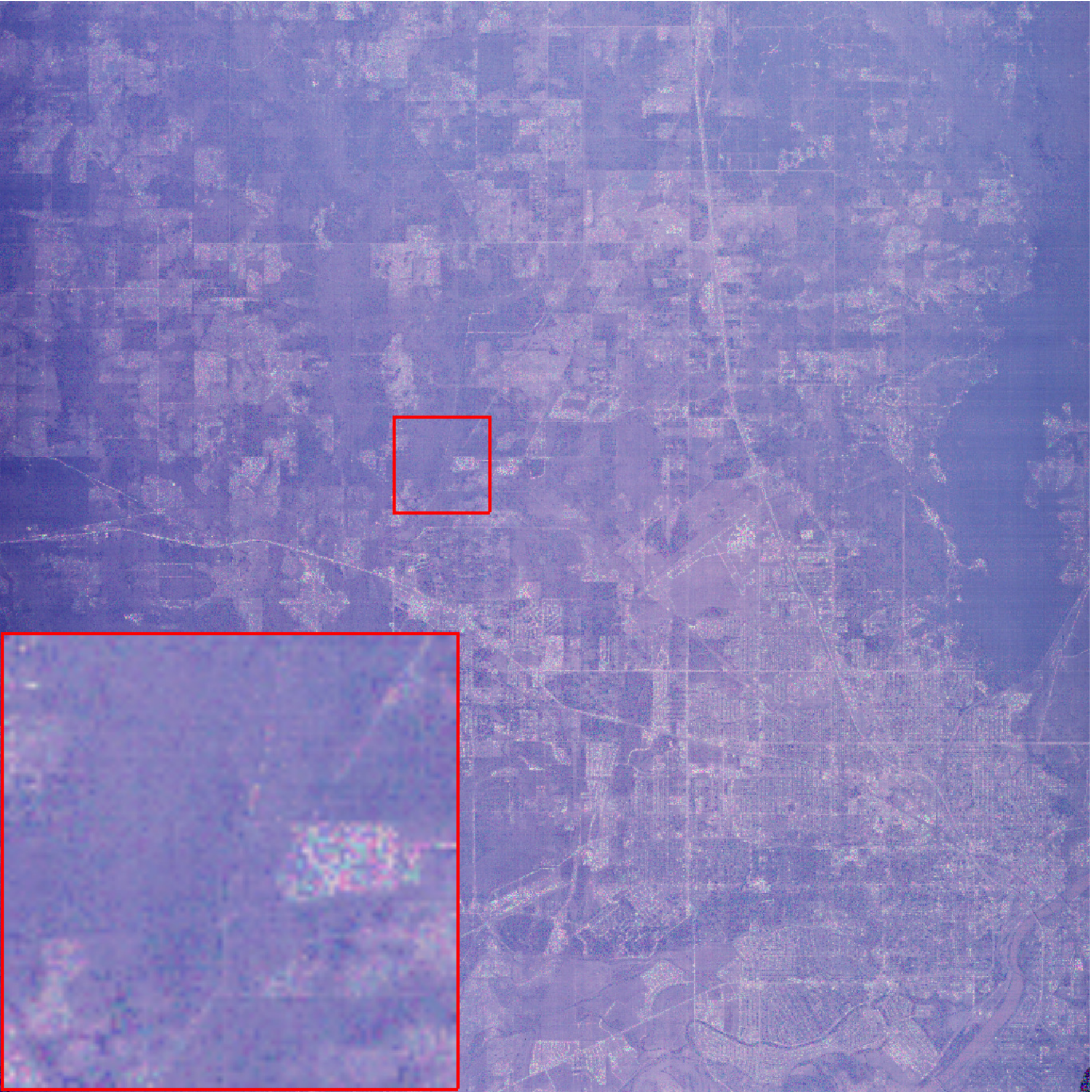}\vspace{0pt}
	\includegraphics[width=\linewidth]{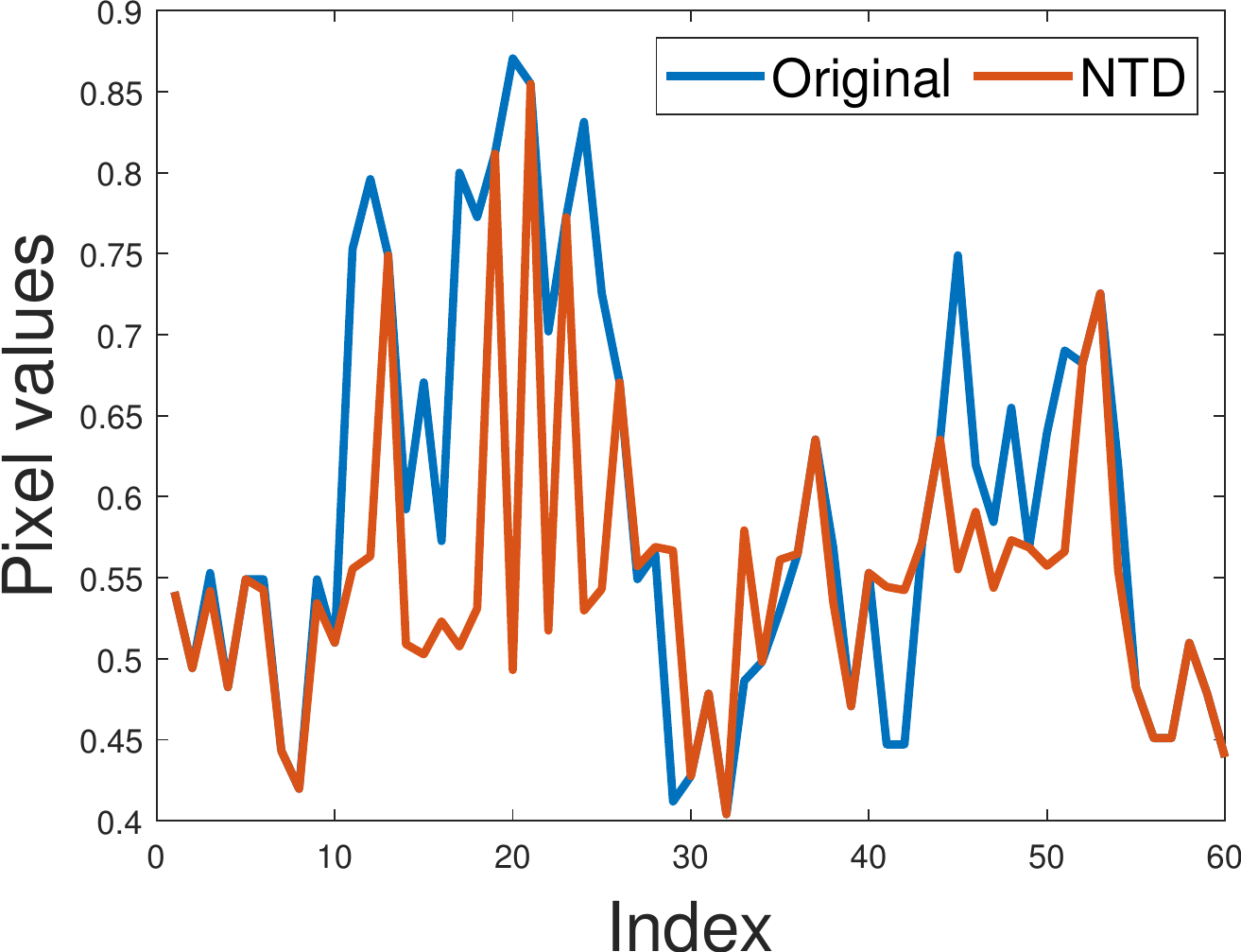}\vspace{0pt}
	\includegraphics[width=\linewidth]{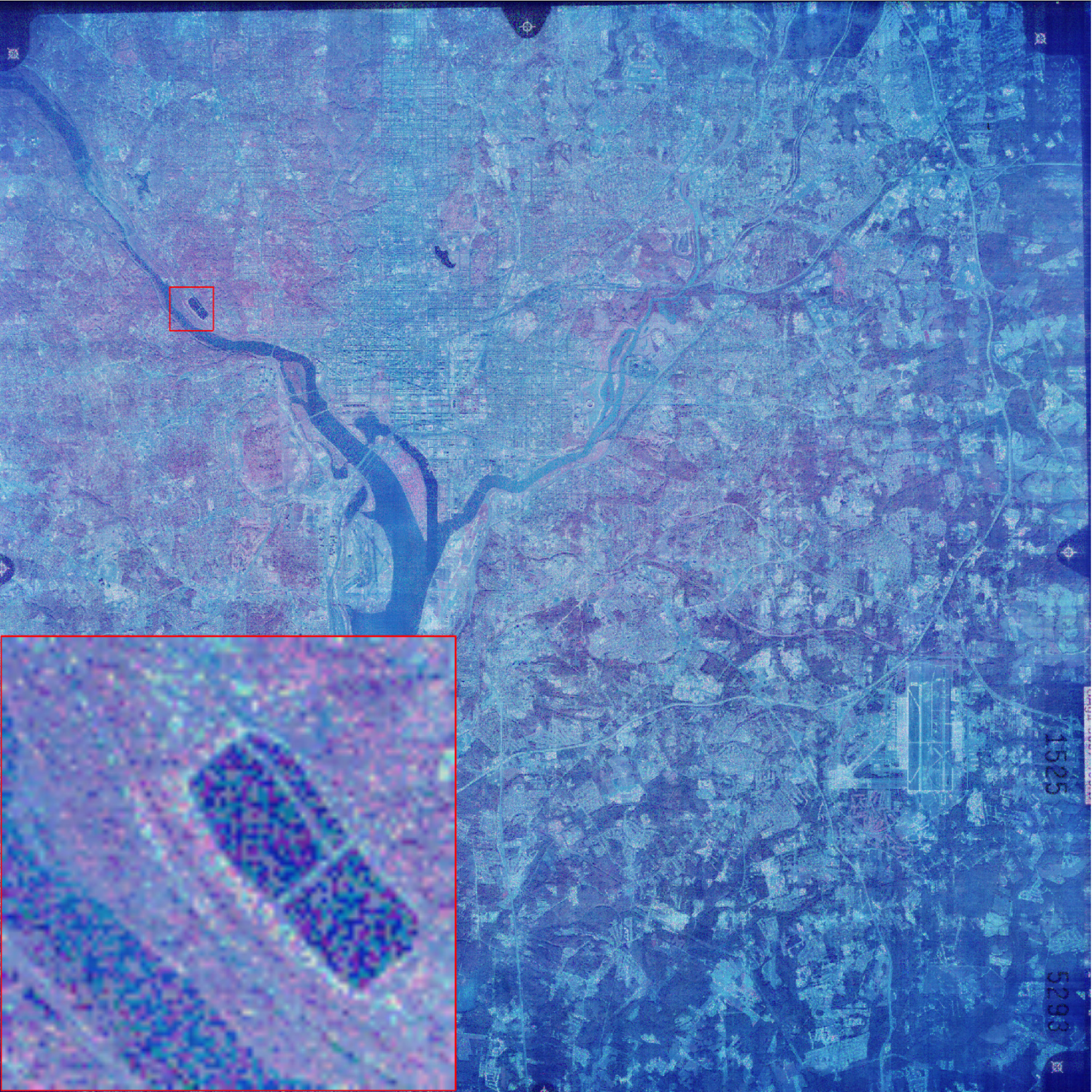}\vspace{0pt}
	\includegraphics[width=\linewidth]{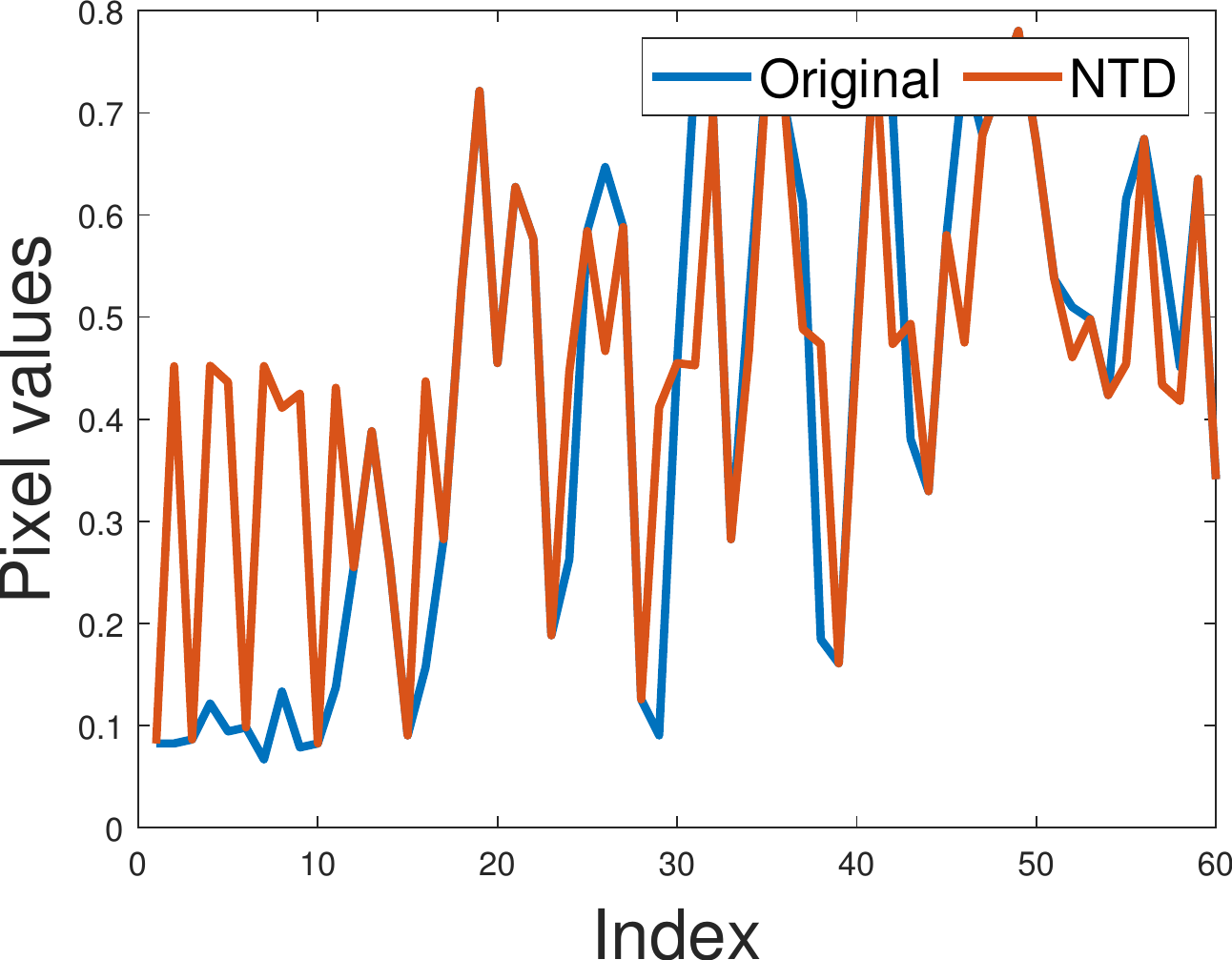}
	\caption{NTD}
\end{subfigure}					
	\end{subfigure}
	\vfill
	\caption{Examples of high altitude aerial image inpainting with $ p=50\% $. From top to bottom are respectively corresponding to ``San Francisco", ``Richmond", ``Shreveport" and ``Wash". For better visualization, we show the zoom-in region and the corresponding partial residuals of the region.}
	\label{fig:color_image_inpainting}
\end{figure*}

\begin{table*}[htbp]
	\centering
	\caption{HIGH ALTITUDE AERIAL IMAGE INPAINTING PERFORMANCE COMPARISON: PSNR, SSIM, FSIM AND RUNNING TIME}\setlength{\tabcolsep}{1.5mm}{
	\begin{tabular}{c|c|cccc|cccc|cccc}
		\hline
		\multirow{2}{*}{Picture} & \multirow{2}{*}{Methods} & \multicolumn{4}{c|}{$ p=40\% $}    & \multicolumn{4}{c|}{$ p=45\% $}   & \multicolumn{4}{c}{$ p=50\% $} \\
		\cline{3-14}          &       & PSNR  & SSIM  & FSIM  & Time  & PSNR  & SSIM  & FSIM  & Time  & PSNR  & SSIM  & FSIM  & Time \\
		\hline
		\multirow{6}{*}{San Francisco} & DTRTC & 29.637  & 0.804  & \textbf{0.982 } & \textbf{11.037 } & 30.666  & \textbf{0.838 } & 0.986  & \textbf{9.105 } & 31.637  & \textbf{0.866 } & 0.990  & \textbf{8.047 } \\
		& WSTNN & \textbf{29.917 } & \textbf{0.806 } & \textbf{0.982 } & 345.400  & \textbf{30.918 } & 0.836  & \textbf{0.988 } & 312.461  & \textbf{31.812 } & 0.858  & \textbf{0.991 } & 316.849  \\
		& TCTF  & 27.144  & 0.752  & 0.914  & 15.294  & 27.663  & 0.776  & 0.930  & 14.821  & 28.928  & 0.803  & 0.969  & 15.115  \\
		& TNN   & 28.838  & 0.775  & 0.972  & 261.645  & 29.560  & 0.804  & 0.979  & 242.107  & 30.309  & 0.831  & 0.984  & 249.143  \\
		& NCPC  & 26.165  & 0.698  & 0.895  & 39.202  & 26.779  & 0.728  & 0.915  & 36.672  & 27.384  & 0.754  & 0.933  & 38.689  \\
		& NTD   & 25.481  & 0.699  & 0.878  & 13.828  & 26.180  & 0.728  & 0.900  & 13.320  & 26.927  & 0.761  & 0.920  & 16.231  \\
		\hline
		\multirow{6}{*}{Richmond} & DTRTC & \textbf{28.671 } & 0.800  & \textbf{0.986 } & \textbf{10.324 } & 29.560  & 0.832  & \textbf{0.990 } & \textbf{8.871 } & 30.384  & 0.857  & \textbf{0.992 } & \textbf{7.756 } \\
		& WSTNN & 28.657  & \textbf{0.816 } & 0.980  & 325.396  & \textbf{29.974 } & \textbf{0.858 } & 0.989  & 314.337  & \textbf{30.985 } & \textbf{0.882 } & \textbf{0.992 } & 325.966  \\
		& TCTF  & 24.790  & 0.661  & 0.874  & 15.521  & 25.709  & 0.700  & 0.919  & 15.040  & 27.011  & 0.743  & 0.963  & 15.742  \\
		& TNN   & 27.596  & 0.750  & 0.974  & 253.573  & 28.395  & 0.786  & 0.981  & 244.574  & 29.232  & 0.818  & 0.987  & 257.293  \\
		& NCPC  & 24.298  & 0.619  & 0.870  & 44.169  & 24.908  & 0.657  & 0.894  & 41.027  & 25.430  & 0.693  & 0.915  & 39.716  \\
		& NTD   & 23.556  & 0.622  & 0.835  & 15.662  & 24.267  & 0.654  & 0.868  & 13.444  & 24.602  & 0.679  & 0.889  & 14.272  \\
		\hline
		\multirow{6}{*}{Shreveport} & DTRTC & 29.411  & 0.807  & \textbf{0.988 } & \textbf{10.938 } & 30.369  & 0.842  & \textbf{0.991 } & \textbf{9.316 } & 31.260  & 0.869  & \textbf{0.994 } & \textbf{8.024 } \\
		& WSTNN & \textbf{29.665 } & \textbf{0.828 } & 0.984  & 325.316  & \textbf{30.643 } & \textbf{0.857 } & \textbf{0.990 } & 312.816  & \textbf{31.505 } & \textbf{0.878 } & \textbf{0.993 } & 326.884  \\
		& TCTF  & 26.463  & 0.686  & 0.929  & 15.476  & 26.980  & 0.716  & 0.942  & 15.077  & 28.385  & 0.767  & 0.978  & 15.626  \\
		& TNN   & 28.245  & 0.752  & 0.976  & 251.147  & 28.989  & 0.786  & 0.983  & 243.248  & 29.730  & 0.817  & 0.987  & 263.028  \\
		& NCPC  & 25.194  & 0.628  & 0.883  & 42.986  & 25.826  & 0.669  & 0.907  & 40.977  & 26.435  & 0.704  & 0.929  & 38.202  \\
		& NTD   & 24.710  & 0.632  & 0.843  & 13.754  & 25.078  & 0.656  & 0.868  & 13.508  & 25.734  & 0.691  & 0.898  & 13.683  \\
		\hline
	\multirow{6}{*}{Wash} & DTRTC & \textbf{21.946 } & \textbf{0.692 } & \textbf{0.990 } & \textbf{54.578 } & \textbf{22.571 } & \textbf{0.729 } & \textbf{0.992 } & \textbf{46.910 } & \textbf{23.202 } & \textbf{0.762 } & \textbf{0.994 } & \textbf{31.456 } \\
& WSTNN & 13.700  & 0.372  & 0.910  & 2869.261  & 15.014  & 0.427  & 0.939  & 2685.948  & 16.492  & 0.485  & 0.961  & 2648.372  \\
& TCTF  & 19.596  & 0.542  & 0.888  & 74.882  & 19.881  & 0.575  & 0.894  & 73.590  & 20.584  & 0.623  & 0.929  & 75.920  \\
& TNN   & 21.729  & 0.644  & 0.980  & 2661.034  & 22.426  & 0.690  & 0.986  & 2542.311  & 23.150  & 0.732  & 0.990  & 2524.334  \\
& NCPC  & 19.346  & 0.527  & 0.877  & 241.108  & 19.841  & 0.573  & 0.901  & 226.601  & 20.398  & 0.615  & 0.925  & 224.523  \\
& NTD   & 18.958  & 0.521  & 0.819  & 55.018   & 19.194  & 0.551  & 0.834  & 50.740   & 19.959  & 0.600   & 0.891 &
46.979 \\
	\hline
	\end{tabular}}%
	\label{tab:image}
\end{table*}%

\par We present the image inpainting results of the four tested images in Table \ref{tab:image}, and the best results are highlighted in bold. For visual comparisons, we show the images of the recovered high altitude aerial images by different methods for $ p=50\% $ in Figure \ref{fig:color_image_inpainting}. The proposed DTRTC algorithm can be seen to achieve the best performance.  The four methods based on tubal rank DTRTC, WSTNN, TCTF, and TNN perform better on PSNR, SSIM, and FSIM values than the method based on CP rank, NCPC, and the method based on Tucker rank, NTD except for the ``Wash" image.
Furthermore, TCTF and TNN do not use all low rank structures of tensors \cite{ZHZ20}, DTRTC and WSTNN are more comprehensive to preserve all low rank structures of tensor data. However, it can be seen from Lemma \ref{lem:wstnn} that WSTNN over-utilizes the low rank information of the tensor, resulting in too long running time and little improvement in PSNR, SSIM and FSIM values.
For the large scale ``Wash" image, the recovery of WSTNN, TCTF, NCPC and NTD is unsatisfactory, but DTRTC and TNN can successfully recover the image. However, since TNN and WSTNN require T-SVD decomposition at each step, as the tensor size increases, its calculation time increases significantly. As a result, DTRTC both produces excellent inpainting results and runs extremely fast.

\subsection{Video Inpainting}
We evaluate our proposed method DTRTC on the widely used YUV Video Sequences\footnote{\url{http://trace.eas.asu.edu/yuv/.}}. There are at least
150 frames in each video sequences. We pick the first $ 120 $ frames from them.
In the experiments, we test our proposed method and other methods on ``Bridge" video with $ 288 \times 352 $ pixels. 
We test the video with random missing data of the sampling ratio $ p=20\%, 25\%, 30\% $.  
Set the initial double tubal rank $ \bm{r}_\X^0=\left(120,70,\ldots,70\right),\,\bm{r}_{\tilde\X}^0=\left(10,10,\ldots,10\right) \, (\Y\in\bR^{n_3\times (n_2n_1/3) \times 3})$ in DTRTC, the initial tubal rank $ \left(50,8,\ldots,8\right) $ in TCTF, the initial CP rank $ 50 $ in NCPC and the initial Tucker rank $ \left(30,30,5\right) $ in NTD. In experiments, the maximum iterative number is set to be $300$ and precision $ \varepsilon $ is set to be 1e-4.

\begin{figure*}[htbp]
	\centering
	\begin{subfigure}[b]{1\linewidth}
		\begin{subfigure}[b]{0.118\linewidth}
			\centering
			\includegraphics[width=\linewidth]{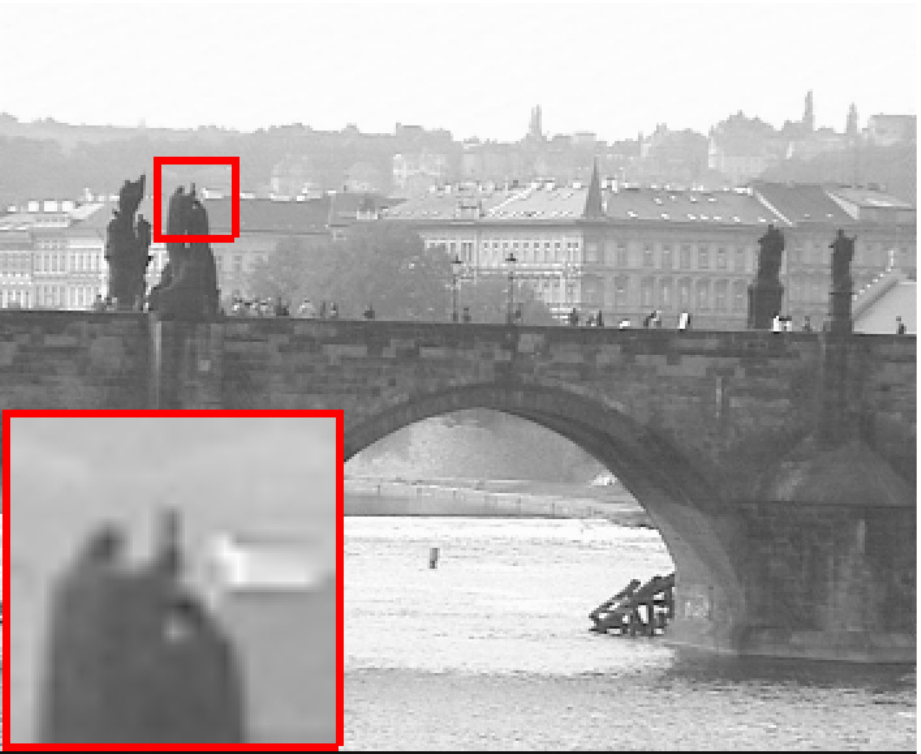}\vspace{0pt}
			\includegraphics[width=\linewidth]{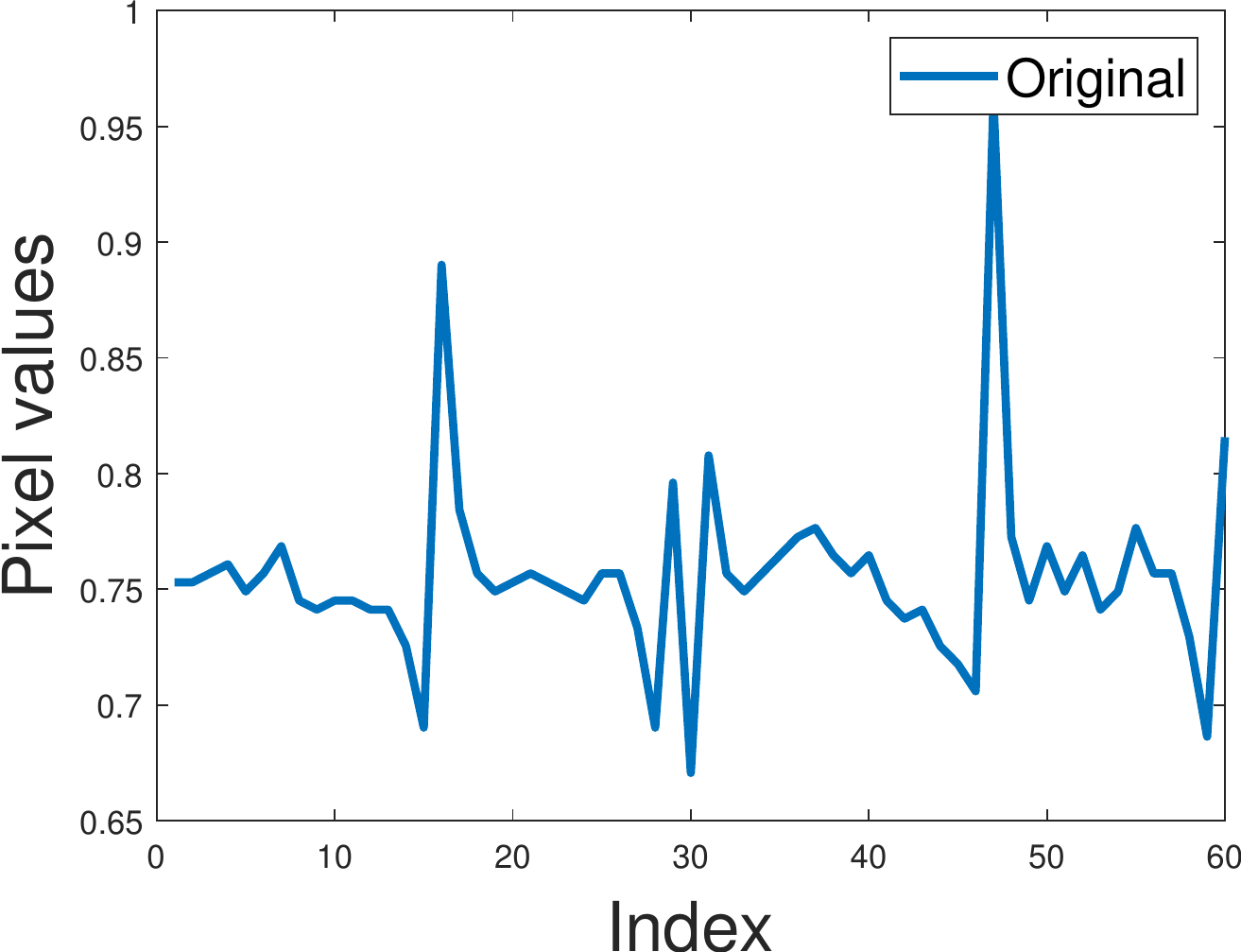}
			\caption{Original}
		\end{subfigure}   	
		\begin{subfigure}[b]{0.1175\linewidth}
			\centering
			\includegraphics[width=\linewidth]{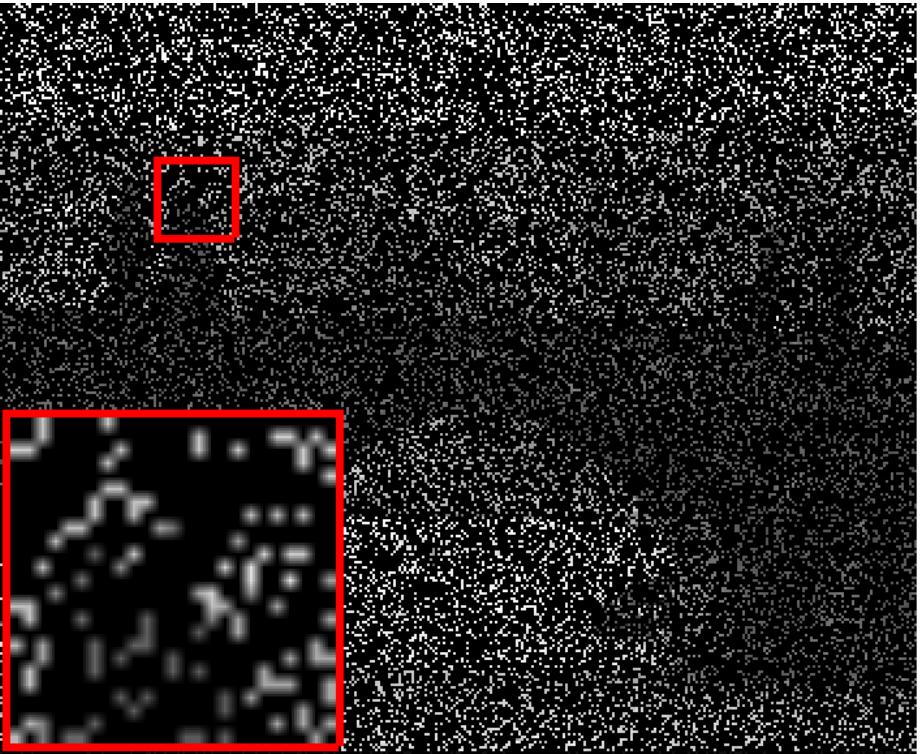}\vspace{0pt}
			\includegraphics[width=\linewidth]{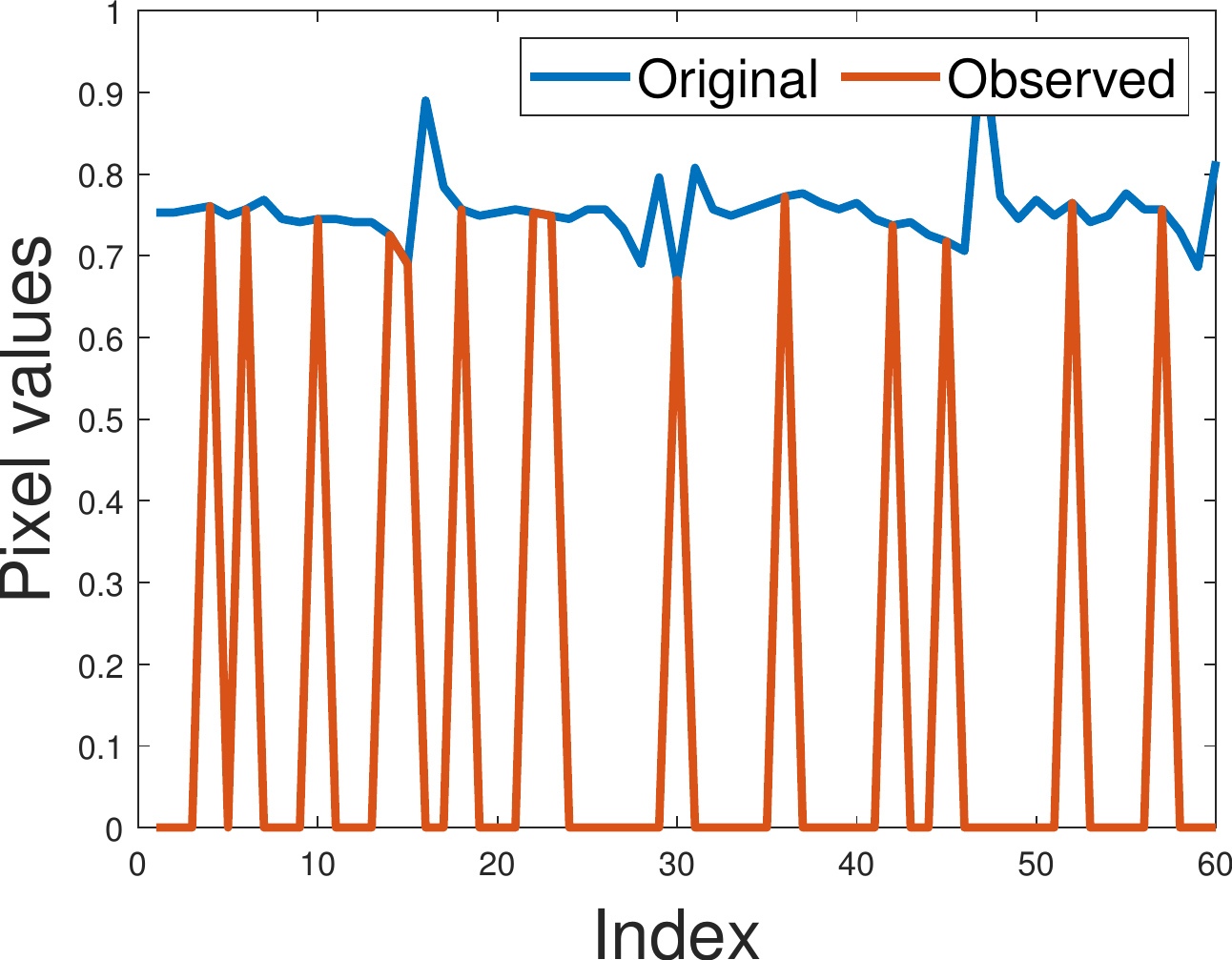}
			\caption{Observed}
		\end{subfigure}
		\begin{subfigure}[b]{0.118\linewidth}
			\centering
			\includegraphics[width=\linewidth]{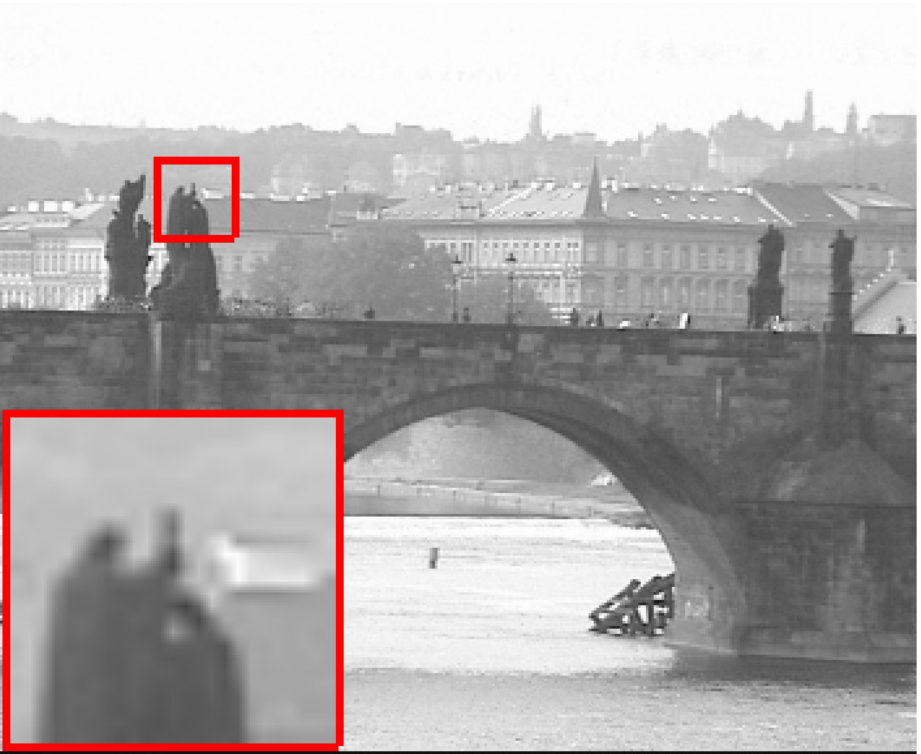}\vspace{0pt}
			\includegraphics[width=\linewidth]{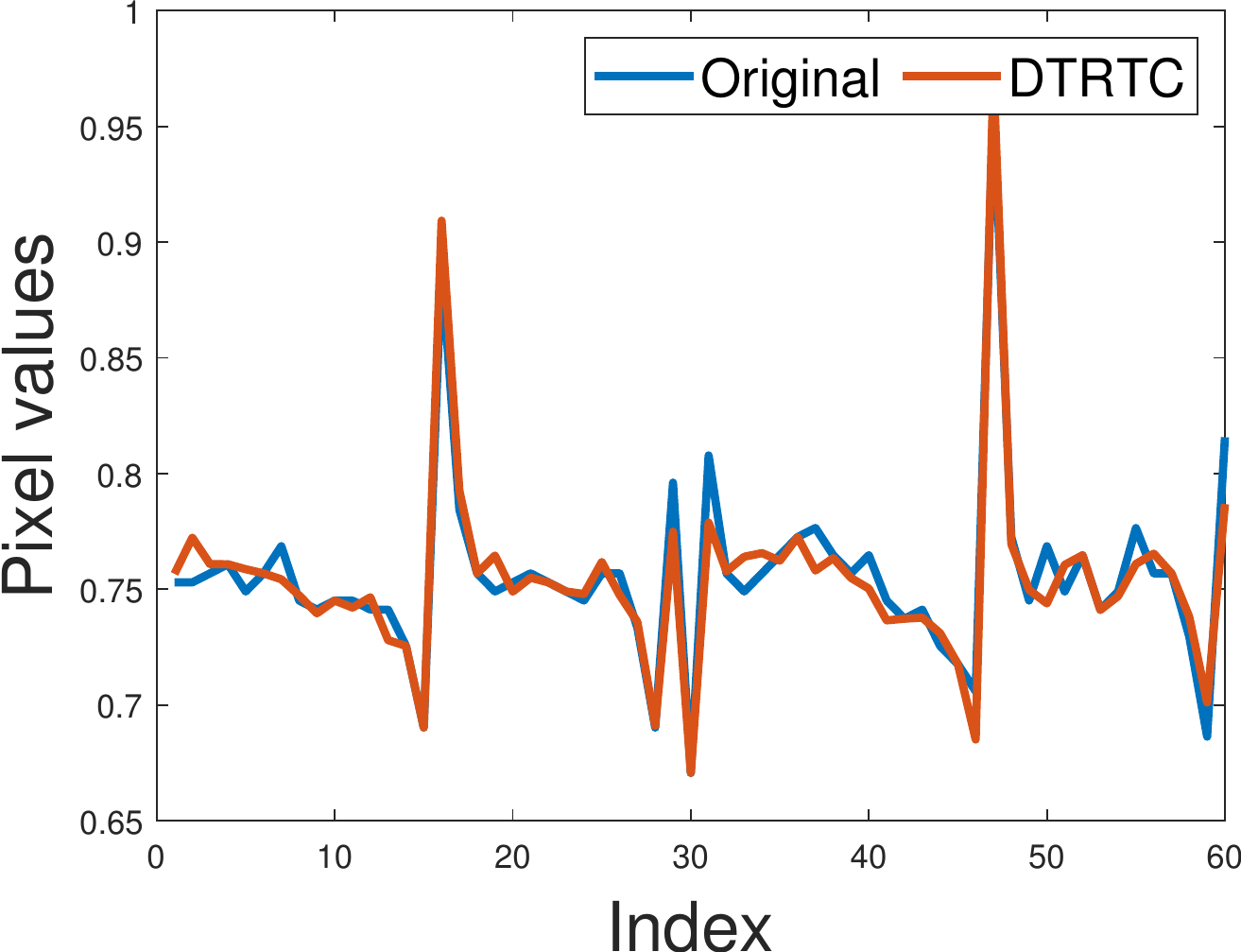}
			\caption{DTRTC}
		\end{subfigure}
		\begin{subfigure}[b]{0.118\linewidth}
			\centering
			\includegraphics[width=\linewidth]{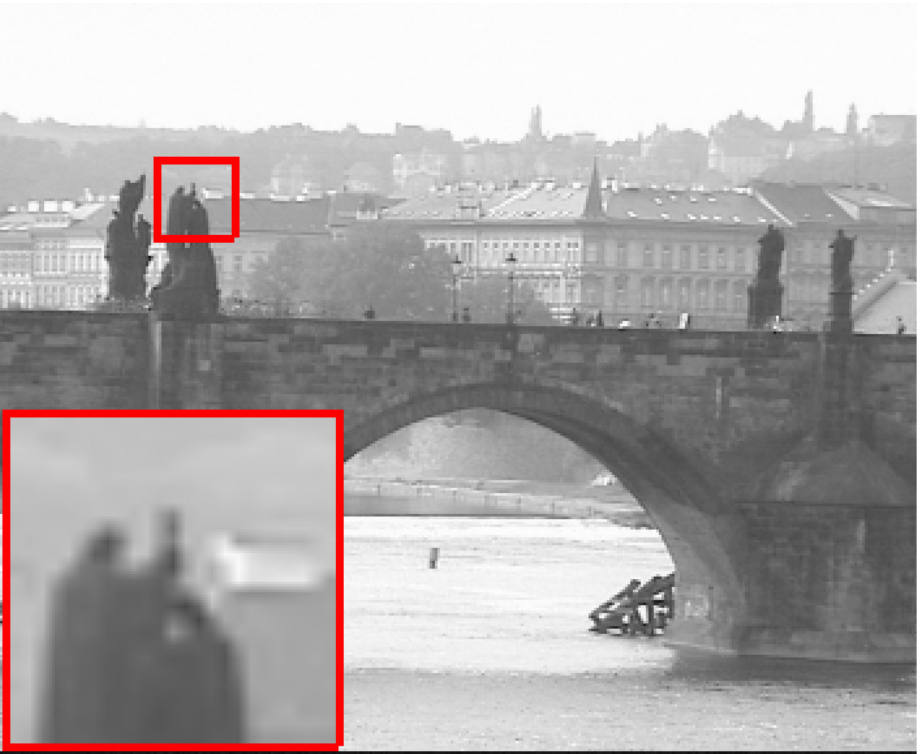}\vspace{0pt}
			\includegraphics[width=\linewidth]{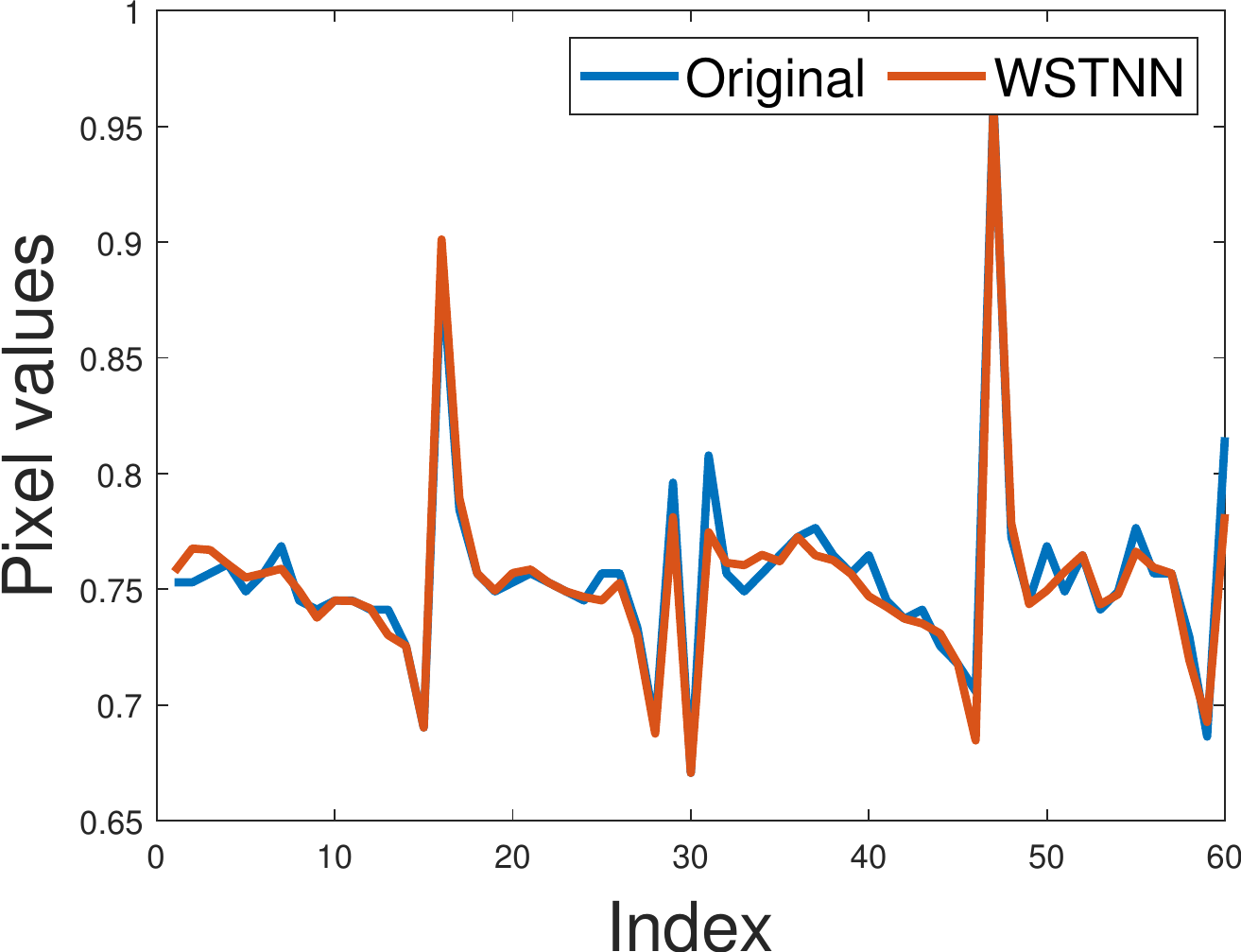}
			\caption{WSTNN}
		\end{subfigure}	
		\begin{subfigure}[b]{0.118\linewidth}
			\centering			
			\includegraphics[width=\linewidth]{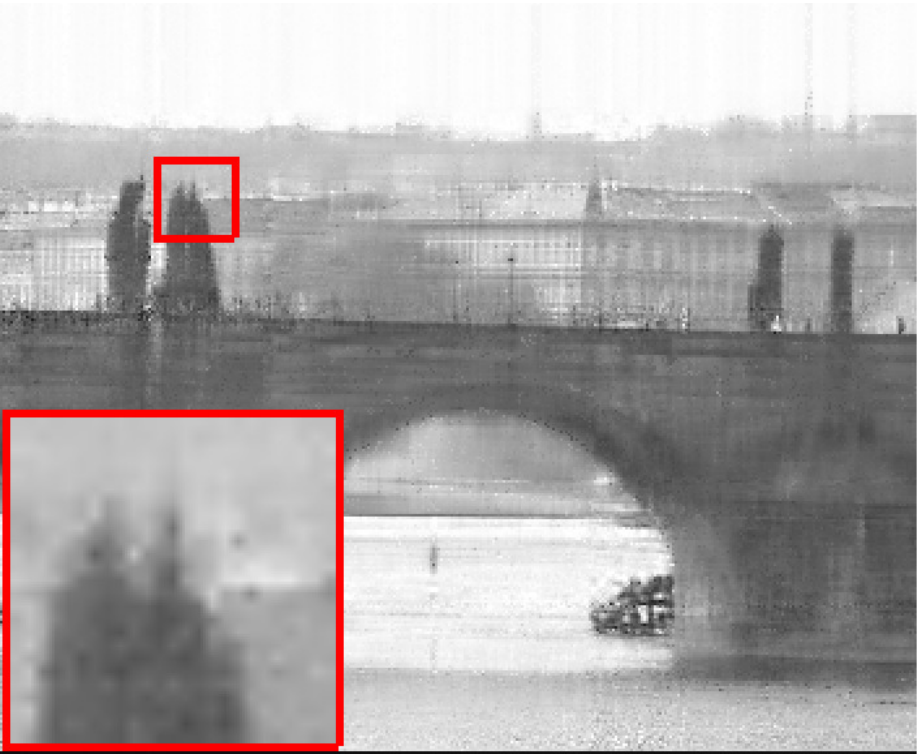}\vspace{0pt}
			\includegraphics[width=\linewidth]{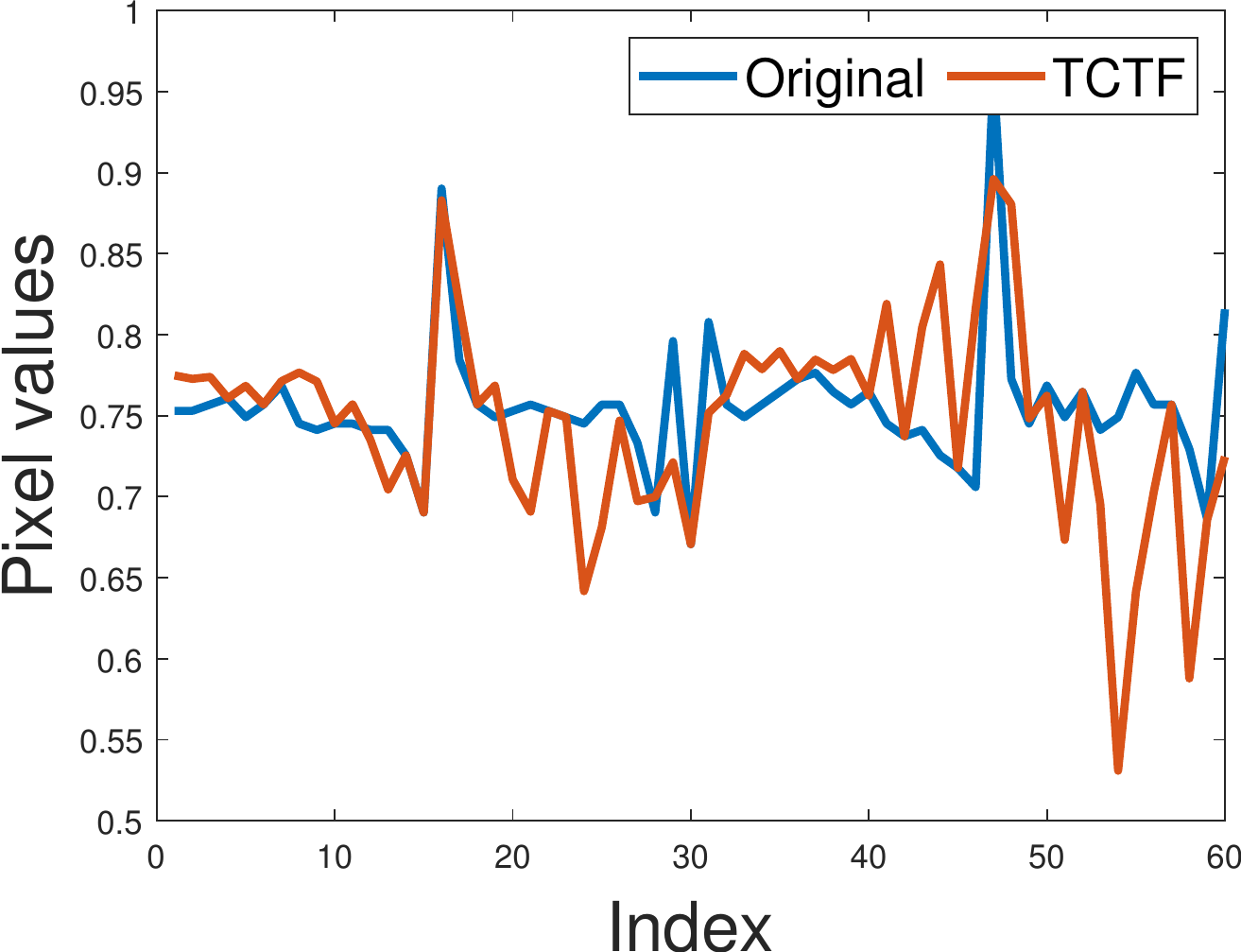}
			\caption{TCTF}
		\end{subfigure}
		\begin{subfigure}[b]{0.118\linewidth}
			\centering		
			\includegraphics[width=\linewidth]{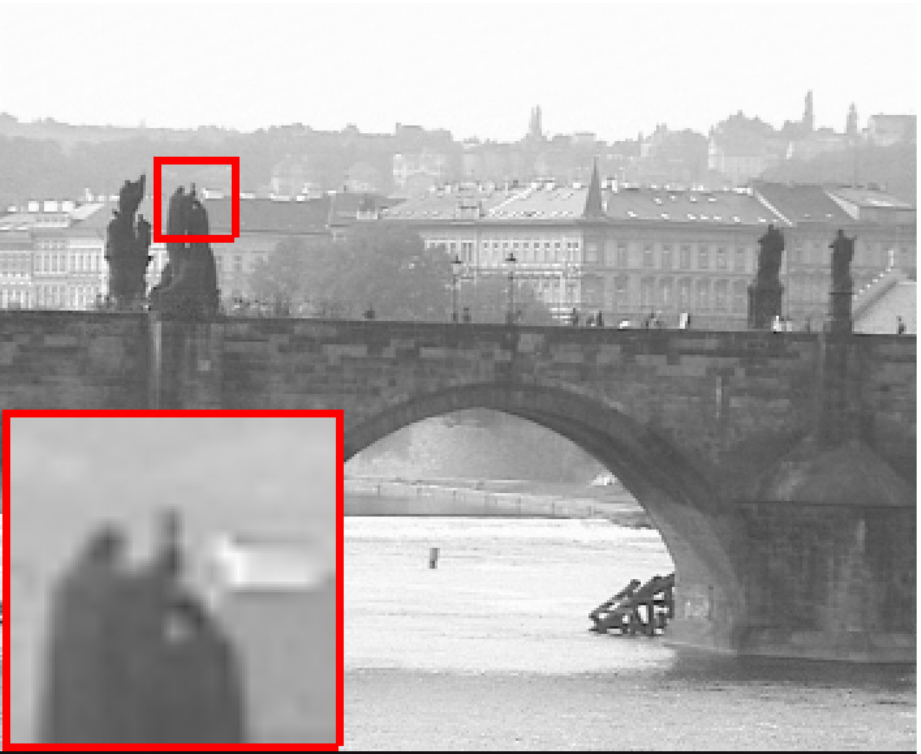}\vspace{0pt}
			\includegraphics[width=\linewidth]{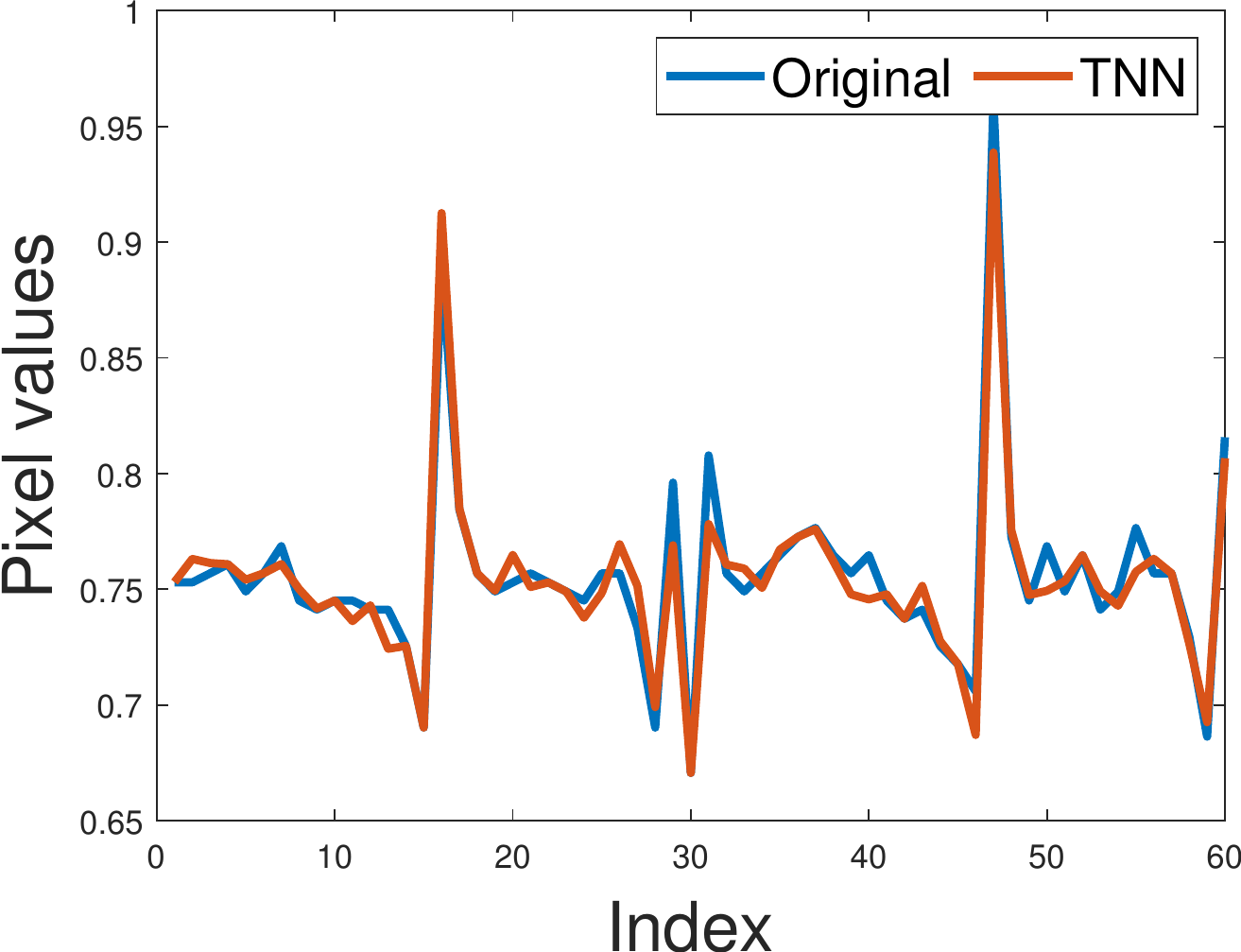}
			\caption{TNN}
		\end{subfigure}
		\begin{subfigure}[b]{0.118\linewidth}
			\centering		
			\includegraphics[width=\linewidth]{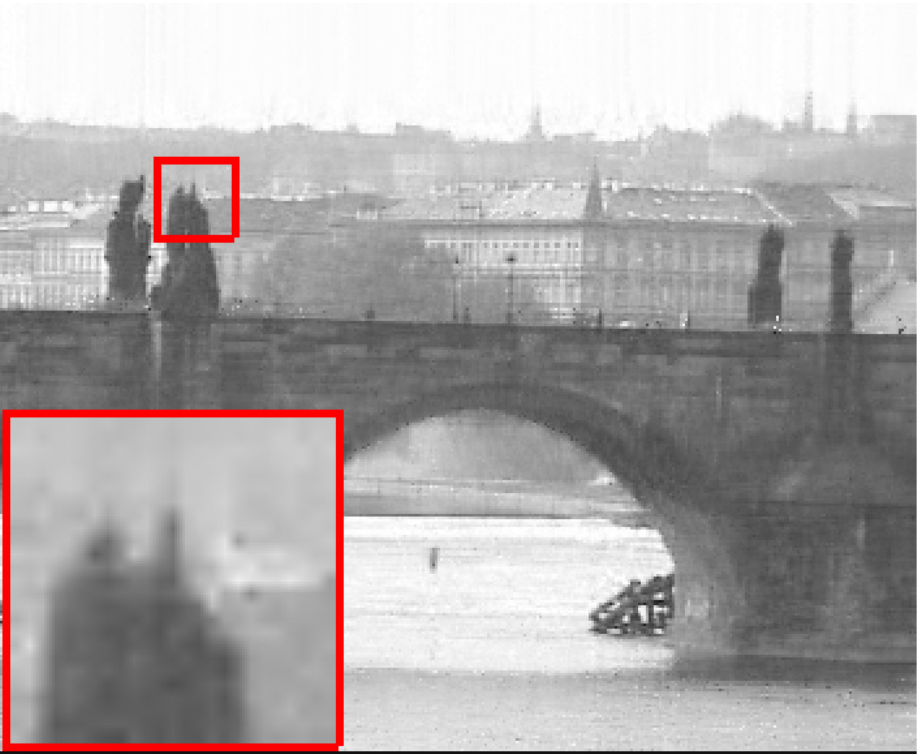}\vspace{0pt}
			\includegraphics[width=\linewidth]{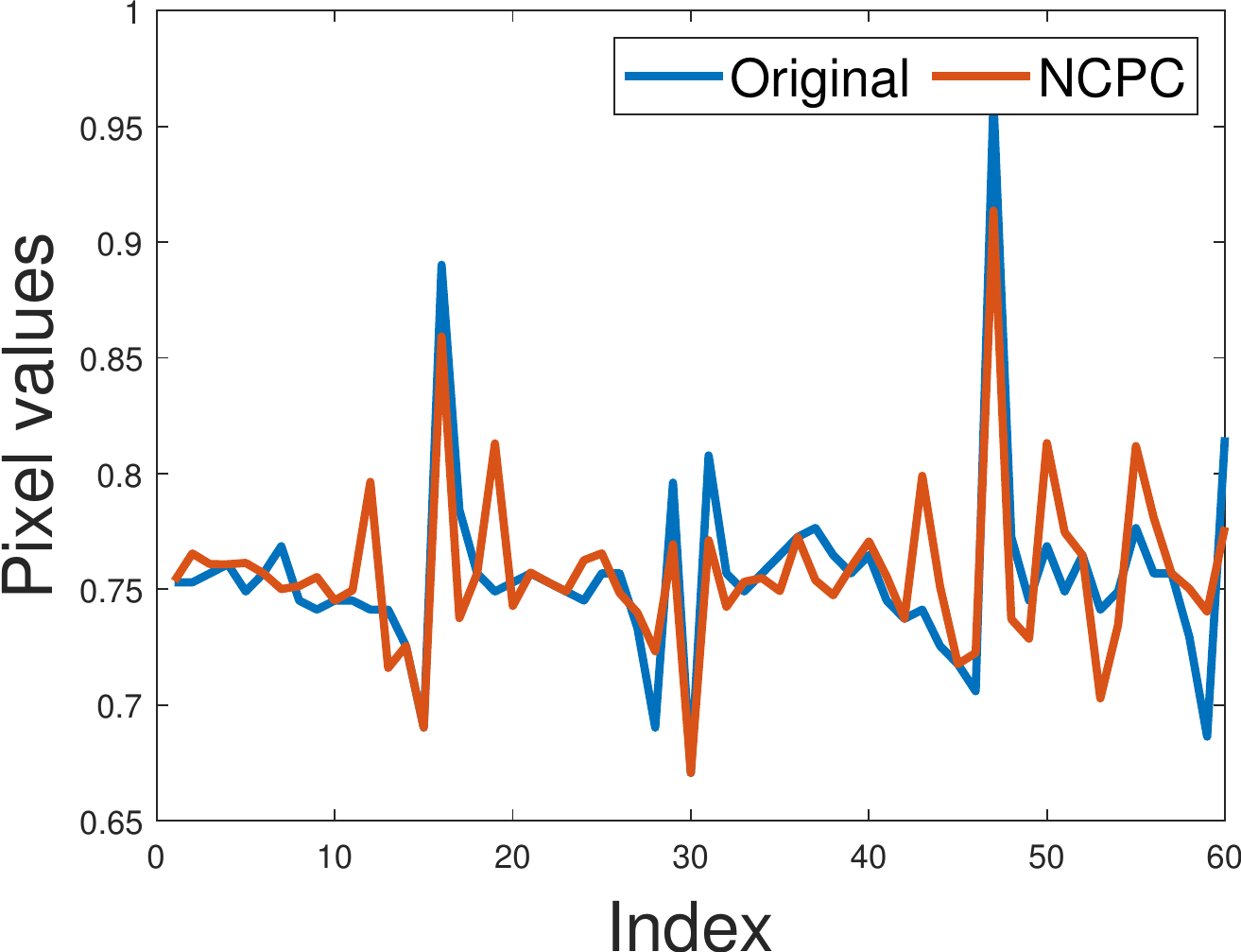}
			\caption{NCPC}
		\end{subfigure}
		\begin{subfigure}[b]{0.118\linewidth}
			\centering			
			\includegraphics[width=\linewidth]{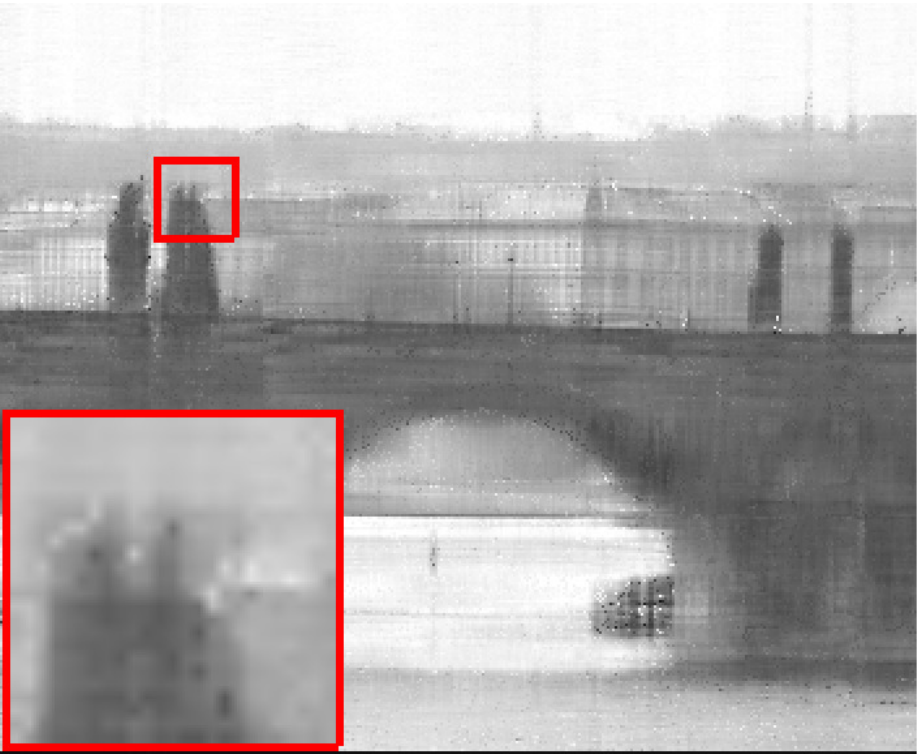}\vspace{0pt}
			\includegraphics[width=\linewidth]{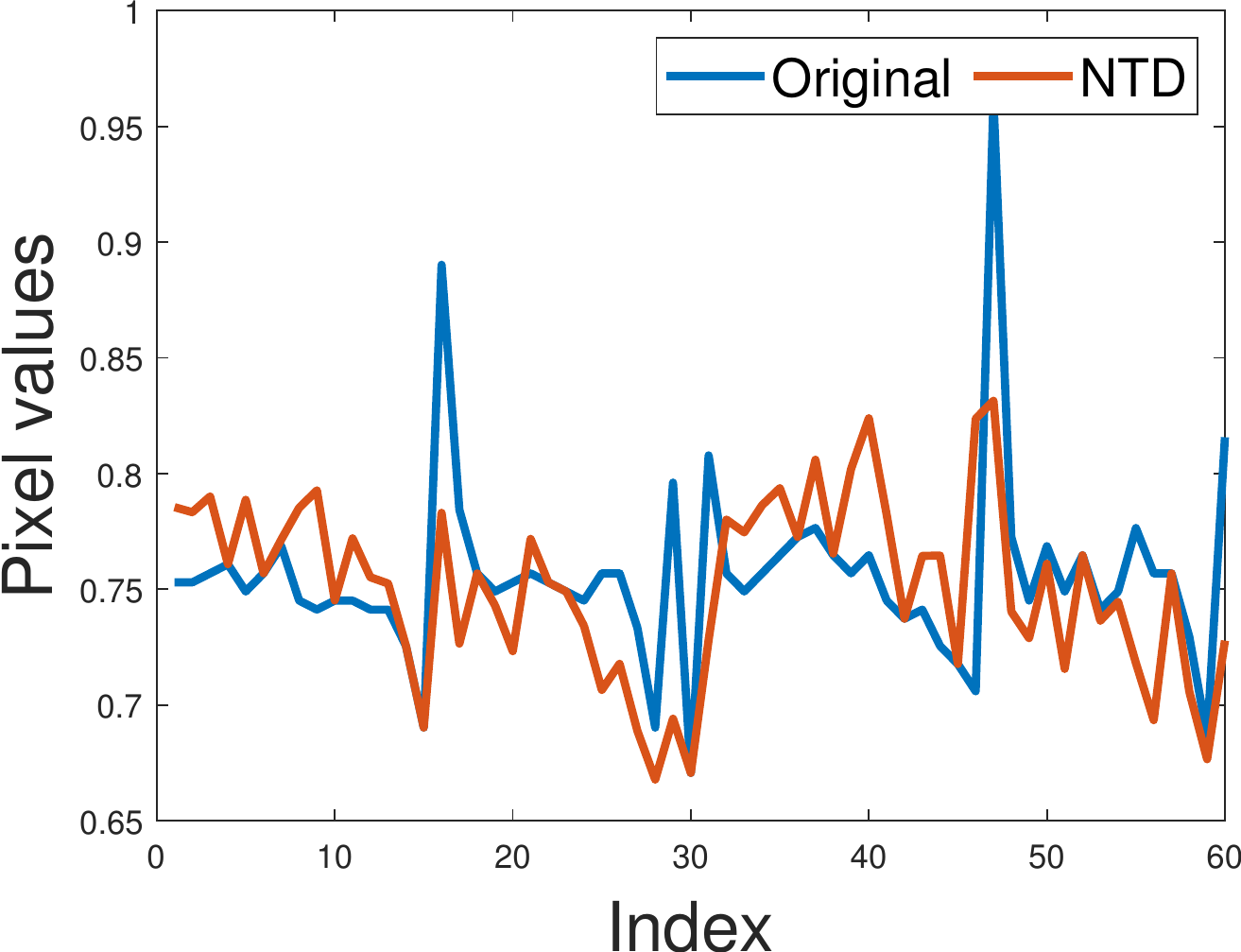}
			\caption{NTD}
		\end{subfigure}					
	\end{subfigure}
	\vfill
	\caption{Examples of video ``Bridge" inpainting with $ p=20\% $. For better visualization, we show the zoom-in region and the corresponding partial residuals of the region.}
    \label{fig:test-video}
\end{figure*}

\begin{table*}[htbp]
	\centering
	\caption{VIDEO INPAINTING PERFORMANCE COMPARISON: PSNR, SSIM, FSIM AND RUNNING TIME}
	\setlength{\tabcolsep}{1.5mm}{
	\begin{tabular}{c|c|cccc|cccc|cccc}
		\hline
		\multirow{2}{*}{Video} & \multirow{2}{*}{Methods} & \multicolumn{4}{c|}{$ p=20\% $}    & \multicolumn{4}{c|}{$ p=25\% $}   & \multicolumn{4}{c}{$ p=30\% $} \\
		\cline{3-14}          &       & PSNR  & SSIM  & FSIM  & Time  & PSNR  & SSIM  & FSIM  & Time  & PSNR  & SSIM  & FSIM  & Time \\
		\hline
		\multirow{6}{*}{Bridge} & DTRTC & 33.289  & 0.931  & 0.968  & \textbf{44.648 } & 33.752  & 0.937  & 0.971  & \textbf{30.767 } & 34.317  & 0.945  & 0.974  & \textbf{25.980 } \\
		& WSTNN & \textbf{33.840 } & \textbf{0.943 } & \textbf{0.972 } & 404.738  & \textbf{34.544 } & \textbf{0.951 } & \textbf{0.976 } & 309.719  & \textbf{35.228 } & \textbf{0.957 } & \textbf{0.980 } & 291.561  \\
		& TCTF  & 27.292  & 0.758  & 0.877  & 120.391  & 28.156  & 0.792  & 0.889  & 111.535  & 22.725  & 0.625  & 0.808  & 113.067  \\
		& TNN   & 33.696  & 0.932  & 0.967  & 1710.926  & 34.414  & 0.941  & 0.971  & 1485.431  & 35.110  & 0.949  & 0.976  & 1439.404  \\
		& NCPC  & 29.722  & 0.836  & 0.914  & 70.474  & 30.188  & 0.853  & 0.924  & 59.200  & 30.598  & 0.866  & 0.932  & 62.777  \\
		& NTD   & 26.814  & 0.736  & 0.858  & 70.616  & 27.478  & 0.766  & 0.876  & 64.893  & 27.963  & 0.786  & 0.889  & 69.222  \\
		\hline
	\end{tabular}}
	\label{tab:test-video}
\end{table*}

\par  Figure \ref{fig:test-video} shows the 18-th frame of ``Bridge", which shows that DTRTC performs better in filling the missing values of the tested sequence and recovers the details better. On the PSNR, SSIM and FSIM metric, DTRTC achieves similar effects to WSTNN and TNN, consistent with the observation in Table \ref{tab:test-video}. On time consumption, DTRTC is the fastest method, about $ 9 $ times faster than WSTNN and at least $ 38 $ times faster than TNN. Clearly, the video inpainting results are also consistent with the image inpainting results, and all these demonstrate that DTRTC can perform tensor completion better with less consumed time.

\section{Conclusion}
In this paper, we established a relationship between matrix rank and tensor tubal rank. After that, we modeled the matrix completion problem as a third order tensor completion problem and proposed a two-stage  tensor factorization based  algorithm, which made a drastic reduction on the dimension of data and hence cut down on the running time. For low rank tensor completion problem, we introduced double tubal rank. Compared to tubal rank, 3-tubal rank and tensor fibered rank, double tubal rank can not only fully exploit the low rank structures of the tensor but also avoid the low rank structures redundancy. Based on this rank, we modified the proposed  tensor factorization based algorithm for tensor completion problem.  The reported experiments demonstrated that our proposed algorithms were much more efficient than the most state-of-the-art matrix/tensor completion algorithms.

\appendices
\section{Proof of Theorem}
\begin{proof}
	According to $ f^t $, we have that
	\begin{equation}\label{fs}
		\begin{aligned}
			&f^t-f^{t+1}\\=&\frac{1}{2}\left(\left\|\P^{t}\ast \Q^{t}-\X^{t}\right\|_F^2-\left\|\P^{t+1}\ast \Q^{t+1}-\X^{t+1}\right\|_F^2\right) 	\\
			&+\frac{\gamma}{2}\left(\left\|\U^{t}\ast \V^{t}-\tilde\X^{t}\right\|_F^2-\left\|\U^{t+1}\ast \V^{t+1}-\tilde\X^{t+1}\right\|_F^2\right). 		
		\end{aligned}
	\end{equation}
	In step 9 of Algorithm 4.1, since $ \X^{t+1} $ is an optimal solution of $\X$-subproblem, we have
	\begin{equation*}
		\begin{aligned}
			&f\left(\P^{t+1},\Q^{t+1},\U^{t+1},\V^{t+1},\X^{t+1}\right)\\\le &f\left(\P^{t+1},\Q^{t+1},\U^{t+1},\V^{t+1},\X^{t}\right).
		\end{aligned}
	\end{equation*}
	Then,
	\begin{equation}\label{xt}
		\begin{aligned}
			&\left\|\P^{t+1}\ast\Q^{t+1}-\X^{t+1}\right\|_F^2+\gamma\left\|\U^{t+1}\ast\V^{t+1}-\tilde\X^{t+1}\right\|_F^2
			\\\le&\left\|\P^{t+1}\ast\Q^{t+1}-\X^{t}\right\|_F^2+\gamma\left\|\U^{t+1}\ast\V^{t+1}-\tilde\X^{t}\right\|_F^2.	
		\end{aligned}
	\end{equation}
	According to the computation of $\P^{t+1}, \Q^{t+1}$ and Lemma 3 in \cite{ZLLZ18}, we have
	\begin{equation}\label{xt1}
		\begin{aligned}
			&\left\|\P^{t}\ast \Q^{t}-\X^{t}\right\|_F^2-\left\|\P^{t+1}\ast \Q^{t+1}-\X^{t+1}\right\|_F^2	\\
			=&\left\|\P^{t+1}\ast\Q^{t+1}-\X^{t}\right\|_F^2-\left\|\P^{t+1}\ast\Q^{t+1}-\X^{t+1}\right\|_F^2\\&+\frac{1}{n_3}\left\|\hat{P}^{t+1} \hat{Q}^{t+1}-\hat{P}^{t} \hat{Q}^{t}\right\|_F^2.	
		\end{aligned}
	\end{equation}
	Similar result can be obtained that
	\begin{equation}\label{yt}
		\begin{aligned}
			&\left\|\U^{t}\ast \V^{t}-\tilde\X^{t}\right\|_F^2-\left\|\U^{t+1}\ast \V^{t+1}-\tilde\X^{t+1}\right\|_F^2	\\
			=&\left\|\U^{t+1}\ast \V^{t+1}-\tilde\X^{t}\right\|_F^2-\left\|\U^{t+1}\ast \V^{t+1}-\tilde\X^{t+1}\right\|_F^2\\&+\frac{1}{q}\left\|\hat{U}^{t+1} \hat{V}^{t+1}-\hat{U}^{t} \hat{V}^{t}\right\|_F^2.	
		\end{aligned}
	\end{equation}
	Combining \eqref{fs}-\eqref{yt}, it holds
	\begin{equation}\label{ft}
		\begin{aligned}
			f^t-f^{t+1} \ge& \frac{1}{2n_3}\left\|\hat{P}^{t+1} \hat{Q}^{t+1}-\hat{P}^{t} \hat{Q}^{t}\right\|_F^2\\&+\frac{\gamma}{2q}\left\|\hat{U}^{t+1} \hat{V}^{t+1}-\hat{U}^{t} \hat{V}^{t}\right\|_F^2\ge 0.	
		\end{aligned}
	\end{equation}
	Summing all the inequality \eqref{ft} for all $ t $, we obtain
	\begin{equation}
		\begin{aligned}
			f^1-f^{n+1} \ge&
			\frac{1}{2n_3}\sum_{t=1}^n\left\|\hat{P}^{t+1} \hat{Q}^{t+1}-\hat{P}^{t} \hat{Q}^{t}\right\|_F^2\\&+\frac{\gamma}{2q}\sum_{t=1}^n\left\|\hat{U}^{t+1} \hat{V}^{t+1}-\hat{U}^{t} \hat{V}^{t}\right\|_F^2.
		\end{aligned}
	\end{equation}
	Thus, we can obtain the following equation:
	\begin{equation}\label{infp}
		\begin{aligned}
			&\lim\limits_{t\to+\infty}\left\|\hat{P}^{t+1} \hat{Q}^{t+1}-\hat{P}^{t} \hat{Q}^{t}\right\|_F^2=0,\\ &\lim\limits_{t\to+\infty}\left\|\hat{U}^{t+1} \hat{V}^{t+1}-\hat{U}^{t} \hat{V}^{t}\right\|_F^2=0.
		\end{aligned}
	\end{equation}
	Similar to the analysis of Equation (38)-(46) in \cite{ZLLZ18}, ones have
	\begin{equation*}
		\begin{aligned}
			&\lim _{t \to+\infty}\left(\bar{X}^{t}-\hat{P}^{t} \hat{Q}^{t}\right)\left(\hat{Q}^{t}\right)^{*}=0,\\
			&\lim _{t \to+\infty}\left(\hat{P}^{t+1}\right)^{*}\left(\bar{X}^{t}-\hat{P}^{t} \hat{Q}^{t}\right)=0.		
		\end{aligned}
	\end{equation*}
	Since the sequence $\left\{\P^t,\Q^t,\U^t,\V^t,\X^t\right\}$ generated by Algorithm 4.1 is bounded, there is a subsequence $\left\{\P^{t_j},\Q^{t_j},\U^{t_j},\V^{t_j},\X^{t_j}\right\}$ that
	converges to a point $\left(\P_{\star},\Q_{\star},\U_{\star},\V_{\star},\X_{\star}\right)$. Therefore, the following two equations hold:
	\begin{equation}\label{pq}
		\begin{aligned}
			&\left(\bar{X}_{\star}-\hat{P}_{\star} \hat{Q}_{\star}\right)\left(\hat{Q}_{\star}\right)^{*}=0,\quad
			\left(\hat{P}_{\star}\right)^{*}\left(\bar{X}_{\star}-\hat{P}_{\star} \hat{Q}_{\star}\right)=0.		
		\end{aligned}
	\end{equation}
	Similarly, we have
	\begin{equation}\label{uv}
		\left(\bar{\tilde X}_{\star}-\hat{U}_{\star} \hat{V}_{\star}\right)\left(\hat{V}_{\star}\right)^{*}=0,\quad
		\left(\hat{U}_{\star}\right)^{*}\left(\bar{\tilde X}_{\star}-\hat{U}_{\star} \hat{V}_{\star}\right)=0.		
	\end{equation}
	On the other hand, we update $ \X^{t+1}=\frac{1}{1+\gamma}P_{\Omega^c}\left(\P^{t+1} \ast\Q^{t+1}+\gamma fold_3\left[\left(\U^{t+1} \ast\V^{t+1}\right)_{(1)}\right]\right)+{P_\Omega }(\mathcal M) $ at each iteration. Thus, $ \X_\star $ always satisfies the following two equations
	\begin{equation}\label{xy}
		\begin{aligned}
			&P_{\Omega^{c}}\left(\mathcal{X}_{\star}-\frac{1}{1+\gamma}\left(\P_{\star} \ast\Q_{\star}+\gamma fold_3\left[\left(\U_{\star} \ast\V_{\star}\right)_{(1)}\right]\right)\right)=0, \\
			&P_{\Omega}\left(\mathcal{X}_{\star}-\mathcal{M}\right)=0.
		\end{aligned}
	\end{equation}
	Furthermore, there exists $\Lambda_{\star}$ such that
	\begin{equation}\label{xys}
		\begin{aligned}
		P_{\Omega}\left(\mathcal{X}_{\star}-\frac{1}{1+\gamma}\left(\P_{\star} \ast\Q_{\star}+\gamma fold_3\left[\left(\U_{\star} \ast\V_{\star}\right)_{(1)}\right]\right)\right)\\+\Lambda_{\star}=0.	
		\end{aligned}
	\end{equation}
	By \eqref{pq}-\eqref{xys},  $ \left(\P_\star,\Q_\star,\U_\star,\V_\star,\X_\star\right) $ is a KKT point of problem \eqref{lm}.
\end{proof}


\ifCLASSOPTIONcaptionsoff
  \newpage
\fi



%
\bibliographystyle{IEEEtran}
\bibliography{my}

%




\end{document}